\documentclass{amsart}
\usepackage[utf8]{inputenc}
\usepackage{amsmath, amssymb, amsfonts, amsthm}
\usepackage{enumerate}
\usepackage{bbold}
\usepackage{color}
\usepackage{graphicx}
\usepackage{fullpage}
\usepackage{tikz}
\usepackage{blkarray, bigstrut} %
\usepackage{tikz-cd}
\usepackage{hyperref}
\usepackage{lipsum}

\usepackage{bbm}

\allowdisplaybreaks

\DeclareMathOperator{\C}{\mathbb{C}}

\DeclareMathOperator{\Cc}{\mathcal{C}}

\DeclareMathOperator{\Vect}{\textrm{Vec}}
\DeclareMathOperator{\End}{\textrm{End}}

\DeclareMathOperator{\id}{\textrm{id}}

\DeclareMathOperator{\Hom}{\textrm{Hom}}

\DeclareMathOperator{\ev}{\operatorname{ev}}
\DeclareMathOperator{\coev}{\operatorname{coev}}

\newcommand{\q}[1]{[#1]_q}

\usepackage{array}   % for \newcolumntype macro
\newcolumntype{L}{>{$}l<{$}}

\newcommand{\Pp}{\mathcal{P}}
\newcommand\scalemath[2]{\scalebox{#1}{\mbox{\ensuremath{\displaystyle #2}}}}

\newcommand{\nothing}[1]{#1}

\newcommand{\cC}{\mathcal{C}}

\theoremstyle{plain}
\newtheorem{thm}{Theorem}[section]
\newtheorem{lem}[thm]{Lemma}

\theoremstyle{definition}
\newtheorem{defn}[thm]{Definition}

\newtheorem{ex}[thm]{Example}
\newtheorem{remark}[thm]{Remark}

\title{Cell Systems for $\overline{\operatorname{Rep}(U_q(\mathfrak{sl}_N))}$ Module Categories}
\author{Daniel Copeland and Cain Edie-Michell}
\address{Cain Edie-Michell\\
University of New Hampshire\\
Durham, 
New Hampshire}
\email{cain.edie-michell@unh.edu}
\address{Daniel Copeland}
\email{daniel.copeland@gmail.com}
\date{}

\begin{document}

\maketitle
\begin{abstract}
In this paper, we define the \textit{KW cell system} on a graph $\Gamma$, depending on parameters $N\in \mathbb{N}$, $q$ a root of unity, and $\omega$ an $N$-th root of unity. This is a polynomial system of equations depending on $\Gamma$ and the parameters. Using the graph planar algebra embedding theorem, we prove that when $q = e^{2\pi i \frac{1}{2(N+k)}}$, solutions to the KW cell system on $\Gamma$ classify module categories over $\overline{\operatorname{Rep}(U_q(\mathfrak{sl}_N))^\omega}$ whose action graph for the object $\Lambda_1$ is $\Gamma$. The KW cell system is a generalisation of the Etingof-Ostrik and the De Commer-Yamashita classifying data for $\overline{\operatorname{Rep}(U_q(\mathfrak{sl}_2))}$ module categories, and Ocneanu's cell calculus for $\overline{\operatorname{Rep}(U_q(\mathfrak{sl}_3))}$ module categories.

To demonstrate the effectiveness of this cell calculus, we solve the KW cell systems corresponding to the exceptional module categories over $\overline{\operatorname{Rep}(U_q(\mathfrak{sl}_4))}$ when $q= e^{2\pi i \frac{1}{2(4+k)}}$, as well as for all three infinite families of charge conjugation modules. Building on the work of the second author, this explicitly constructs and classifies all irreducible module categories over $\mathcal{C}(\mathfrak{sl}_4, k)$ for all $k\in \mathbb{N}$. These results prove claims made by Ocneanu on the \textit{quantum subgroups} of $SU(4)$. We also construct exceptional module categories over $\overline{\operatorname{Rep}(U_q(\mathfrak{sl}_4))^\omega}$ where $\omega\in \{-1, \mathbf{i}, -\mathbf{i}\}$. Two of these module categories have no analogue when $\omega=1$.

The main technical contributions of this paper are a proof of the graph planar algebra embedding theorem for oriented planar algebras, and a refinement of Kazhdan and Wenzl's skein theory presentation of the category $\overline{\operatorname{Rep}(U_q(\mathfrak{sl}_N))^\omega}$. We also explicitly describe the subfactors coming from a solution to a KW cell system.
\end{abstract}
\section{Introduction}
One of the largest (and most interesting) classes of tensor categories comes from the representation theory of the quantum groups $U_q(\mathfrak{g})$ at roots of unity $q$. Namely, one looks at the category of tilting modules of these objects, and takes an appropriate quotient. Equivalently, these categories can also be described as the category of level-$k$ integrable representations of $\hat{\mathfrak{g}}$, with non-standard tensor product given by the level-preserving fusion \cite{MR1384612}. These categories are typically denoted by either $\overline{\operatorname{Rep}(U_q(\mathfrak{sl}_N))}$ or $\mathcal{C}(\mathfrak{g}, k)$, depending on the context. There are many appearances of these categories in various areas of mathematics \cite{Wasser}, as well as physics \cite{ADE}. Notably, these categories are the representation theory of the Wess-Zumino-Witten chiral conformal field theories $\mathcal{V}(\mathfrak{g}, k)$.

A module category over a tensor category $\mathcal{C}$ is a natural categorification of a module over a ring or group \cite{OstMod}. More specifically, it is a monoidal functor
\[   \mathcal{C} \to \operatorname{End}(\mathcal{M})   \]
where $\mathcal{M}$ is some abelian category. The module categories over $\mathcal{C}$ have various applications. In particular, when $\mathcal{C}$ is the representation theory of a chiral conformal field theory $\mathcal{V}$, the module categories over $\mathcal{C}$ classify full conformal field theories with a chiral half $\mathcal{V}$ \cite{Full}.

In the last several years, there has been a revitalisation in the program to classify module categories over the quantum group categories $\mathcal{C}(\mathfrak{g}, k)$ (building on older works e.g. \cite{ADE, Ocneanu}). This is mainly due to works of Schopieray \cite{LevelAndy}, and Gannon \cite{LevelTerry}. In particular, the latter work classifies and constructs all of the \textit{type I} module categories\footnote{These are module categories which have the additional compatible structure of a tensor category.} for the simple Lie algebras of rank $\leq 6$. 

Recent work of the second author and Gannon extended the results of Gannon to classify all module categories for the Lie algebras $\mathfrak{sl}_N$ for $N\leq 7$ for all $k$, as well as for all $N$ for sufficiently large $k$ \cite{ModulesPt1,ModulesPt2}. However, this classification result is non-constructive, as it uses the bijection between Lagrangian algebras in the centre, and indecomposible modules over a category \cite{LagrangeGerman, LagrangeUkraine}.

The purpose of this paper is to develop an efficient method of explicitly constructing module categories over $\mathcal{C}(\mathfrak{sl}_N, k)$ (and more generally, the twisted categories $\overline{\operatorname{Rep}(U_q(\mathfrak{sl}_N))^\omega}$). We achieve this in the following theorem. This gives a system of polynomial equations, the solutions of which (in the unitary setting) classify module categories over $\overline{\operatorname{Rep}(U_q(\mathfrak{sl}_N))^\omega}$. 
\begin{thm}\label{thm:main}
Let $N\in \mathbb{N}_{\geq 2}$, $q=e^{2\pi i \frac{1}{2(N+k)}}$ for some $k\in \mathbb{N}$, $\omega$ an $N$-th root of unity, and $\Gamma$ a finite graph with norm $[N]_q$. There is a bijective correspondence between
\begin{enumerate}
    \item pivotal $\overline{\operatorname{Rep}(U_q(\mathfrak{sl}_N))^\omega}$-modules $\mathcal{M}$ whose module fusion graph for action by $\Lambda_1$ is $\Gamma$, and
    \item solutions for the \textit{Kazhdan-Wenzl cell system} on $\Gamma$
\end{enumerate}
where the Kazhdan-Wenzl cell system on $\Gamma$ is defined in Definition~\ref{defn:KW}.

The equivalence relation on 1) is equivalence of module categories, and the equivalence relation on 2) can be found in Definition~\ref{defn:KWequiv}.
\end{thm}
The Kazhdan-Wenzl cell system on $\Gamma$ is a polynomial system of equations. These polynomial equations are fairly reasonable to solve, as demonstrated in Section~\ref{sec:examples} where we find many solutions. 

\begin{remark}
The reader may be interested in constructing module categories over $\overline{\operatorname{Rep}(U_q(\mathfrak{sl}_N))^\omega}$ in the non-unitary setting (i.e. when $q\neq e^{2 \pi i \frac{1}{2(N+k)}}$). We offer two remedies.

The first is Lemma~\ref{lem:Galois}, which shows that when $q$ is a root of unity, $\overline{\operatorname{Rep}(U_q(\mathfrak{sl}_N))^\omega}$ is Galois conjugate to $\overline{\operatorname{Rep}(U_{q'}(\mathfrak{sl}_N))^{\omega'}}$ where $q' = e^{2 \pi i \frac{1}{2(N+k)}}$ for some $k$, and $\omega'$ some $N$-th root of unity. For this Galois conjugate the full strength of Theorem~\ref{thm:main} applies, and the module categories are classified by solutions to KW cell systems on graphs. As Galois conjugate categories have the same representation theory, this allows the representation theory of the non-unitary categories to be determined with our theory.

The second approach is discussed in Remark~\ref{rmk:nonUni}. This remark explains how 
 the KW cell system makes sense in the non-unitary setting, and how an additional equation can be added to a KW cell system. A solution to this larger system of equations then guarantees the existence of a module category even in the non-unitary setting. This additional polynomial equation has degree the length of the longest word in the symmetric group $S_N$. In practice this additional equation can take weeks to verify on a computer.
\end{remark}

The result of Theorem~\ref{thm:main} reduces the construction of such a module category to a polynomial system of equations which we call a Kazhdan-Wenzl cell system. The Kazhdan-Wenzl cell system depends on the parameters $N$, $q$, $\omega$ and $\Gamma$, and is a degree 3 polynomial system. In the $N=2$ case, our polynomial system of equations is related to Etingof and Ostrik's classifying data for $\overline{\operatorname{Rep}(U_q(\mathfrak{sl}_2))}$ module categories \cite{OstReich} (see also \cite{MR3420332} for the case where $|q| \leq 1$). In the $N=3$ case, our polynomial system is related to Ocneanu's cell calculus for $\overline{\operatorname{Rep}(U_q(\mathfrak{sl}_3))}$ module categories. See \cite{SU3} for solutions in the $SU(3)$ case. Note that in \cite{Ocneanu}, Ocneanu claims a cell calculus for $\overline{\operatorname{Rep}(U_q(\mathfrak{sl}_4))}$ module categories, but no definition is given. Our definition holds for all $N$, and hence generalises the above definitions. See also \cite{EvansSO3} for a cell calculus for $SO(3)$ module categories.

Our definition of a KW cell system can be naturally broken into two pieces. The first is a path representation of the Hecke algebra, satisfying the Markov property, and the second is the solution to a linear system, along with a normalisation convention. Solutions to the first piece have appeared many times in the literature \cite{HansThesis,MR2021644,HansMod}, including in the physics literature \cite{Stat} where the connection to integrable lattice modules in explained. A solution to this first piece can be thought of as the data of a ``$\overline{\operatorname{Rep}(U_q(\mathfrak{gl}_N))}$" module. The second piece of data (to our best knowledge) is completely new, and is precisely the data to extend a ``$\overline{\operatorname{Rep}(U_q(\mathfrak{gl}_N))}$" module to a $\overline{\operatorname{Rep}(U_q(\mathfrak{sl}_N))^\omega}$ module.

In order to obtain our polynomial system, we follow the strategy of \cite{EH}, using the theory of graph planar algebra embeddings. Let us briefly describe the philosophy of this strategy.

Recall a module category for a tensor category $\mathcal{C}$ is equivalent to a monoidal functor
\[  \mathcal{C}\to \operatorname{End}(\mathcal{M})  \]
where $\mathcal{M}$ is a semi-simple category. This is directly analogous to a module over a group $G$, which is described by a homomorphism
\[G \to \operatorname{End}(V). \]
Given an explicit group, say $D_{n}$, the most efficient way to build a module is to use a nice presentation, say $\langle r , s \mid r^n = e, rs = sr^{-1}\rangle $. A module can then be built by given the images of $r$ and $s$ in $\operatorname{End}(V)$, and verifying these images satisfy the relations in the presentation.

As introduced in \cite{OstReich} and \cite{EH}, an analogous idea holds for modules over a tensor category. In particular for us, we use the Kazhdan-Wenzl presentation \cite{SovietHans} for the category $\overline{\operatorname{Rep}(U_q(\mathfrak{sl}_N))^\omega}$, which has a single object generator $\Lambda_1$, and two morphism generators. The image of $\Lambda_1$ in $\operatorname{End}(\mathcal{M})$ can be described as an oriented graph $\Gamma$ (whose vertices are the simple objects of $\mathcal{M}$, and edges determine the action of the endofunctor). The images of the two generating morphisms live in a distinguished subcategory of $\operatorname{End}(\mathcal{M})$ known as the graph planar algebra on $\Gamma$, which we denote\footnote{To distinguish it from the closely related, but distinct, non-oriented version $GPA(\Gamma)$ \cite{OGGPA}. The oriented version was known to Jones, and variants have been defined in \cite{Morrisey, Emily}} $oGPA(\Gamma)$. 

As seen in \cite{OGGPA,EH}, the distinguished subcategory $oGPA(\Gamma)$ has an incredibly explicit description in terms of loop on the graph $\Gamma$. This allows us to describe the images of the two generating morphisms as linear functionals of the space of certain loops in $\Gamma$. We can then use the (a refinement of) the relations of Kazhdan-Wenzl to obtain polynomial equations that these functionals must satisfy. Extracting all this data gives us our definition of a Kazhdan-Wenzl cell system.

There are several natural choices for a presentation for the category $\overline{\operatorname{Rep}(U_q(\mathfrak{sl}_N))^\omega}$. In order to obtain an efficient cell calculus, we desire several conditions on the presentation
\begin{itemize}
\item The presentation is uniform for $\overline{\operatorname{Rep}(U_q(\mathfrak{sl}_N))^\omega}$ as $N$ and $q$ vary,
\item The presentation contains as few generating objects and morphisms as possible,
\item The relations the generating morphisms satisfy must live in Hom spaces between objects of as small length as possible.
\end{itemize}

The most obvious presentation is the $6-j$ symbol presentation, where all simple objects are generating objects, and the collection of all trivalent vertices are the generating morphisms. This presentation only satisfies the third condition, and blows out on the first two. Further, to the authors knowledge, this presentation is only explicitly described for $\mathfrak{sl}_2$. This immediately rules out this choice. 

Another option is the Cautis-Kamnitzer-Morrison presentation \cite{SlnWebs}. Here the generating objects are the fundamental representations $\Lambda_i$, and the generating morphisms are trivalent vertices between them. This presentation satisfies the third point, and is uniform with respect to $q$. The practical issue occurs as the number of generating objects and morphisms grow with $N$. This makes determining the images of these generators in the graph planar algebra unfeasible in general (see \cite{Emily} for a specific example where this is achieved).

If we were to follow Ocneanu directly, then we can use a presentation of $\overline{\operatorname{Rep}(U_q(\mathfrak{sl}_N))^\omega}$ with generating object $\Lambda_1$, and single generating morphism the intertwiner $\Lambda_1^{\otimes N}\to \mathbf{1}$. For $\mathfrak{sl}_3$ this is exactly Kuperburgs presentation for the $\mathfrak{sl}_3$ spider \cite{Spider}. This seems ideal at first, however this presentation is not uniform with respect to $N$ at all. To the authors best knowledge, a presentation is only known for $N\in \{2,3,4\}$. We suspect that the cell system claimed to exist by Ocneanu in \cite{Ocneanu} was based on this $\mathfrak{sl}_4$ presentation.

Finally we have the Kazhdan-Wenzl presentation for $\overline{\operatorname{Rep}(U_q(\mathfrak{sl}_N))^\omega}$ from \cite{SovietHans} (see also \cite{Anup}). This has a single generating object which is $\Lambda_1$, and two generating morphisms; the projection onto $\Lambda_2$, and the intertwiner $\Lambda_1^{\otimes N}\to \mathbf{1}$. While this may seem more complicated than the Kuperburg style presentation, the additional generating morphism allows a presentation which is uniform across all $N$. For this reason, we use this presentation to describe the module categories.

The major downside of the Kazhdan-Wenzl presentation is that two of the relations occur in $\operatorname{End}(\Lambda_1^{\otimes N})$ and $\operatorname{Hom}(\Lambda_1^{\otimes N+1}\to \Lambda_1)$. As $N$ grows, these relations will be computationally infeasible to verify inside $\operatorname{End}(\mathcal{M})$. To rectify this, we show that these two relations can be replaced with three much simpler relations. This is our first technical result, and can be found in Section~\ref{sec:KW}.

%A key difference between our setting, and the setting of \cite{EH} is that the generating object $\Lambda_1$ of the Kazhdan-Wenzl presentation is not self-dual. This means we have to generalise the techniques developed in \cite{EH} to the oriented setting. This involves defining the oriented graph planar algebra $oGPA(\Gamma)$s (which has also appeared in \cite{morrisey, Emily}, as well being known by Jones), and proving the oriented version of the graph planar algebra embedding theorem from \cite{EH}. This theorem gives a bijection between embeddings into 

One of the motivation behind this work was to improve on the second authors results of \cite{ModulesPt1,ModulesPt2}. These results abstractly classify module categories over $\mathcal{C}(\mathfrak{sl}_N, k)$ for small $N$. In particular for $N=4$ we have the following.
\begin{thm}\cite{ModulesPt2}\label{thm:absClassIntro}
Let $k\geq 0$, and $\mathcal{C}(\mathfrak{sl}_4, k)$ the category of level $k$ integrable representation of $\widehat{\mathfrak{sl}_4}$. Then there are exactly 
\begin{center}
\begin{tabular}{ |c|c c c c c c c| } 
 \hline
$k$ & 1 & 2 & 4 & 6 & 8& $k>1$ odd & $k >8$ even \\\hline
$\#$ of Modules & 2&3&7&8&9&4&6 \\ 
 \hline
\end{tabular}
\end{center}
%\[ \begin{cases}
%2 & \quad\text{ if $k=1$}\\
%3 & \quad\text{ if $k=2$}\\
%4 & \quad\text{ if $k\neq 1$ is odd}\\
%6 & \quad\text{ if $k\notin\{2,4,6,8\}$ is even}\\
%7 & \quad\text{ if $k=4$}\\
%8 & \quad\text{ if $k=6$}\\
%9 & \quad\text{ if $k=6$}.
%\end{cases}\]
irreducible module categories over $\mathcal{C}(\mathfrak{sl}_4, k)$ up to equivalence.
\end{thm}
The proof of this theorem is non-constructive, as it uses the correspondence between Lagrangian algebras in $\mathcal{Z}(\mathcal{C})$, and irreducible module categories over $\mathcal{C}$ \cite{LagrangeGerman,LagrangeUkraine}. With the results of this paper, we can explicitly construct all of these modules in the sense that we fully describe the functor $\mathcal{C}(\mathfrak{sl}_4, k) \to \operatorname{End}(\mathcal{M})$. Hence we upgrade the abstract classification to a concrete classification.

\begin{remark}
We would like to highlight some relevant work regarding the module categories of $\mathcal{C}(\mathfrak{sl}_4, k)$. We first note that in \cite{Ocneanu} a complete description and classification of $\mathcal{C}(\mathfrak{sl}_4, k)$ module categories was claimed. No proofs were supplied. 

The three type $I$ exceptional modules can be constructed via conformal inclusions \cite{Xu}. However this construction does not immediately give the full structure of the module category. Also note that in \cite{LiuYB} a planar algebra presentation for the exceptional type $I$ module when $k=6$ is given. The full structure of this module was determined in \cite{LiuRing}. Several of these graphs are discussed in \cite[Section 6]{ModInv1} and \cite[Section 8]{ModInv2}.

In \cite{HansMod} a family of module categories over $\mathcal{C}(\mathfrak{sl}_4, k)$ is constructed. We expect that this family corresponds to the second infinite family of graphs below. In this same paper two families of modules over $\mathcal{C}(\mathfrak{sl}_4, k)^{\text{ad}}$ are constructed. These families are the restrictions of the first and third families of modules below, from $\mathcal{C}(\mathfrak{sl}_4, k)$ down to $\mathcal{C}(\mathfrak{sl}_4, k)^{\text{ad}}$.

\end{remark}

In Section~\ref{sec:examples}, we construct KW cell systems on the following families of graphs:
\[\raisebox{-.5\height}{ \includegraphics[scale = .4]{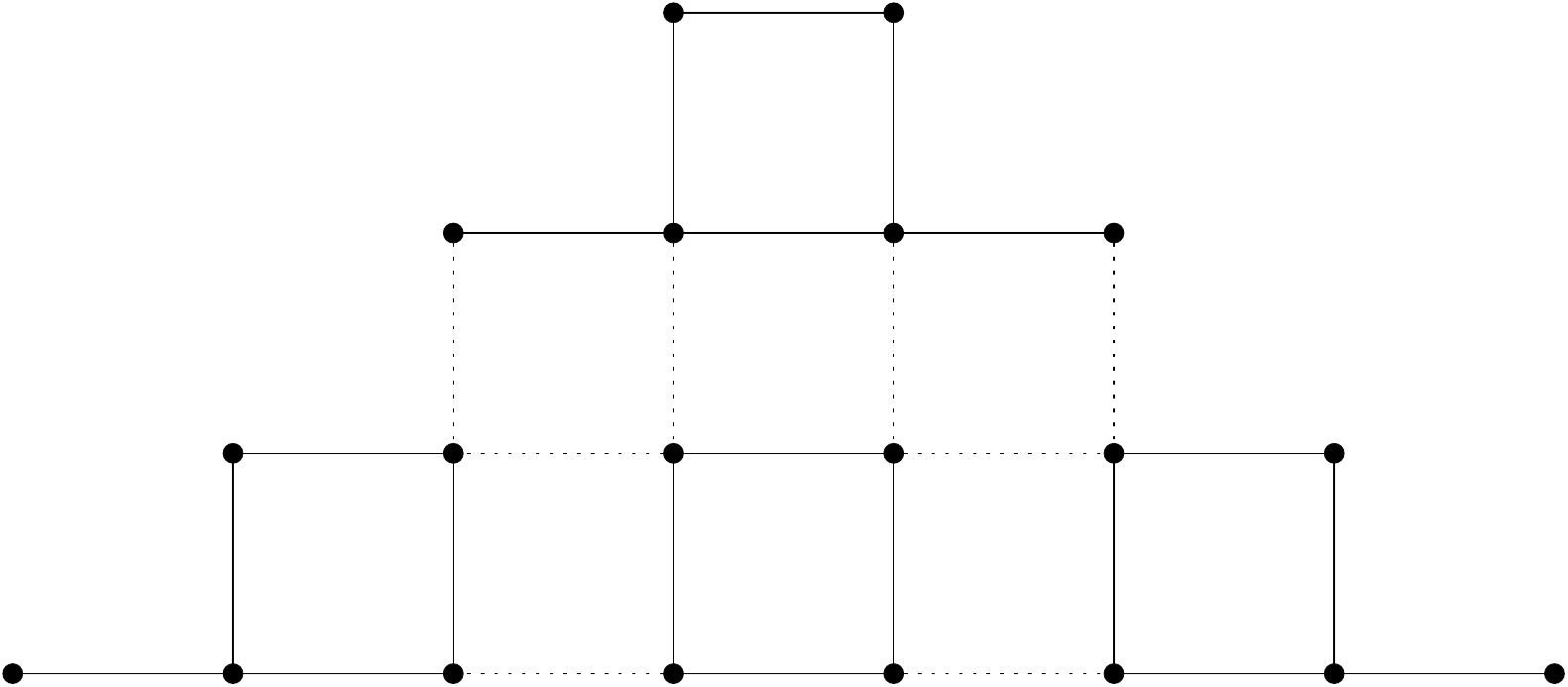}}\quad \raisebox{-.5\height}{ \includegraphics[scale = .5]{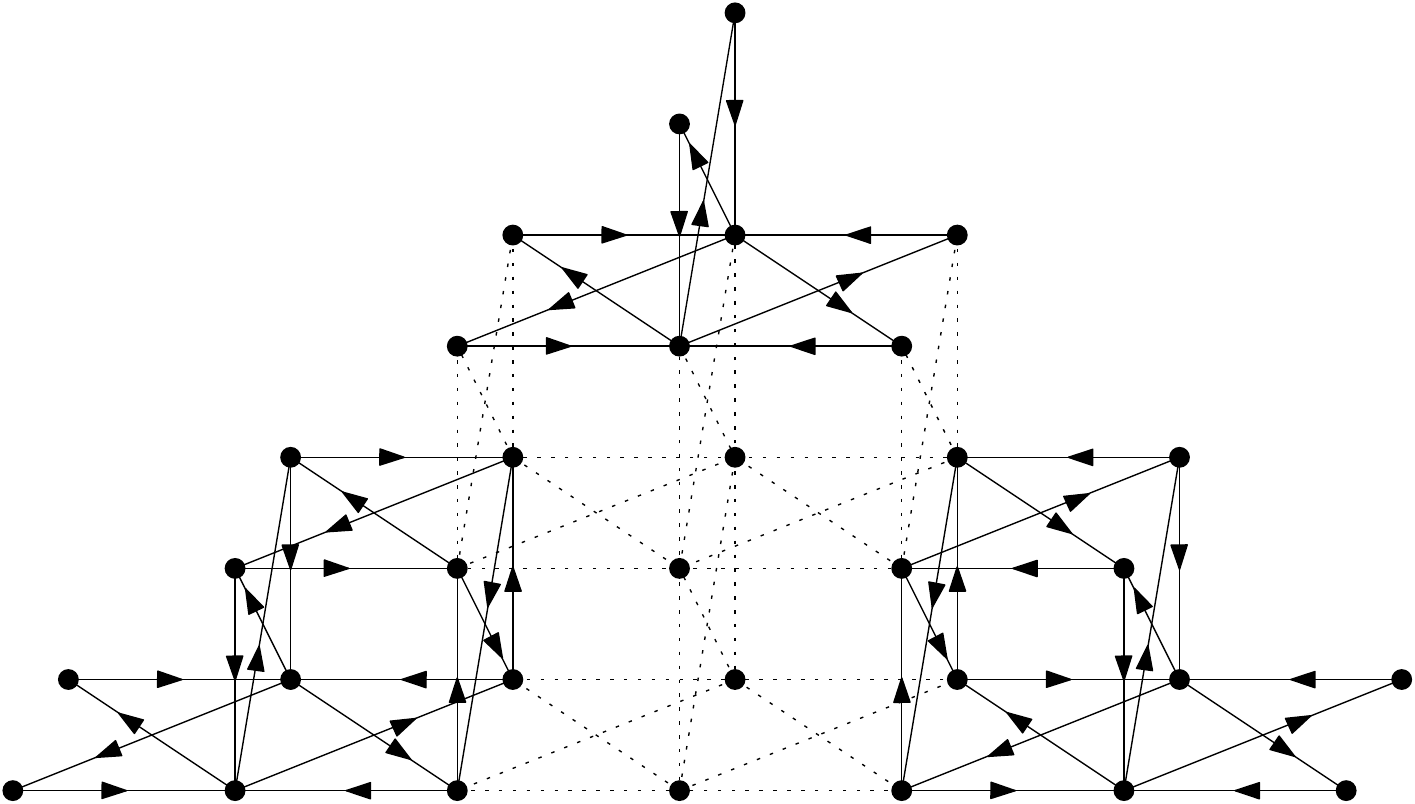}}\]
for all $k$ (constructing two families of charge conjugation modules), the family of graphs
\[ \raisebox{-.5\height}{ \includegraphics[scale = .5]{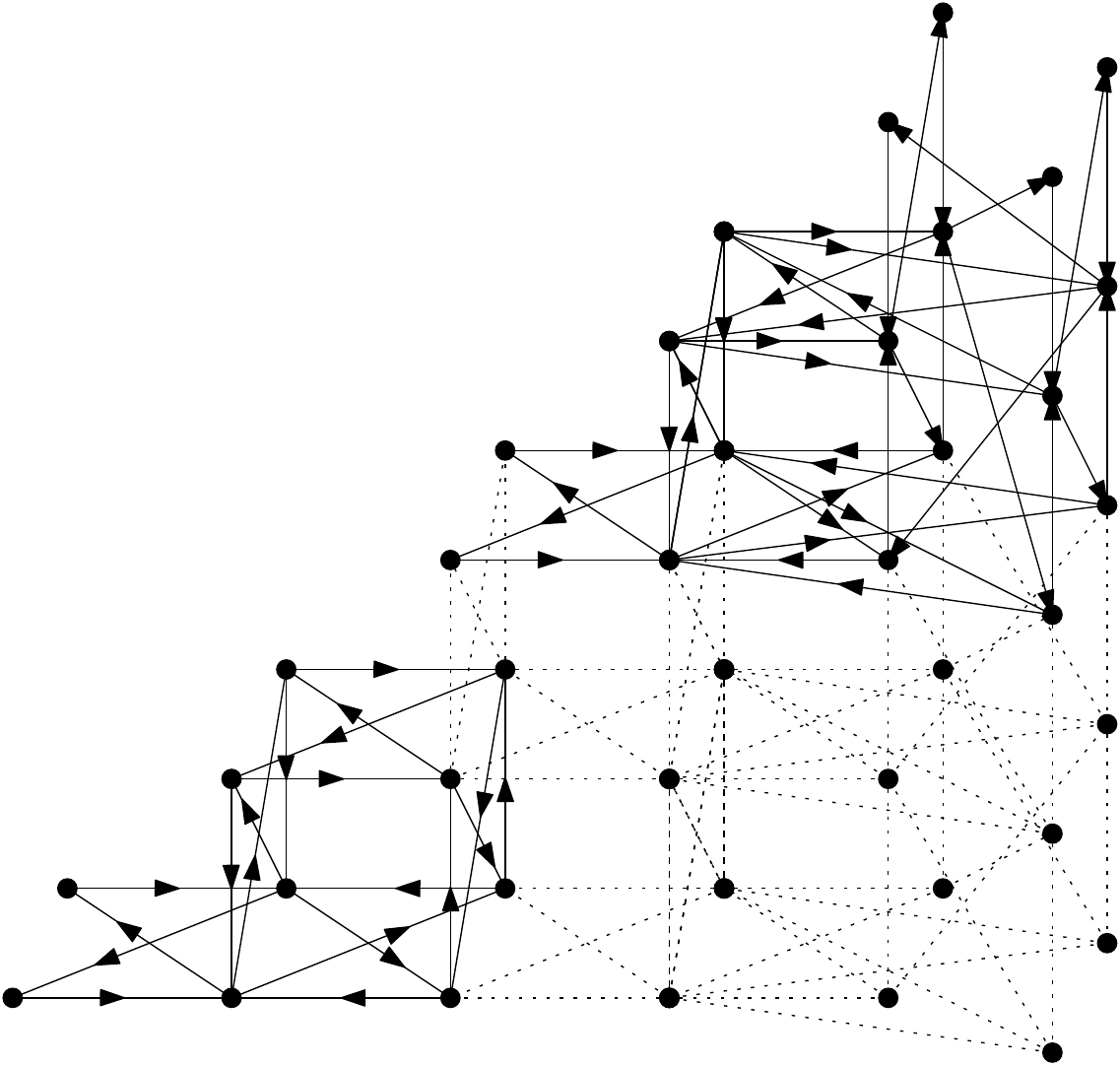}} \]
when $k$ is even (constructing a third family of charge conjugation modules), the graph
\[   \raisebox{-.5\height}{ \includegraphics[scale = .4]{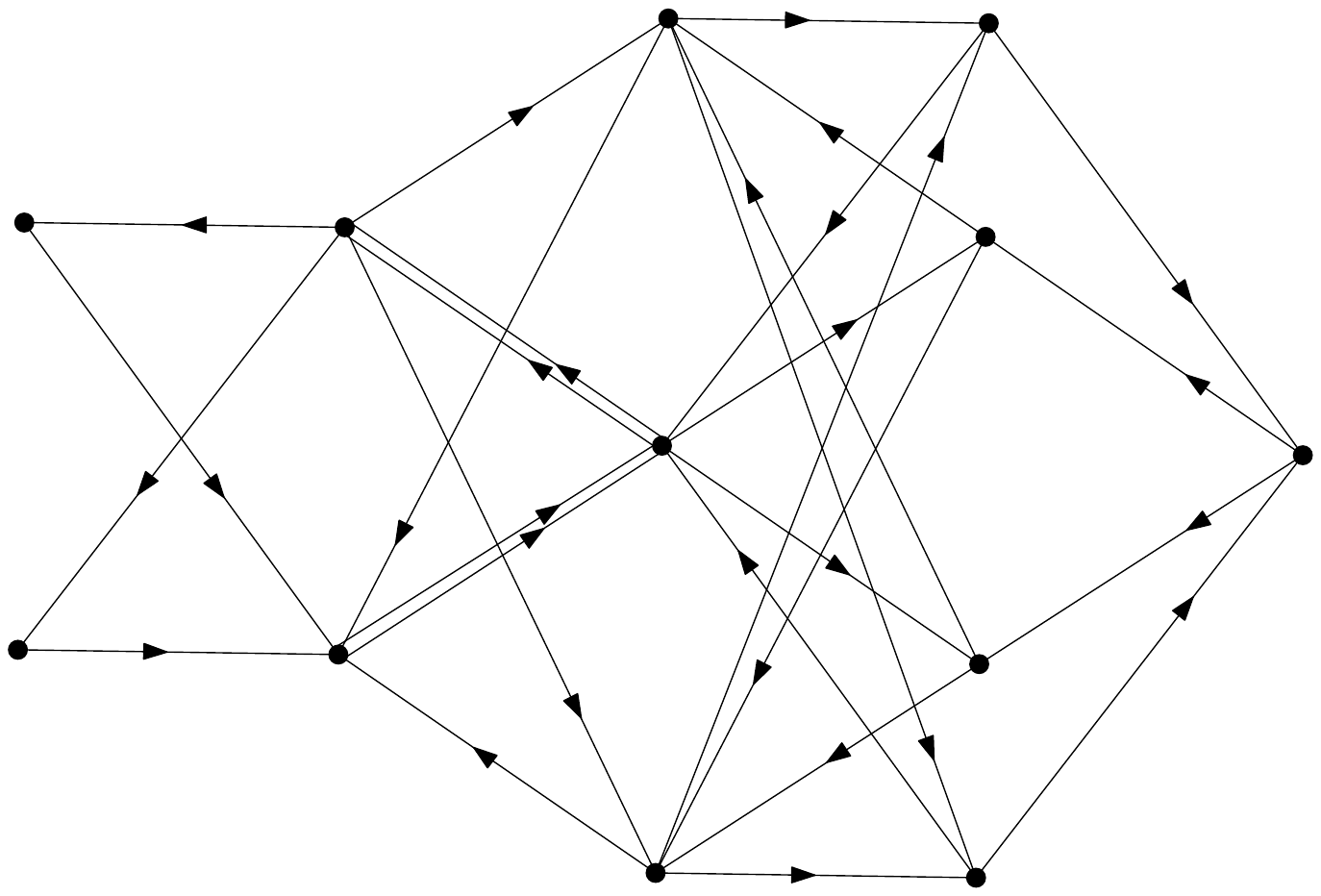}}  \]
when $k=4$ (constructing the sole exceptional module for $\mathcal{C}(\mathfrak{sl}_4, 4)$), the graphs
\[    \raisebox{-.5\height}{ \includegraphics[scale = .4]{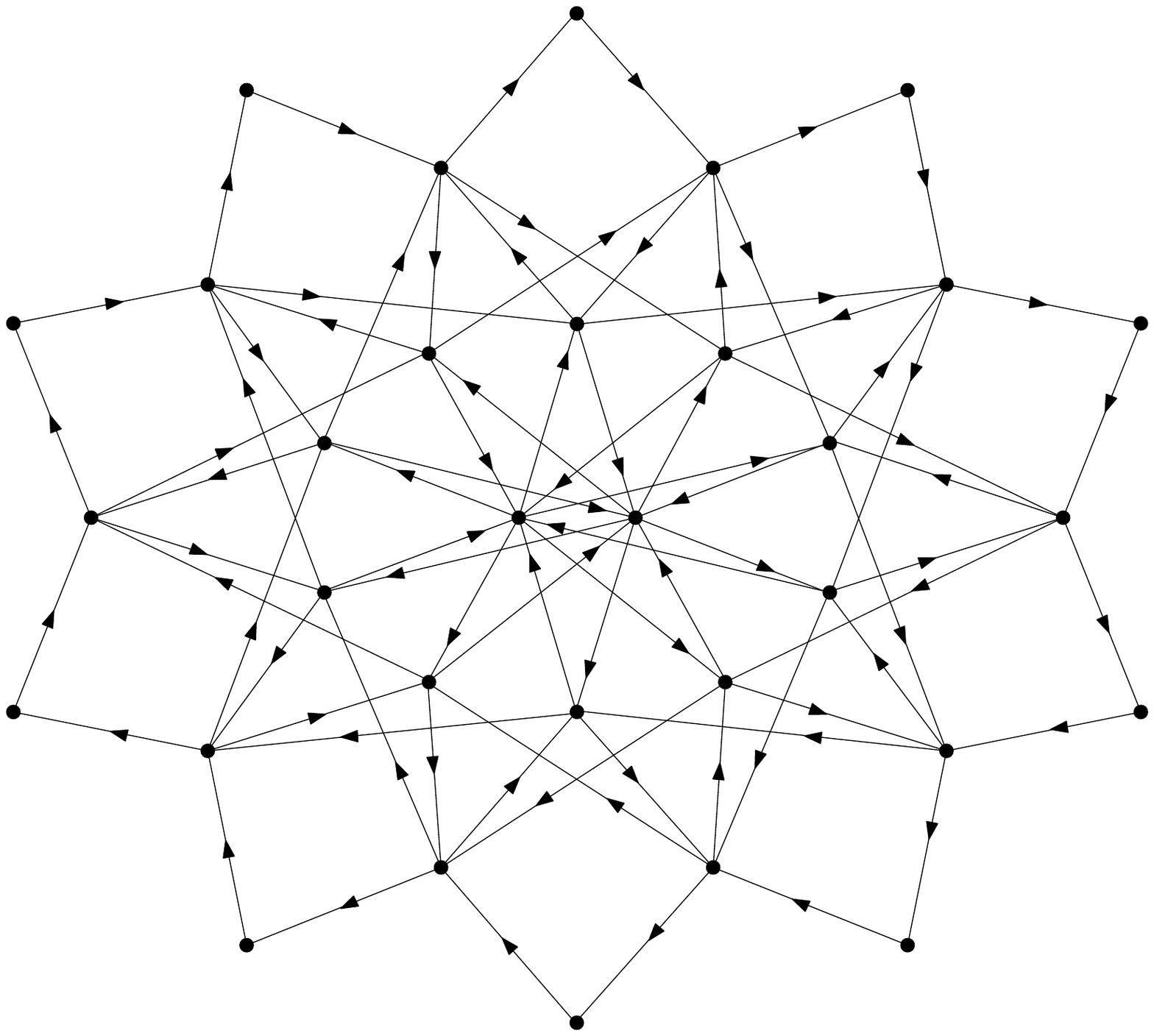}} \quad \raisebox{-.5\height}{ \includegraphics[scale = .4]{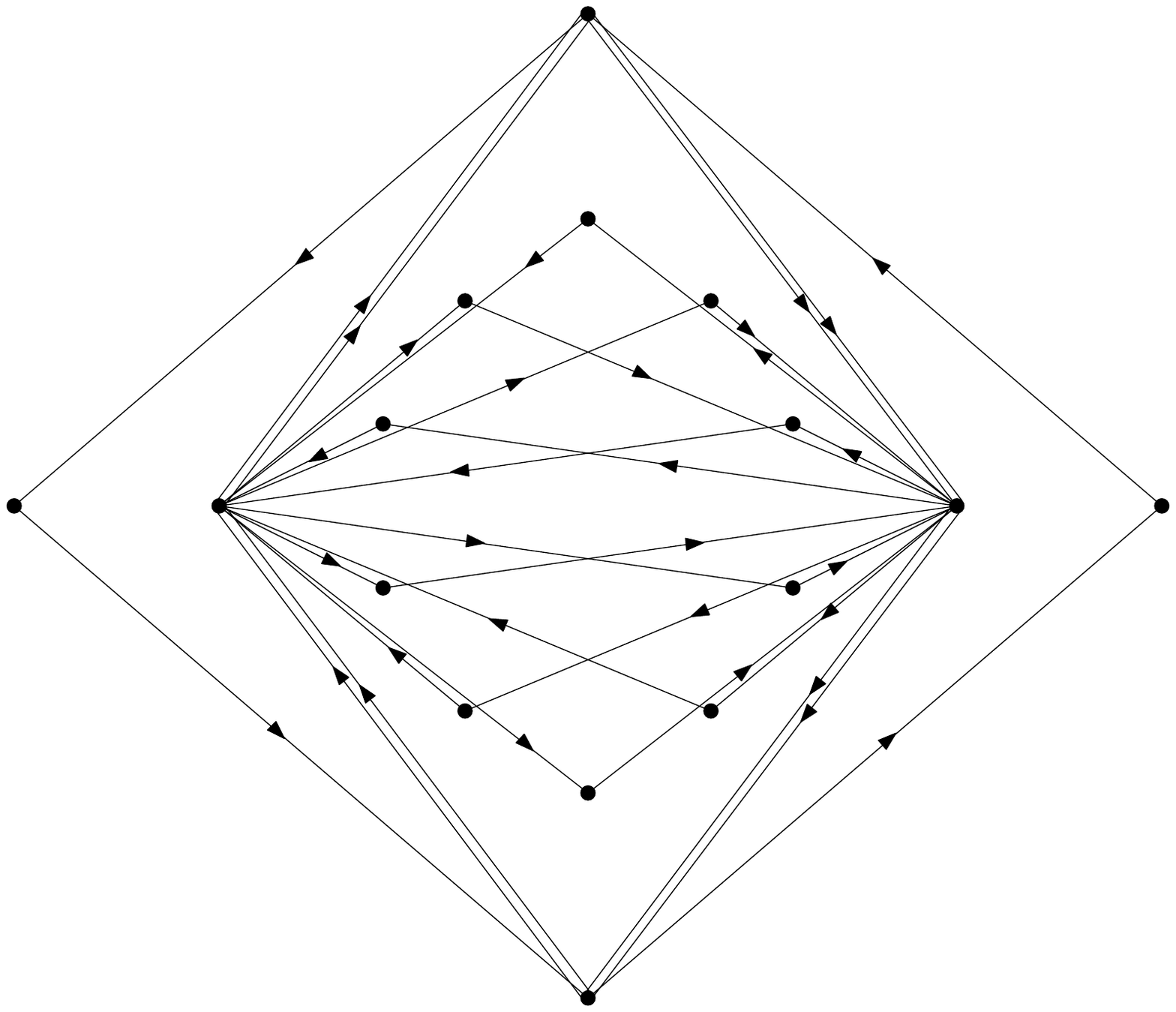}}   \]
when $k=6$ (constructing the two exceptional modules for $\mathcal{C}(\mathfrak{sl}_4, 6)$), and the graphs
\[  \raisebox{-.5\height}{ \includegraphics[scale = .25]{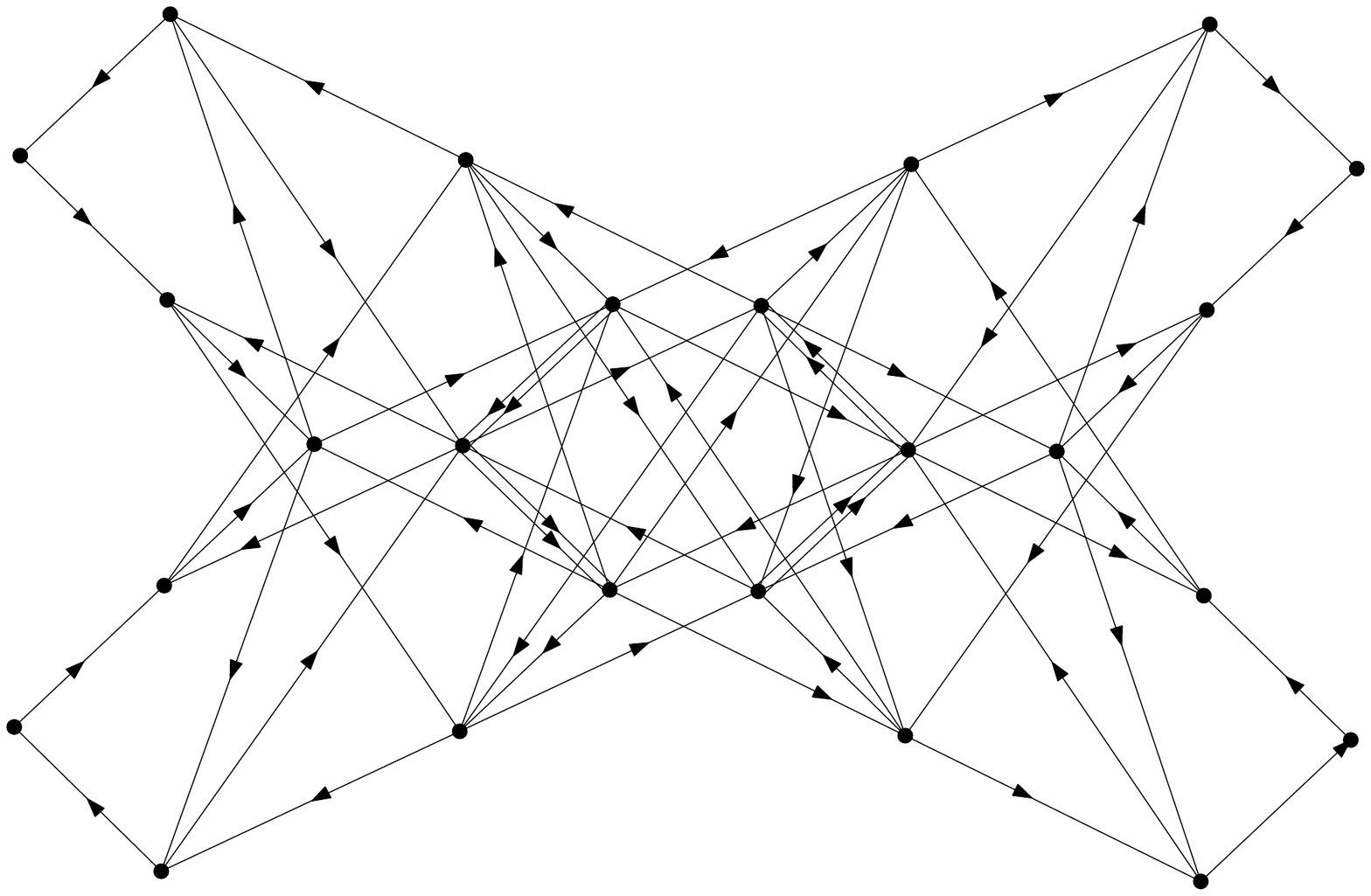}} \quad \raisebox{-.5\height}{ \includegraphics[scale = .25]{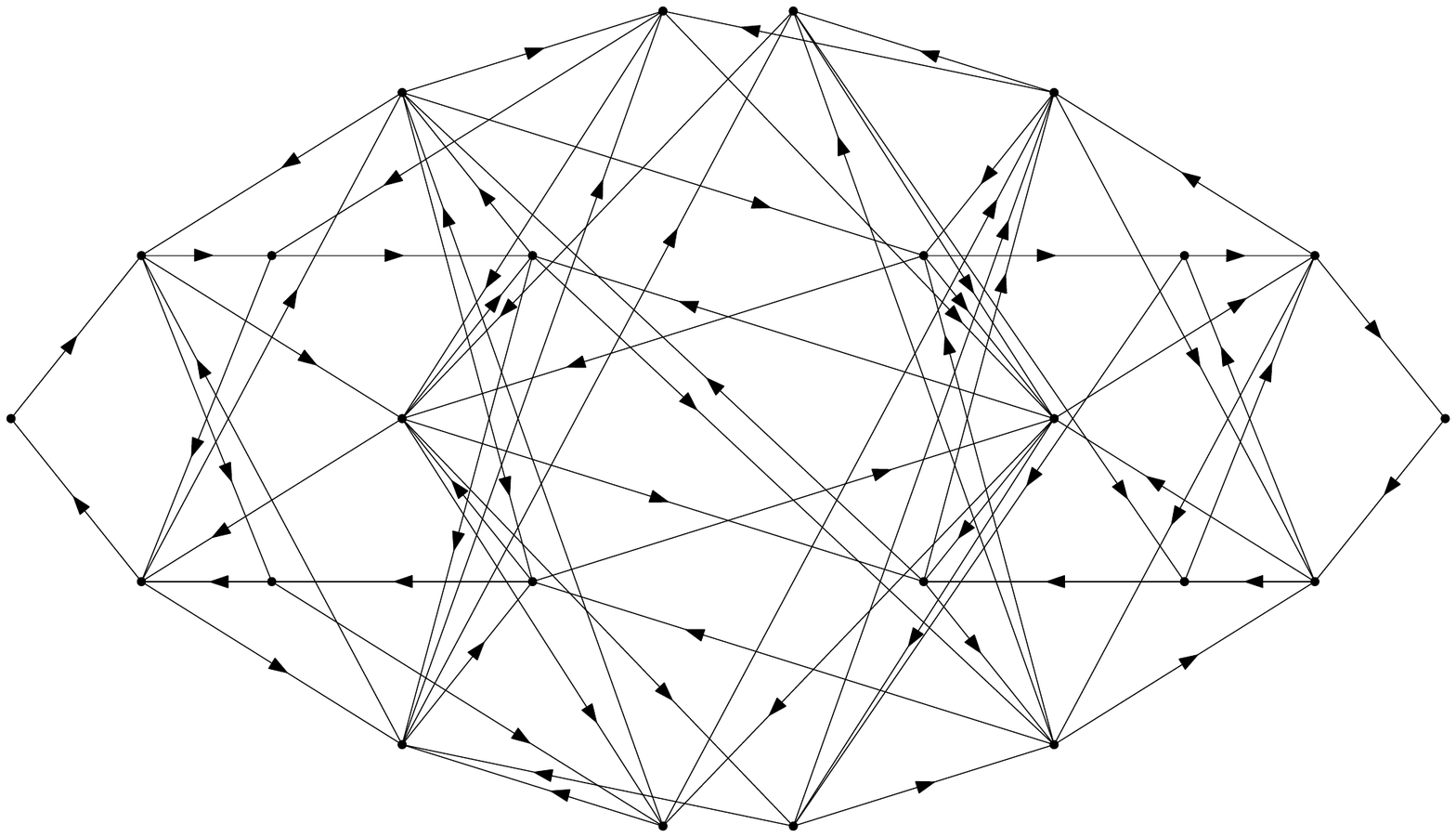}} \quad\raisebox{-.5\height}{ \includegraphics[scale = .25]{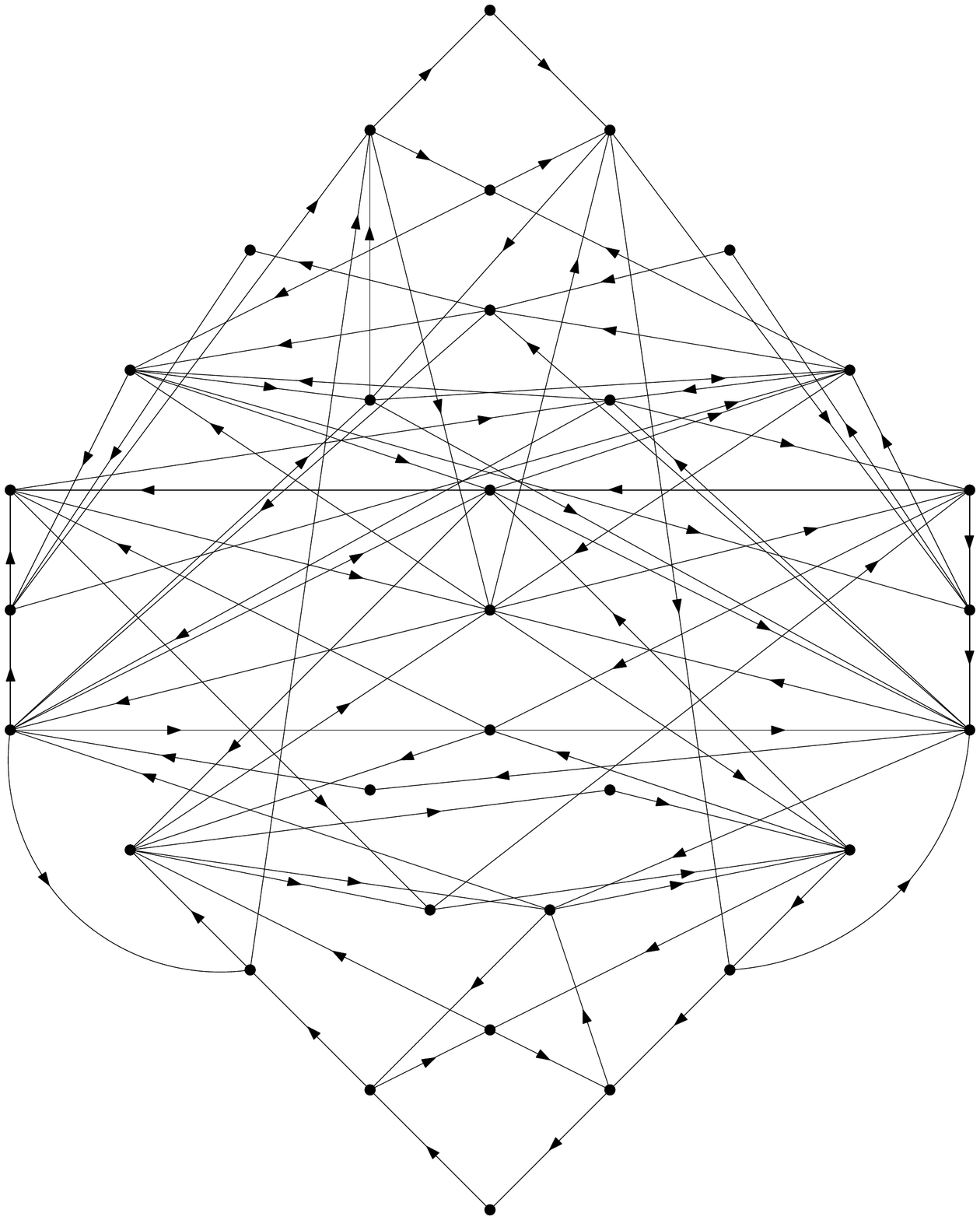}}   \]
when $k=8$ (constructing the three exceptional modules for $\mathcal{C}(\mathfrak{sl}_4, 8)$).

As the module structure of $\mathcal{C}(\mathfrak{sl}_4, k)$ acting on itself is well known (see Figure~\ref{cap:std}), we neglect to solve the KW cell system in these cases. We expect that such a computation should be routine. In fact, a solution for a path representation of the Hecke algebra on the fusion graph for $\Lambda_1$ is given for all $N$ and $k$ in \cite{HansThesis}. Hence solving the remainder of the KW cell system on this graph just requires solving a linear system.
\begin{figure}[h!]
\begin{center}\raisebox{-.5\height}{ \includegraphics[scale = .3]{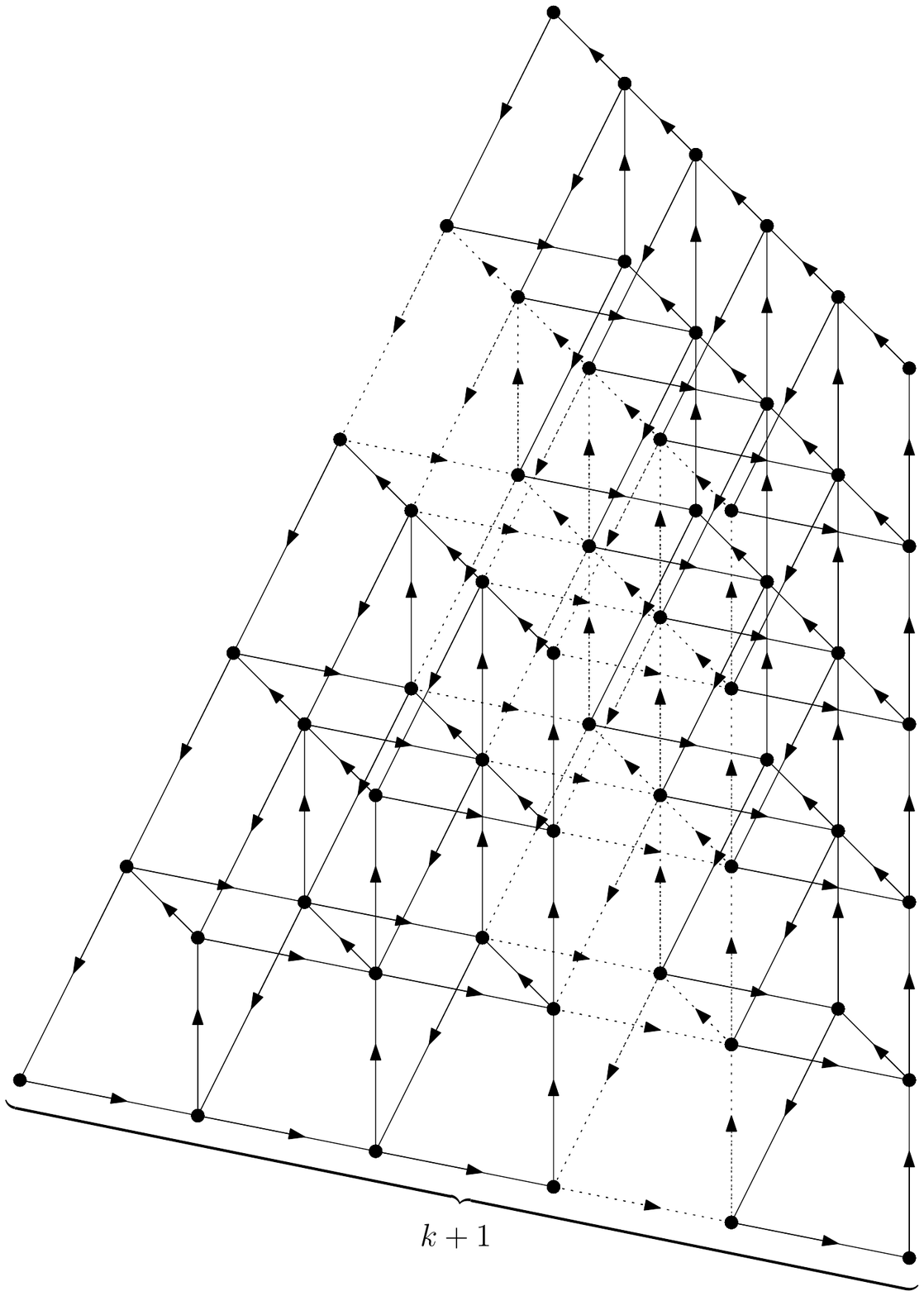}}\end{center}%\quad\text{ or??? }\quad  \raisebox{-.5\height}{ \includegraphics[scale = .4]{std2.pdf}}   
\caption{The module fusion graph of $\mathcal{C}(\mathfrak{sl}_4, k)$ acting on itself for the object $\Lambda_1$.\label{cap:std}}
\end{figure}

The action of $\mathcal{C}(\mathfrak{sl}_4, k)$ on the de-equivariantisations $\mathcal{C}(\mathfrak{sl}_4, k)_{\operatorname{Rep}(\mathbb{Z}_m)}$ is also well understood. These exist for $m=4$ when $2 \mid k$, and for $m=2$ for all $k$. The structure of the category $\mathcal{C}(\mathfrak{sl}_4, k)_{\operatorname{Rep}(\mathbb{Z}_m)}$ is well known. In particular, the module fusion graph for action by $\Lambda_1$ is the orbifold of the graph in Figure~\ref{cap:std} by the canonical $\mathbb{Z}_m$ action.

A quick count-up shows that the number of modules we have constructed above is exactly the number of modules classified abstractly in Theorem~\ref{thm:absClassIntro}. Hence the modules appearing above (and explicitly constructed in Section~\ref{sec:examples}) provide a classification of semi-simple module categories over $\mathcal{C}(\mathfrak{sl}_4, k)$ for all $k$. This confirms claims made by Ocneanu regarding the ``quantum subgroups'' of $SU(4)$ \cite{Ocneanu}.

Furthermore, for the 6 exceptional graphs above, we also find solutions to the Kazhdan-Wenzl cell systems when $\omega\in \{-1,\mathbf{i},-\mathbf{i}\}$. This gives exceptional module categories over $\overline{\operatorname{Rep}(U_q(\mathfrak{sl}_N))^\omega}$ at the appropriate $q$ values. These modules cannot be constructed via conformal inclusions\footnote{In fact, when $\omega^2\neq 1$ these categories are not braided, and can not be the representation of a conformal field theory.}. To the best of our knowledge, this is the first construction of these module categories. We also find KW cell system solutions on the following graph when $q = e^{2\pi i \frac{1}{20}}$ and $\omega = \pm \mathbf{i}$.
 \[\raisebox{-.5\height}{ \includegraphics[scale = .4]{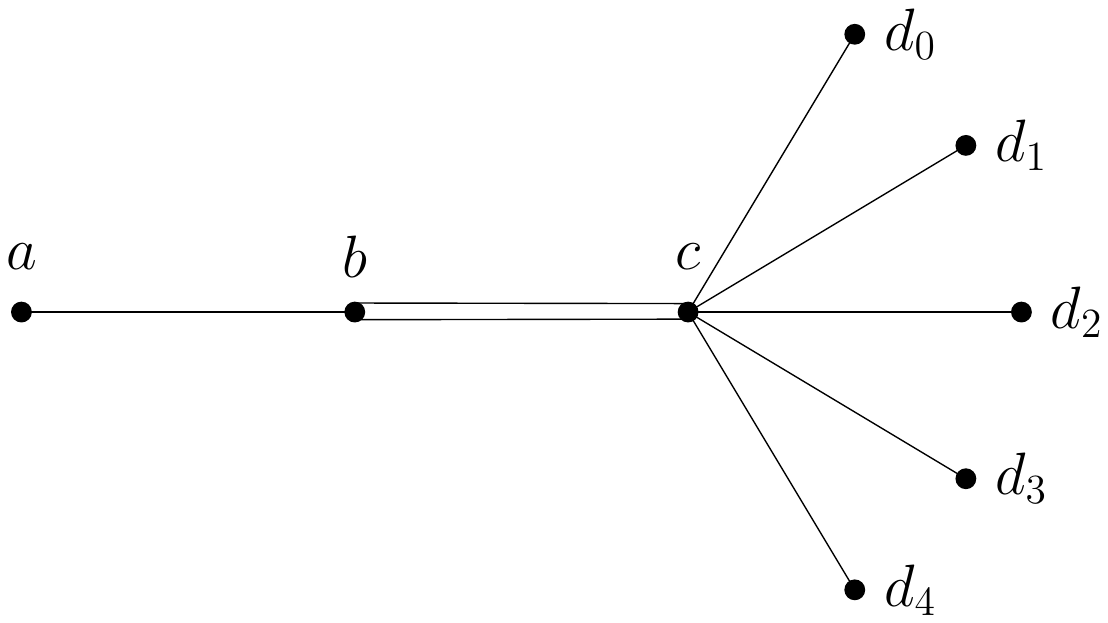}} \]

\subsection*{Acknowledgements}
The second author would like to thank Dietmar Bisch for suggesting this problem back in 2019, Dave Penneys for several useful comments on graph planar algebras, Gwen McKinley for advice on drawing graphs in LaTeX, David Evans for useful feedback on an earlier version of this manuscript, as well as BIRS for hosting them while part of this project was completed. The second author was supported by NSF grant DMS 2245935.

Both authors would like to thank Hans Wenzl for many illuminating conversations, as well as for comments on a preliminary draft of this paper.

\section{Preliminaries}\label{sec:prelims}

We refer the reader to \cite{Book} for the basics of tensor categories and module categories. In this paper, a multi-tensor category is a $\C$-linear, locally finite, rigid monoidal category. A tensor category, for us, is a multi-tensor category whose unit object is simple.
\subsection{Oriented Planar Algebras}
In this section we introduce oriented planar algebras following Jones \cite[Notes 3.12.9]{VaughanCourse}, and Morrison \cite{Morrisey}. Our definition is technically different, yet essentially identical to both of the cited definitions.

\begin{defn}
An {\bf oriented planar algebra} is a strict monoidal, strictly pivotal $\C$-linear category whose objects are parameterized by finite sequences $(\epsilon_1, \epsilon_2, \dots, \epsilon_r)$ with $\epsilon_i \in \{\pm 1\}$. The tensor product is given by concatenation of sequences:
$$
(\epsilon_1, \dots, \epsilon_r) \otimes (\delta_1, \dots, \delta_s) = (\epsilon_1, \dots, \epsilon_r, \delta_1, \dots, \delta_s).
$$
The unit is given by the empty sequence, and is denoted by $(\emptyset)$ or $\mathbbm{1}$.
The dual of an object $(\epsilon_1, \dots, \epsilon_k)$ is obtained by reversing signs and order:
$$
(\epsilon_1, \dots, \epsilon_r)^* = (-\epsilon_r, \dots, -\epsilon_1).
$$
\end{defn}
\begin{remark}
The strictly pivotal assumption means that there are fixed duality morphisms that are compatible with tensor product. % {\color{red} citation, be more careful here?}.
\end{remark}

%\begin{remark}
%This definition of an oriented planar algebra is equivalent to the ``tangle operadic" definitions suggested by Jones [???] and Morrison [???].
%\end{remark}

\begin{remark} The adjective ``oriented" comes from the fact that it is traditional to use oriented strands in the graphical calculus to specify objects. For instance, reading morphisms bottom to top, the following diagram represents a morphism $(1,1,-1,1) \to (-1)$:
\[\raisebox{-.5\height}{ \includegraphics[scale = .4]{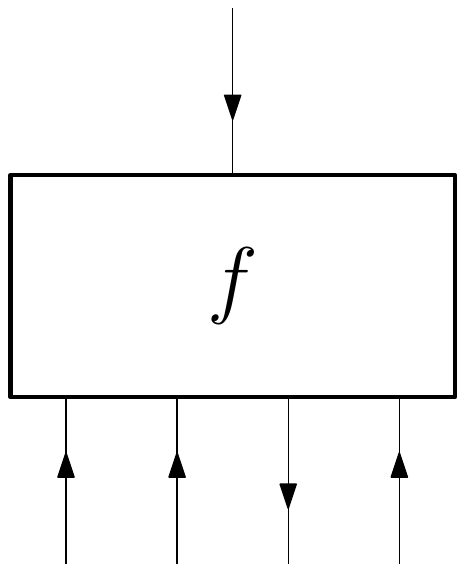}} .\]
\end{remark}

Oriented planar algebras are prevalent: they are strictifications of pivotal categories tensor generated by an object $X$ and its dual $X^*$. 

\begin{defn}\label{defn:oP}
Given any pivotal, monoidal $\C$-linear category $\Cc$, and an object $X$ in $\Cc$, we can form an oriented planar algebra generated by $X$ and $X^*$, denoted $\Pp_{\cC;X}$, as follows (this strictification construction is due to Ng and Schauenberg \cite[Theorem 2.2]{Richard}). The oriented planar algebra $\Pp_{\cC;X}$ is defined by
$$
\Hom_{\Pp_{\cC;X}}((\epsilon_1, \dots, \epsilon_r), (\delta_1, \dots, \delta_s)) :=
\Hom_{\Cc}\big(( \dots (X^{\epsilon_1} \otimes X^{\epsilon_2} ) \dots ) \otimes X^{\epsilon_r}, (\dots (X^{\delta_1} \otimes X^{\delta_2}) \otimes \dots ) \otimes X^{\delta_s}\big),
$$
where we set $X^1 := X$ and $X^{-1} := X^*$. The composition and tensor product of morphisms in $\Pp_{\cC;X}$ is obtained from the composition and tensor product of morphisms in $\Cc$\footnote{The definition of tensor product of morphisms requires using the associativity constraints in $\Cc$.}. A choice of duality maps in $\Cc$ for $X$, say $\coev_X: \mathbbm{1} \to X \otimes X^*$ and $\ev_X: X^* \otimes X \to \mathbbm{1}$ can be uniquely extended to duality maps for every object in $\Pp_X$ in a way that makes $\Pp_{\cC;X}$ strictly pivotal (see \cite[Theorem 2.2]{Richard} for details). Thus $\Pp_{\cC;X}$ is an oriented planar algebra. 
\end{defn}

If $\Cc$ is tensor generated by $X$ and $X^*$, then it is well-known that the Cauchy completion of $\Pp_{\cC;X}$ is equivalent to $\Cc$.

\begin{remark}
When the ambient category $\cC$ is clear and unambiguous, we will use the shorthand notation $\mathcal{P}_X$ instead of $\Pp_{\cC;X}$.
\end{remark}

Oriented planar algebras form a category, whose morphisms are strictly pivotal, strict monoidal functors \cite[Section 1.7.5]{Tur} which act as the identity on objects. In particular, strictly pivotal functors are required to preserve the choice of dualty morphisms (not just be compatible with duality functors). The construction described above extends to a ``strictification functor" from the category of pointed pivotal monoidal categories $(\Cc, X)$ to the category of oriented planar algebras. When restricted to the class of categories tensor-generated by $X$ and $X^*$, the functor is an equivalence, and this establishes the following folklore result (cf. \cite[Section 3]{EH} or \cite[Theorem 4.1]{David})
\begin{thm}\label{thm:folk}
The map $(\Cc, X) \mapsto \Pp_{\cC;X}$ extends to an equivalence of categories
$$
\begin{Bmatrix}
\textrm{Pairs  $(\Cc, X)$ with $\Cc$ a pivotal multi-tensor} \\
\textrm{category generated by $X$ and $X^*$}
\end{Bmatrix} \cong
\{ \textrm{Oriented planar algebras}\}.
$$
\end{thm}

\subsection{The Oriented Graph Planar Algebra}

The oriented graph planar algebra associated to a finite directed graph $\Gamma$, which we denote $oGPA(\Gamma)$, is an important example of a unitary oriented planar algebra. In Section~\ref{sec:GPA} we will explain the connection of this oriented planar algebra to the classification of module categories.

For the oriented GPA, it is important to consider paths on our graph which traverse edges backwards. To formalize this, we introduce new edges corresponding to the original edges, but with their directions reversed. This results in a signed graph (a graph where the edges are labeled by $\pm 1$), with the original edges labelled $+1$ and the new edges labeled $-1$. More precisely, we make the following definitions:

\begin{defn}
Let $\Gamma = (V, E)$ be a directed graph. Given $e \in E$, let $\overline{e}$ denote a new edge with the source and target of $e$ swapped. The {\bf signed graph associated to $\Gamma$} is given by
$$
\overline{\Gamma} = (V, E \cup \overline{E}),
$$
where $\overline{E} = \{ \overline{e}\ : e \in E\}$. Edges in $E$ are given the sign $+1$ while edges in $\overline{E}$ are given the sign $-1$.
\end{defn}

\begin{defn} Suppose $\epsilon = (\epsilon_1, \dots, \epsilon_r)$ is a sequence of $1$'s and $-1$'s. An {\bf $\epsilon$-path} is a path $(f_1, \dots, f_r)$ in $\overline{\Gamma}$ such that
$$
\operatorname{sign}(f_i) = \epsilon_i \quad \text{ for all } i.
$$
When $\epsilon = (\emptyset)$ then an $\epsilon$-path is a path of length zero, ie a vertex in $\Gamma$. We denote such paths by vertex labels, ie $v \in V$.
\end{defn}
Any path in $\overline{\Gamma}$ is an $\epsilon$-path for some $\epsilon$. If $p$ is a path, let $s(p)$ denote the first vertex of the path and $t(p)$ the final vertex. If $p = v$ is a path of length $0$ then $s(p)=t(p) = v$.
\begin{defn}
Let $\Gamma = (V, E)$ be a finite directed graph. The {\bf oriented graph planar algebra} associated to $\Gamma$, denoted $oGPA(\Gamma)$, is an oriented planar algebra defined as follows. The objects of $oGPA(\Gamma)$ are finite sequences $\epsilon = (\epsilon_1, \dots, \epsilon_r)$ of $+1$'s and $-1$'s. Given two objects $\epsilon$ and $\delta$, define a vector space
$$
\Hom_{oGPA(\Gamma)}(\epsilon, \delta) := \textrm{span} \{ (p,q): \text{$p$ is an $\epsilon$-path, $q$ is a $\delta$-path, $s(p) = s(q)$, and $t(p) = t(q)$}\}.
$$
Composition is defined as follows: given $(p,q) \in \Hom_{oGPA(\Gamma}(\epsilon, \delta)$ and $(p',q') \in \Hom_{oGPA(\Gamma}(\delta, \gamma)$, the composition $(p',q') \circ (p,q) \in \Hom_{oGPA(\Gamma)}(\epsilon, \gamma)$ is defined as
\begin{equation}
(p',q') \circ (p,q) = \delta_{q', p}(p', q).
\end{equation}
Extending this linearly makes $oGPA(\Gamma)$ into a $\C$-linear category.

The tensor structure is defined as follows: given $(p,q) \in \Hom_{oGPA(\Gamma}(\epsilon, \delta)$ and $(p',q') \in \Hom_{oGPA(\Gamma)}(\gamma, \beta)$, define
\begin{equation}
(p,q) \otimes (p',q') = \delta_{t(p), s(p')} \delta_{t(q), s(q')} (pp', qq').
\end{equation}
Here $pp'$ and $qq'$ denotes the concatenation of paths. Extending linearly, this definition makes $oGPA(\Gamma)$ into a strict monoidal category.

We define a dagger structure on $oGPA(\Gamma)$ as the anti-linear extension of 
\[  (p,q)^\dag = (q,p).  \]
As the morphisms $(p,q)$ are a full basis of matrix units for $\End_{oGPA(\Gamma)}(\delta)$, we have that $oGPA(\Gamma)$ is semisimple. We also immediately see that these algebras are $C^*$-algebras. 
%If $p$ and $q$ are both $\delta$ paths, then $(p,q)$ are matrix units for the endomorphism algebras $\End(\delta)$, we have that $\dag$ endows $\End(\delta)$ with the structure of a $C^*$-algebra. An application of Roberts $2x2$ trick \cite{}, then gives that the entire $oGPA(\Gamma)$ is a unitary tensor category 

To define a pivotal structure on $oGPA(\Gamma)$, let $\lambda=(\lambda_1, \dots, \lambda_k)$ be the positive Frobenius-Perron eigenvector of $\Gamma$. It is uniquely defined up to multiplication by a positive real number. As $oGPA(\Gamma)$ is an oriented planar algebra, we have that $\epsilon^* = (\epsilon_1, \cdots, \epsilon_n)^* = (-\epsilon_n, \cdots, -\epsilon_1)$. We define
\begin{align}\label{eq:oGPAcupcap}
\operatorname{ev}_{(+,-)}   :=   \sum_{\text{$(e,\overline{e})$ a $(1,-1)$-path}} \sqrt{\frac{\lambda_{t(e)}}{\lambda_{s(e)}}}((e,\overline{e}),s(e)):   (1,-1) \to \mathbbm{1}\\
\operatorname{coev}_{(-,+)}   :=   \sum_{\text{$(\overline{e},e)$ a $(-1,1)$-path}} \sqrt{  \frac{\lambda_{s(e)}}{\lambda_{t(e)}}}( t(e),(\overline{e},e)):  \mathbbm{1} \to  (-1,1)  \\
\operatorname{ev}_{(-,+)}   :=   \sum_{\text{$(\overline{e},e)$ a $(-1,1)$-path}}  \sqrt{\frac{\lambda_{s(e)}}{\lambda_{t(e)}}}((\overline{e},e), t(e)):   (-1,1) \to \mathbbm{1}\\
\operatorname{coev}_{(+,-)}   :=   \sum_{\text{$(e,\overline{e})$ a $(1,-1)$-path}} \sqrt{\frac{\lambda_{t(e)}}{\lambda_{s(e)}}}( s(e),(e,\overline{e})):  \mathbbm{1} \to  (1,-1) 
\end{align}
\end{defn}
Clearly these definitions do not change if $\lambda$ is rescaled by a positive real number. A simple computation shows that these maps satisfy the zig-zag relations. 

A direct computation shows that the identity map is a monoidal natural isomorphism $**\to \id_{oGPA(\Gamma)}$. Hence we choose this as our pivotal structure. Note that $\operatorname{ev}_{(-,+)}^\dag= \operatorname{coev}_{(-,+)}$, and thus our chosen pivotal structure is a unitary pivotal structure in the sense of \cite[Definition 3.11]{David}.

Finally, we verify that $oGPA(\Gamma)$ is a unitary category. From the explicit basis of the hom spaces, the inner product coming from the $\dag$-structure is easily seen to be positive definite. This then implies unitarity by \cite[Lemma 3.51.]{EH}.

\subsection{The multi-tensor category $M_k(\Vect)$}\label{subsec:Mk}

In this subsection we introduce the multi-tensor category $M_k(\Vect)$. As we will see in Section~\ref{sec:GPA} (following ideas of \cite{EH}), there is a close connection between $M_k(\Vect)$ and the graph planar algebra for $\Gamma$.

The category $M_k(\Vect)$ is a semisimple multi-tensor category which is a categorification of the ring $M_k(\mathbb{N})$. Informally, we replace natural numbers by vector spaces, addition of natural numbers by direct sum, and multiplication of natural numbers by tensor product. The category is recognizable as the category of endomorphisms in the 2-category $2\Vect$ (specifically, the 2-category $2\Vect_c$ in \cite{Kap, elG}). Equivalently, $M_k(\mathbb{N})$ is monoidally equivalent to $\operatorname{End}(\mathcal{M})$ where $M$ is the unique semisimple category of rank $k$. More formally, the category is defined as follows.

The objects are $k \times k$ matrices whose entries are (finite-dimensional) Hilbert spaces. The morphisms are $k \times k$ matrices of linear transformations. The composition of morphisms is given by entry-wise composition of linear transformations. This category is $\C$-linear and semisimple, with the direct sum of objects given by entry-wise direct sum of vector spaces. Every simple object is isomorphic to an object with a copy of $\C$ in one entry of the matrix, and the $0$ vector space in all other entries. The simple object whose non-zero entry occurs in the $(i,j)$-th entry is denoted $E_{ij}$.

The tensor structure on $M_k(\Vect)$ is defined as follows. Given two objects, say
$$
A = \begin{bmatrix}A_{11} & \dots& A_{1k} \\
\vdots & \ddots & \vdots \\
A_{k1} & \dots & A_{kk}
\end{bmatrix}, \quad
B = \begin{bmatrix}B_{11} & \dots& B_{1k} \\
\vdots & \ddots & \vdots \\
B_{k1} & \dots & B_{kk}
\end{bmatrix}
$$
then the object $A \otimes B$ is defined by
$$
A \otimes B = [(A \otimes B)_{ij}] := \left[ \bigoplus_{l=1}^k A_{il} \otimes B_{lj}\right]_{ij}.
$$
Similarly, given two morphisms, say $f = [f_{ij}]_{i,j}$ and $g = [g_{ij}]_{i,j}$ (where $f_{ij}$ and $g_{ij}$ denote linear transformations), define
\begin{equation*}
f \otimes g = [(f \otimes g)_{ij}]_{i,j} := \left[\bigoplus_{l=1}^k f_{i,l} \otimes g_{l,j}\right]_{i,j}.
\end{equation*}
The unit for the category is given by $\mathbbm{1} = E_{11} \oplus \dots \oplus E_{kk}$ (i.e. the identity matrix, with a copy of $\mathbb{C}$ in each diagonal entry). The tensor structure is not strict, but has standard associators and unitors coming from the standard associativity and distributivity isomorphisms in $\Vect$.

% The category has a dagger structure inherited from the dagger structure of the category $\Vect$. Explicitly, given a morphism $f = [f_{ij}]_{i,j}: A \to B$, define
% $$
% f^{\dag} = [(f^{\dag}_{ij})] := [f_{i,j}^{\dag}]: B \to A,
% $$
% where $f_{i,j}^{\dag}$ denotes the usual Hermitian adjoint of a bounded map between Hilbert spaces. This dagger structure makes $M_k(\Vect)$ into a $C^*$-multitensor category [use unitary instead??].

The category is rigid. If $A = [A_{ij}]_{ij}$ is an object, then the dual object is obtained by transposing the matrix for $A$ and applying the duality functor in $\Vect$ to every entry:
$$
A^* := [A_{ji}^*]_{ij}.
$$

 The category $M_k(\Vect)$ has a dagger structure which makes it a unitary category. It is defined by
 $$
 \begin{bmatrix}
 f_{11} & \dots & f_{1k} \\
 \vdots & \ddots & \vdots \\
 f_{k1} & \dots & f_{kk}
 \end{bmatrix}^{\dag} = \begin{bmatrix}
 f_{11}^{\dag} & \dots & f_{1k}^{\dag} \\
 \vdots & \ddots & \vdots \\
 f_{k1}^{\dag} & \dots & f_{kk}^{\dag}
 \end{bmatrix}
 $$
 where $f_{ij}^{\dag}$ denotes the usual complex conjugate of a matrix. It is easily checked this gives $M_k(\Vect)$ the structure of a unitary category.

The category admits a pivotal structure. We fix explicit standard (left) duality morphisms. Given an object $A = [A_{ij}]_{ij}$, define
\begin{align*}
\ev^{std}_A &:= \begin{bmatrix}
\bigoplus_{l=1}^k \ev_{A_{l1}} & & \\
& \ddots & \\
& & \bigoplus_{l=1}^k \ev_{A_{lk}}
\end{bmatrix}: A^* \otimes A \to \mathbbm{1}, \\
\coev^{std}_A &:= \begin{bmatrix}
\bigoplus_{l=1}^k \coev_{A_{1l}} & & \\
& \ddots & \\
& & \bigoplus_{l=1}^k \coev_{A_{kl}}
\end{bmatrix}: \mathbbm{1} \to A \otimes A^*,
\end{align*}
where $\ev_{A_{ij}}: A_{ij}^* \otimes A_{ij} \to \mathbb{C}$ and $\coev_{A_{ij}}: \C \to A_{ij} \otimes A_{ij}^*$ denote the standard left duality morphisms in $\Vect$. The pivotal structure $A \to A^{**}$ is inherited from the usual natural isomorphism between a vector space and its double dual. It is straightforward to check this choice of pivotal structure is spherical, and every simple object has dimension $\id_{\mathbbm{1}} \in \End_{M_k(\Vect)}(\mathbbm{1})$. The right duality maps corresponding to this pivotal structure are given by
\begin{align*}
\widetilde{\ev}^{std}_A &:= \begin{bmatrix}
\bigoplus_{l=1}^k \widetilde{\ev}_{A_{1l}} & & \\
& \ddots & \\
& & \bigoplus_{l=1}^k \widetilde{\ev}_{A_{kl}}
\end{bmatrix}: A \otimes A^* \to \mathbbm{1}, \\
\widetilde{\coev}^{std}_A &:= \begin{bmatrix}
\bigoplus_{l=1}^k \widetilde{\coev}_{A_{l1}} & & \\
& \ddots & \\
& & \bigoplus_{l=1}^k \widetilde{\coev}_{A_{lk}}
\end{bmatrix}: \mathbbm{1} \to A^* \otimes A^*,
\end{align*}
where $\widetilde{\ev}_{A_{ij}}: A_{ij} \otimes A_{ij}^* \to \C$ and $\widetilde{\coev}_{A_{ij}}: \C \to A_{ij}^* \otimes A_{ij}$ are the standard right duality morphisms in $\Vect$.

The pivotal structure chosen above is not unique. Given any vector $\lambda = (\lambda_1, \dots, \lambda_k)$ of non-zero complex numbers, we may define new left and right duality morphisms by
\begin{align}\label{eq:MkVeccupcap}
\ev^{\lambda}_A &:= \begin{bmatrix}
\bigoplus_{l=1}^k \sqrt{\frac{\lambda_1}{\lambda_l}}\ev_{A_{l1}} & & \\
& \ddots & \\
& & \bigoplus_{l=1}^k \sqrt{\frac{\lambda_k}{\lambda_l}}\ev_{A_{lk}}
\end{bmatrix}: A^* \otimes A \to \mathbbm{1}, \\
\widetilde{\ev}^{\lambda}_A &:= \begin{bmatrix}
\bigoplus_{l=1}^k \sqrt{\frac{\lambda_l}{\lambda_1}}\widetilde{\ev}_{A_{1l}} & & \\
& \ddots & \\
& & \bigoplus_{l=1}^k \sqrt{\frac{\lambda_l}{\lambda_k}}\widetilde{\ev}_{A_{kl}}
\end{bmatrix}: A \otimes A^* \to \mathbbm{1}
\end{align}
The modified coevaluation maps are fixed by requiring the zig-zag relations hold. The choices for square roots are required to satisfy $\sqrt{\frac{\lambda_i}{\lambda_j}}\sqrt{\frac{\lambda_j}{\lambda_i}} =1$. If $\lambda$ consists of all positive real numbers, we pick the positive square roots. These duality maps only depend on $\lambda$ up to multiplication by a scalar.

In general, this pivotal structure is not spherical. We denote the pivotal category with this choice of pivotal structure $(M_k(\Vect), \lambda)$.

\subsection{Kazhdan-Wenzl Skein Theory}\label{sec:preKW}

In this subsection we describe (a slight modification of\footnote{The results of \cite{SovietHans} gives a presentation only allowing upwards pointing strands. We present a slight generalisation here which allows for strands in any orientation. As such, we have to give slight extensions to the results of \cite{SovietHans} throughout this section. These extensions are all routine.}) the Kazhdan-Wenzl presentation \cite{SovietHans} for the tensor category $\overline{\operatorname{Rep}(U_q(\mathfrak{sl}_N))^\omega}$.

We begin by presenting the pre-semisimplified version of this category. Let $N\in \mathbb{N}_{\geq 2}$, $q\in \mathbb{C}$, and $\omega$ an $N$-th root of unity, and let $\Lambda_1\in \operatorname{Rep}(U_q(\mathfrak{sl}_N))^\omega$ be the ``vector representation''. Our first goal is to describe the oriented $\dag$-planar algebra $\mathcal{P}_{\operatorname{Rep}(U_q(\mathfrak{sl}_N))^\omega;\Lambda_1}$. The results of \cite{SovietHans} give generators for this planar algebra.

\begin{lem}\cite{SovietHans}
The oriented $\dag$-planar algebra $\mathcal{P}_{\operatorname{Rep}(U_q(\mathfrak{sl}_N))^\omega;\Lambda_1}$ is generated by the projection \[\raisebox{-.5\height}{ \includegraphics[scale = .3]{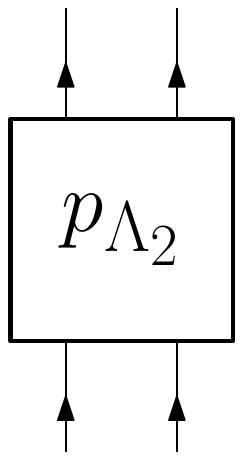}}\in \Hom_{\operatorname{Rep}(U_q(\mathfrak{sl}_N))^\omega}(\Lambda_1^{\otimes 2} \to\Lambda_1^{\otimes 2})\] onto $\Lambda_2$, along with an element (unique up to scalar)
\[\raisebox{-.5\height}{ \includegraphics[scale = .25]{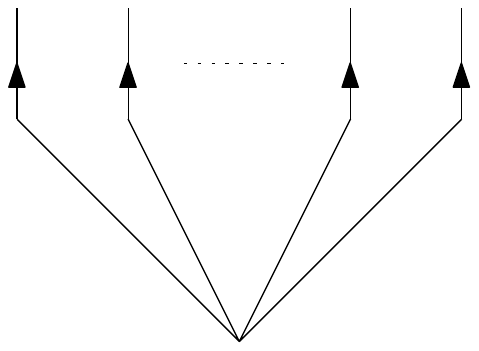}} \in \Hom_{\operatorname{Rep}(U_q(\mathfrak{sl}_N))^\omega}(\mathbbm{1} \to \Lambda_1^{\otimes N}).\]
\end{lem}
\begin{proof}
It is shown in \cite[Theorem 4.1 and Proposition 2.2]{SovietHans} that these morphisms generate all the spaces $\Hom_{\operatorname{Rep}(U_q(\mathfrak{sl}_N))^\omega}(\Lambda_1^{\otimes n} \to\Lambda_1^{\otimes m})$ with $n,m\geq 0$. This result then extends to all spaces in $\mathcal{P}_{\operatorname{Rep}(U_q(\mathfrak{sl}_N))^\omega;\Lambda_1}$ (which recall allows homs from arbitrary strings of $\Lambda_1$ and $\Lambda_1^*$) using the rigidity maps, and the element 
\[ \left(1 + q^{-2}\right)\raisebox{-.5\height}{ \includegraphics[scale = .3]{pl2.pdf}} -\raisebox{-.5\height}{ \includegraphics[scale = .3]{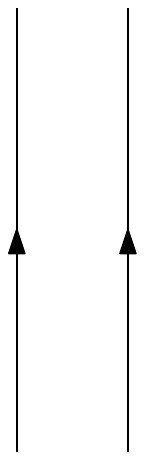}}    \]
which is a braiding on the subcategory generated by $\raisebox{-.5\height}{ \includegraphics[scale = .3]{pl2.pdf}}$.
\end{proof}
We will write $ \raisebox{-.5\height}{ \includegraphics[scale = .25]{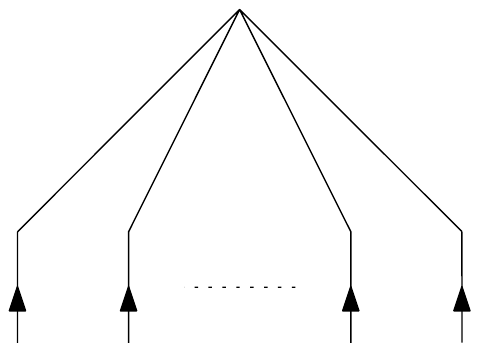}}$ for $\left( \raisebox{-.5\height}{ \includegraphics[scale = .25]{triv.pdf}}\right)^\dag$. To simplify our relations, we use the rescaled generator
\[ \raisebox{-.5\height}{ \includegraphics[scale = .3]{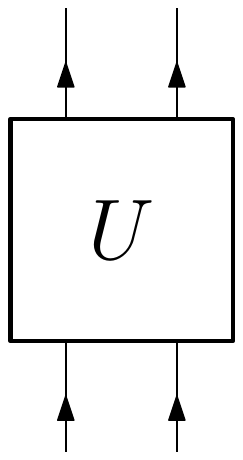}}:=[2]_q\raisebox{-.5\height}{ \includegraphics[scale = .3]{pl2.pdf}}    \]
in our presentation instead of the projection. These generators satisfy the following relations:
\begin{align*}
(\textrm{R1}) :  &\quad\raisebox{-.5\height}{ \includegraphics[scale = .3]{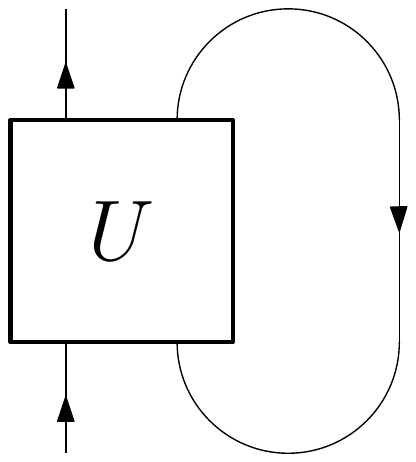}} = \raisebox{-.5\height}{ \includegraphics[scale = .3]{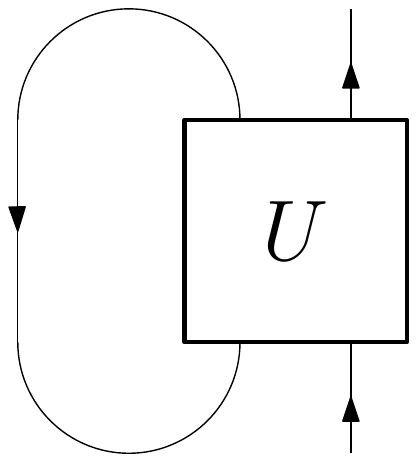}}= [N-1]_q\raisebox{-.5\height}{ \includegraphics[scale = .3]{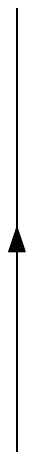}}  & (\textrm{R2}) :  &\quad \left(\raisebox{-.5\height}{ \includegraphics[scale = .3]{UU.pdf}}\right)^\dag = \raisebox{-.5\height}{ \includegraphics[scale = .3]{UU.pdf}}\\
(\textrm{R3}) :  &\quad\raisebox{-.5\height}{ \includegraphics[scale = .3]{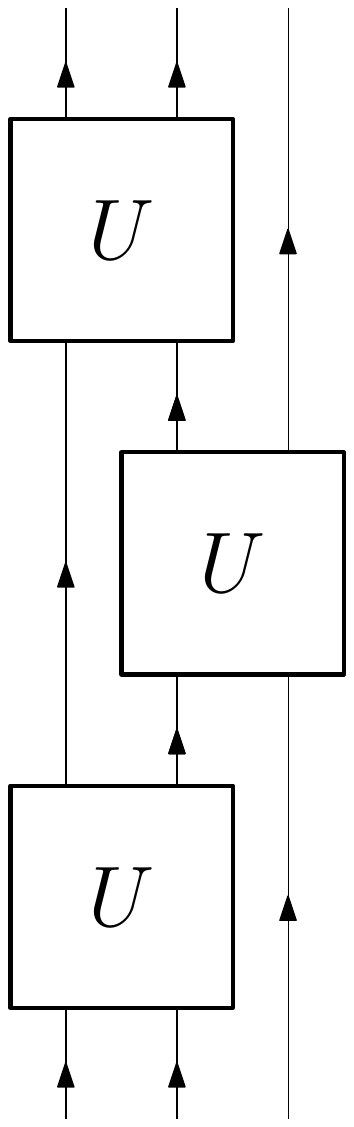}}  -\raisebox{-.5\height}{ \includegraphics[scale = .3]{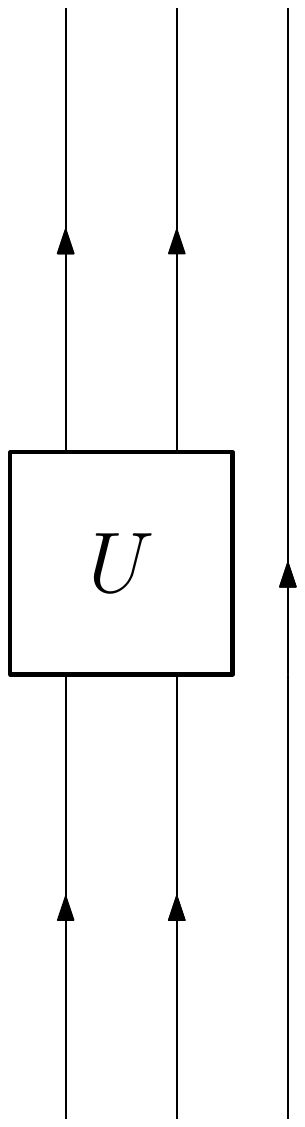}}  =\raisebox{-.5\height}{ \includegraphics[scale = .3]{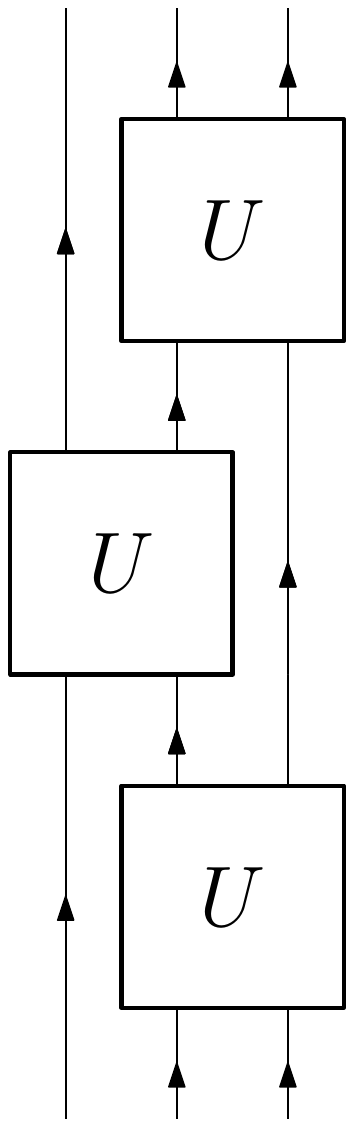}}  -\raisebox{-.5\height}{ \includegraphics[scale = .3]{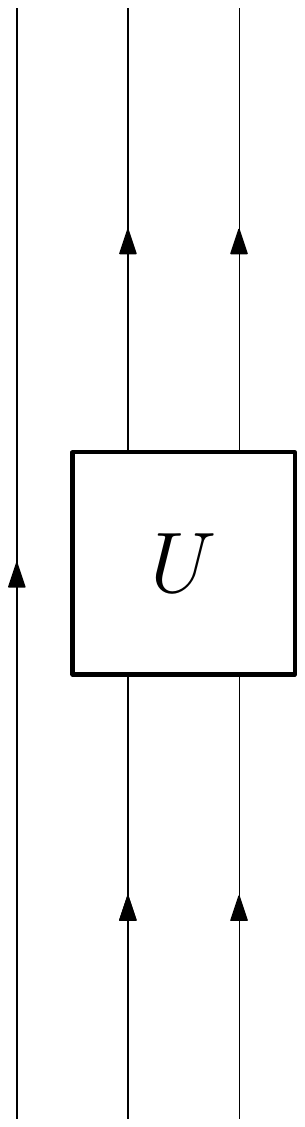}} &(\textrm{Hecke}) :  &\quad\raisebox{-.5\height}{ \includegraphics[scale = .3]{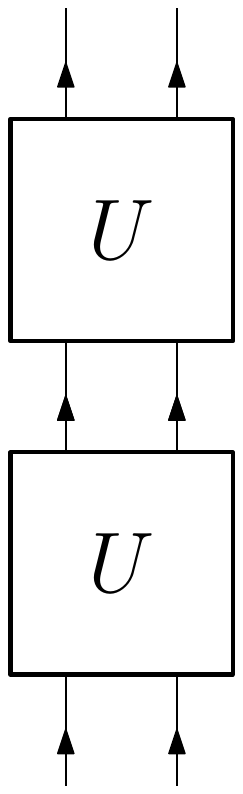}}= [2]_q\raisebox{-.5\height}{ \includegraphics[scale = .3]{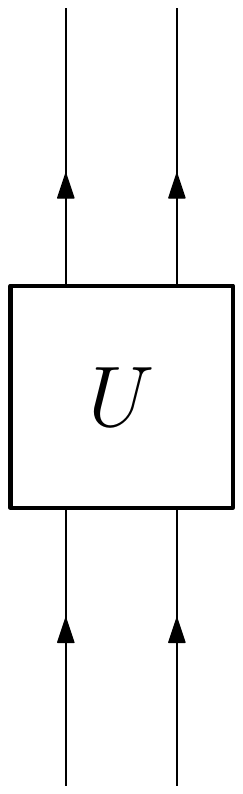}}\\
(\textrm{Over Braid}) :  &\quad\omega\raisebox{-.5\height}{ \includegraphics[scale = .3]{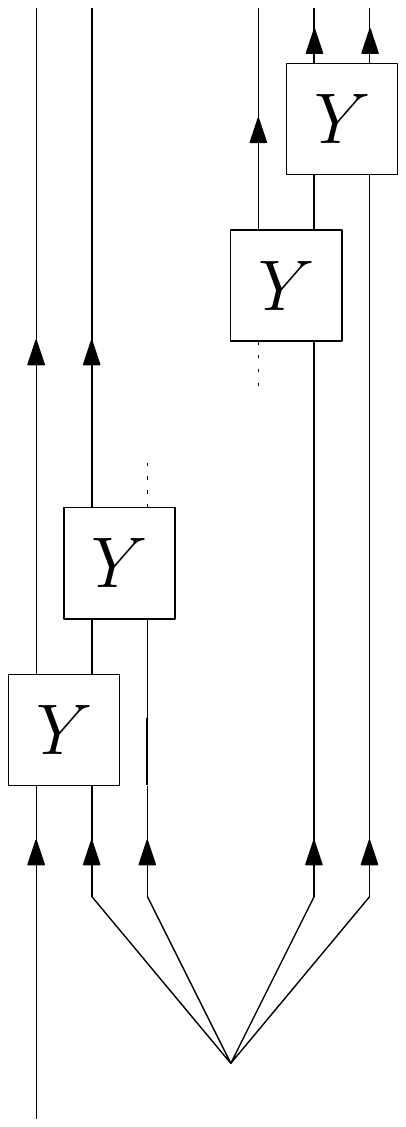}}= (-1)^N q^{1-N}\raisebox{-.5\height}{ \includegraphics[scale = .3]{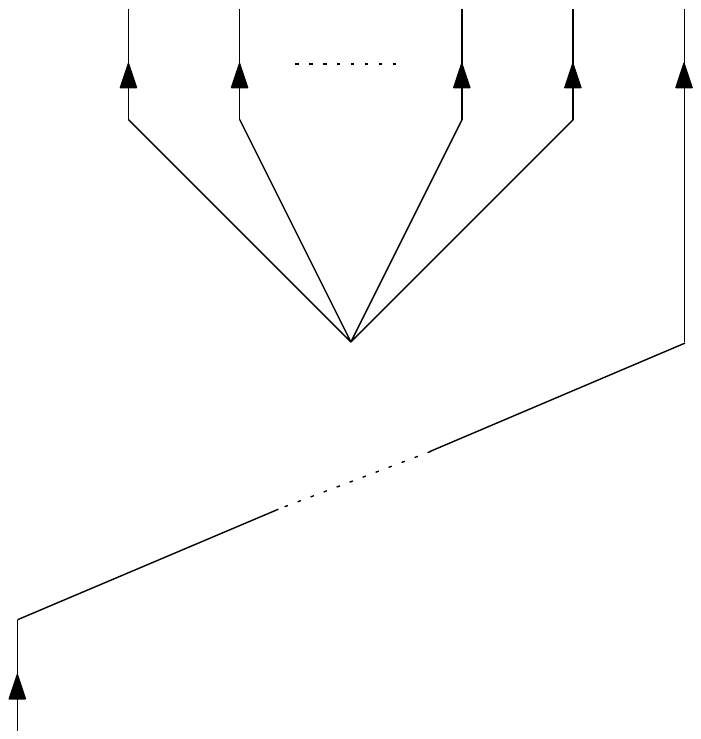}}&(\textrm{Anti-Sym 1}) : &\quad \raisebox{-.5\height}{ \includegraphics[scale = .3]{antiSymL.pdf}}= \raisebox{-.5\height}{ \includegraphics[scale = .3]{antisymR.pdf}}\\
(\textrm{Anti-Sym 2}) : &\quad \raisebox{-.5\height}{ \includegraphics[scale = .3]{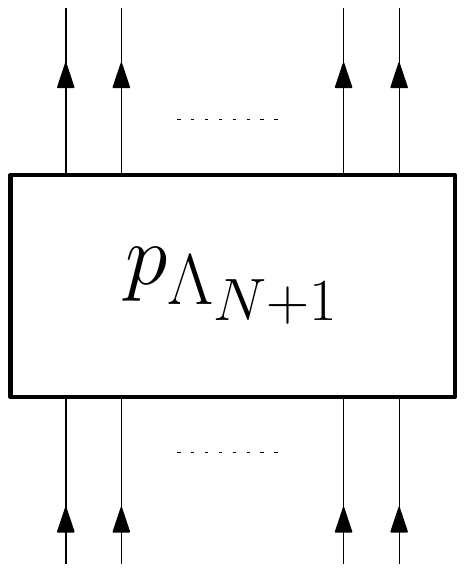}}=0&(\textrm{oR2})\footnotemark : &\quad \raisebox{-.5\height}{ \includegraphics[scale = .3]{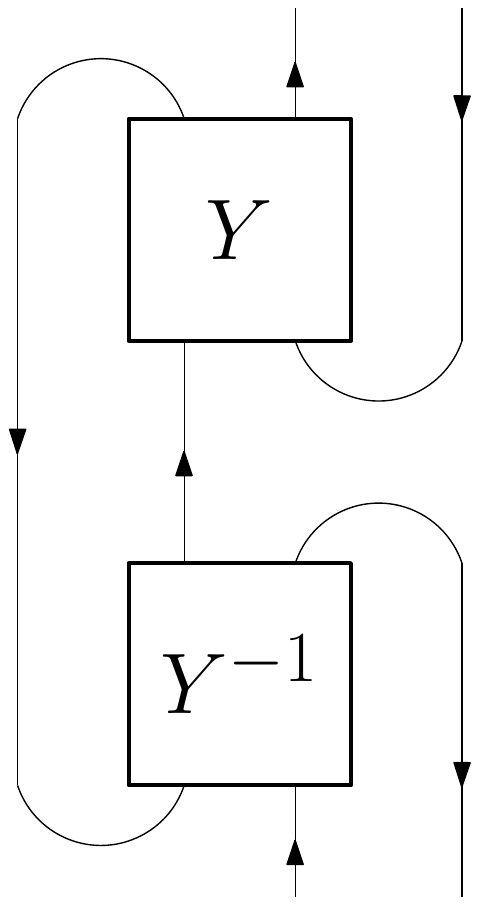}}= \raisebox{-.5\height}{ \includegraphics[scale = .3]{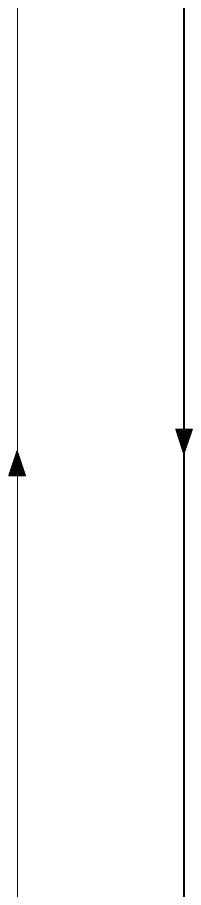}}\\
\end{align*}
\footnotetext{This relation actually follows from the other 7 relations. We include it to make evaluation arguement simpler. Surprisingly, this relation doesn't follow from the first four relations.}
Here $Y$ is the invertible element
\[     \raisebox{-.5\height}{ \includegraphics[scale = .3]{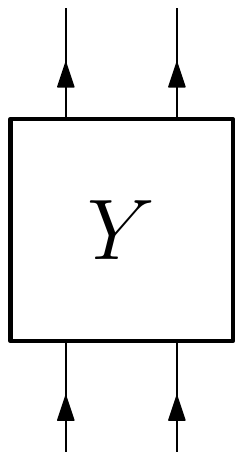}} :=  \frac{1}{q} \raisebox{-.5\height}{ \includegraphics[scale = .3]{UU.pdf}}-\raisebox{-.5\height}{ \includegraphics[scale = .3]{id.pdf}}. \]

\begin{remark}\label{rmk:braid}
Note that if $\omega=1$, and a choice of $N$-th root of $q$ is picked (i.e. a choice of braiding on $\operatorname{Rep}(U_q(\mathfrak{sl}_N))$), then $Y$ is a scalar multiple of the braid $\sigma_{V,V}$. This scalar is $-q^{\frac{ N-1}{N}}$, and the twist of $V$ with respect to this braiding is $q^{\frac{N^2-1}{N}}$. In particular this element is a braiding on the subcategory generated by $U$ for all $\omega$.
\end{remark}
The elements $p_{\Lambda_i}$ are the projections in $\operatorname{End}_{\operatorname{Rep}(U_q(\mathfrak{sl}_N))^\omega}\left(\Lambda_1^{\otimes i}\right)$ onto the simple objects $\Lambda_i$. These projections can be defined recursively \cite[Theorem 6]{Paggo} by
\[   \raisebox{-.5\height}{ \includegraphics[scale = .3]{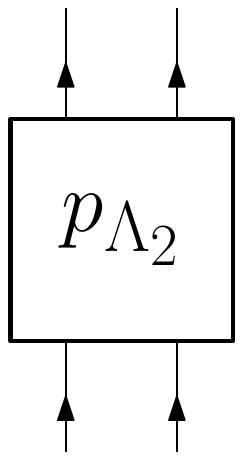}} := \frac{1}{[2]_q} \raisebox{-.5\height}{ \includegraphics[scale = .3]{UU.pdf}}, \quad\text{and} \quad\raisebox{-.5\height}{ \includegraphics[scale = .3]{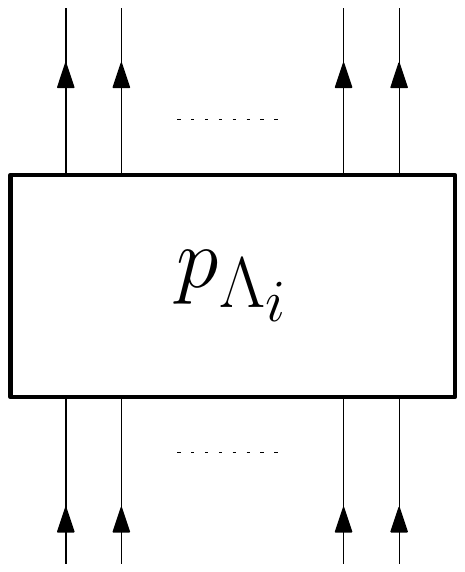}} :=\frac{1}{\sum_{j=0}^{i-1} q^{-2j}} \left( \raisebox{-.5\height}{ \includegraphics[scale = .3]{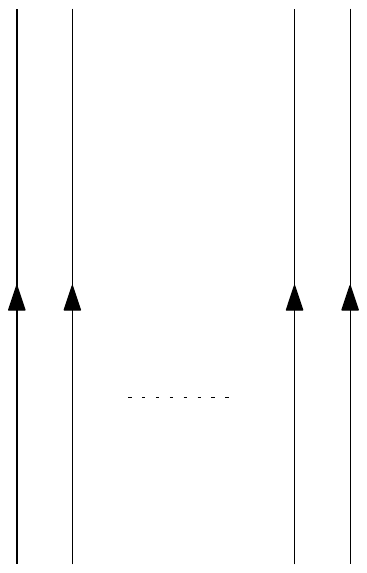}} +\raisebox{-.5\height}{ \includegraphics[scale = .3]{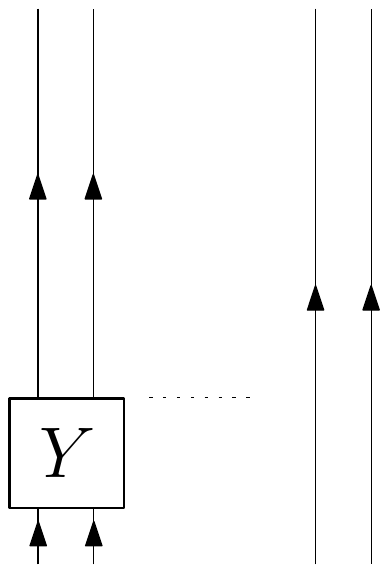}}+\raisebox{-.5\height}{ \includegraphics[scale = .3]{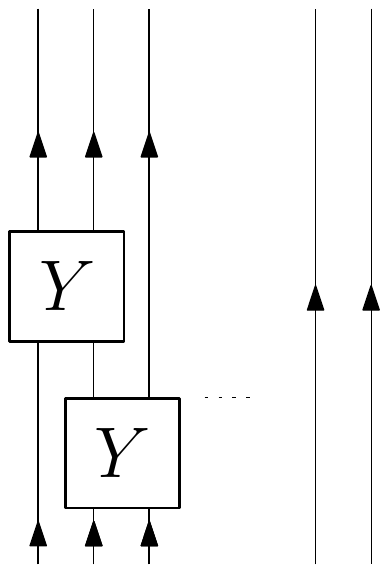}}+\cdots + \raisebox{-.5\height}{ \includegraphics[scale = .3]{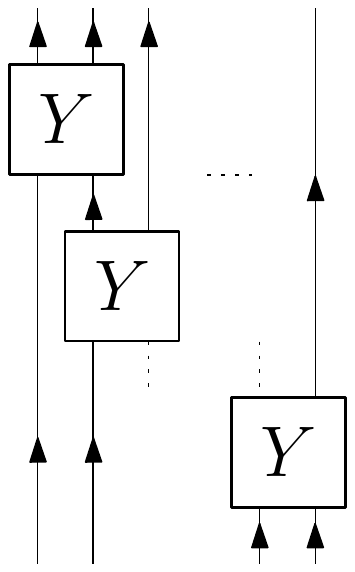}}   \right)\circ \raisebox{-.5\height}{ \includegraphics[scale = .3]{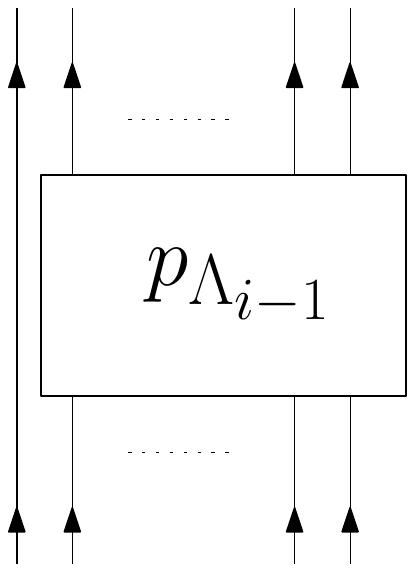}}.  \]
\begin{remark}\label{rem:ASdag}
    From the above recursive definition, along with relations (R1), (R2), (R3), and (Hecke), we can inductively compute that $p_{\Lambda_i}^\dag = p_{\Lambda_i}$, $\operatorname{tr}(p_{\Lambda_N}) = 1$, $\operatorname{tr}(p_{\Lambda_{N+1}}) = 0$, and $p_{\Lambda_i}\circ p_{\Lambda_i} = p_{\Lambda_i}$. We leave these computations to a bored reader.
\end{remark}

\begin{remark}
The decision to describe the Hecke algebra portion of the skein theory in $U$ basis, as opposed to $Y$ basis (as in \cite{SovietHans}) is entirely a practical one. We have found in concrete examples that the image of $U$ embedded inside a graph planar algebra has nicer coefficients than the embedding of $Y$. This is because the eigenvalues of $U$ are $0$ and $[2]_q$, while the eigenvalues of $Y$ are $-1$ and $q^2$. Note that we have kept the relation names corresponding to the $Y$ basis.
\end{remark}

Apart from (R1) and (oR2), these relations can all be found in \cite{SovietHans}. The relation (R1) follows from the quantum dimension of $\Lambda_2$ being $\frac{\q{N}\q{N-1}}{\q{2}}$. The relation (oR2) can be deduced from the other relations. As this presentation contains more hom spaces than the presentation originally given in \cite{SovietHans}, we have to prove that we have given sufficient relations. This is a fairly routine evaluation argument which we neglect to write down.

In this paper, we are interested in the semi-simple quotient of $\operatorname{Rep}(U_q(\mathfrak{sl}_N))^\omega$. That is, the category $\overline{
\operatorname{Rep}(U_q(\mathfrak{sl}_N))^\omega}$ in the notation of \cite{Simp}. This category is obtained by quotienting out by the negligible ideal \cite[Definition 2.1]{Simp} of $\operatorname{Rep}(U_q(\mathfrak{sl}_N))^\omega$. The planar algebra $\mathcal{P}_{\overline{
\operatorname{Rep}(U_q(\mathfrak{sl}_N))^\omega},\Lambda_1} 
$ is equal (by definition) to $\overline{\mathcal{P}_{
\operatorname{Rep}(U_q(\mathfrak{sl}_N))^\omega,\Lambda_1} }
$
where $\overline{\mathcal{P}}$ is the planar quotient by the planar ideal of negligible elements. In this paper we do not have to worry about explicitly describing the negligible ideal of $\operatorname{Rep}(U_q(\mathfrak{sl}_N))^\omega$. This is because we will be working with unitary planar algebras, whose canonical inner-product is positive definite. In this setting, negligible elements are necessarily $0$.

It is shown in \cite{JamsHans} that if $q = e^{2\pi i \frac{1}{2(N+k)}}$ then $\overline{\operatorname{Rep}(U_q(\mathfrak{sl}_N))}$ is unitary. This immediately implies that $\overline{\operatorname{Rep}(U_q(\mathfrak{sl}_N))^\omega}$ is unitary for any choice of $\omega$. This is because the twisting is equivalent to twisting the associator by an element of $H^3(\mathbb{Z}_N, U(1))$. The reader may be interested in the case when $q$ is not of the above form. For this, we provide the following result which shows that any such category is always Galois conjugate to a unitary one (and in particular has the same representation theory as the unitary one).

\begin{lem}\label{lem:Galois}
Let $N\geq 2$, $\omega$ an $N$-th root of unity, and $q$ a root of unity of order $2(N+k)$. Then $\overline{
\operatorname{Rep}(U_q(\mathfrak{sl}_N))^\omega}$ is Galois conjugate to $\overline{
\operatorname{Rep}(U_{q'}(\mathfrak{sl}_N))^{\omega'}}$ where $q' = e^{2\pi i \frac{1}{2(N+k)}}$ and $\omega'$ is some $N$-root of unity.
\end{lem}
\begin{proof}
From \cite{SovietHans}, the category $\overline{\operatorname{Rep}(U_q(\mathfrak{sl}_N))^\omega}$ can be defined over $\mathbb{Q}[q, \omega]$, and hence over $\mathbb{Q}[e^{2\pi i \frac{1}{\operatorname{LCM}(2(N+k), N)}}]$. By the Chinese remainder theorem, we have that $\mathbb{Z}_{  \operatorname{LCM}(2(N+k), N)   } \cong \mathbb{Z}_{ 2(N+k)  }\times \mathbb{Z}_{\frac{\operatorname{LCM}(2(N+k), N)}{2(N+k)}  }$. As $q$ is a primitive $2(N+k)$-th root of unity, there exists an element of $\ell \in \mathbb{Z}_{ 2(N+k)  }^\times$ such that $q^\ell = e^{2\pi i \frac{1}{2(N+k)}}$. From the above isomorphism of groups, we get an element $\ell' \in \mathbb{Z}_{\operatorname{LCM}(2(N+k), N)}^\times$ such that $q^{\ell'} = e^{2\pi i \frac{1}{2(N+k)}}$. Hence we can Galois conjugate $\overline{\operatorname{Rep}(U_q(\mathfrak{sl}_N))^\omega}$ by $\ell'$ to obtain the category $\overline{\operatorname{Rep}(U_{ e^{2\pi i \frac{1}{2(N+k)}}}(\mathfrak{sl}_N))^{\omega'}}$ where $\omega' = \omega^{\ell'}$.
\end{proof}
\begin{remark}
Note that the above lemma implies that the inner product coming from the $\dag$ structure on $\overline{
\operatorname{Rep}(U_q(\mathfrak{sl}_N))^\omega}$ is non-degenerate for all $q$.
\end{remark}

In the unitary setting (and more generally when the inner product in non-degenerate), we obtain the relation (Anti-Sym 2) for free, and also that all negligibles are 0. This gives the following result, which is essentially shown in \cite{SovietHans}.
\begin{lem}\cite{SovietHans}\label{lem:KW}
    Let $\mathcal{P}$ be an oriented unitary planar algebra generated by morphisms 
    \[  \raisebox{-.5\height}{ \includegraphics[scale = .3]{UU.pdf}} \in \mathcal{P}_{++\to ++},\quad \text{and}\quad \raisebox{-.5\height}{ \includegraphics[scale = .3]{triv.pdf}} \in \mathcal{P}_{\emptyset\to (+)^N}    \]
    satisfying relations (R1), (R2), (R3), (Hecke), (Over Braid), and (Anti-Sym 1). Then 
    \[   \mathcal{P} \cong \mathcal{P}_{\overline{
\operatorname{Rep}(U_q(\mathfrak{sl}_N))^\omega},\Lambda_1}.  \]
\end{lem}
\begin{proof}
This is a standard technique. First note that (Anti-Sym 2) holds in $\mathcal{P}$ as $p_{\Lambda_{N+1}}^\dag =p_{\Lambda_{N+1}}$, and so $\langle p_{\Lambda_{N+1}}, p_{\Lambda_{N+1}}\rangle = \operatorname{tr}(p_{\Lambda_{N+1}}) = 0$ as the inner product is positive definite. Similarly any negligible element $f\in\mathcal{P}$ has $\langle f, f\rangle = 0$, and hence is 0. We thus have a surjective map
\[  \mathcal{P} \to \mathcal{P}_{\overline{
\operatorname{Rep}(U_q(\mathfrak{sl}_N))^\omega},\Lambda_1}. \]
Applying \cite[Proposition 3.5]{EHOld}, the kernel of this map is a sub-ideal of the negligible ideal of $\mathcal{P}$ as $\mathcal{P}_{\emptyset}$ is 1-dimensional. Hence the kernel is trivial. Thus the map is an isomorphism.
\end{proof}

\section{A Refinement of Kazhdan-Wenzl}\label{sec:KW}

There are two scary relations (from a GPA point of view) in the Kazhdan-Wenzl presentation for $\mathcal{P}_{\overline{
\operatorname{Rep}(U_q(\mathfrak{sl}_N))^\omega},\Lambda_1}$ described in Subsection~\ref{sec:preKW}. These are (Anti-Sym 1) and (Over Braid). This is due to the fact that to solve for (Anti-Sym 1) in the graph planar algebra of $\Gamma$, we have to consider loops of length $2N$ in $\Gamma$, as well as summing over all internal configurations of the morphism $p_{\Lambda_N}$. From a computational point of view this is impractical for large $N$. The relation (Over Braid) is not as bad, but still requires solving a degree $N$ polynomial. Again this will scale badly with $N$. 

The goal of this section is to replace the two relations (Anti-Sym 1) and (Over Braid) with simpler relations (from a GPA point of view). In our setting this means relations with fewer external boundary edges, and fewer internal faces. The main result of this section shows that we can achieve this with the three relations below.

\begin{lem}\label{lem:anti-sym}
We have the following relations in $\mathcal{P}_{\overline{\operatorname{Rep}(U_q(\mathfrak{sl}_N))^\omega};\Lambda_1}$ for all $q$:
\begin{align*}
&\raisebox{-.5\height}{ \includegraphics[scale = .3]{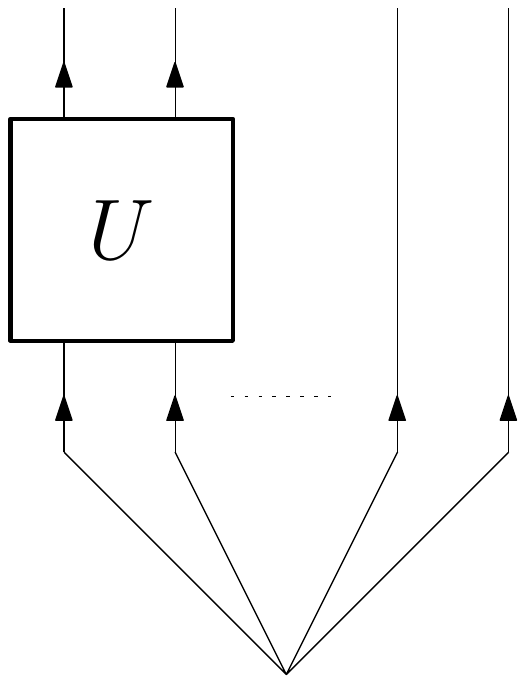}} = [2]_q \raisebox{-.5\height}{ \includegraphics[scale = .3]{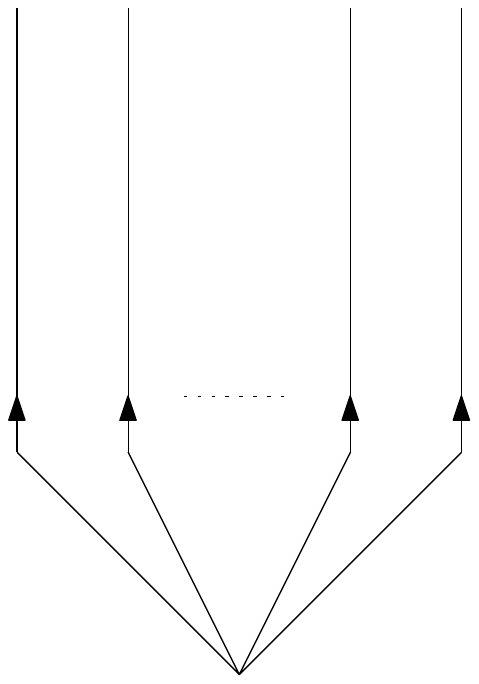}},  &\raisebox{-.5\height}{ \includegraphics[scale = .3]{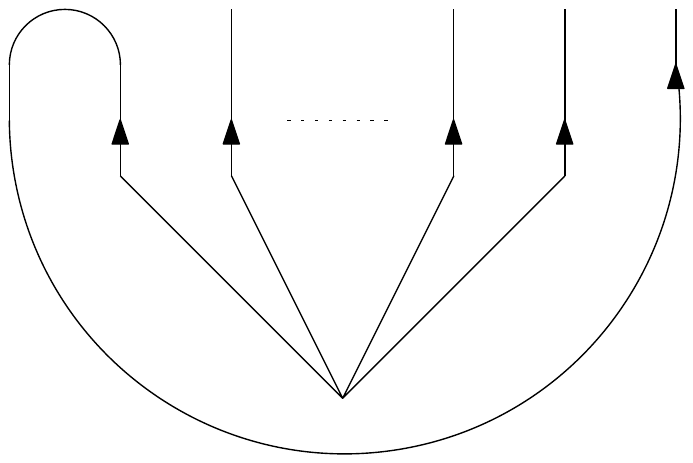}} =  (-1)^{N+1}\omega\raisebox{-.5\height}{ \includegraphics[scale = .3]{triv.pdf}}, &\qquad \raisebox{-.5\height}{ \includegraphics[scale = .3]{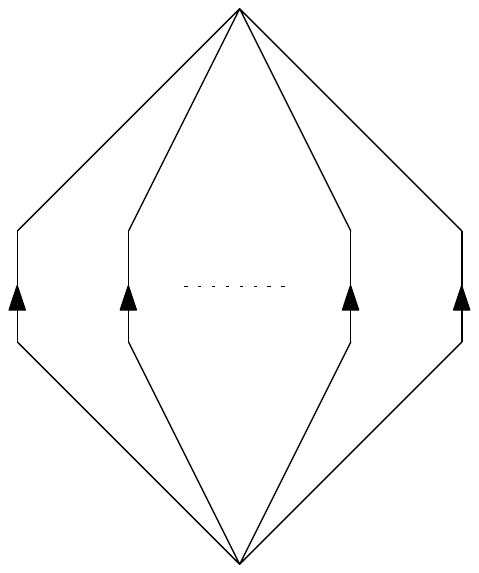}}=1\\
&(\textrm{Braid Absorption}) & (\textrm{Rotational Invariance}) & \qquad(\textrm{Norm})
\end{align*}
\end{lem}
\begin{proof}
The relation (Norm) comes from taking the categorical trace of relation (Anti-Sym 1). Then by glueing a $\raisebox{-.5\height}{ \includegraphics[scale = .3]{triv.pdf}}$ to the bottom of relation (Anti-Sym 1), we obtain
\[\raisebox{-.5\height}{ \includegraphics[scale = .3]{triv.pdf}} = \raisebox{-.5\height}{ \includegraphics[scale = .3]{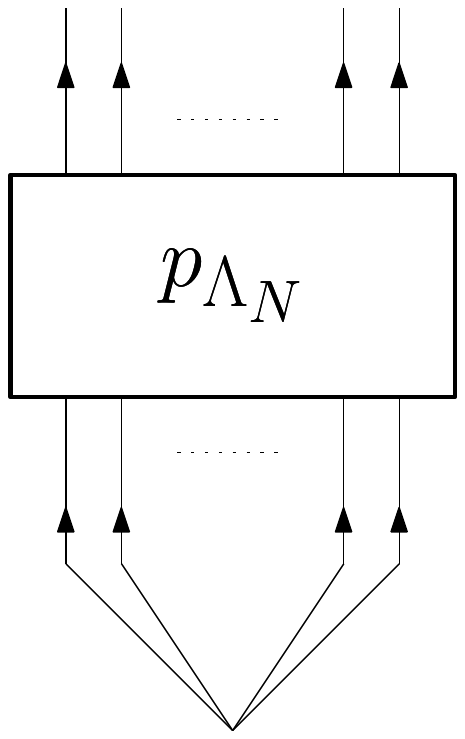}}.   \]
As sticking a $\raisebox{-.5\height}{ \includegraphics[scale = .3]{UU.pdf}}$ on top of $p_{\Lambda_N}$ gives $[2]_q p_{\Lambda_N}$ \cite[Theorem 4]{Paggo} (their crossing is $-q^2 Y$ in our basis), we get (Braid Absorption).

To get (Rotational Invariance), we take the left partial trace of (Over Braid) to get
\[ \omega \raisebox{-.5\height}{ \includegraphics[scale = .3]{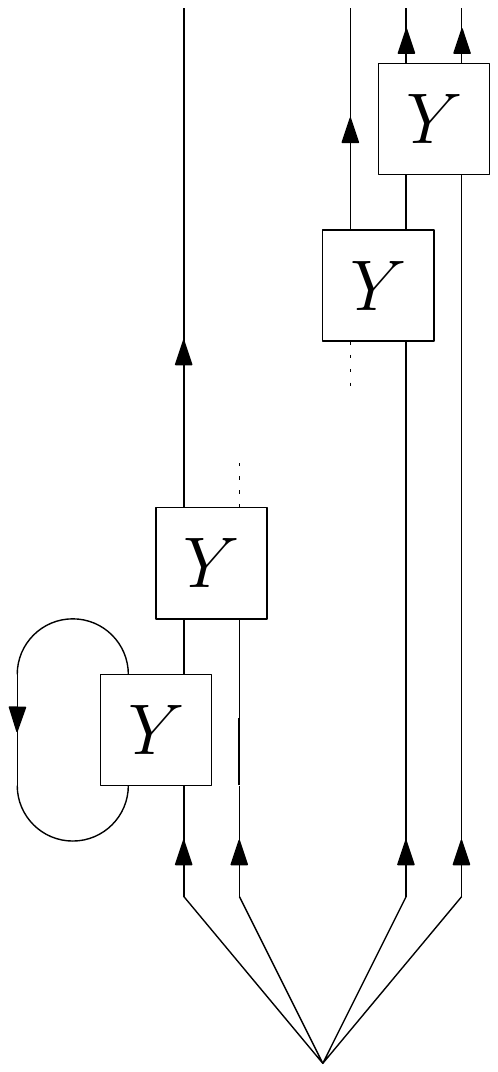}} = (-1)^Nq^{1-N}   \raisebox{-.5\height}{ \includegraphics[scale = .3]{exRel2.pdf}} . \]
The left hand side simplifies to $- q^{1-N}\raisebox{-.5\height}{ \includegraphics[scale = .3]{triv.pdf}}$ using relations (R1) and (Braid Absorption).
\end{proof}

With some work, we can show that in the non-classical case, (Over Braid) follows from these new relations.

\begin{lem}
Let $\mathcal{P}$ be an oriented planar algebra satisfying relations (R1), (R2), (R3), (Hecke), (Anti-Sym 1), (Anti-Sym 2), (Braid Absorption), (Rotational Invariance), and (Norm). Suppose $q^2\neq 1$, then $\mathcal{P}$ satisfies (Over Braid).
\end{lem}
\begin{proof}
From (Hecke) we get that $Y$ is invertible, with inverse
\[     \raisebox{-.5\height}{ \includegraphics[scale = .3]{YU.pdf}}^{-1} =  q \raisebox{-.5\height}{ \includegraphics[scale = .3]{UU.pdf}}-\raisebox{-.5\height}{ \includegraphics[scale = .3]{id.pdf}}. \]
We then have the equation
\[   q\raisebox{-.5\height}{ \includegraphics[scale = .3]{YU.pdf}}  -q^{-1}\raisebox{-.5\height}{ \includegraphics[scale = .3]{YU.pdf}}^{-1}  = (q^{-1} - q)\raisebox{-.5\height}{ \includegraphics[scale = .3]{id.pdf}}. \]

From \cite[Corollary 2.1]{Clam} \footnote{We have to be extremely careful here, as \cite{Clam} assumes the oriented Reidemeister move (oR2) that we do not have without the assumption of (Over Braid). However carefully working through the details of their proof show that only the relations we have listed are required. Morally the reason we don't require (oR2) is because this relation takes place in the Kazhdan-Wenzl subcategory where all strands are upwards pointing.
}that
\[ \raisebox{-.5\height}{ \includegraphics[scale = .3]{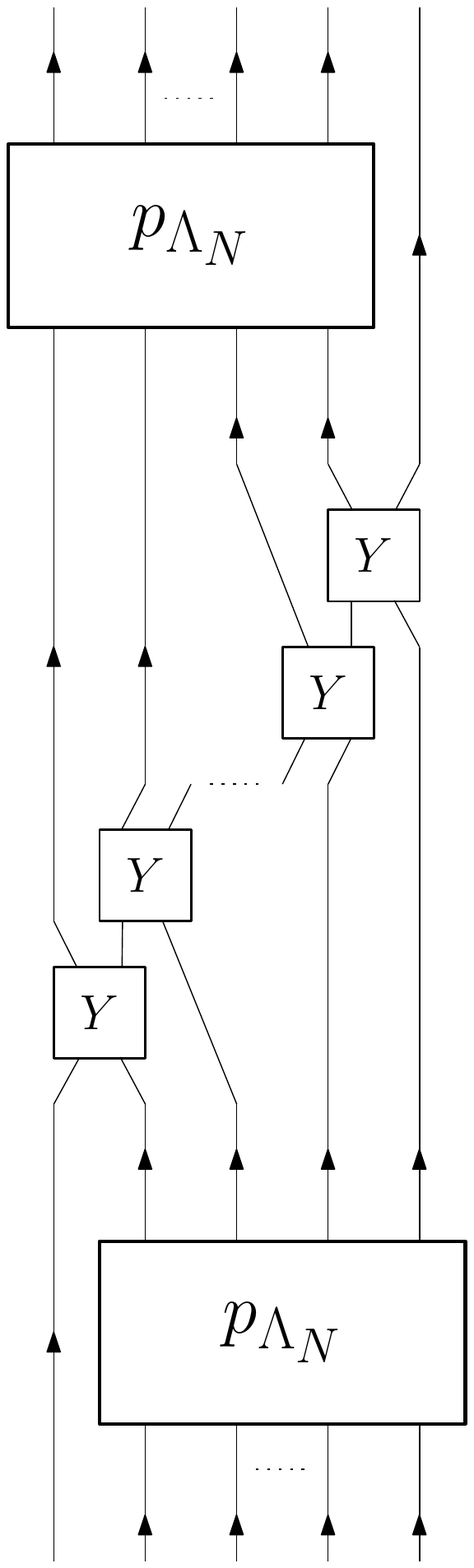}}\quad=\quad q^{2-2N}\raisebox{-.5\height}{ \includegraphics[scale = .3]{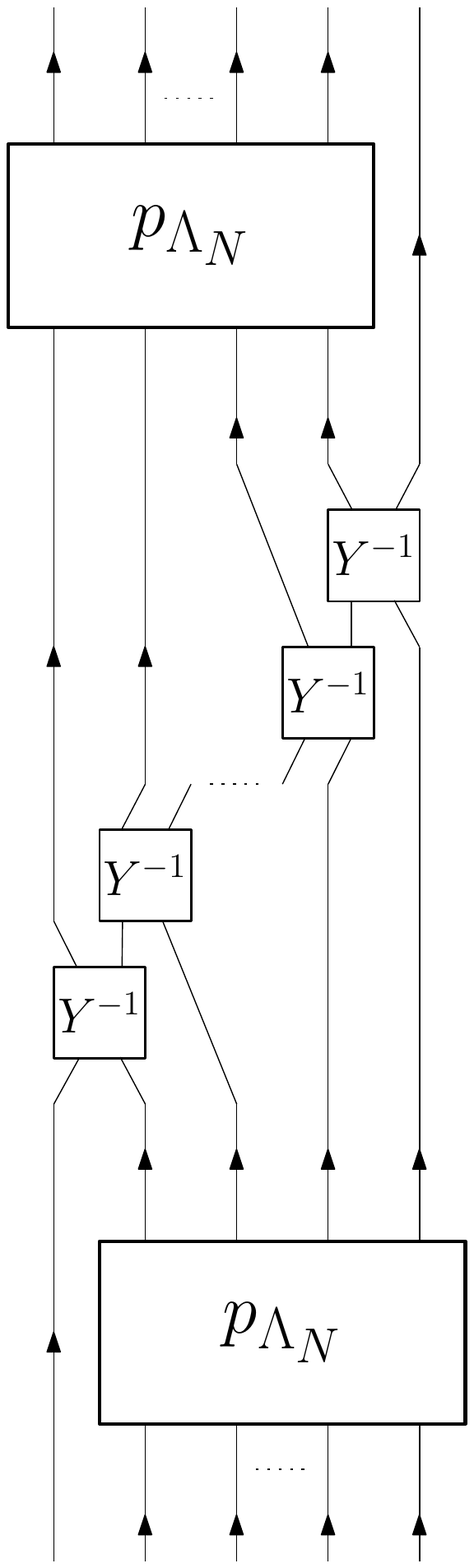}}  . \]
It follows from (Rotational Invariance) and (Braid Absorption) that $\raisebox{-.5\height}{ \includegraphics[scale = .3]{triv.pdf}}$ absorbs a $Y$ in any position at the cost of $\frac{1}{q^2}$. Using the recursive formula for $p_{\Lambda_N}$ this implies the relations
\[\raisebox{-.5\height}{ \includegraphics[scale = .3]{triv.pdf}} = \raisebox{-.5\height}{ \includegraphics[scale = .3]{newform.pdf}}\qquad \text{and}\qquad \raisebox{-.5\height}{ \includegraphics[scale = .3]{trivdag.pdf}} = \raisebox{-.5\height}{ \includegraphics[scale = .3]{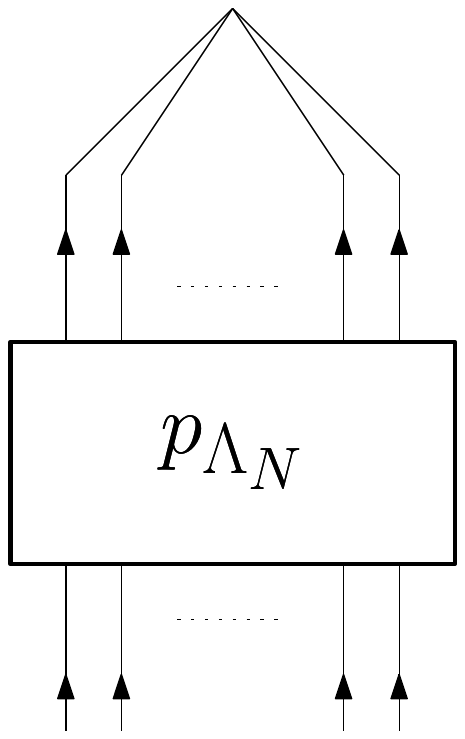}} .  \]
We can now compute
\[ \raisebox{-.5\height}{ \includegraphics[scale = .3]{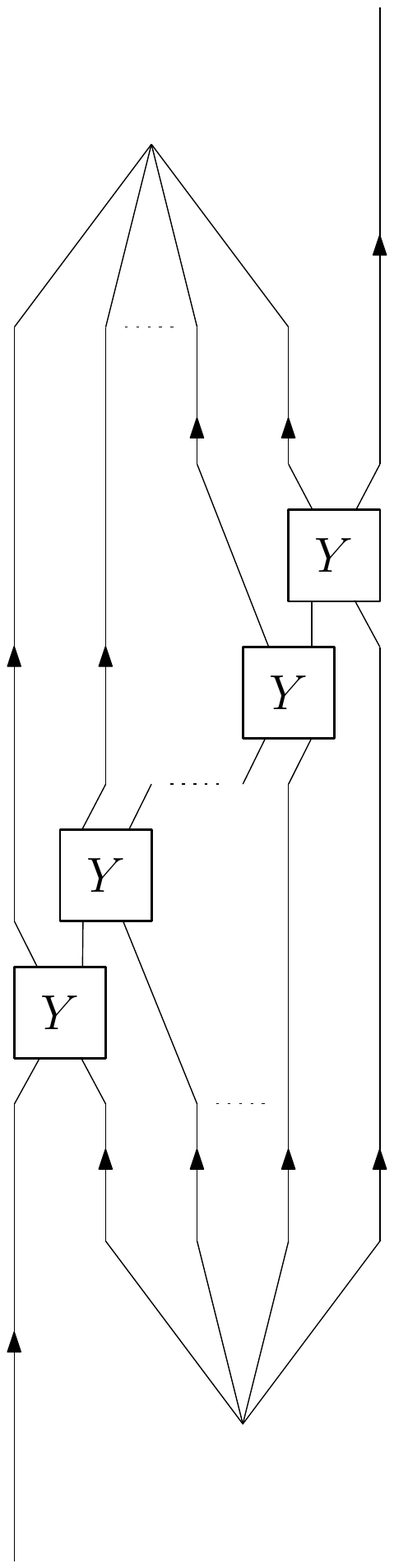}}  \quad =\quad \raisebox{-.5\height}{ \includegraphics[scale = .3]{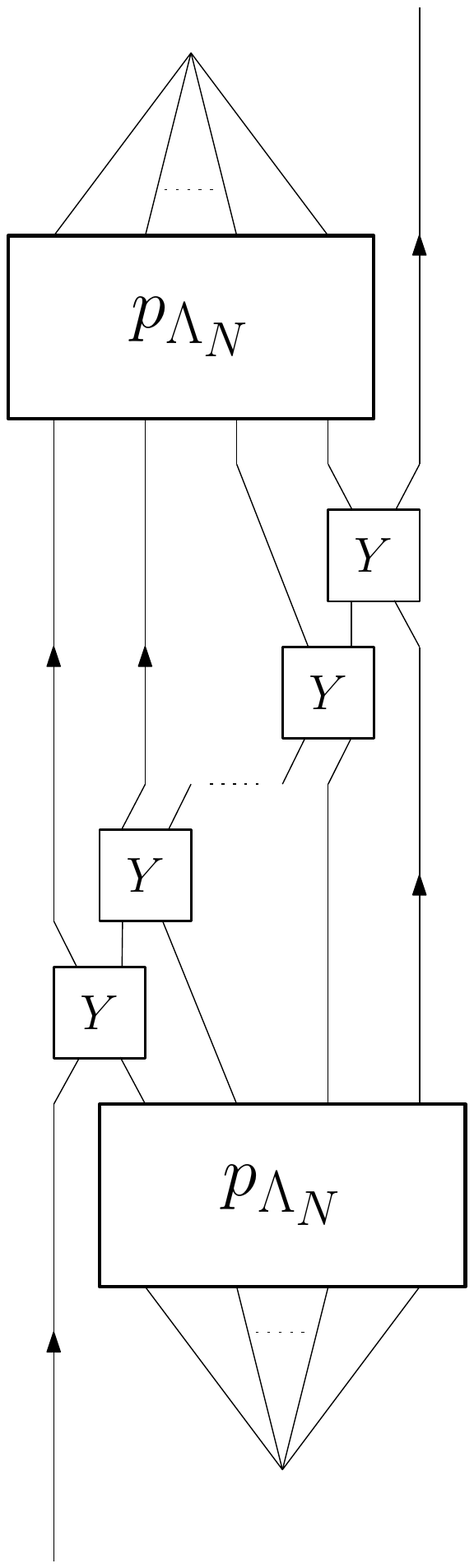}}\quad = \quad q^{2-2N}\raisebox{-.5\height}{ \includegraphics[scale = .3]{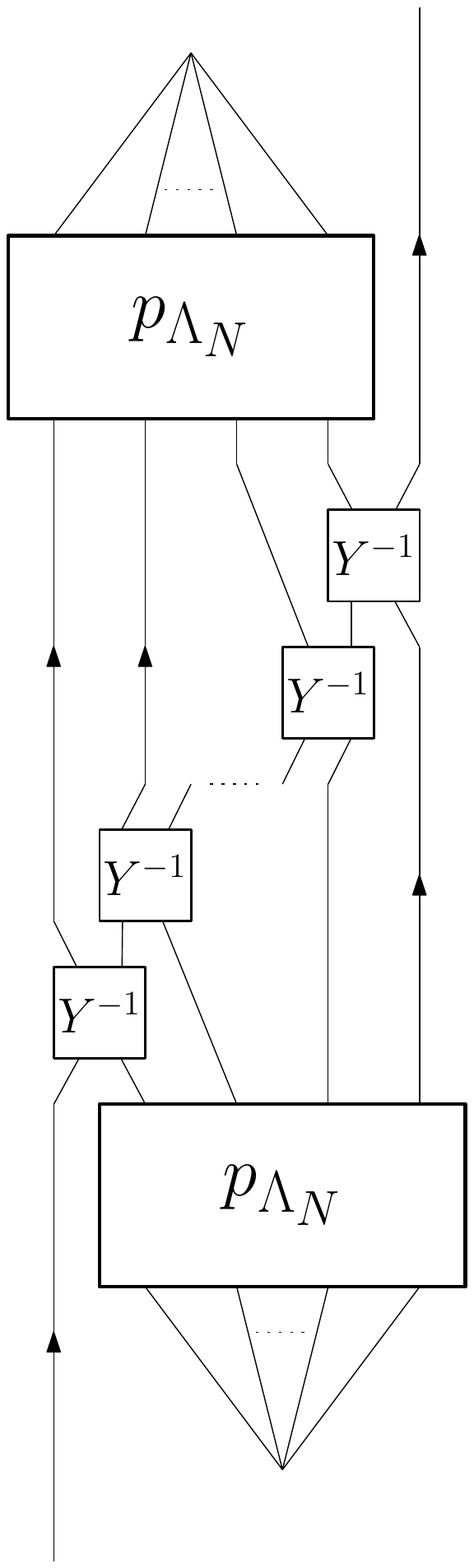}} \quad = \quad q^{2-2N} \raisebox{-.5\height}{ \includegraphics[scale = .3]{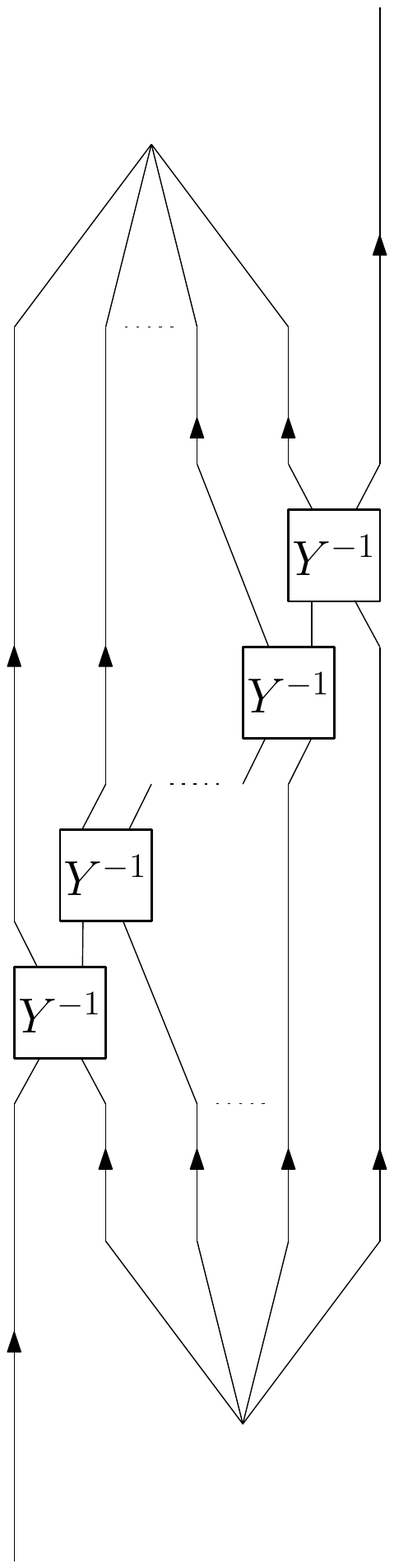}}.     \]

We also compute 
\[  \raisebox{-.5\height}{ \includegraphics[scale = .25]{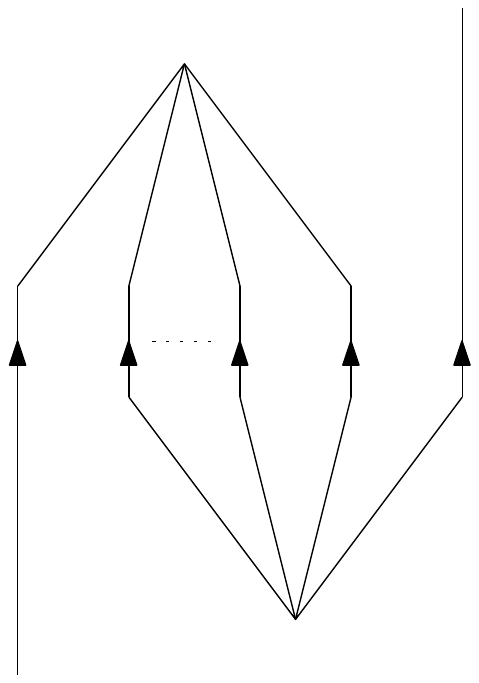}} = (-1)^{N+1}\overline{\omega}\raisebox{-.5\height}{ \includegraphics[scale = .25]{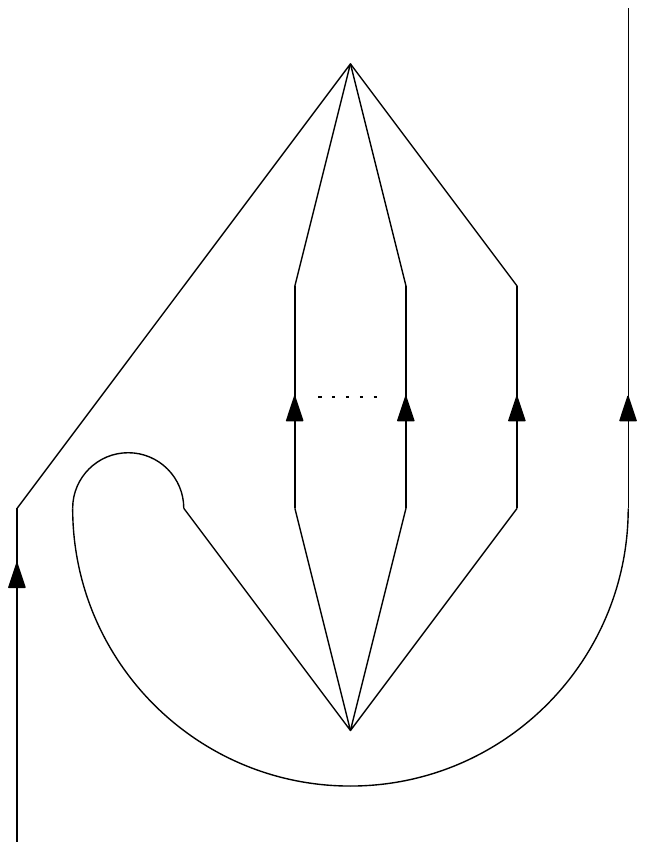}} =  (-1)^{N+1}\overline{\omega} \raisebox{-.5\height}{ \includegraphics[scale = .25]{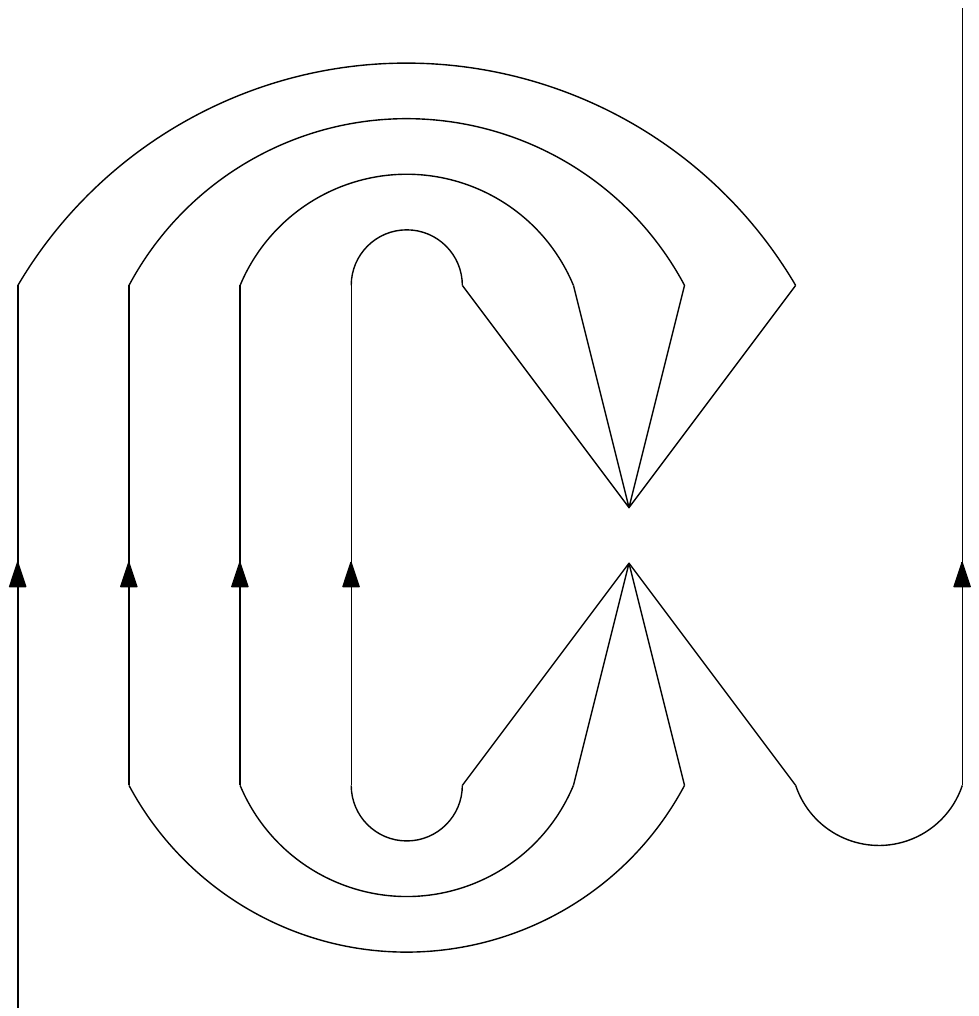}} =  (-1)^{N+1}\overline{\omega}\raisebox{-.5\height}{ \includegraphics[scale = .25]{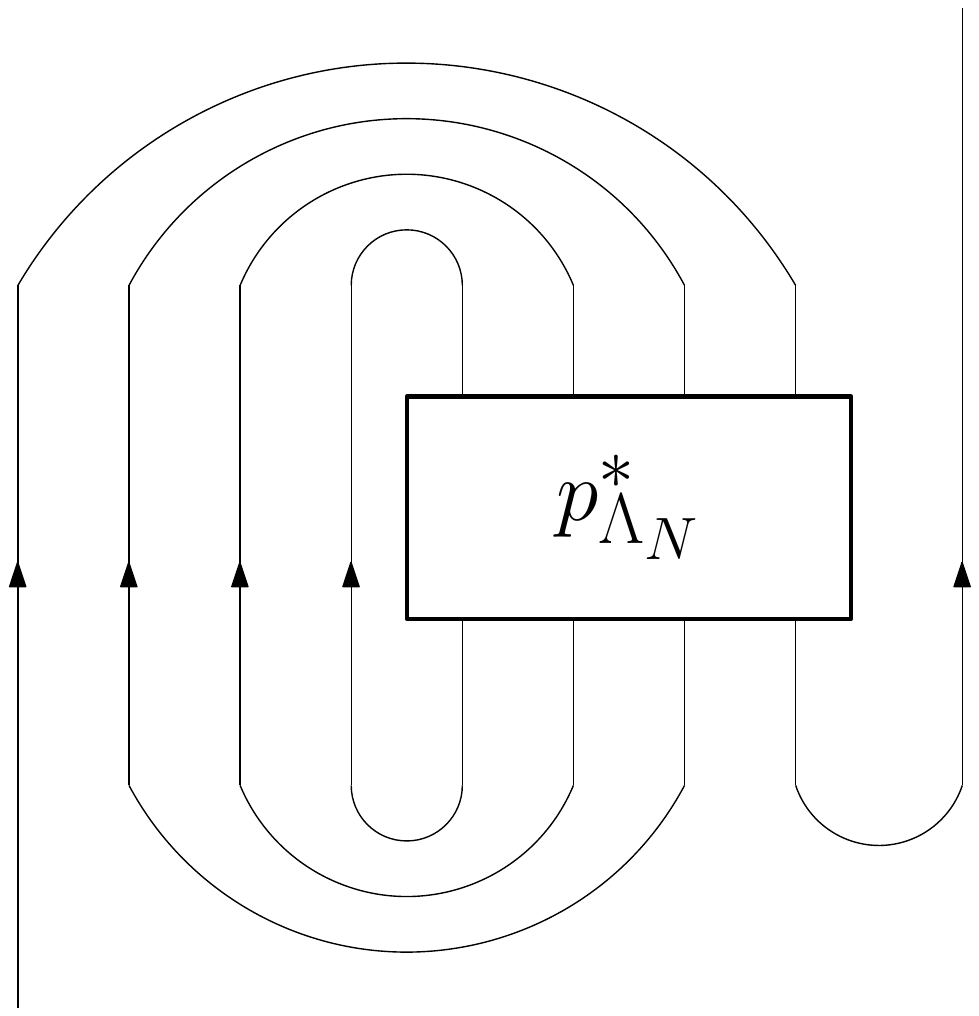}}=\frac{(-1)^{N+1}\overline{\omega}}{\q{N}}\raisebox{-.5\height}{ \includegraphics[scale = .25]{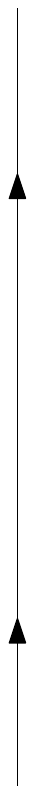}}.\] 
Here the first equality follows from (Rotational Invariance), the second is rigid isotopy, the third is from (Anti-Sym 1), and the fourth follows from (R1) and the recursive definition of $p_{\Lambda_N}$.

We now use the relation
\[   \raisebox{-.5\height}{ \includegraphics[scale = .3]{YU.pdf}}    = q^{-2}\raisebox{-.5\height}{ \includegraphics[scale = .3]{YU.pdf}}^{-1}+( q^{-2}-1)\raisebox{-.5\height}{ \includegraphics[scale = .3]{id.pdf}} \]
to compute
\[  \raisebox{-.5\height}{ \includegraphics[scale = .3]{blancheteq1.pdf}} = q^{-2N} \raisebox{-.5\height}{ \includegraphics[scale = .3]{blancheteq4.pdf}} + \sum  \raisebox{-.5\height}{\includegraphics[scale = .3]{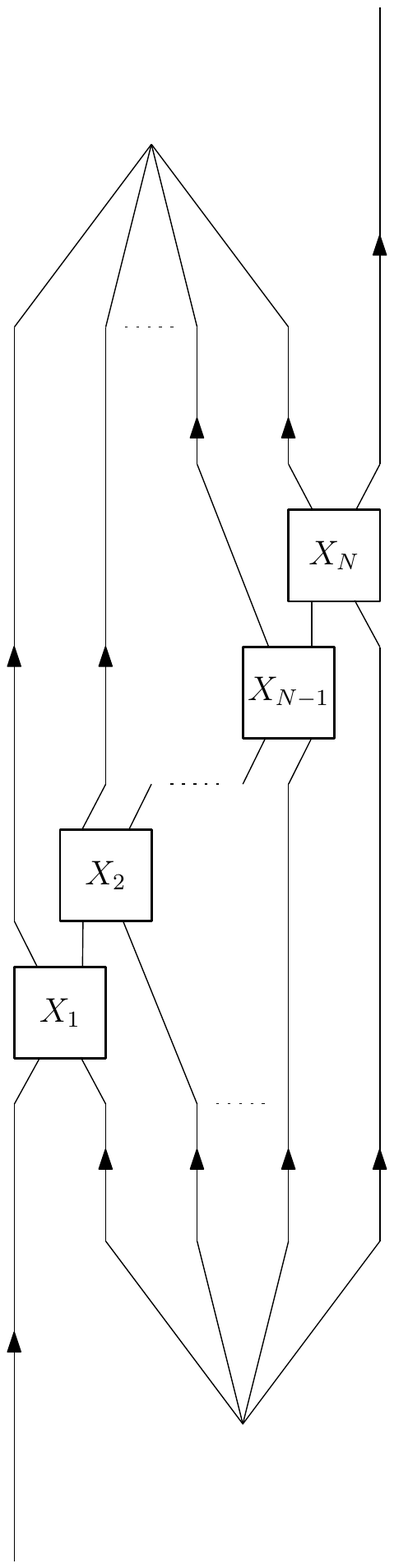}}\]
where 
\[X_i \in \left\{ q^{-2}\raisebox{-.5\height}{ \includegraphics[scale = .2]{YU.pdf}}^{-1},( q^{-2}-1)\raisebox{-.5\height}{ \includegraphics[scale = .2]{id.pdf}}   \right\}.\]
and the sum is taken over the $2^N - 1$ possibilities for $X_i$ where not all $X_i = q^{-2}Y^{-1}$. As each term in the above sum contains as least one $X_i$ with identity strands, we can use (Braid Absorption), along with the earlier relations to simplify the right hand side to obtain
\[ q^{-2} \raisebox{-.5\height}{ \includegraphics[scale = .3]{blancheteq1.pdf}} + \sum_{i=0}^{N-1} (q^{-2}-1)^{N-i}q^{-2i}q^{2i}\binom{N}{i} \raisebox{-.5\height}{\includegraphics[scale = .3]{shift1.pdf}}.\]
We can simplify the last term as follows
\begin{align*}  
  \sum_{i=0}^{N-1} (q^{-2}-1)^{N-i}q^{-2i}q^{2i}\binom{N}{i} \raisebox{-.5\height}{\includegraphics[scale = .3]{shift1.pdf}}
&=\frac{(-1)^{N+1}\overline{\omega}}{\q{N}}(q^{-2}-1)^{N}\sum_{i=0}^{N-1} \left(\frac{1}{q^{-2}-1}\right)^{i}\binom{N}{i}  \raisebox{-.5\height}{\includegraphics[scale = .3]{wrap7.pdf}}\\
&=\frac{(-1)^{N+1}\overline{\omega}}{\q{N}}(q^{-2}-1)^{N}\left(\left(1 +\frac{1}{q^{-2}-1}\right)^N - \left(\frac{1}{q^{-2}-1}\right)^N\right)  \raisebox{-.5\height}{\includegraphics[scale = .3]{wrap7.pdf}}\\
&=(-1)^{N+1}\overline{\omega}\frac{\left(q^{-N-1}-q^{N-1}  \right) \left(q^{2} -1 \right)}{q^{2N}-1}  \raisebox{-.5\height}{\includegraphics[scale = .3]{wrap7.pdf}}.
\end{align*}
With this simplification, we can simplify and rearrange the original equation to obtain
\[   \left(1-q^{-2}\right)\raisebox{-.5\height}{ \includegraphics[scale = .3]{blancheteq1.pdf}}=(-1)^{N+1}\overline{\omega}\frac{\left(q^{-N-1}-q^{N-1}  \right) \left(q^{2} -1 \right)}{q^{2N}-1}  \raisebox{-.5\height}{\includegraphics[scale = .3]{wrap7.pdf}}.   \]
As $q^2\neq 1$, we can divide to obtain
\[  \raisebox{-.5\height}{ \includegraphics[scale = .3]{blancheteq1.pdf}}=(-1)^{N+1}\overline{\omega}\frac{\left(q^{-N-1}-q^{N-1}  \right) \left(q^{2} -1 \right)}{\left(q^{2N}-1\right)\left( 1 - q^{-2} \right)}  \raisebox{-.5\height}{\includegraphics[scale = .3]{wrap7.pdf}}= (-1)^{N}\overline{\omega}q^{1-N}\raisebox{-.5\height}{\includegraphics[scale = .3]{wrap7.pdf}}. \]

Finally we get the desired
\[  (-1)^{N}\overline{\omega} q^{1-N}\raisebox{-.5\height}{\includegraphics[scale = .3]{OBR.pdf}}\quad=\quad  \raisebox{-.5\height}{\includegraphics[scale = .3]{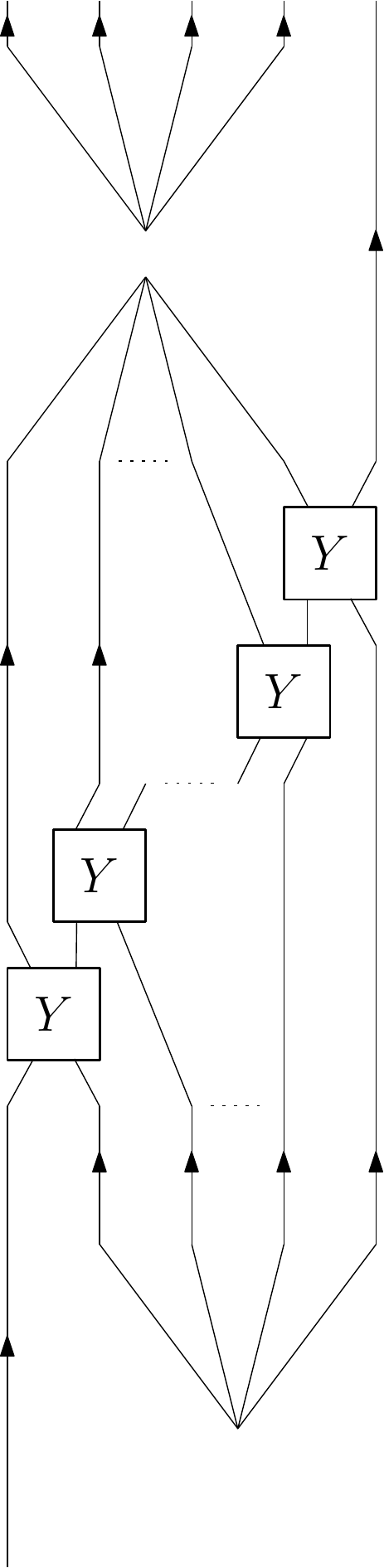}}\quad=\quad  \raisebox{-.5\height}{\includegraphics[scale = .3]{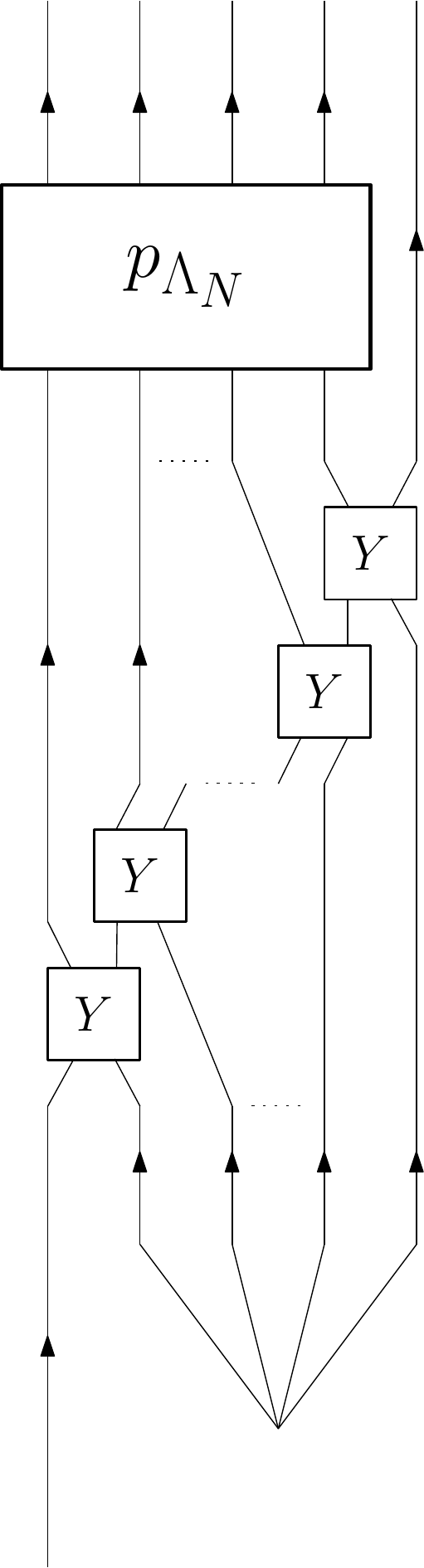}}\quad=\quad  \raisebox{-.5\height}{\includegraphics[scale = .3]{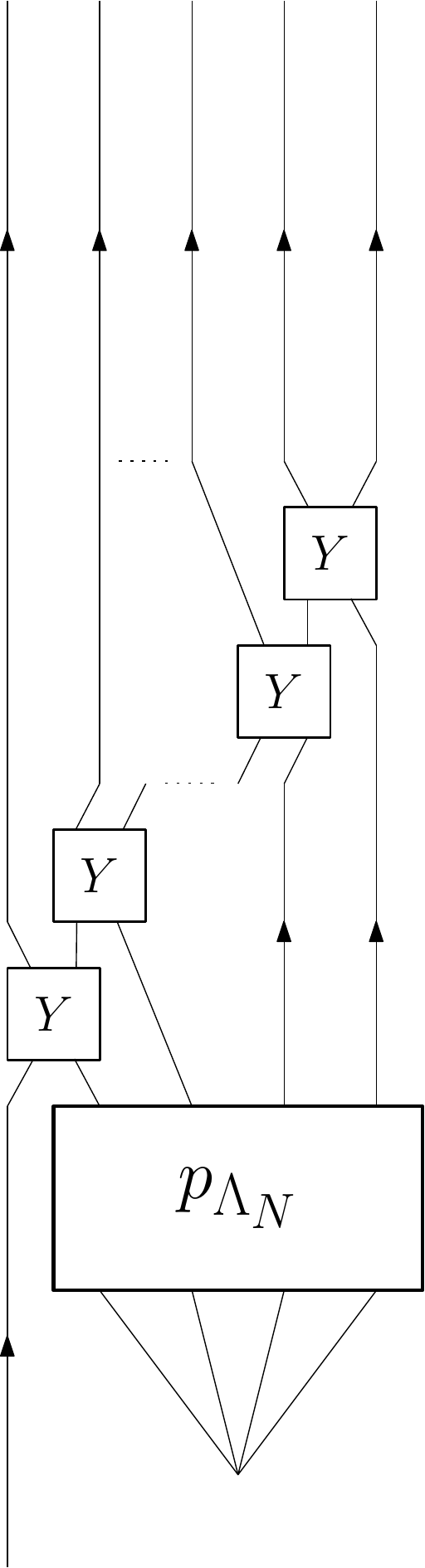}}\quad=\quad  \raisebox{-.5\height}{\includegraphics[scale = .3]{OBL.pdf}}\]
which is exactly (Over-Braid) up to rearranging.
\end{proof}

If we assume that our planar algebra is unitary (which will come for free in our setting where we have our planar algebra realised as $\dag$-planar algebra of the unitary $oGPA(\Gamma)$), it is fairly easy to show that (Anti-Sym 1) is a consequence of these three new relations.
\begin{lem}
Let $\mathcal{P}$ be an oriented unitary planar algebra satisfying relations (R1), (R2), (R3), (Hecke), (Braid Absorption), (Rotational Invariance), and (Norm). Then $\mathcal{P}$ satisfies (Anti-Sym 1).
\end{lem}
\begin{proof}
To show that (Anti-Sym 1) holds, we use the standard inner product trick (see \cite{EHOld} for an example). Let \[f:= \raisebox{-.5\height}{ \includegraphics[scale = .3]{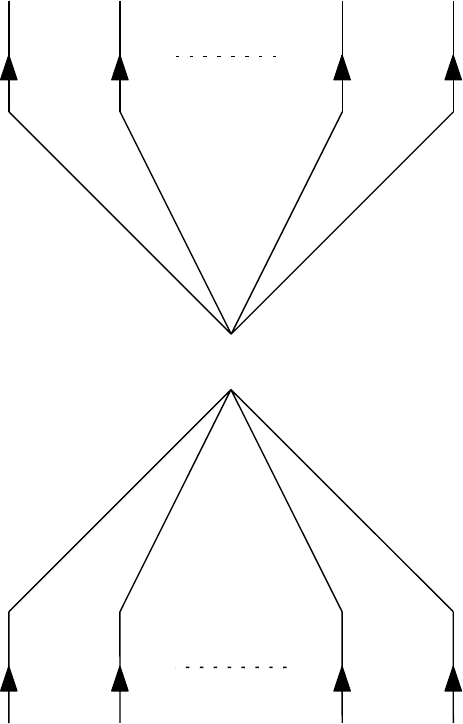}}    -   \raisebox{-.5\height}{ \includegraphics[scale = .3]{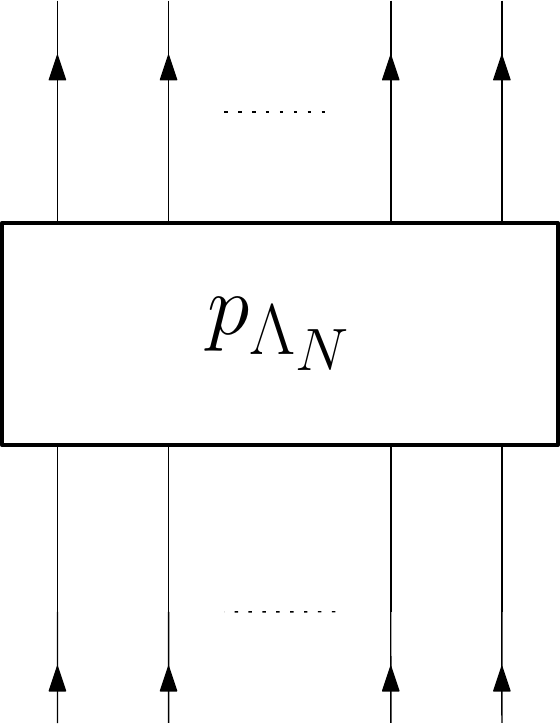}}\]
From Remark~\ref{rem:ASdag} we know $p_{\Lambda_N}^\dag = p_{\Lambda_N}$, and so $f^\dag = f$. Not assuming (Anti-Sym 1) (nor (Over Braid)), we compute
\[
\left\langle f,f\right  \rangle = \raisebox{-.5\height}{ \includegraphics[scale = .3]{exRel3.pdf}}+\raisebox{-.5\height}{ \includegraphics[scale = .3]{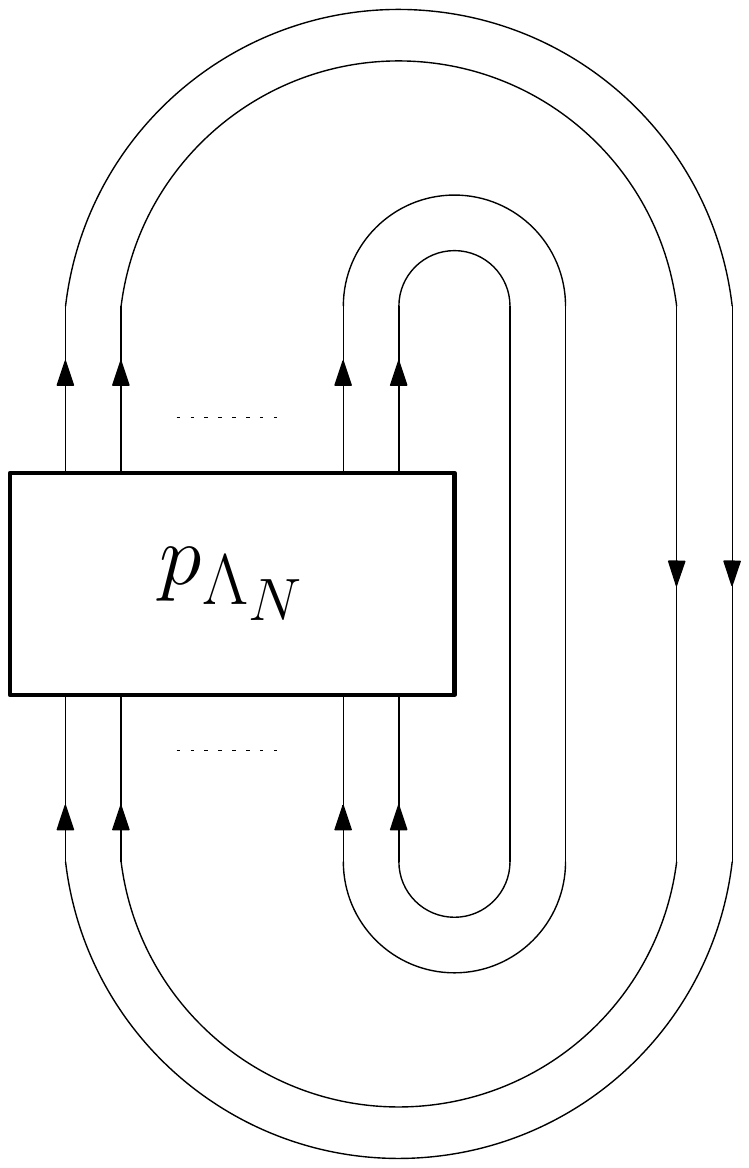}}-2\raisebox{-.5\height}{ \includegraphics[scale = .3]{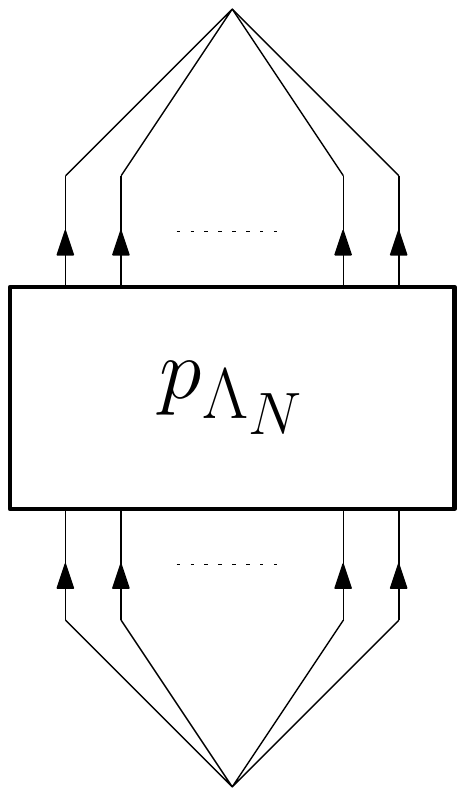}}=2-2\raisebox{-.5\height}{ \includegraphics[scale = .3]{trpnx.pdf}}.\]
Note that (Braid Absorption) with (Rotational Invariance) imply that $\raisebox{-.5\height}{ \includegraphics[scale = .3]{triv.pdf}}$ absorbs a $Y$ in any position at the cost of $q^{-2}$. With this we compute  
\[  \raisebox{-.5\height}{ \includegraphics[scale = .3]{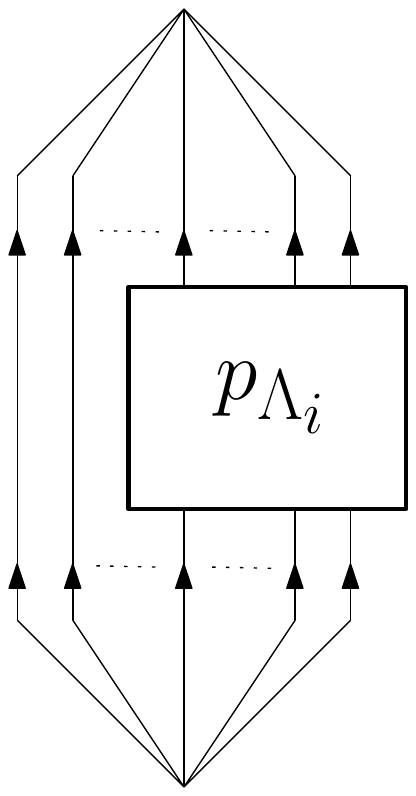}}  = \frac{1}{\sum_{j=0}^{i-1}q^{-2j}}\left( \raisebox{-.5\height}{ \includegraphics[scale = .3]{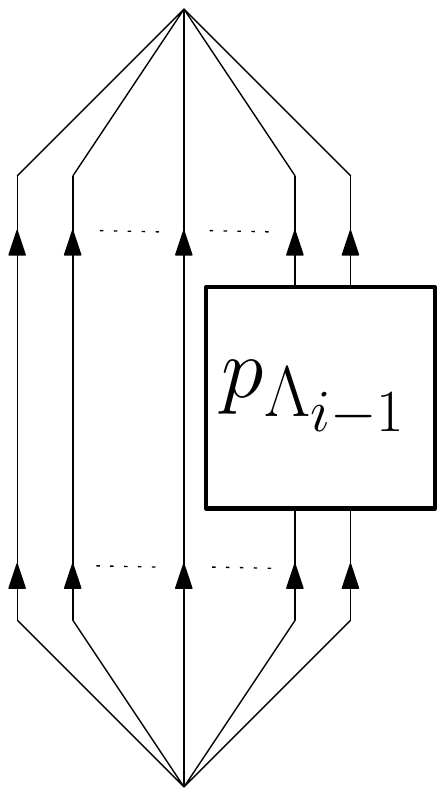}}  + \cdots + \raisebox{-.5\height}{ \includegraphics[scale = .3]{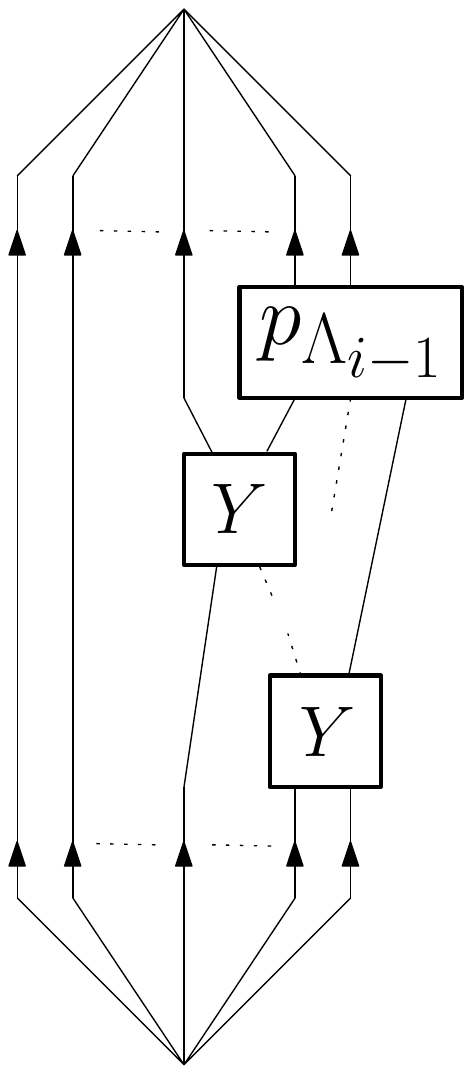}}\right)= \frac{1}{\sum_{j=0}^{i-1}q^{-2j}}\left(\sum_{j=0}^{i-1} q^{-2j} \raisebox{-.5\height}{ \includegraphics[scale = .3]{trpnx3.pdf}}\right)=\raisebox{-.5\height}{ \includegraphics[scale = .3]{trpnx3.pdf}}.
           \]
This recursively shows that 
\[   \raisebox{-.5\height}{ \includegraphics[scale = .3]{trpnx.pdf}} = \raisebox{-.5\height}{ \includegraphics[scale = .3]{exRel3.pdf}} = 1.    \]
As $\mathcal{P}$ is unitary, we have that $f = 0$, and thus
\[\raisebox{-.5\height}{ \includegraphics[scale = .3]{AntiSymL.pdf}}    =   \raisebox{-.5\height}{ \includegraphics[scale = .3]{AntiSymR.pdf}}.     \]
\end{proof}

%As this proof did not use the (Over Braid) relation, it is natural to ask the following:
%\begin{question}
%Does (Over Braid) follow as a consequence from the remaining relations?
%\end{question}
%A positive answer to this question would greatly simplify the KW-cell calculus of Section~\ref{sec:KWcells}. In every example we have studied so far, it suffices to use only relations (R1), (R2), (R3), (Hecke), (Norm), (Braid Absorption), and (Rotational Invariance) to pin down our GPA embeddings (up to gauge choice), and we only have to verify (Over Braid) holds for these solutions. This provides a point of evidence that there is a positive answer to the above question. Verifying (Over Braid) can take up to 2 weeks of computer time for the examples we study in this paper. Hence showing (Over Braid) holds automatically would be of great practical interest.
Summarising the results of this section, we have the following:
\begin{thm}\label{thm:pres}
    Let $\mathcal{P}$ be an oriented unitary planar algebra, generated by morphisms 
    \[  \raisebox{-.5\height}{ \includegraphics[scale = .3]{UU.pdf}} \in \mathcal{P}_{++\to ++},\quad \text{and}\quad \raisebox{-.5\height}{ \includegraphics[scale = .3]{triv.pdf}} \in \mathcal{P}_{\emptyset\to (+)^N}    \]
   satisfying relations (R1), (R2), (R3), (Hecke), (Braid Absorption), (Rotational Invariance), and (Norm). If $q^2 \neq 1$ then 
    \[   \mathcal{P} \cong \mathcal{P}_{\overline{\operatorname{Rep}(U_q(\mathfrak{sl}_N))^\omega};\Lambda_1}.\]
\end{thm}
\begin{proof}
The results of this section show that (Over Braid) and (Anti-Sym 1) hold in this setting. The result then follows from Lemma~\ref{lem:KW}.
\end{proof}

\section{The oriented module embedding theorem}\label{sec:GPA}
The goal of this section is the prove the equivalence between $\mathcal{C}$-module categories, and embeddings of a planar algebra of $\mathcal{C}$ in $oGPA(\Gamma)$. We do this by following the methods of \cite{EH}. Let us briefly outline this strategy.

A $\mathcal{C}$-module category $\mathcal{M}$ with $k$ distinct simples is described by a monoidal functor (not necessarily strict)
$$
F: \mathcal{C} \to \End(\mathcal{M}).
$$
Recall from Subsection~\ref{subsec:Mk} that $\End(\mathcal{M})$ is equivalent to $M_k(\Vect)$ as a multi-tensor category. Suppose $X$ is a $\otimes$-generator for $\cC$, and $\mathcal{P}_{X}$ the associated oriented planar algebra. Then $X$ will map to $F(X) := \Gamma \in \End(\mathcal{M})$, and $\mathcal{P}_X$ will embed into $\mathcal{P}_{\Gamma}$. The main substance of this section of then showing that $\mathcal{P}_{\Gamma}$ is equivalent to $oGPA(\Gamma)$. This shows that embeddings $\mathcal{P}_X \to oGPA(\Gamma)$ give $\cC$-module categories, and \textit{vice versa}. With the high level argument in mind, let us now proceed with the details.

Let $\Gamma$ be a directed graph and let $d_{ij}$ denote the number of edges from $i$ to $j$. Abusing notation, we also let $\Gamma$ denote the following object of $M_k(\Vect)$:
$$
\Gamma = \begin{bmatrix}
V_{11} & \dots & V_{1k} \\
\vdots & \ddots & \vdots \\
V_{k1} & \dots & V_{kk},
\end{bmatrix}
$$
where each $V_{ij}$ is an arbitrary, but fixed, vector space of dimension $d_{ij}$. 

Let $\lambda = (\lambda_1, \dots, \lambda_k)$ be the positive eigenvector for $\Gamma$. Using the construction in Definition~\ref{defn:oP} we have the oriented planar algebra $\Pp_{\Gamma}$ in $(M_k(\Vect), \lambda)$.
\begin{thm}\label{thm:GPAiso}
There is a $\dag$-isomorphism of $\dag$-planar algebras
$$
\Pp_\Gamma\cong oGPA(\Gamma).
$$
\end{thm}
\begin{proof}
We construct an isomorphism explicitly, following the same strategy as \cite{EH} which covered the unoriented case. This is essentially identical to their proof, with minor adjustments to account for non self-dual objects.

Suppose $\epsilon = (\epsilon_1, \dots, \epsilon_r)$ and $\delta = (\delta_1, \dots, \delta_s)$ are two sequences of $\pm 1$'s (possibly with $r$ or $s$ equal to $0$). We describe linear bijections
$$
\Hom_{oGPA(\Gamma)}(\epsilon, \delta) \to \Hom_{\Pp_\Gamma}(\epsilon, \delta)
$$
and check these give an oriented planar algebra isomorphism. First, we establish notation.

Let $(V_{ij})^{-1} := V_{ji}^*$ and $\Gamma^{-1} = \Gamma^*$. In this notation,
$$
\Gamma^{-1} = \begin{bmatrix}
V_{11}^{-1} & \dots & V_{1k}^{-1} \\
\vdots & \ddots & \vdots \\
V_{k1}^{-1} & \dots & V_{kk}^{-1}.
\end{bmatrix}.
$$
Given $\epsilon = (\epsilon_1, \dots, \epsilon_r)$, we have $\Gamma^{\epsilon} = \Gamma^{\epsilon_1} \otimes \dots \otimes \Gamma^{\epsilon_r}$. Hence the $(i,j)$-th entry of the object $\Gamma^{\epsilon}$ may be written
\begin{equation}\label{eq:ijGamma}
(\Gamma^{\epsilon})_{ij} = \bigoplus_{l_1, \dots, l_{r-1} = 1}^k V_{il_1}^{\epsilon_1} \otimes V_{l_1 l_2}^{\epsilon_2} \otimes \dots \otimes V_{l_{r-1}j}^{\epsilon_r}.
\end{equation}
If $r = 0$, then $\Gamma^{\epsilon}$ is the unit in $M_k(\Vect)$, so $(\Gamma^{\epsilon})_{ij} = (\mathbbm{1})_{ij} = \delta_{ij} \C$. For each $V_{ij}$, let
$$
\{v^l_{i,j} \in V_{ij}\ |\ l = 1, \dots, k\}
$$
denote an arbitrary, but fixed, basis of $V_{ij}$. We assume the edges in $\Gamma$ from $i$ to $j$ are labeled by $\{1, \dots, d_{ij}\}$, so there is a natural bijection between these edges and the basis of $V_{ij}$ chosen above. With these bases chosen let
$$
\{\overline{v}^l_{i,j} \in V^*_{ij} \ |\ l = 1, \dots, k\}
$$
denote the corresponding dual bases of $V_{ij}^*$. The vectors we have chosen index the edges in $\Gamma \cup \overline{\Gamma}$. Explicitly, there is a bijection
$$
\psi: \{\text{edges in $\Gamma \cup \overline{\Gamma}$}\} \to \bigsqcup_{i,j,l} \{v_{i,j}^l, \overline{v}_{i,j}^l\}.
$$
This map $\psi$ extends to all paths in $\Gamma \cup \overline{\Gamma}$: given a path $p = (p_1, \dots, p_r)$ of type $\epsilon$ from $i$ to $j$, define
$$
\psi(p) := \begin{cases} \psi(p_1) \otimes \dots \otimes \psi(p_r) \in (\Gamma^{\epsilon})_{ij} & r > 0 \\
1 \in (\mathbbm{1})_{ii} & r = 0.
\end{cases}
$$
By Eq. (\ref{eq:ijGamma}), the set
$$
\{\psi(p)\ |\ p \text{ is an $\epsilon$-path from $i$ to $j$} \}
$$
is a basis of the vector space $(\Gamma^{\epsilon})_{ij}$.

We are ready to define an isomorphism
$$
oGPA(\Gamma) \xrightarrow{F} \Pp_\Gamma.
$$
Given $(p,q) \in \Hom_{oGPA(\Gamma)}(\epsilon, \delta)$ with $s(p) = s(q) = i$ and $t(p) = t(q) = j$, define the linear map $F_{pq}: (\Gamma^{\epsilon})_{ij} \to (\Gamma^{\delta})_{ij}$ which acts on basis elements by
$$
F_{p,q}(\psi(p')) = \begin{cases} \psi(q) & \text{ if } p' = p \\
0 & \text{ otherwise.}
\end{cases}
$$
Finally, define $F((p,q))$ by
$$
F((p,q)) := \begin{blockarray}{cccccc}
 & & j & & & \\
\begin{block}{[ccccc]c}
   &  &  &  &  &  \\
  &  &  & &  &  \\
  &  & F_{p,q} & & & i \\
  &  &  &  &  &  \\
  &  &  &  &  &  \\
\end{block}
\end{blockarray} \in \Hom_{\Pp_\Gamma}(\epsilon, \delta).
$$
Extending this definition linearly defines $F$ on all of $oGPA(\Gamma)$. It is clearly a linear bijection on hom-spaces. We must check that it provides an oriented planar algebra isomorphism, or in other words a pivotal strict monoidal functor. The proof that $F$ is a strict monoidal functor is identical to \cite[Section 3]{EH} and omitted.

We check that $F$ is strictly pivotal. It suffices to check $F$ preserves cups, i.e. $F(\ev_{(+,-)}) = \ev^{\lambda}_{\Gamma} \in \Hom_{\Pp_\Gamma}((+,-), \mathbbm{1})$ and similarly for the right evaluation morphisms. Comparing Eqs. (\ref{eq:oGPAcupcap}) and (\ref{eq:MkVeccupcap})), it suffices to prove
$$
F\left(\sum_{e:\ s(e) = i,\ t(e) = j}((\overline{e}, e), s(e))\right) = \bigoplus_{l=1}^k \ev_{V_{ij}}.
$$
By the definition of $F_{p,q}$, if $e$ is an edge in $\Gamma$, then $F_{(e, \overline{e}), s(e)}$ is the linear map $V_{s(e), t(e)}^* \otimes V_{s(e), t(e)} \to \C$ which sends $\psi(\overline{e}) \otimes \psi(e)$ to $\psi(s(e)) = 1$. Therefore
\begin{equation}
\sum_{e:\ s(e) = i,\ t(e) = j} F_{(e, \overline{e}), s(e)} = \ev_{V_{ij}}.
\end{equation}
%{\color{red} PROBABLY ENOUGH OF AN ARGUMENT, BUT MAKE SURE SUBSCRIPTS ARE CORRECT.}
This proves that $F$ preserves the left duality caps. The argument for the right duality caps is similar.

Finally from the explicit description of the $\dag$ structure on both $oGPA(\Gamma)$ and $M_k(\Vect)$, we see this isomorphism preserves these dagger structures.
\end{proof}
With the above theorem in hand, we now obtain the oriented version of the graph planar algebra embedding theorem.
\begin{thm}\label{thm:oGPA}
Let $\mathcal{C}$ be a (unitary) pivotal fusion category, and $X\in \mathcal{C}$ a $\otimes$-generator. Then there is a bijective correspondence between
\begin{enumerate}
    \item semisimple pivotal ($C^*$-)module categories $\mathcal{M}$ over $\mathcal{C}$ whose module fusion graph for $X$ is $\Gamma$, and
    \item embeddings of oriented (unitary) planar algebras $\mathcal{P}_X \to oGPA(\Gamma)$.
\end{enumerate}
The equivalence relation on 1) is (unitary) equivalence of module categories, and the equivalence relation on 2) is (unitary) natural isomorphism of planar algebra morphisms.
\end{thm}
\begin{proof}
By Theorem~\ref{thm:folk} and \cite[Corollary 3.53]{EH}, a ($C^*$)- module category of 1) is equivalent to a ($\dag$) planar algebra morphism
\[\mathcal{P}_X \to \mathcal{P}_\Gamma.\]
We then have from Theorem~\ref{thm:GPAiso} the ($\dag$) isomorphism
\[\mathcal{P}_\Gamma\cong oGPA(\Gamma).\]
Hence the module category of 1) is equivalent to a ($\dag$) morphism 
\[  \mathcal{P}_X\to oGPA(\Gamma). \]
As $\mathcal{C}$ is semisimple and fusion this morphism must be injective. Hence we have the data of 2).

The equivalence relation is obtained by pushing through the definition of module category equivalence through the above chain of isomorphisms and equivalences.
\end{proof}

Before we end this section, we would like to prove one more general result regarding GPA embeddings. This result is well-known to experts\footnote{We thank Hans Wenzl for informing us of the following lemma.}, however we could not find a proof in the literature. This lemma is useful, as it allows categorical data to be deduced from combinatorial data. In the reverse direction, this lemma allows the module fusion graphs for every object of $\cC$ to be determined from the GPA embedding.
\begin{lem}\label{lem:Hans}
Let $X\in \cC$, and $p_Z\in \End_\cC(X^{\otimes n})$ a minimal projection onto $Z\in \cC$. Let $\mathcal{M}$ be a $\mathcal{C}$-module, and $\Gamma_Y$ the module fusion graph for action by $Y\in \cC$. Let
\[\phi : \mathcal{P}_X \to oGPA( \Gamma_X) \]
be an embedding corresponding to the $\cC$-module $\mathcal{M}$ under the bijection of Theorem~\ref{thm:oGPA}. Let $M_1,M_2$ simple objects of $\mathcal{M}$, then we have
\[  \sum_{(q,q) :\quad  q \text{ is a $+^n$-path}, s(q) =  M_1, t(q)=M_2}   \phi(p_Z)[(q,q)] = \left(\Gamma_Z \right)_{M_1\to M_2}.  \]
\end{lem}
\begin{proof}
Consider the commutative diagram of semi-simple algebras
\[\begin{tikzcd}
& & & oGPA(\Gamma)_{+^n\to +^n}\arrow[d]\\
\left(\mathcal{P}_X\right)_{+^n\to +^n} \arrow[r,"="] \arrow[rrru,bend left = 5]&\End_\cC(X^{\otimes n}) \arrow[r,hookrightarrow]  & \End_{\mathcal{M}}(X^{\otimes n}\otimes M_1)\arrow[r,"="] & \End_{\End(\mathcal{M})}(\Gamma_X^{\otimes n}[M_1])
\end{tikzcd}\]
The top most arrow is exactly the map $\phi$. The downward arrow is the restriction of a functional to basis elements $(p,q)$ where both $p$ and $q$ begin at the vertex $M_1$ (and implicitly using the isomorphism from Theorem~\ref{thm:GPAiso}). The bottom inclusion is the natural embedding $f\mapsto f\otimes \operatorname{id}_{M_1}$. This diagram commutes due to the bijection between $\mathcal{C}$-module categories and monoidal functors $\cC\to \End(\mathcal{M})$.

Consider the projection $p_Z\in \End_{\mathcal{C}}(X^{\otimes n})$ from the statement of the lemma. By definition of the module fusion rules, we have that $p_Z$ maps to $\sum_{M_j \in \mathcal{M}}\sum_{1\leq \ell \leq \left(\Gamma_Z\right)_{M_1\to M_j}}p_{M_j}^\ell$ where the $p_{M_j}^\ell$ is some set of minimal projections onto $M_j$ in $\End_{\mathcal{M}}(X^{\otimes n}\otimes M_1)$. Restricting to the block corresponding to the subobject $M_2$ hence gives $\sum_{1\leq \ell \leq \left(\Gamma_Z\right)_{M_1\to M_2}}p_{M_j}^\ell$. The trace of this morphism (calculated w.r.t the fixed basis of matrix units) is thus $\left(\Gamma_Z\right)_{M_1\to M_2}$.

On the other hand, the projection $p_Z$ maps to $\phi(p_Z) \in oGPA(\Gamma)$, and restricting this morphism to the block of $\End_{\End(\mathcal{M})}(\Gamma_X^{\otimes n}[M_1])$ corresponding to the summand $M_2$ is (again by the equivalence between $\mathcal{C}$-module categories and monoidal functors $\cC\to \End(\mathcal{M})$) gives the restriction of $\phi(p_Z)$ to the space of loops $(p,q)$ with source $M_1$ and target $M_2$. Taking the trace of this restriction (w.r.t. the basis of matrix units $(p,q)$) gives
\[  \sum_{(q,q) :\quad  q \text{ is a $+^n$-path}, s(q) =  M_1, t(q)=M_2}   \phi(p_Z)[(q,q)].  \]
From the commutativity of the diagram at the start of the proof, our two expressions for the trace are equal. This gives the statement of the lemma.
\end{proof}

\section{Kazhdan-Wenzl Cells}
In this section we introduce the definition of a KW cell system on a graph $\Gamma$. This is a polynomial system of equations depending on parameters $N\geq 2$ a natural number, $q$ a root of unity, and $\omega$ an $N-th$ root of unity. 

When $q = e^{2 \pi i \frac{1}{2(N+k)}}$ for some $k\geq 0$, the data of a KW cell system is (by definition, and from the results of the previous section) an embedding 
\[\mathcal{P}_{\overline{
\operatorname{Rep}(U_q(\mathfrak{sl}_N))^\omega},\Lambda_1}\to  oGPA(\Gamma).  \]
Hence by Theorem~\ref{thm:oGPA} a KW cell system on $\Gamma$ give a module category over $\overline{
\operatorname{Rep}(U_q(\mathfrak{sl}_N))^\omega}$ whose module fusion graph for $\Lambda_1$ is $\Gamma$. We also define the notion of equivalence of KW cell systems, which is defined to be the pull-back of equivalence of module categories.  Hence solutions to KW cell systems (up to equivalence) with $q = e^{2 \pi i \frac{1}{2(N+k)}}$ classify module categories (up to equivalence) over $\overline{\operatorname{Rep}(U_q(\mathfrak{sl}_N))^\omega}$. This is all expanded on in the proof of Theorem~\ref{thm:main} at the end of this section.

\begin{remark}\label{rmk:nonUni}
    Our definition of KW cell system also makes sense for roots of unity $q$ which are not of the form $e^{2 \pi i \frac{1}{2(N+k)}}$. However, 
    there are two issues which stop us from obtaining module categories for these $q$ values. The first is that the implicit image of the cups and caps in $oGPA(\Gamma)$ no longer satisfy the correct loop parameter to take a homomorphism from $\mathcal{P}_{\overline{
\operatorname{Rep}(U_q(\mathfrak{sl}_N))^\omega},\Lambda_1}$. This can be fixed by choosing a different pivotal structure on $oGPA(\Gamma)$. The more serious issue is that if we change the pivotal structure on $oGPA(\Gamma)$, then it is no longer a unitary pivotal structure. In particular, we can't assume the image of the elements specified by a KW cell system form a unitary subcategory. Hence we do not have that (Anti-Sym 1) holds for free. To show existence of module categories over these non-unitary $q$'s, one would have to verify that (Anti-Sym 1) holds in the graph planar algebra manually. We have done this computation for several examples, and unfortunately the relation (Anti-Sym 1) can takes weeks of computer time to verify.
\end{remark}

Our definition of a KW cell system is as follows.

\begin{defn}\label{defn:KW}
Let, $N\in \mathbb{N}_{\geq 2}$, $q$ a root of unity, and $\omega$ an $N$-th root of unity. Further, let $\Gamma$ be a graph with norm $\q{N}$, and let $\lambda$ be the positive Frobenius-Perron eigenvector of $\Gamma$.

A \textit{Kazhdan-Wenzl cell system} with parameters $(N,q,\omega)$ on the graph $\Gamma$ is a map $KW$ which assigns to every loop $\raisebox{-.5\height}{ \includegraphics[scale = .4]{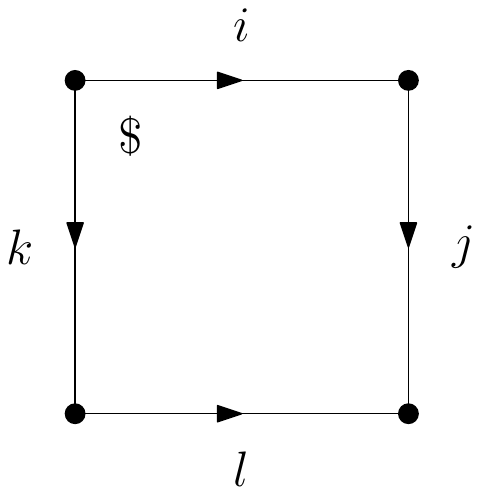}}$ in $\Gamma$, a complex scalar
\[  KW\left( \raisebox{-.5\height}{ \includegraphics[scale = .4]{KWcell4.pdf}}\right)\in \mathbb{C} , \]
and to every loop of length $N$ $\raisebox{-.5\height}{ \includegraphics[scale = .4]{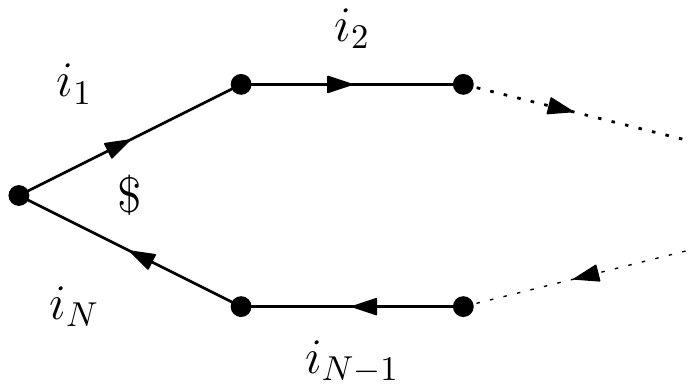}}$ in $\Gamma$, a complex scalar
\[  KW\left( \raisebox{-.5\height}{ \includegraphics[scale = .4]{nbox.pdf}}\right)\in \mathbb{C} . \]
These scalars satisfy the following conditions:
\begin{align*}
(R1)&:\quad  \sum_{p} \frac{\lambda_{{\text{source}(p)}}}{\lambda_{{\text{target}(p)}}} KW\left(\raisebox{-.5\height}{ \includegraphics[scale = .4]{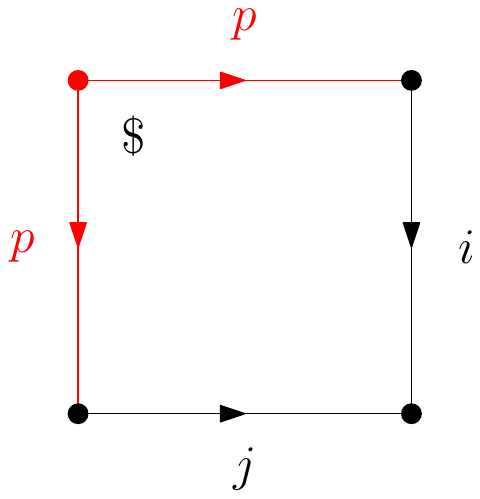}}\right)\quad = \sum_{p}\frac{\lambda_{{\text{target}(p)}}}{\lambda_{{\text{source}(p)}}}KW\left(\raisebox{-.5\height}{ \includegraphics[scale = .4]{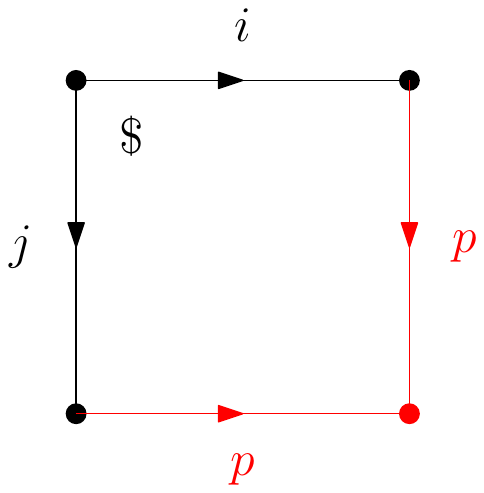}}\right)\quad=\quad [N-1]_q\delta_{i,j}\\
(R2)&:\quad   KW\left( \raisebox{-.5\height}{ \includegraphics[scale = .4]{KWcell4.pdf}}\right)\quad =\quad \overline{KW\left( \raisebox{-.5\height}{ \includegraphics[scale = .4]{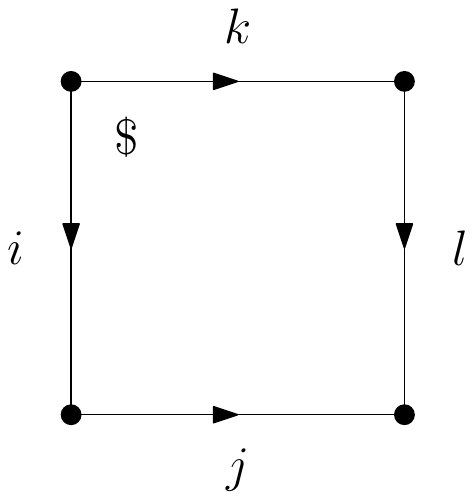}}\right)}\\
(R3)&:\quad  \sum_{p,q,r}KW\left(\raisebox{-.5\height}{ \includegraphics[scale = .4]{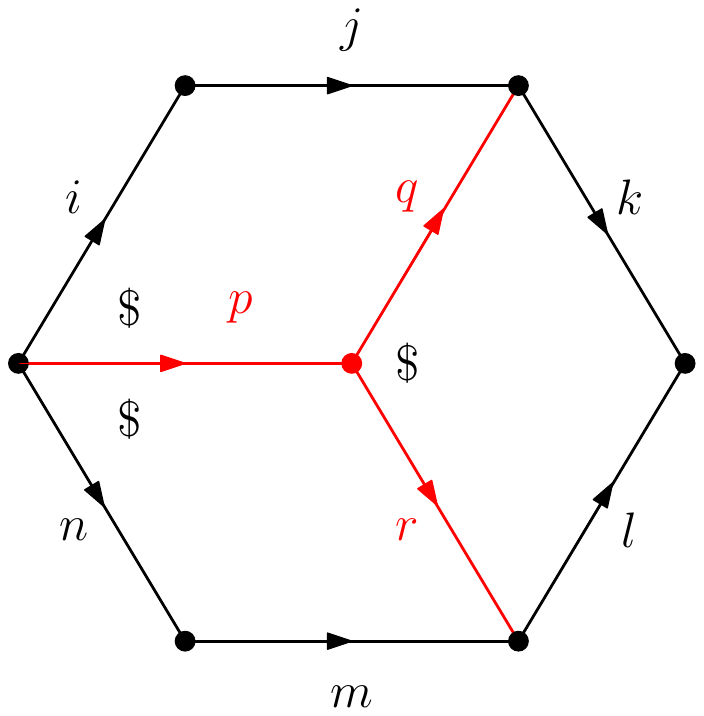}}\right)- \delta_{k,l}  KW\left( \raisebox{-.5\height}{ \includegraphics[scale = .4]{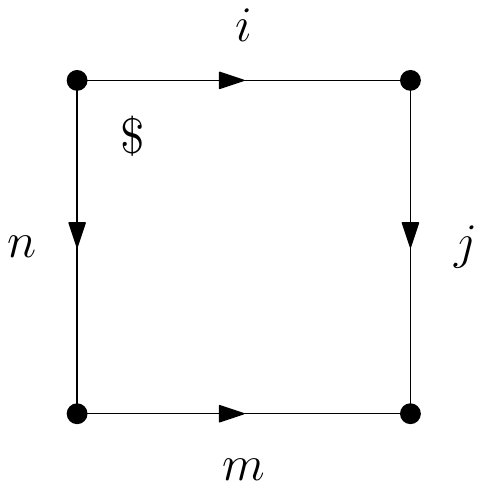}}\right)\quad =\\&\quad \sum_{p',q',r'} KW\left(\raisebox{-.5\height}{ \includegraphics[scale = .4]{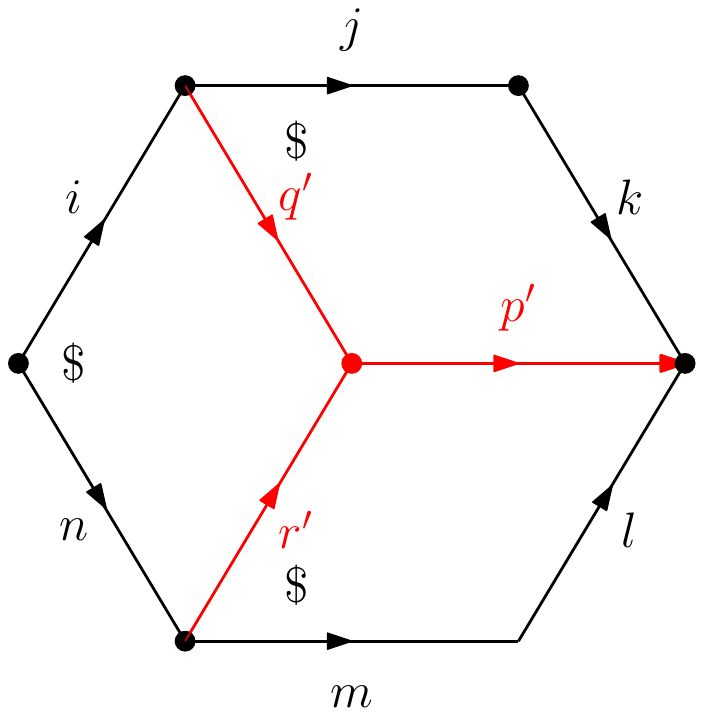}}\right) - \delta_{i,n} KW\left( \raisebox{-.5\height}{ \includegraphics[scale = .4]{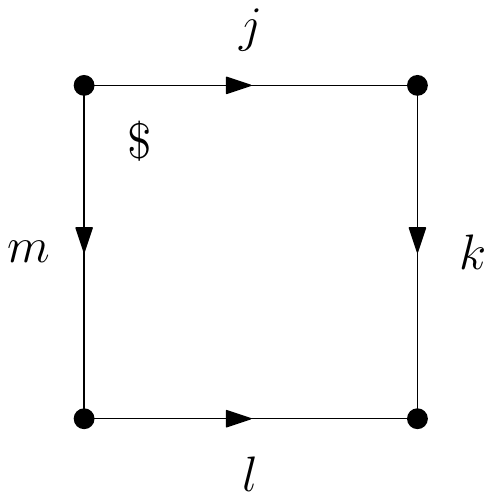}}\right)\\
(\text{Hecke})&: \quad  \sum_{p,q}KW\left(\raisebox{-.5\height}{ \includegraphics[scale = .4]{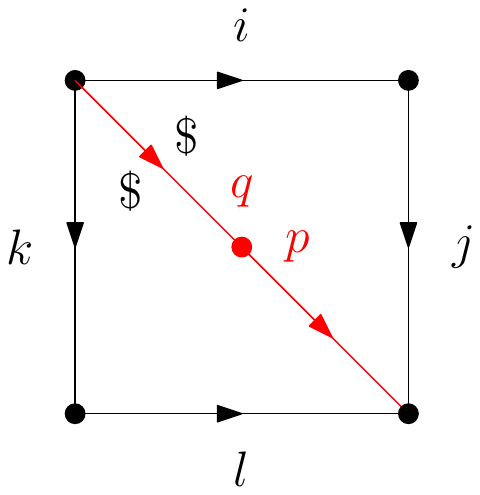}}\right)\quad =\quad [2]_qKW\left(\raisebox{-.5\height}{ \includegraphics[scale = .4]{KWcell4.pdf}}\right)\\
\text{(RI)}&:\quad (-1)^{N+1}\omega \cdot KW\left(\raisebox{-.5\height}{ \includegraphics[scale = .4]{nbox.pdf}}\right)=\frac{\lambda_{{\text{source}(i_N)}}}{\lambda_{{\text{source}(i_1)}}}KW\left(\raisebox{-.5\height}{ \includegraphics[scale = .4]{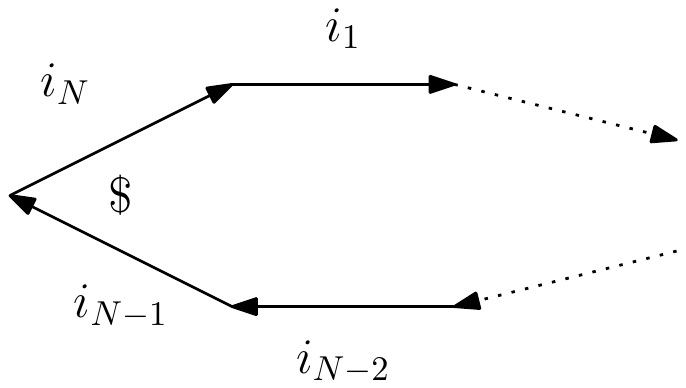}} \right)\\
\text{(BA)}&:\sum_{p,q}KW\left(\raisebox{-.5\height}{ \includegraphics[scale = .4]{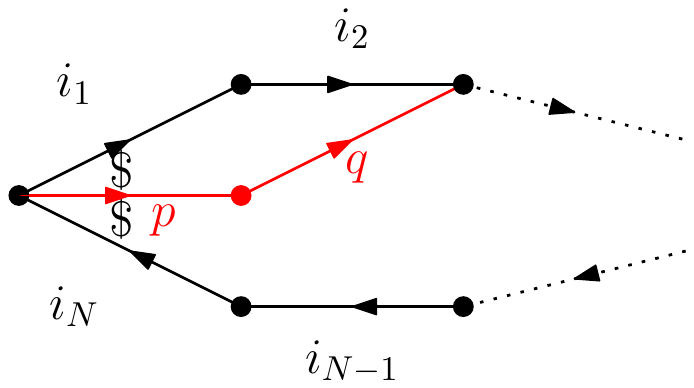}} \right)=[2]_q\quad KW\left(\raisebox{-.5\height}{ \includegraphics[scale = .4]{nbox.pdf}}\right)\\
\text{(N)}&:\sum_{i_j : 1\leq j \leq N}\quad \left|KW\left(\raisebox{-.5\height}{ \includegraphics[scale = .4]{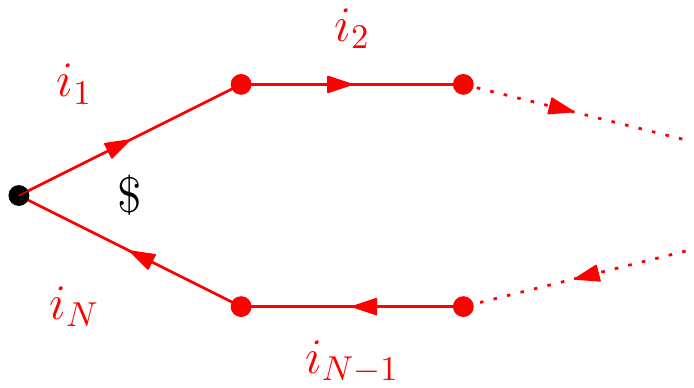}}\right)\right|^2=1\\
%\text{(OB)}&: \quad  \sum_{j_t, k_t : 1\leq t\leq N}KW\left(\raisebox{-.5\height}{ \includegraphics[scale = .4]{overBraid.pdf}}\right)\quad =\quad (-1)^N \omega q^{1-N}\delta_{l,m}KW\left(\raisebox{-.5\height}{ \includegraphics[scale = .4]{nbox.pdf}}\right) 
\end{align*}

%where
%\[   KW\left(\raisebox{-.5\height}{ \includegraphics[scale = .4]{Y.pdf}}\right) := \frac{1}{q}KW\left(\raisebox{-.5\height}{ \includegraphics[scale = .4]{KWcell4.pdf}}\right) - \delta_{i,k}\delta_{j,l}.  \]
\end{defn}

\begin{remark}
    The above definition uses the statistical mechanics ``Boltzmann weight'' notation (see \cite{Stat,SU3} for examples). Typically this notation is just used for (R2) and (R3), where a solution is exactly solution to the quantum Yang-Baxter equation (see \cite{Kap}), however the extension we use is well defined, and compact.
    
    Let us briefly explain how to interpret this notation. The value of a graph containing multiple $\raisebox{-.5\height}{ \includegraphics[scale = .3]{KWcell4.pdf}}$ and $\raisebox{-.5\height}{ \includegraphics[scale = .4]{nbox.pdf}}$ is defined to be the product of the KW cells on these individual loops. For example
    \[KW\left(\raisebox{-.5\height}{ \includegraphics[scale = .4]{nboxsquare1.pdf}} \right):= KW\left(\raisebox{-.5\height}{ \includegraphics[scale = .4]{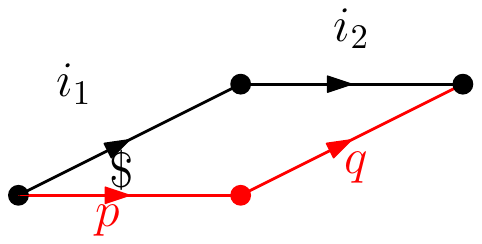}} \right)\cdot KW\left(\raisebox{-.5\height}{ \includegraphics[scale = .4]{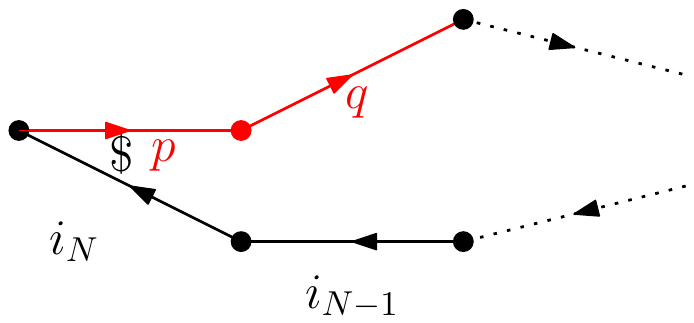}} \right).\]
    Due to the loops we consider in this cell calculus, this definition is unambiguous, and well defined.

    Typically the equations of Definition~\ref{defn:KW} involve sums of KW cells. The edges and vertices in a graph which are fixed for a given equation are coloured black, while the vertices and edges which are summed over in the equation are coloured red. The sums are taken over all red edges and vertices, such that there is a graph homomorphism from the graph into $\Gamma$ which agrees on the vertex and edge labels.
\end{remark}
This definition at first glance appears to give an incredibly difficult system of polynomial equations to solve. However, notice that the maximum degree of this polynomial system is 3 (from (R3)). Further, a large number of these equations are linear ((R1), (R2), and (RI)). Further, the system can be solved in two steps. First the 4-path cells can be solved with relations (R1), (R2), (R3), and (Hecke). Once these cells are determined, the equation (BA) is also linear. The solution for the $N$-path cells is then obtained as a solution to the linear system to (RI) and (BA) (with (N) used to normalise the solution). We refer to the 4-path cells as the \textit{U-cells}, and the $N$-path cells as the \textit{B-cells}. We refer the reader to Section~\ref{sec:examples} where several solutions are determined for examples of this 2-step procedure.

\begin{remark}\label{rmk:matrix}
To present a solution to a KW cell system, it can be convenient to use matrix notation. For the cells corresponding to loops of the form $\raisebox{-.5\height}{ \includegraphics[scale = .4]{KWcell4.pdf}}$ this notation wraps the solution of the cells system into a family of square matrices indexed by the upper left, and bottom right vertex. The rows and columns of this matrix are indexed by paths of length two between the two labelled vertices. The entries of the matrix are then exactly the values $KW\left( \raisebox{-.5\height}{ \includegraphics[scale = .4]{KWcell4.pdf}}\right)$. Typically we denote these matrices $U^{a}_{\quad b}$ where $a$ and $b$ are vertices in the graph.

Several of the relations (but not all) can be fully formulated in terms of standard linear algebra. The relation (R2) is equivalent to each matrix $U^{a}_{\quad b}$ being Hermitian, and (Hecke) is equivalent to each matrix satisfying $U^{a}_{\quad b}\cdot U^{a}_{\quad b} = \q{2}U^{a}_{\quad b}$.

For cells corresponding to loops of the form $\raisebox{-.5\height}{ \includegraphics[scale = .4]{nbox.pdf}}$ this notation wraps the solution into a family of matrices indexed by the $\$$ vertex, and the third vertex from the $\$$ vertex. The rows are indexed by paths of length 2 between these two vertices, and the columns are indexed by paths of length $N-2$ between these two vertices. The entries of the matrix are then the cells $ KW\left( \raisebox{-.5\height}{ \includegraphics[scale = .4]{nbox.pdf}}\right)$. We denote these matrices $B_{a,\underline{\hspace{.5em}}, b, \underline{\hspace{1em}}}$. The relation (BA) is then equivalent to the matrix $B_{a,\underline{\hspace{.5em}}, b, \underline{\hspace{1em}}}$ being an eigenmatrix with eigenvalue $\q{2}$ for the matrix $U^{a}_{\quad b}$ for all pairs of vertices $a,b$.

This matrix notation is a compact way to present a solution to a KW cell system. For the examples in Section~\ref{sec:examples} we use this notation. 
\end{remark}

We also introduce the following notion of equivalence of KW cell systems. 

\begin{defn}\label{defn:KWequiv}
Let $KW^1$ and $KW^2$ be two KW cell systems on $\Gamma$ with parameters $(N, q, \omega)$. An equivalence $KW^1 \to KW^2$ is a map $V$ which assigns to every loop of the form $\raisebox{-.5\height}{ \includegraphics[scale = .4]{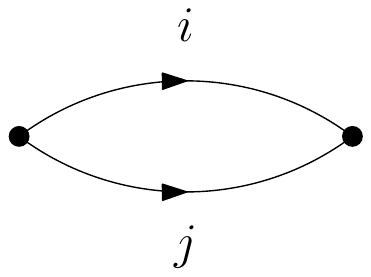}}$ in $\Gamma$ a complex scalar:
\[ V\left( \raisebox{-.5\height}{ \includegraphics[scale = .4]{V1.pdf}}  \right) \in \mathbb{C} \]

These scalars must satisfy
\[  \sum_k V\left( \raisebox{-.5\height}{ \includegraphics[scale = .3]{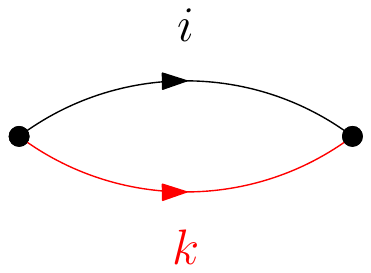}} \right)\overline{V\left( \raisebox{-.5\height}{ \includegraphics[scale = .3]{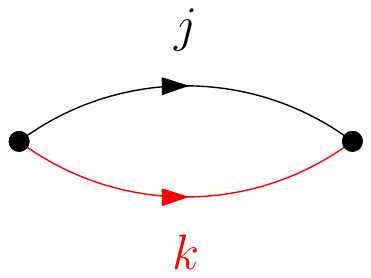}}  \right)}   = \delta_{i,j}=\sum_k V\left( \raisebox{-.5\height}{ \includegraphics[scale = .3]{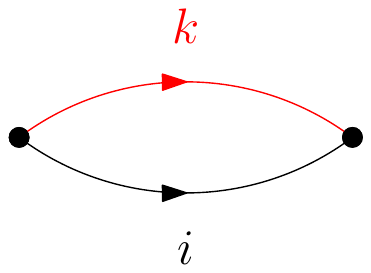}} \right)\overline{V\left( \raisebox{-.5\height}{ \includegraphics[scale = .3]{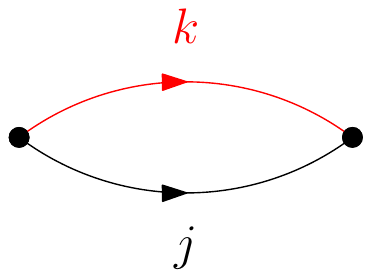}}  \right)}   \]

\[  KW^1\left( \raisebox{-.5\height}{ \includegraphics[scale = .3]{KWcell4.pdf}}   \right) = \sum_{i',j',k',l'}  KW^2\left( \raisebox{-.5\height}{ \includegraphics[scale = .3]{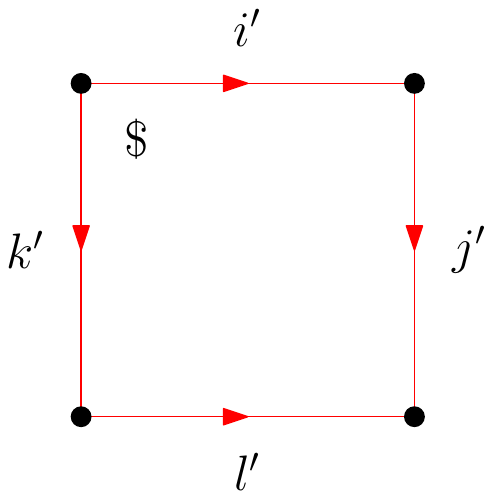}}\right) \overline{V\left( \raisebox{-.5\height}{ \includegraphics[scale = .3]{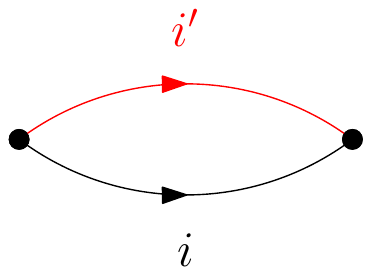}} \right)V\left( \raisebox{-.5\height}{ \includegraphics[scale = .3]{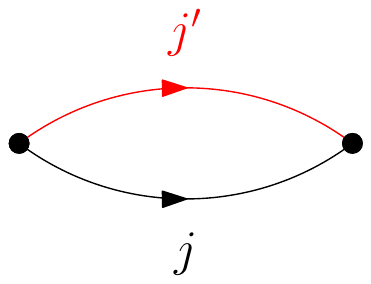}} \right)}V\left( \raisebox{-.5\height}{ \includegraphics[scale = .3]{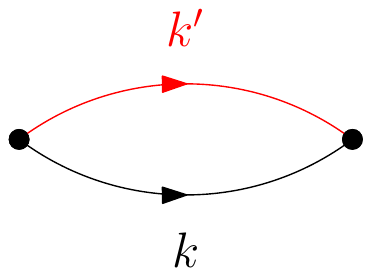}} \right)V\left( \raisebox{-.5\height}{ \includegraphics[scale = .3]{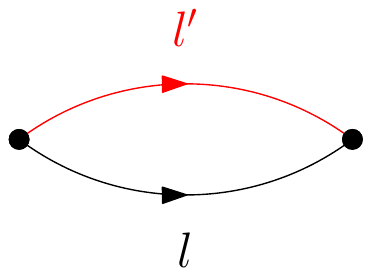}} \right)     \]
\[  KW^1\left( \raisebox{-.5\height}{ \includegraphics[scale = .3]{nbox.pdf}}   \right) = \sum_{i_1',\cdots , i_N'}  KW^2\left( \raisebox{-.5\height}{ \includegraphics[scale = .3]{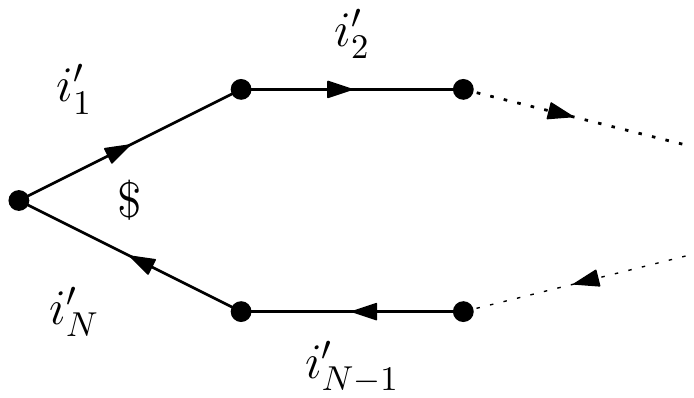}}\right)V\left( \raisebox{-.5\height}{ \includegraphics[scale = .3]{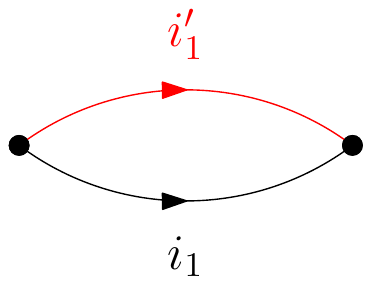}} \right)V\left( \raisebox{-.5\height}{ \includegraphics[scale = .3]{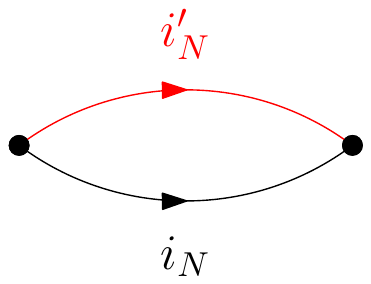}} \right) .     \]
\end{defn}
In our matrix interpretation, $V$ can be interpreted as a square matrix indexed by pairs of vertices connected by an edge. The rows and columns of this matrix are indexed by the paths between these vertices. The first relation above simply state that this matrix is a unitary matrix. The other two relations don't appear to have a nice matrix interpretation.

Now that we have defined KW cell systems up to equivalence, we can prove the main theorem of this paper. 
\begin{proof}[Proof of Theorem~\ref{thm:main}]
Throughout this proof, $N\geq 2$ will be an integer, $q = e^{2 \pi i \frac{1}{2(N+k)}}$ for some positive integer $k$, and $\omega$ is an $N$-th root of unity.

Suppose $\mathcal{M}$ is a module category over $\overline{\operatorname{Rep}(U_q(\mathfrak{sl}_N))^\omega}$ whose module fusion graph for $\Lambda_1$ is $\Gamma$. Then the data of this module  is equivalent to a pivotal monoidal functor
\[  \overline{\operatorname{Rep}(U_q(\mathfrak{sl}_N))^\omega} \to \operatorname{End}(\mathcal{M}),\]
such that the image of $\Lambda$ is graph $\Gamma$. Furthermore, by \cite{Whale}, we can assume this functor is a $\dag$-functor. By Theorem~\ref{thm:oGPA}, this gives an embedding of oriented unitary planar algebras
\[  \phi: \mathcal{P}_{\overline{\operatorname{Rep}(U_q(\mathfrak{sl}_N))^\omega};\Lambda_1}\to oGPA(\Gamma).  \]
In particular, the image of $\raisebox{-.5\height}{ \includegraphics[scale = .3]{UU.pdf}}$ and $\raisebox{-.5\height}{ \includegraphics[scale = .3]{triv.pdf}}$ give distinguished elements of $oGPA(\Gamma)$ satisfying relations (R1), (R2), (R3), (Hecke), (Braid Absorption), (Rotational Invariance), and (Norm). We define a KW cell system on $\Gamma$ by setting \begin{align*}
 KW\left( \raisebox{-.5\height}{ \includegraphics[scale = .4]{KWcell4.pdf}}\right) &:= \phi\left[ \raisebox{-.5\height}{ \includegraphics[scale = .3]{UU.pdf}}     \right]\left(\raisebox{-.5\height}{ \includegraphics[scale = .4]{KWcell4.pdf}} \right)\\
KW\left( \raisebox{-.5\height}{ \includegraphics[scale = .4]{nbox.pdf}}\right) &:= \phi\left[ \raisebox{-.5\height}{ \includegraphics[scale = .3]{triv.pdf}}     \right]\left(\raisebox{-.5\height}{ \includegraphics[scale = .4]{nbox.pdf}} \right).
\end{align*}
By the definition of the planar algebra $oGPA(\Gamma)$, the map KW satisfies all of the relations required to be a KW cell system on $\Gamma$ with parameters $(N,q,\omega)$. 

Let $\phi_1, \phi_2$ be two embeddings $\mathcal{P}_{\overline{\operatorname{Rep}(U_q(\mathfrak{sl}_N))^\omega};\Lambda_1}\to oGPA(\Gamma)$ corresponding to equivalent module categories. Then by definition, there exists a unitary element $V\in oGPA(\Gamma)_{+\to +}$ such that
\begin{align*}
\raisebox{-.5\height}{ \includegraphics[scale = .4]{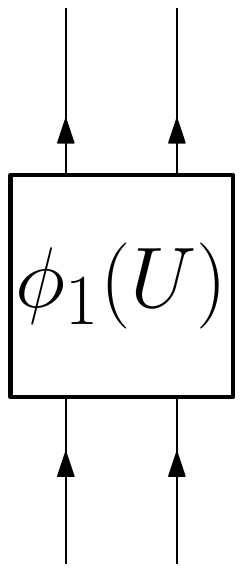}} & = \raisebox{-.5\height}{ \includegraphics[scale = .4]{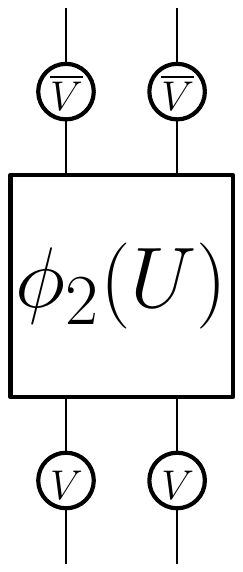}}\\
\raisebox{-.5\height}{ \includegraphics[scale = .4]{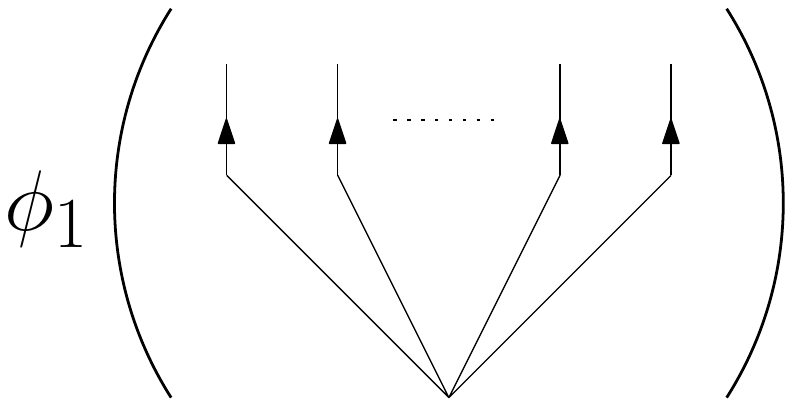}} & = \raisebox{-.5\height}{ \includegraphics[scale = .4]{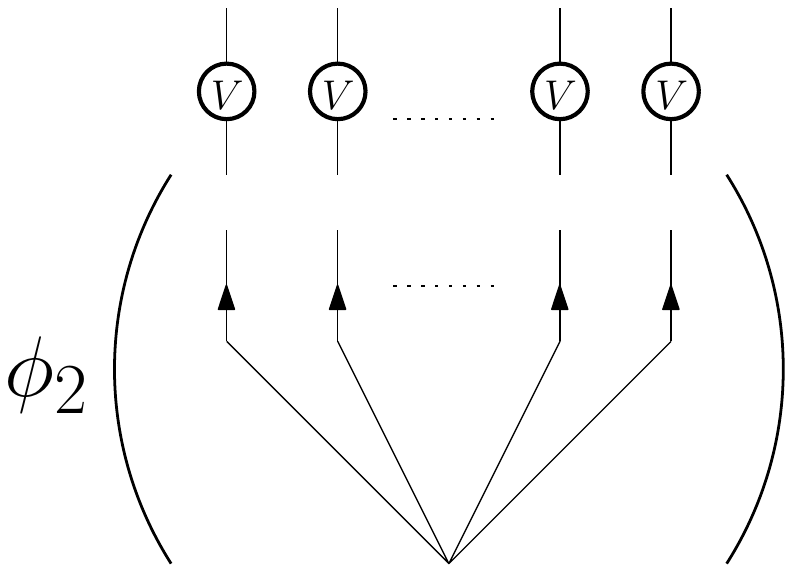}}.
\end{align*}
Unpacking these equations using the definition of $oGPA(\Gamma)$ gives exactly Definition~\ref{defn:KWequiv}. Thus equivalent module categories give rise to equivalent KW cell systems under the above construction.

Conversely, a KW cell system with parameters $(N,q,\omega)$ on a graph $\Gamma$ is by definition a pair of elements in $oGPA(\Gamma)$ satisfying the relations (R1), (R2), (R3), (Hecke), (Rotational Invariance), (Braid Absorption), and (Norm) of Section~\ref{sec:KW}. As the $\dag$-structure on $oGPA(\Gamma)$ is unitary, it restricts to a unitary $\dag$ structure on the $\dag$-subcategory generated by the two elements specified by the KW cell system solution. We then apply Theorem~\ref{thm:pres} to see this subcategory is isomorphic to $\mathcal{P}_{\overline{\operatorname{Rep}(U_q(\mathfrak{sl}_N))^\omega};\Lambda_1}$. We thus have an embedding 
\[  \mathcal{P}_{\overline{\operatorname{Rep}(U_q(\mathfrak{sl}_N))^\omega};\Lambda_1}\to oGPA(\Gamma).\]
Hence Theorem~\ref{thm:oGPA} gives us a module category over $\overline{\operatorname{Rep}(U_q(\mathfrak{sl}_N))^\omega}$ such that the module fusion graph for $\Lambda_1$ is $\Gamma$. 

The same argument as in the converse case (in reverse), shows that equivalent KW cell systems give rise to equivalent module categories over $\overline{\operatorname{Rep}(U_q(\mathfrak{sl}_N))^\omega}$. 
\end{proof}

In order to find a solution to a KW cell system on a given graph, we often need some addition equations. In the situation where we know the action graphs for $\Lambda_2$ and $\Lambda_3$, we have the following additional equations. These equations first appeared in \cite{Stat} as a conjecture. With our technical machinery, we can easily prove the equations always hold.
\begin{lem}\label{lem:L2equations}
Suppose $\mathcal{M}$ is a module category for $\overline{\operatorname{Rep}(U_q(\mathfrak{sl}_N))^\omega}$, with fusion graph for the action of $X$ given by $\Gamma_{X}$. Then any cell system corresponding to the module $\mathcal{M}$ satisfies the following relations:  
\begin{align*}
\text{(Tr$(U_1)$)}:\quad \sum_{i,j}KW\left(\raisebox{-.5\height}{ \includegraphics[scale = .4]{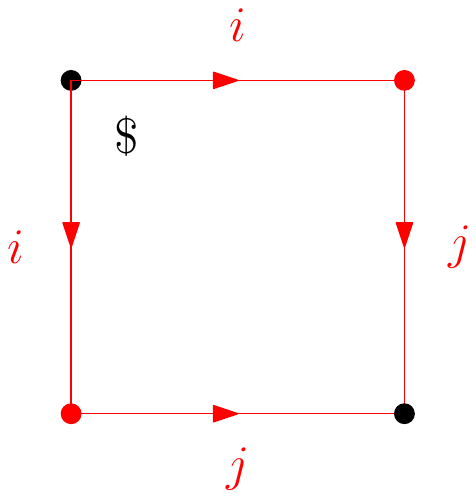}}\right)&=   \left(\Gamma_{\Lambda_2}\right)_{s(i), r(j)} \cdot  [2]_q\\
\text{(Tr$(U_1U_2)$)}:\quad \sum_{i,j,k}KW\left(\raisebox{-.5\height}{ \includegraphics[scale = .4]{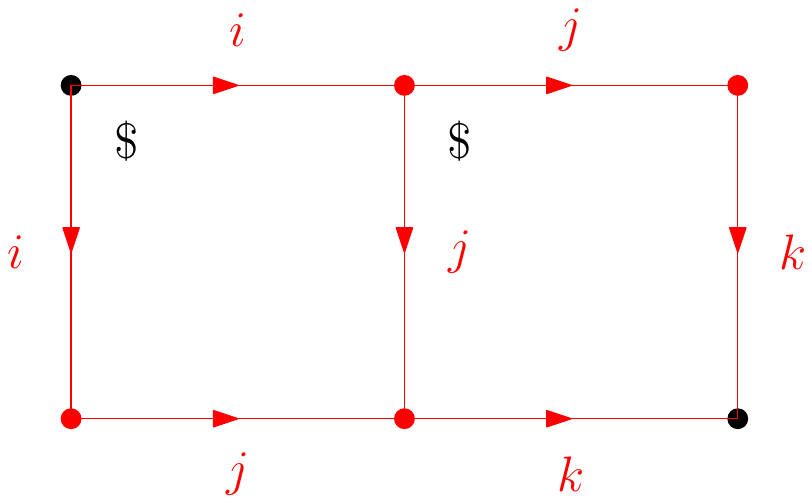}}\right)&=  \left( \Gamma_{\Lambda_3}\right)_{s(i), r(k)} \cdot  [2]^2_q + \left( \Gamma_{\Lambda_1}\cdot\Gamma_{\Lambda_2}   - \Gamma_{\Lambda_3}   \right)_{s(i), r(k)}.
\end{align*}
\end{lem}
\begin{proof}
These are a consequence of Lemma~\ref{lem:Hans}, which determines the trace of the embedding of an idempotent in terms of the module fusion rules. The first equation holds as we have 
\[ \raisebox{-.5\height}{ \includegraphics[scale = .3]{UU.pdf}}=[2]_q\raisebox{-.5\height}{ \includegraphics[scale = .3]{pl2.pdf}} .   \]
The second holds as 
\[ \raisebox{-.5\height}{ \includegraphics[scale = .3]{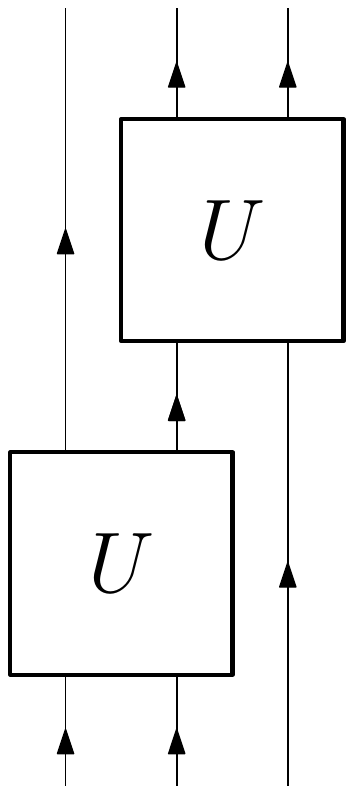}}=[2]^2_q\raisebox{-.5\height}{ \includegraphics[scale = .3]{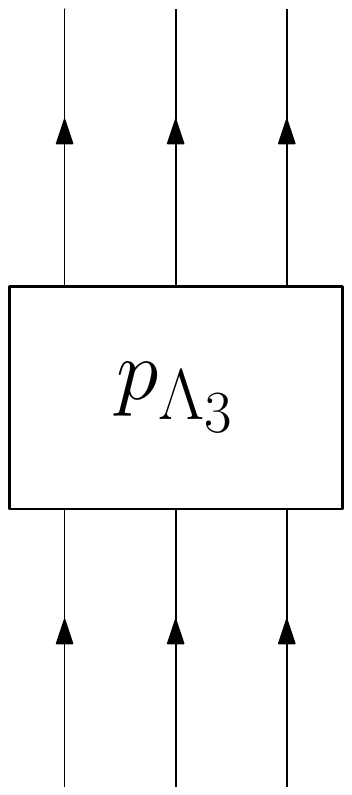}} +  \raisebox{-.5\height}{ \includegraphics[scale = .3]{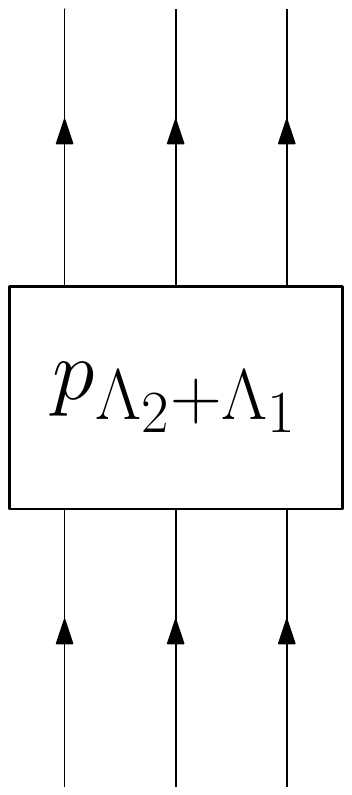}}  \]
along with the fusion rule $\Lambda_1 \otimes \Lambda_2 \cong \Lambda_{3} \oplus (\Lambda_{2}+\Lambda_1 )$.
\end{proof}
Certainly many more equation of this form can be derived using Lemma~\ref{lem:Hans}. For the examples considered in Section~\ref{sec:examples}, these two equations were sufficient (and incredibly useful).

\section{Examples}\label{sec:examples}
In this section we compute several solutions to a KW cell system on a variety of graphs. We restrict our attention to the $\mathfrak{sl}_4$ case. As this is the first case which hasn't been solved before. Solutions for the $\mathfrak{sl}_3$ case can be found in \cite{SU3}. 

The graph in the $\mathfrak{sl}_4$ case we take from the work of Ocneanu \cite{Ocneanu}. We can also find the graphs for action by $\Lambda_2$ in this work, allowing us to apply the equations from Lemma~\ref{lem:L2equations}. Note that the results of this section are not dependent on the correctness of \cite{Ocneanu}\footnote{No proofs were given in this paper.}. However our results in this section verify that Ocneanu's claims were correct.

We assume that Ocneanu found these graphs by solving the modular splitting equation \cite{Ocneanu,Xu} for the $SU(4)$ modular invariants. If one wanted to extend the results of this section to higher $SU(N)$, the modular splitting equation should allow one to obtain the graphs corresponding to the higher $SU(N)$ modular invariants. Discussion on this problem can be found in \cite[Section 6]{ModInv1} and \cite[Section 8]{ModInv2}.

 From \cite{ModulesPt2}, we have Theorem~\ref{thm:absClassIntro} which abstractly classifies irreducible module categories over $\mathcal{C}(\mathfrak{sl}_4, k)$. To remind the reader, we have
\begin{center}
\begin{tabular}{ |c|c c c c c c c| } 
 \hline
$k$ & 1 & 2 & 4 & 6 & 8& $k> 1$ odd & $k >8$ even \\\hline
$\#$ of Modules & 2&3&7&8&9&4&6 \\ 
 \hline
\end{tabular}
\end{center}
irreducible module categories over $\mathcal{C}(\mathfrak{sl}_4, k)$ up to equivalence.

In this section we construct all of these abstractly classified module categories of $\mathcal{C}(\mathfrak{sl}_4, k)$. Hence we upgrade the abstract classification to a concrete classification.

\begin{remark}We will use the matrix notation from Remark~\ref{rmk:matrix} to express our solutions. Recall that we will refer to the entries of the $U$ matrices as $U$-cells, and the entries of the $B$ matrices as $B$-cells. In the interest of space we do not include the solution to the B-cells for the exceptional solutions. These are easily obtained by solving the linear system (RI) + (BA) once the U-cells have been determined. The full solutions to the KW cell systems on the exceptional graphs can be found in Mathematica notebooks attached to the arXiv submission of this paper. We also include the Mathematica files ``CellLibrary.nb'' and ``CellLibrary-MultiEdge.nb'' which contain general functions that takes as input a graph (or multi-edged graph) and returns the polynomial equations for a KW system on that graph.
\end{remark}

For the 6 exceptional modules, we also construct KW cell system solutions with $\omega \in \{-1,\mathbf{i}, -\mathbf{i}\}$. That is, exceptional module categories for $\overline{\operatorname{Rep}(U_q(\mathfrak{sl}_4))^\omega}$ at the appropriate $q$ values. As far as we know, these modules over the twisted categories cannot be constructed by conformal inclusions, and this paper is the first time they have been constructed in any capacity. \textit{A-priori} we should expect to have to find a new solution to both the U and B cells for each value of $\omega$. We are fortunate in the sense that for each exceptional graph, the four solutions for each value of $\omega$ share the same U-cell solution. Hence once the U-cell solution is found, the four B-cell solutions can be found by simply solving the linear system (RI) + (BA). In slightly different language, this means that all four $\overline{\operatorname{Rep}(U_q(\mathfrak{sl}_4))^\omega}$ modules restrict to the same $\overline{\operatorname{Rep}(U_q(\mathfrak{gl}_4))}$ module.

For several of our solutions, we use the notion of \textit{orbifolding} to help us obtain a solution. We expect this to be equivariantisation for module categories (see \cite{Emily}). The concept of orbifolding has been explored in the physics context \cite{Orbo} and in the mathematical context \cite{MR1301620}. In the later paper, the orbifold procedure is made rigorous for relations occurring in the endomorphism algebras of the graph planar algebra. For us this encompasses relations (R2), (R3), and (H). For relations occurring in more general spaces, to the best of our knowledge, the orbifold procedure remains un-rigorous. Still, we can use this un-rigorous procedure to help determine a solution to a KW cell system. To make everything rigorous again, we simply explicitly verify all the relations for the solution. We use this orbifolding procedure to help find solutions in Subsections~\ref{sub:k4excep}, \ref{sub:k6excep2}, and \ref{sub:k8excep2}.

%\subsection*{$\mathfrak{sl}_3$ at level $5$}

%Let 
%\[\Gamma= \raisebox{-.5\height}{ \includegraphics[scale = .6]{SU35.pdf}}\]
%As $\Gamma$ is multiplicity free, we can uniquely specify our loops by a list of vertices. We will write 
%\[ U^{a,b}_{c,d} :=   \quad B_{a,b,c} :=  \]
%We have the following KW cell system on $\Gamma$:

%\begin{align*}
   % U^{1}_{\quad 2} &= U^{1}_{\quad 4} = U^{2}_{\quad 4} = U^{3}_{\quad 1} = U^{4}_{\quad 1} = 0\\
   % U^{1}_{\quad 3}&=U^{2}_{\quad 1} = U^{3}_{\quad 4} = U^{4}_{\quad 2} = [2]\\
  %  U^{2}_{\quad 2}&= \begin{bmatrix}
%\frac{[3]}{[2]} & \frac{\sqrt{[3]}}{[2]} \\
%\frac{\sqrt{[3]}}{[2]}&  \frac{1}{[2]} 
%\end{bmatrix}\\
%U^{3}_{\quad 3}&= \begin{bmatrix}
 %\frac{1}{[2]} & \frac{\sqrt{[3]}}{[2]} \\
%\frac{\sqrt{[3]}}{[2]}&  \frac{[3]}{[2]}
%\end{bmatrix}\\
%U^{2}_{\quad 3}&= \begin{bmatrix}
 %\frac{1}{[2]} & \frac{1}{[2]}& -\frac{1}{\sqrt{[3]}}\\
%\frac{1}{[2]} &  \frac{1}{[2]}&-\frac{1}{\sqrt{[3]}}\\
%-\frac{1}{\sqrt{[3]}}& -\frac{1}{\sqrt{[3]}} & \frac{[2]}{[3]}
%\end{bmatrix}\\
%U^{3}_{\quad 2}&= \begin{bmatrix}
%\frac{[2]}{[3]}  & \frac{1}{\sqrt{[3]}}&\frac{1}{\sqrt{[3]}}\\
 %\frac{1}{\sqrt{[3]}} & \frac{1}{[2]} &\frac{1}{[2]}\\
%\frac{1}{\sqrt{[3]}}&\frac{1}{[2]} & \frac{1}{[2]}
%\end{bmatrix}\\
%\end{align*}
%\begin{align*}
%B_{1,2,3} &=- B_{4,3,2} = 1\\
%B_{2,2,2}&=-B_{3,3,3} =\frac{1}{[2]}\\
%B_{2,3,1}&= B_{3,1,2} = - B_{2,4,3} =  -B_{3,2,4} =  \frac{1}{[3]}\\
%B_{2,2,3}  &=  B_{2,3,2}= B_{2,3,3}= B_{3,2,2}  = B_{3,2,3} = B_{3,3,2} = \frac{1}{[2]\sqrt{[3]}}   
%\end{align*}

\subsection{An Exceptional Module for $SU(4)$ at Level $4$}\label{sub:k4excep}
In this section we construct a cell system with parameters $q = e^{2\pi i \frac{1}{16}}$ on the following graph
\[\Gamma_{\Lambda_1}^{4,4,\subset} =\raisebox{-.5\height}{ \includegraphics[scale = .6]{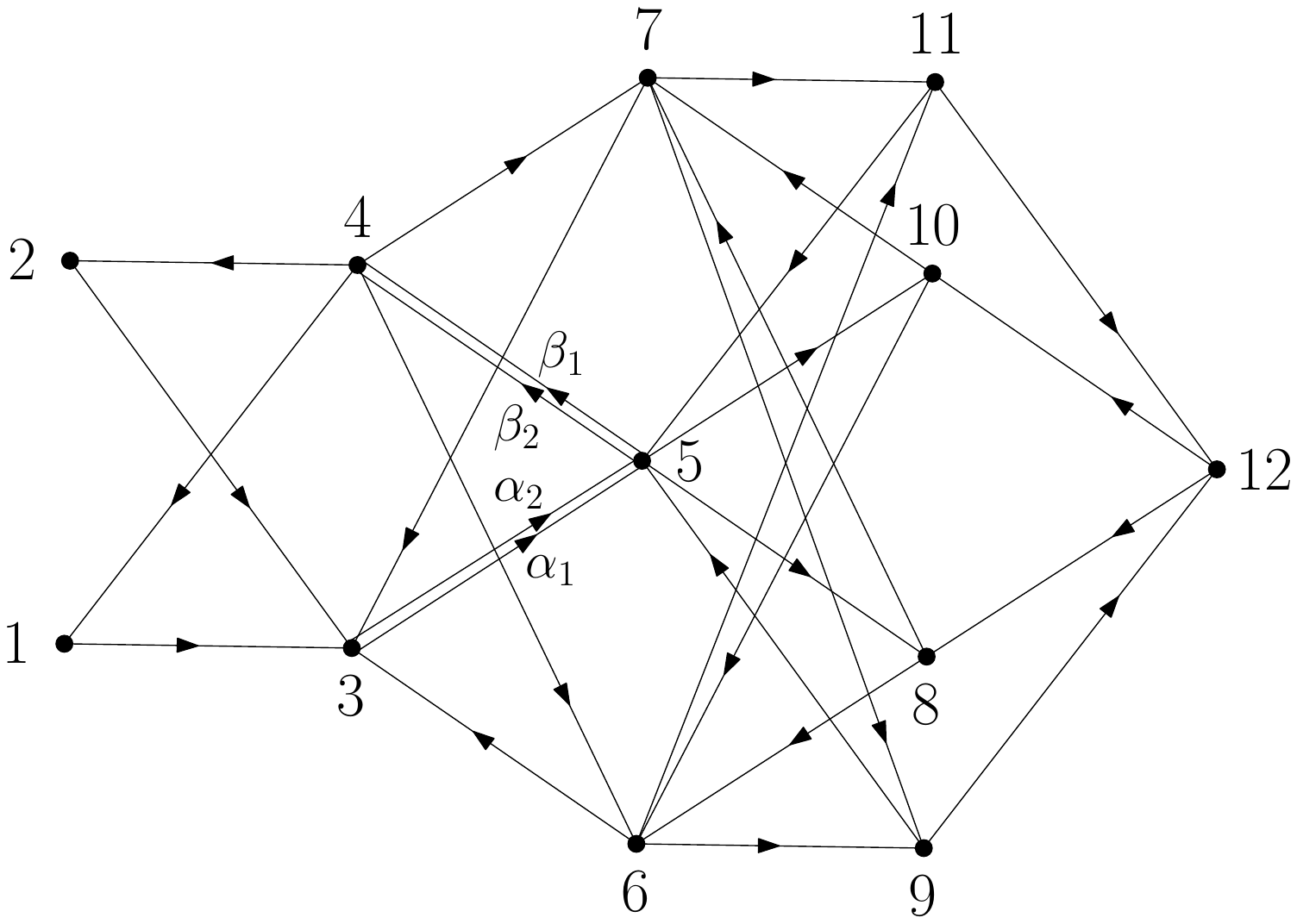}} \]
The unique positive eigenvector is
\[\lambda = \left\{1,1, \q{4},\q{4}, \frac{\q{3}\q{4}}{\q{2}}, \q{3}, \q{3}, \q{2}, \q{2}, \q{2},\q{2},\frac{\q{4}}{\q{2}}\right\}.\]
We assume that graph for action by $\Lambda_2$ is
\[\Gamma_{\Lambda_2}^{4,4,\subset} =\raisebox{-.5\height}{ \includegraphics[scale = .6]{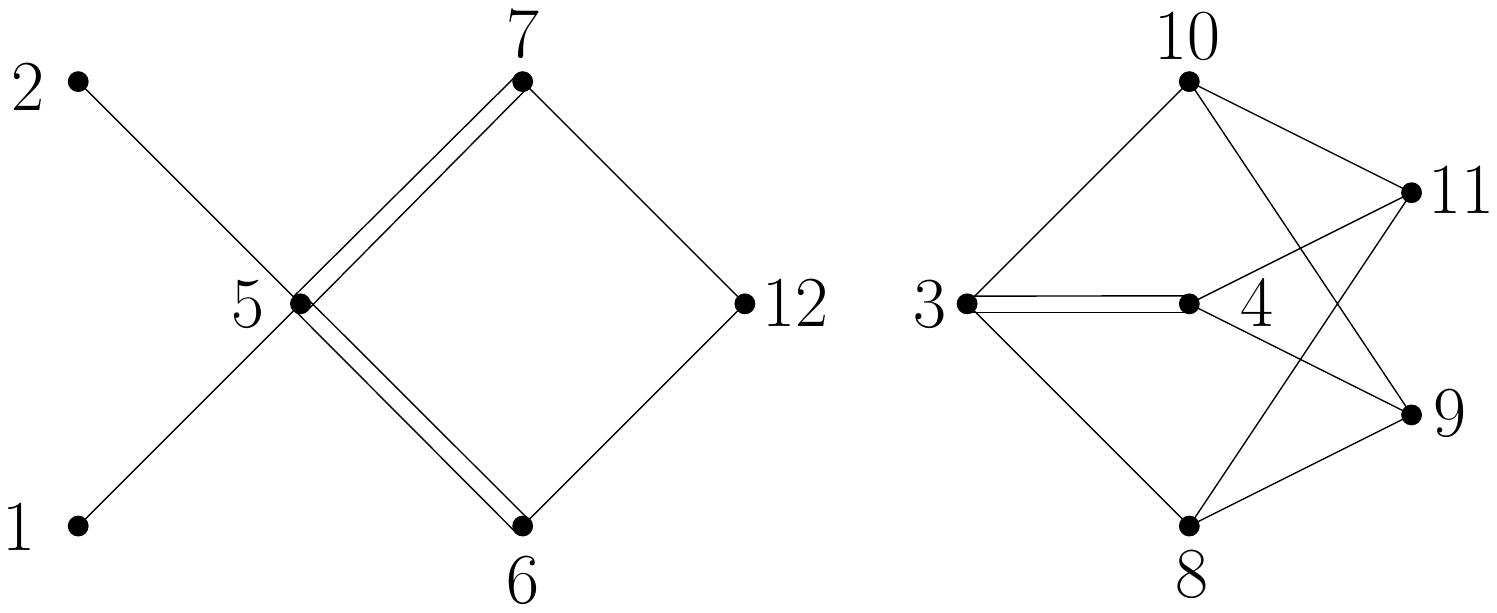}} \]
The above data for this graph can also be found in the Mathematica file ``k=4/Conformal Inclusion/Data.nb".

To solve for the U-cells on this graph, we first observe the following $\mathbb{Z}_2$ symmetry on $\Gamma_{\Lambda_1}^{4,4,\subset}$:
\[ 1\longleftrightarrow 2\qquad 6\longleftrightarrow 7 \qquad 8\longleftrightarrow 10\qquad 9\longleftrightarrow 11.\]
The orbifold graph of $\Gamma_{\Lambda_1}^{4,4,\subset}$ with respect to this symmetry, is again $\Gamma_{\Lambda_1}^{4,4,\subset}$. This suggests that there is a U-cell solution on $\Gamma_{\Lambda_1}^{4,4,\subset}$ which comes from orbifolding a U-cell solution invariant under the $\mathbb{Z}_2$ symmetry. Hence we have two avenues of attack. We can either try assume that our solution looks like an orbifold solution (and hence contains many 0's), or that the solution is invariant under the $\mathbb{Z}_2$ symmetry. 

We first attempted the approach of assuming the $\mathbb{Z}_2$ symmetry. However we were unable to solve the system in this setting. By assuming our solution comes from an orbifold, we see that the $4\times 4$ block $U^{3}_{\quad 4}$ is of the form
\[ U^{3}_{\quad 4} = \begin{blockarray}{cccc}
{}^{\alpha_1}5^{\beta_1} & {}^{\alpha_1}5^{\beta_2} & {}^{\alpha_2}5^{\beta_1} &{}^{\alpha_2}5^{\beta_2} \\
\begin{block}{[cccc]}
 x &0&0&y\\
 0 & z&w&0\\
 0 &\overline{w}&\q{2}-z\\
 \overline{y} & 0 & 0 &\q{2}-x\\
\end{block}
\end{blockarray}\]
for some complex scalars $x,y,z,w\in  \mathbb{C}$. To determine these scalars, we hand-pick a collection of equations from Tr($U_1U_2$) and (R3) which contain these four variables (along with several other coefficients), and numerically approximate a solution. From this numerical approximation, we can guess that (up to the graph automorphisms exchanging $\alpha_1 \leftrightarrow \alpha_2$ and $\beta_1 \leftrightarrow \beta_2$) that $x =\frac{\q{3}}{\q{4}} +\frac{\sqrt{\q{3}}}{\q{2}} $ and $y =  \frac{\q{3}}{\q{4}}$. We then use (Hecke) to pin down $w$ and $y$ up to phases.

With this seed information, we can then solve Tr($U_1U_2$) completely. This gives the diagonal elements of all of our matrices. We can then use (Hecke) to solve the $2\times 2$ blocks up to phases. At this point, there are enough linear and quadratic equations in (R3) to pin down the five remaining $4\times 4$ blocks (making a couple of arbitary choices). After choosing a concrete gauge choice, we arrive at the following solution\footnote{We should point out that the initial solution we found for the U-cells did not admit a solution for the B-cells. This is because we had actually found the embedding of $p_{2\Lambda_1}$, instead of $p_{\Lambda_2}$. Due to level rank duality these morphisms are indistinguishable with respect to the relations (R1), (R2), (R3), and (Hecke). To obtain the correct solution, we used the formula $p_{\Lambda_2} = \id_{\Lambda_1\otimes \Lambda_1} - p_{2\Lambda_1}$.}:
\begin{align*} U^{3}_{\quad 4} &= \begin{blockarray}{cccc}
{}^{\alpha_1}5^{\beta_1} & {}^{\alpha_1}5^{\beta_2} & {}^{\alpha_2}5^{\beta_1} &{}^{\alpha_2}5^{\beta_2} \\
\begin{block}{[cccc]}
 \frac{\q{3}}{\q{4}} -\frac{\sqrt{\q{3}}}{\q{2}} &0&0&-\frac{1}{\q{4}}\\
 0 & \frac{\q{3}}{\q{4}}&\frac{\q{3}}{\q{4}}&0\\
 0 & \frac{\q{3}}{\q{4}}&\frac{\q{3}}{\q{4}}&0\\
 -\frac{1}{\q{4}} & 0 & 0 &\frac{\q{3}}{\q{4}} +\frac{\sqrt{\q{3}}}{\q{2}} \\
\end{block}
\end{blockarray}\quad U^{4}_{\quad 3} = \begin{blockarray}{cccc}
1& 2 &6 &7 \\
\begin{block}{[cccc]}
\frac{\q{3}}{\q{4}} &-\frac{1}{\q{4}}&-\frac{\sqrt{\q{3}}}{\q{4}}&-\frac{\sqrt{\q{3}}}{\q{4}}\\
 -\frac{1}{\q{4}} &\frac{\q{3}}{\q{4}}&\mathbf{i}\frac{\sqrt{\q{3}}}{\q{4}}&-\mathbf{i}\frac{\sqrt{\q{3}}}{\q{4}}\\
-\frac{\sqrt{\q{3}}}{\q{4}} &- \mathbf{i}\frac{\sqrt{\q{3}}}{\q{4}}&\frac{\q{3}}{\q{4}}&\mathbf{i}\frac{1}{\q{4}}\\
 -\frac{\sqrt{\q{3}}}{\q{4}} & \mathbf{i}\frac{\sqrt{\q{3}}}{\q{4}} &- \mathbf{i}\frac{1}{\q{4}} &\frac{\q{3}}{\q{4}}\\
\end{block}
\end{blockarray}\\
U^{5}_{\quad 6} &= \begin{blockarray}{cccc}
{}^{\beta_1}4 & {}^{\beta_2}4 & 8 &10 \\
\begin{block}{[cccc]}
 \frac{\q{3}}{\q{4}} +\frac{1}{\sqrt{\q{3}\q{4}^2}} &-\frac{1}{\sqrt{\q{2}\q{4}}}&0&-\frac{\sqrt{\q{4}}}{\sqrt{\q{2}\q{3}}}\\
 -\frac{1}{\sqrt{\q{2}\q{4}}} & \frac{\q{3}}{\q{4}} -\frac{1}{\sqrt{\q{3}\q{4}^2}} &-\frac{\sqrt{\q{4}}}{\sqrt{\q{2}\q{3}}}&0\\
 0 &-\frac{\sqrt{\q{4}}}{\sqrt{\q{2}\q{3}}}&\frac{\q{3}}{\q{4}} +\frac{1}{\sqrt{\q{3}\q{4}^2}} &\frac{1}{\sqrt{\q{2}\q{4}}}\\
 -\frac{\sqrt{\q{4}}}{\sqrt{\q{2}\q{3}}} & 0 & \frac{1}{\sqrt{\q{2}\q{4}}} &\frac{\q{3}}{\q{4}} -\frac{1}{\sqrt{\q{3}\q{4}^2}} \\
\end{block}
\end{blockarray}= U^{6}_{\quad 5} = U^{5}_{\quad 7}\\
 U^{7}_{\quad 5}  &= \begin{blockarray}{cccc}
3^{\alpha_1} & 3^{\alpha_2} & 9 &11 \\
\begin{block}{[cccc]}
 \frac{\q{3}}{\q{4}} +\frac{1}{\sqrt{\q{3}\q{4}^2}} &\frac{1}{\sqrt{\q{2}\q{4}}}&-\frac{\sqrt{\q{4}}}{\sqrt{\q{2}\q{3}}}&0\\
\frac{1}{\sqrt{\q{2}\q{4}}} & \frac{\q{3}}{\q{4}} -\frac{1}{\sqrt{\q{3}\q{4}^2}} &0&-\mathbf{i}\frac{\sqrt{\q{4}}}{\sqrt{\q{2}\q{3}}}\\
- \frac{\sqrt{\q{4}}}{\sqrt{\q{2}\q{3}}} &0&\frac{\q{3}}{\q{4}} -\frac{1}{\sqrt{\q{3}\q{4}^2}} &- \mathbf{i}\frac{1}{\sqrt{\q{2}\q{4}}}\\
 0 & \mathbf{i}\frac{\sqrt{\q{4}}}{\sqrt{\q{2}\q{3}}} & \mathbf{i}\frac{1}{\sqrt{\q{2}\q{4}}} &\frac{\q{3}}{\q{4}} +\frac{1}{\sqrt{\q{3}\q{4}^2}}\\ 
\end{block}
\end{blockarray}\quad U^{10}_{\quad 11}  = \begin{blockarray}{cc}
6 & 7 \\
\begin{block}{[cc]}
 \frac{\q{3}}{\q{4}} + \frac{\sqrt{\q{3}}}{\q{2}} &-\frac{1}{\q{4}}\\
-\frac{1}{\q{4}}&   \frac{\q{3}}{\q{4}} - \frac{\sqrt{\q{3}}}{\q{2}} \\
\end{block}
\end{blockarray}  \\
U^{1}_{\quad 5}  &= \begin{blockarray}{cc}
3^{\alpha_1} & 3^{\alpha_2}  \\
\begin{block}{[cc]}
 \frac{\q{3} - \sqrt{\q{3}}}{\q{4}} &   \mathbf{i}\frac{\sqrt{\q{3}}}{\sqrt{\q{2}\q{4}}}\\
- \mathbf{i}\frac{\sqrt{\q{3}}}{\sqrt{\q{2}\q{4}}}&  \frac{\q{3} + \sqrt{\q{3}}}{\q{4}} \\
\end{block}
\end{blockarray}   = U^{4}_{\quad 9}= U^{5}_{\quad 1}= U^{7}_{\quad 12}= U^{8}_{\quad 3}=U^{10}_{\quad 3}\qquad U^{8}_{\quad 9}  = \begin{blockarray}{cc}
6 & 7 \\
\begin{block}{[cc]}
 \frac{\q{3}}{\q{4}} -\frac{\sqrt{\q{3}}}{\q{2}} &\mathbf{i}\frac{1}{\q{4}}\\
-\mathbf{i}\frac{1}{\q{4}}&   \frac{\q{3}}{\q{4}} + \frac{\sqrt{\q{3}}}{\q{2}} \\
\end{block}
\end{blockarray} \\
U^{3}_{\quad 8}  &= \begin{blockarray}{cc}
{}^{\alpha_1}5 & {}^{\alpha_2}5 \\
\begin{block}{[cc]}
 \frac{\q{3} + \sqrt{\q{3}}}{\q{4}} &-\frac{\sqrt{\q{3}}}{\sqrt{\q{2}\q{4}}}\\
 -\frac{\sqrt{\q{3}}}{\sqrt{\q{2}\q{4}}}&  \frac{\q{3} - \sqrt{\q{3}}}{\q{4}} \\
\end{block}
\end{blockarray}   = U^{4}_{\quad 11}= U^{6}_{\quad 12}= U^{9}_{\quad 4}= U^{12}_{\quad 6}=U^{12}_{\quad 7}\qquad U^{8}_{\quad 11}  = \begin{blockarray}{cc}
6 & 7 \\
\begin{block}{[cc]}
 \frac{\q{3}}{\q{4}} & -\frac{\q{3}}{\q{4}}\\
 -\frac{\q{3}}{\q{4}}&   \frac{\q{3}}{\q{4}}\\
\end{block}
\end{blockarray}  \\
U^{2}_{\quad 5}  &= \begin{blockarray}{cc}
{3}^{\alpha_1} & {3}^{\alpha_2} \\
\begin{block}{[cc]}
 \frac{\q{3} - \sqrt{\q{3}}}{\q{4}} &-\mathbf{i}\frac{\sqrt{\q{3}}}{\sqrt{\q{2}\q{4}}}\\
\mathbf{i} \frac{\sqrt{\q{3}}}{\sqrt{\q{2}\q{4}}}&  \frac{\q{3} + \sqrt{\q{3}}}{\q{4}} \\
\end{block}
\end{blockarray}   = U^{5}_{\quad 2}\qquad
U^{9}_{\quad 8}  = \begin{blockarray}{cc}
5 & 12 \\
\begin{block}{[cc]}
 \frac{1}{\q{2}} &-\frac{\sqrt{\q{3}}}{\q{2}}\\
-\frac{\sqrt{\q{3}}}{\q{2}}&  \frac{\q{3}}{\q{2}} \\
\end{block}
\end{blockarray} = U^{9}_{\quad 10} = U^{11}_{\quad 8}\\
U^{3}_{\quad 10}  &= \begin{blockarray}{cc}
{}^{\alpha_1}5 & {}^{\alpha_2}5 \\
\begin{block}{[cc]}
 \frac{\q{3} + \sqrt{\q{3}}}{\q{4}} &\frac{\sqrt{\q{3}}}{\sqrt{\q{2}\q{4}}}\\
 \frac{\sqrt{\q{3}}}{\sqrt{\q{2}\q{4}}}&  \frac{\q{3} - \sqrt{\q{3}}}{\q{4}} \\
\end{block}
\end{blockarray}   = U^{11}_{\quad 4}\qquad U^{11}_{\quad 10}  = \begin{blockarray}{cc}
5 & 12 \\
\begin{block}{[cc]}
 \frac{1}{\q{2}} &\frac{\sqrt{\q{3}}}{\q{2}}\\
\frac{\sqrt{\q{3}}}{\q{2}}&  \frac{\q{3}}{\q{2}} \\
\end{block}
\end{blockarray} \qquad U^{10}_{\quad 9}  = \begin{blockarray}{cc}
6 & 7 \\
\begin{block}{[cc]}
 \frac{\q{3}}{\q{4}} & \mathbf{i}\frac{\q{3}}{\q{4}}\\
-\mathbf{i} \frac{\q{3}}{\q{4}}&   \frac{\q{3}}{\q{4}}\\
\end{block}
\end{blockarray} 
\quad 
\end{align*}
The unlabeled row/column orderings are
\begin{align*}
&&U^{6}_{\quad 5} &: \{3^{\alpha_1},3^{\alpha_2},9,11\}
&&U^{5}_{\quad 7} &&: \{10,8, 5^{\beta_1},5^{\beta_2}\}
&&U^{5}_{\quad 1} &: \{{}^{\beta_1}4,{}^{\beta_2}4\}
&&U^{10}_{\quad 3} &&: \{6,7\}\\
&&U^{8}_{\quad 3} &: \{6,7\}
&&U^{4}_{\quad 9} &&: \{7,6\}
&&U^{7}_{\quad 12} &: \{11,9\}
&&U^{4}_{\quad 11} &&: \{7,6\}\\
&&U^{6}_{\quad 12} &: \{11,9\}
&&U^{9}_{\quad 4} &&: \{5^{\beta_1},5^{\beta_2}\}
&&U^{12}_{\quad 6} &: \{10,8\}
&&U^{12}_{\quad 7} &&: \{8,10\}\\
&&U^{5}_{\quad 2} &: \{{}^{\beta_1}4,{}^{\beta_2}4\}
&&U^{9}_{\quad 10} &&: \{12,5\}
&&U^{11}_{\quad 8} &: \{12,5\}
&&U^{11}_{\quad 4} &&: \{5^{\beta_1},5^{\beta_2}\}
\end{align*}
This solution can be found in the Mathematica file "k=4/Conformal Inclusion/Solutions/w=1/Solution.nb". A computer verifies relations (R1), (R2), (R3), and (Hecke) in just under 3 minutes. A record of this verification can be found in "k=4/Conformal Inclusion/Solutions/w=1/Verification.nb".

Solving the linear equations (RI) and (BA) give a 1-dimensional solution space for the B-cells, which we normalise to satisfy (N). This solution for the B-cells can also be found in the Mathematica file "k=4/Conformal Inclusion/Solutions/w=1/Solution/w=1.nb". A computer verifies relations (RI), (BA), and (N) for this solution in 15 seconds. A record of this can be found in "k=4/Conformal Inclusion/Solutions/w=1/Verification.nb". As we have a solution to the KW cell system on $\Gamma^{4,4,\subset}_{\Lambda_1}$, we have the following theorem.

\begin{thm}
There exists a rank 12 module category $\mathcal{M}$ for $\mathcal{C}(\mathfrak{sl}_4, 4)$ such that the fusion graph for action by $\Lambda_1$ is $\Gamma^{4,4,\subset}_{\Lambda_1}$.
\end{thm}

We also find KW cell system solutions on the graph $\Gamma^{4,4,\subset}_{\Lambda_1}$ when $\omega\in \{-1,\mathbf{i}, -\mathbf{i}\}$. The solutions and verification of these solutions can be found in the folder ``k=4/Conformal Inclusion/Solutions".

\begin{thm}
For each $\omega \in \{-1, \mathbf{i}, -\mathbf{i}\}$ there exists a rank 12 module category $\mathcal{M}$ for $\overline{\operatorname{Rep}(U_{e^{2\pi i \frac{1}{16}}}(\mathfrak{sl}_4))^\omega}$ such that the fusion graph for action by $\Lambda_1$ is $\Gamma^{4,4,\subset}_{\Lambda_1}$.
\end{thm}

\subsection{An Exceptional Module for $SU(4)$ at Level $6$}\label{subsec:6I}
We will construct a cell system on the following graph for $N=4$ and $q = e^{2\pi i \frac{1}{20}}$ (i.e. a module category for $\mathcal{C}(\mathfrak{sl}_4, 6)$).
\[\Gamma_{\Lambda_1}^{4,6,\subset}= \raisebox{-.5\height}{ \includegraphics[scale = .6]{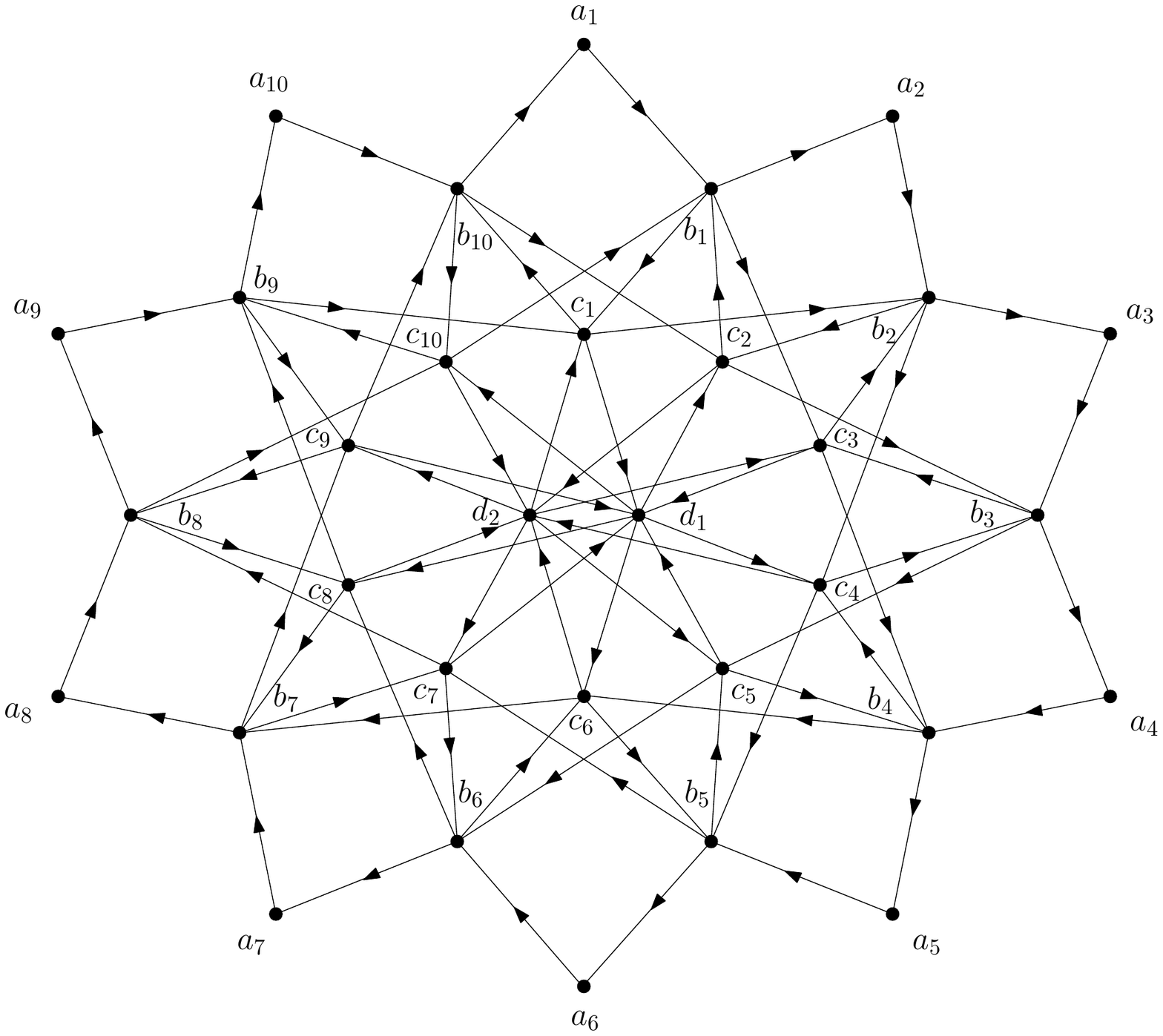}}\]
\begin{remark}
For ease of notation, we will assume the subscripts of the $a,b$, and $c$ vertices are taken mod $10$ (so that $a_{11} = a_{1}$ ect.), and the subscripts of the $d$ vertices are taken mod 2 (so that $d_{3} = d_{1}$).
\end{remark}
This graph has the positive eigenvector $\lambda$ (with eigenvalue $[4]_q$):
\[ \lambda_{a_i} = 1, \qquad \lambda_{b_i} = [4]_q, \qquad \lambda_{c_i} = \frac{[3]_q[4]_q}{[2]_q}, \qquad \lambda_{d_i} = \frac{[2]_q[4]^2_q}{[3]_q}.   \]

The graph for action by $\Lambda_2$ we assume to be 
\[\Gamma_{\Lambda_2}^{4,6,\subset}=\raisebox{-.5\height}{ \includegraphics[scale = .6]{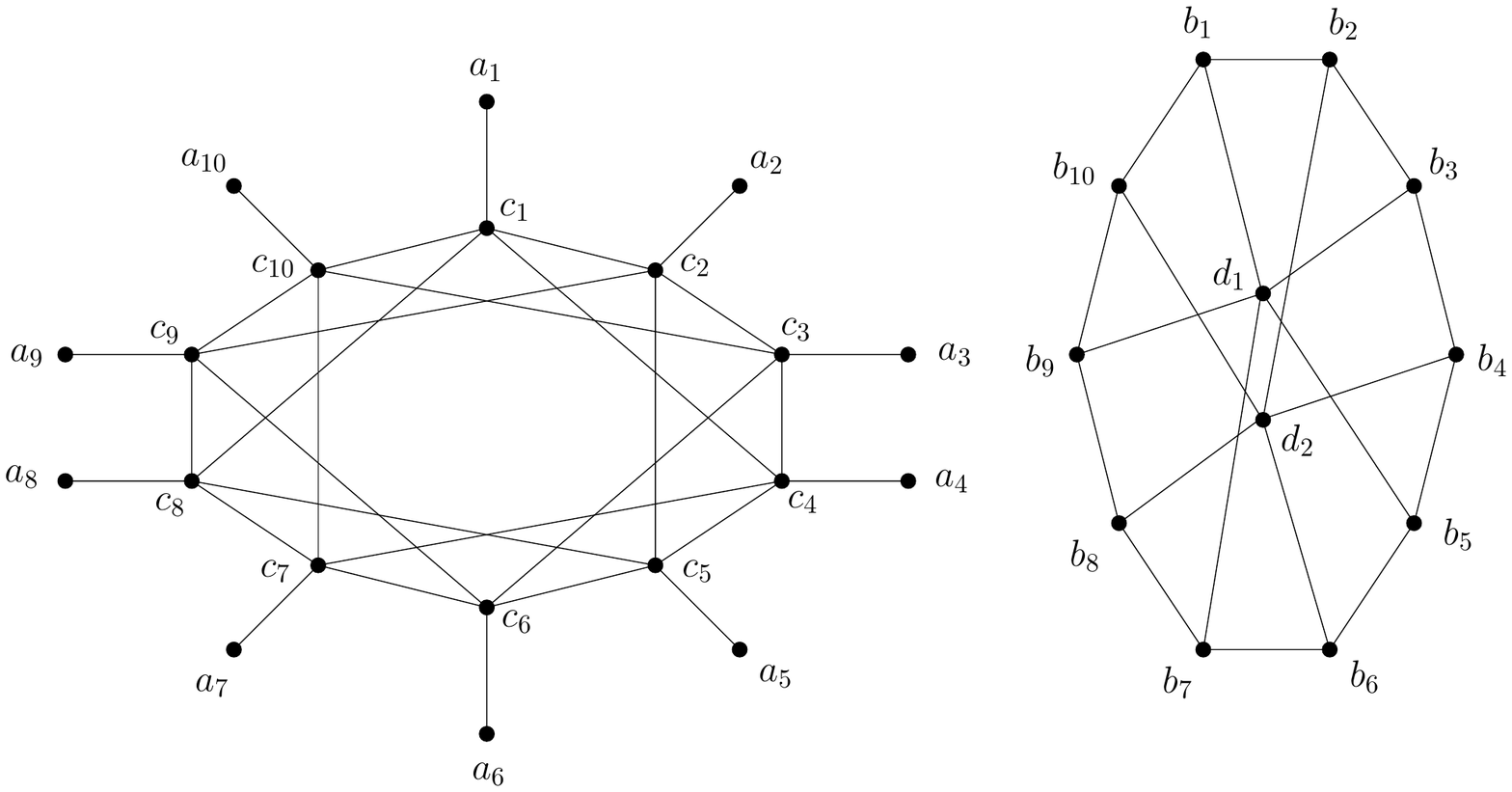}}\]

The above data for this graph can also be found in the Mathematica file ``k=6/Conformal Inclusion/Data.nb".
% From \cite{cain}, we know there is at most one cell system on this graph up to equivalence. As we are only trying to find one cell system, and not directly determine all cell systems, we are free to make as many simplifying assumptions as we like (provided the solution we obtain satisfies all of (R1), (R2), (R3), and (Hecke)).
 
To try find a cell system, we begin by assuming that the $U$ cells are invariant under the $\mathbb{Z}_{10}$ symmetry of the graph, and that the coefficients of the $U$ cells are all real. These assumptions reduce the complexity to a level that a computer can quickly find the following solution (hence justifying our assumptions):
\begin{align*}
     U^{a_i}_{\quad a_{i+1}} &= U^{a_i}_{\quad c_{i+2}}= 0,      & U^{a_i}_{\quad c_{i}} &= [2]_q \\
     U^{b_i}_{\quad b_{i+1}} &=   \frac{1}{[2]^{\frac{3}{2}}_q}\begin{blockarray}{ccc}
a_{i+1} & c_i & c_{i+1}  \\
\begin{block}{[ccc]}
  \sqrt{[2]_q[3]_q}  & \sqrt{[4]_q}& \sqrt{[4]_q}\\
  \sqrt{[4]_q} &\sqrt{[2]_q} &\sqrt{[2]_q}\\
 \sqrt{[4]_q}&\sqrt{[2]_q} &\sqrt{[2]_q}\\
\end{block}
\end{blockarray}
&  U^{b_i}_{\quad b_{i+3}} &=0\\
 U^{b_i}_{\quad d_{i}} &= \frac{1}{[2]_q}\begin{blockarray}{cc}
  c_i & c_{i+2}  \\
\begin{block}{[cc]}
  1  & \sqrt{[3]_q}\\
  \sqrt{[3]_q} &[3]_q\\
\end{block}
\end{blockarray} &  U^{b_i}_{\quad b_{i-1}} &=[2]_q \qquad U^{c_i}_{\quad a_{i}} =
[2]_q\\
  U^{c_i}_{\quad c_{i+1}} &= \frac{1}{[3]^{\frac{3}{2}}_q}\begin{blockarray}{ccc}
b_{i-1} & b_{i+1} & d_{i}  \\
\begin{block}{[ccc]}
 [4]_q  & -[4]_q &-\sqrt{[2]_q [3]_q }\\
 -[4]_q  & [4]_q  &\sqrt{[2]_q [3]_q }\\
-\sqrt{[2]_q [3]_q }&\sqrt{[2]_q [3]_q } &[2]_q\\
\end{block}
\end{blockarray}& U^{c_i}_{\quad a_{i+2}}&=U^{c_i}_{\quad c_{i+5}}=  U^{c_i}_{\quad a_{i-3}} = 0 \\
U^{c_i}_{\quad c_{i+3}} &= \frac{1}{[3]_q}\begin{blockarray}{cc}
b_{i+1}  & d_{i}  \\
\begin{block}{[cc]}
 [4]_q  & \sqrt{[2]_q[4]_q}\\
 \sqrt{[2]_q[4]_q}  & [2]_q  \\
\end{block}
\end{blockarray} & U^{c_i}_{\quad c_{i-1}} &= \frac{1}{[3]_q}\begin{blockarray}{cc}
b_{i-1}  & d_{i}  \\
\begin{block}{[cc]}
 [2]_q  & \sqrt{[2]_q[4]_q}\\
 \sqrt{[2]_q[4]_q}  & [4]_q  \\
\end{block}
\end{blockarray} & & \\
U^{d_i}_{\quad d_{i+1}} &= \frac{1}{[2]_q^2[4]_q}\begin{blockarray}{ccccc}
c_{i+1}  & c_{i+3}& c_{i+5}&c_{i+7}&c_{i+9}  \\
\begin{block}{[ccccc]}
 [3]_q[5]_q  & [3]_q & -[3]^2_q & -[3]_q^2 & [3]_q\\
 \hspace{.01em}[3]_q & [3]_q[5]_q  & [3]_q & -[3]^2_q & -[3]_q^2 \\
 -[3]_q^2 & [3]_q & [3]_q[5]_q  & [3]_q & -[3]^2_q  \\
  -[3]^2_q & -[3]_q^2 & [3]_q & [3]_q[5]_q  & [3]_q   \\
    \hspace{.01em} [3]_q  &   -[3]^2_q & -[3]_q^2 & [3]_q & [3]_q[5]_q \\
\end{block}
\end{blockarray}  &  U^{d_i}_{\quad b_{j}} &= \frac{1}{[2]_q}\begin{blockarray}{cc}
c_{j-1}  & c_{j+1}  \\
\begin{block}{[cc]}
 [3]_q  & -\sqrt{[3]_q}\\
 -\sqrt{[3]_q}  & 1  \\
\end{block}
\end{blockarray}  &   &  \\
\end{align*}
This solution can be found in the Mathematica file "k=6/Conformal Inclusion/Solutions/w=1/Solution.nb". A computer verifies relations (R1), (R2), (R3), and (Hecke) in just over 3 minutes. A record of this verification can be found in "k=6/Conformal Inclusion/Solutions/w=1/Verification.nb".

We can now solve the linear system given by equations (RI) and (BA) to determine the $B$ cells up to a single scalar. Solving relation (N) determines the norm of this scalar. Hence we have a solution for the $B$ cells, up to a single phase. Fixing a natural choice for this phase gives a concrete solution to a KW cell system on $\Gamma_{\Lambda_1}^{4,6,\subset}$. This solution can be found in ``k=6/Conformal Inclusion/Solutions/w=1/Solution.nb".
\begin{remark}
Note that these $B$ coefficients have projective $\mathbb{Z}_{10}$ symmetry, with factor $-1$ for a one-click rotation.
\end{remark}

%To show that the above coefficients give a cell system, we have to verify (OB) and (AS). Due to the the (projective) symmetry of the $U$ and $B$ coefficients, and the form the (OB) and (AS) equations, it suffices to verify these equations when the boundary loop begins with either $a_1$, $b_1$, $c_1$, or $d_1c_2$.

%We use exact arithmetic on Mathematica to verify both of the equations (OB) and (AS), working in a specific degree 16 number field. For (OB) we have 350 individual equations to check. This took just over half an hour on a standard desktop computer. For (AS) we have 3360 individual equations to check. This took just over an hour.
A computer verifies relations (RI), (BA), and (N) for our solution in under half a minute. As a consequence of our computations, we have the following result.
\begin{thm}
There exists a rank 32 module category $\mathcal{M}$ for $\mathcal{C}(\mathfrak{sl}_4, 6)$ such that the fusion graph for action by $\Lambda_1$ is $\Gamma_{\Lambda_1}^{4,6,\subset}$.
\end{thm}
From \cite{ModulesPt2, LevelTerry}, there is a unique rank 32 module category over $\mathcal{C}(\mathfrak{sl}_4, 6)$, and it has the additional structure of a fusion category. Hence the rank 32 module category we have constructed must be this fusion category. In particular, the adjoint subcategory of this fusion category is an interesting $3^{\mathbb{Z}_5}$ quadratic category.

We also find KW cell system solutions on the graph $\Gamma^{4,6,\subset}_{\Lambda_1}$ when $\omega\in \{-1,\mathbf{i}, -\mathbf{i}\}$. The solutions and verification of these solutions can be found in the folder``k=6/Conformal Inclusion/Solutions".

\begin{thm}
For each $\omega \in \{-1, \mathbf{i}, -\mathbf{i}\}$ there exists a rank 32 module category $\mathcal{M}$ for $\overline{\operatorname{Rep}(U_{e^{2\pi i \frac{1}{20}}}(\mathfrak{sl}_4))^\omega}$ such that the fusion graph for action by $\Lambda_1$ is $\Gamma^{4,6,\subset}_{\Lambda_1}$.
\end{thm}

\subsection{A Second Exceptional Module for $SU(4)$ at Level $6$}\label{sub:k6excep2}
We will construct a cell system on the following graph where $N=4$ and $q = e^{2\pi i \frac{1}{20}}$.
\[\Gamma_{\Lambda_1}^{4,6,\subset_{\mathbb{Z}_5}}=\raisebox{-.5\height}{ \includegraphics[scale = .6]{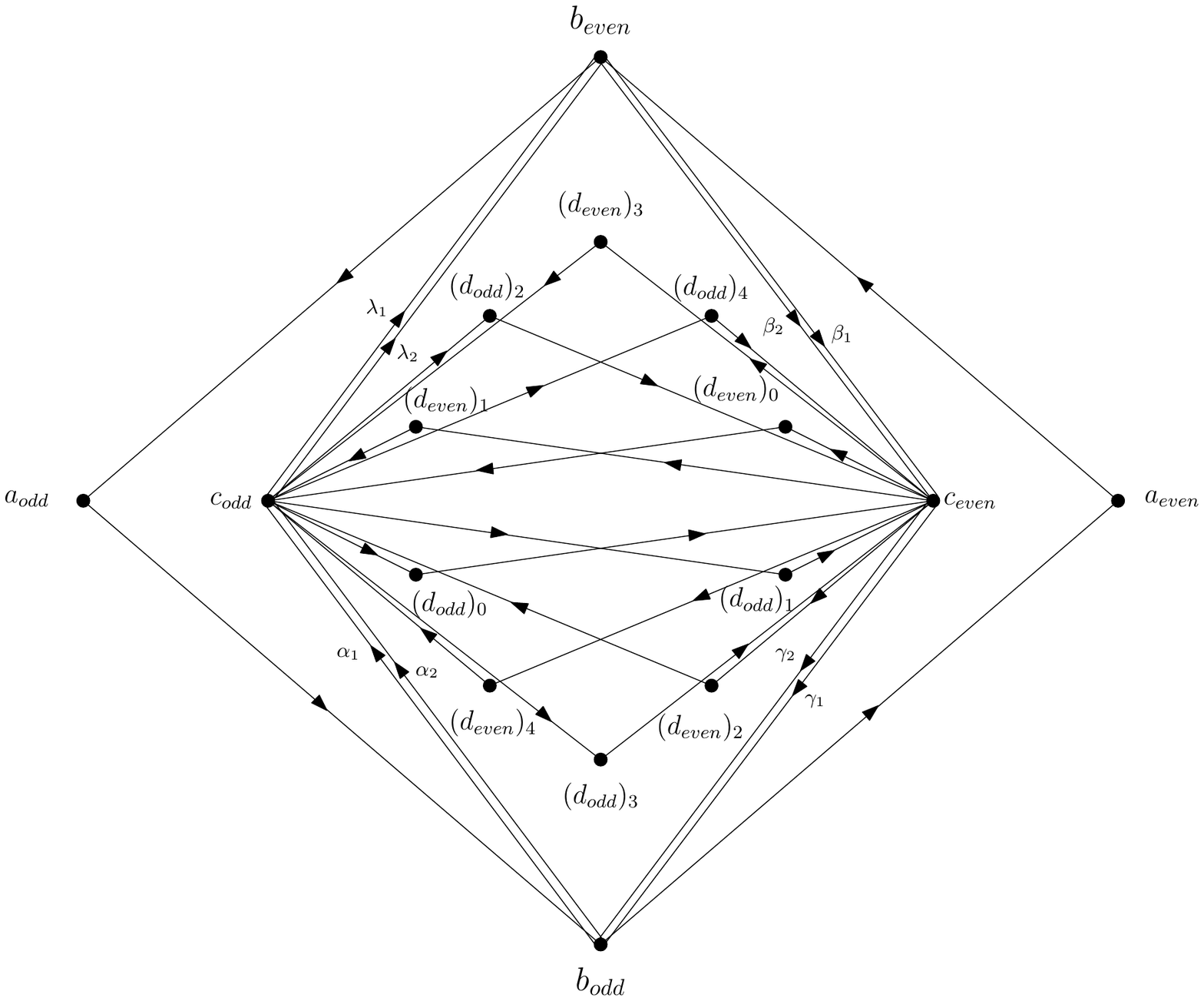}}\]

This graph has the positive eigenvector $\lambda$ (with eigenvalue $[4]_q$):
\[\lambda_{a_{odd}} = \lambda_{a_{even}} = 1, \quad  \lambda_{b_{odd}} = \lambda_{b_{even}} = \q{4},\quad \lambda_{c_{odd}} = \lambda_{c_{even}} =\frac{\q{3}\q{4}}{\q{2}},\quad \lambda_{(d_{odd})_i} = \lambda_{(d_{even})_i} =\frac{\q{3}}{\q{2}}. \]

The graph for action by $\Lambda_2$ we assume to be
\[  \Gamma_{\Lambda_2}^{4,6,\subset_{\mathbb{Z}_5}} =\raisebox{-.5\height}{ \includegraphics[scale = .6]{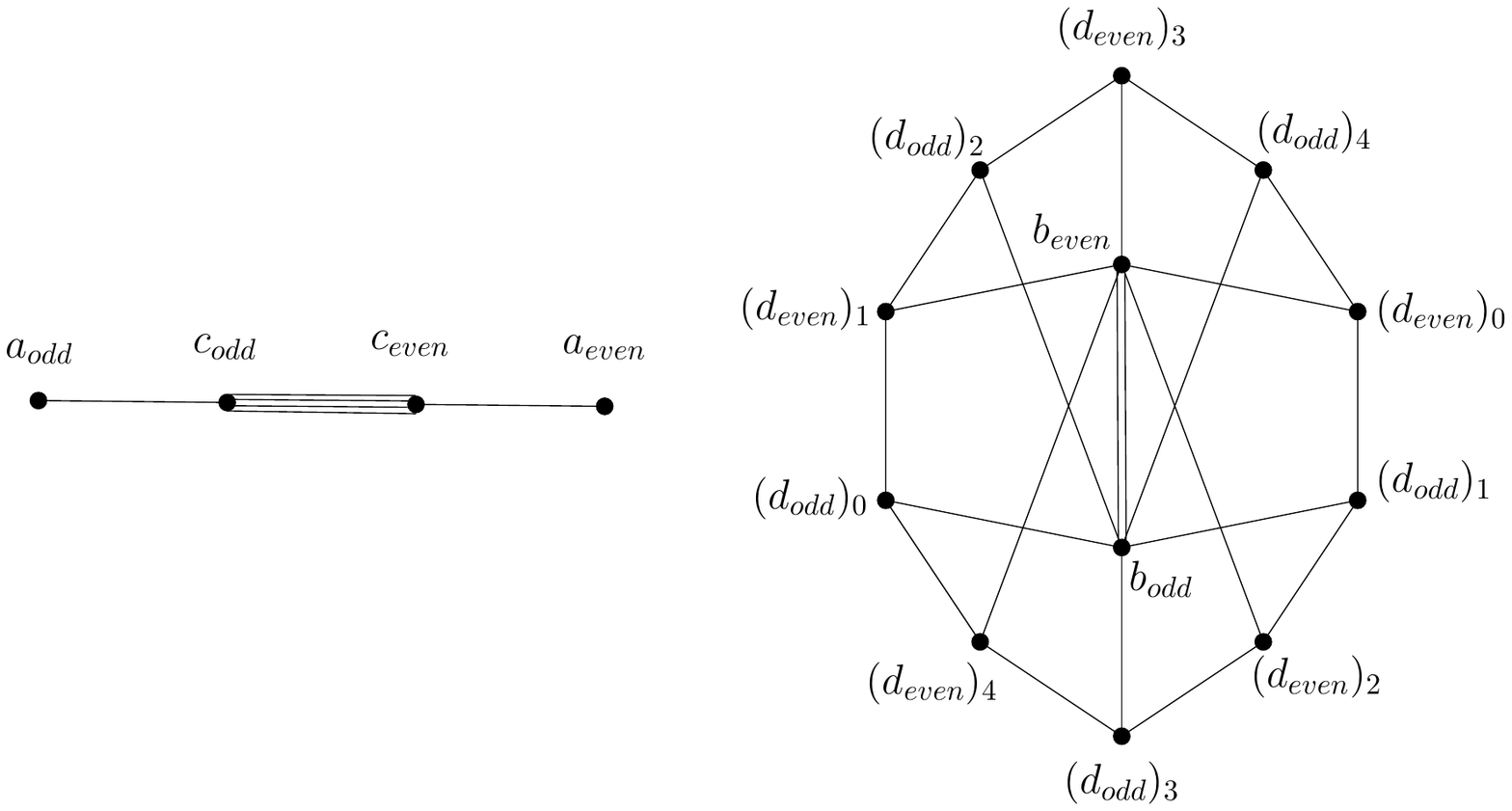}} \]

The above data for this graph can also be found in the Mathematica file ``k=6/Conjugate/Data.nb".

Recall the $\mathbb{Z}_5$ symmetry on $\Gamma_{\Lambda_1}^{4,6,\subset}$ from Subsection~\ref{subsec:6I}. By making the natural orbifold identifications of $a_{odd} \leftrightarrow \{a_i \mid \text{ i odd }\}$ ect. with the $\mathbb{Z}_5$ orbifold of $\Gamma_{\Lambda_1}^{4,6,\subset}$, along with 
\begin{align*}
    \alpha_1 &\leftrightarrow \{b_i \to c_i \mid i \text{ odd}\}\quad
    \alpha_2 \leftrightarrow \{b_i \to c_{ i+2} \mid i\text{ odd}\}\quad
    \beta_1 \leftrightarrow \{b_i \to c_i \mid i\mid i\text{ even}\}\quad
    \beta_2 \leftrightarrow \{b_i \to c_{i \mid i+2} \mid i\text{ even}\}\\
    \gamma_1 &\leftrightarrow \{c_i \to b_{i-1} \mid i\text{ even}\}\quad
    \gamma_2 \leftrightarrow \{c_i \to b_{i+1} \mid i\text{ even}\}\quad
    \lambda_1 \leftrightarrow \{c_i \to b_{i-1} \mid i\text{ odd}\}\quad
    \lambda_2 \leftrightarrow \{c_i \to b_{i+1}\mid i\text{ odd}\}\\
\end{align*} 
we can determine the U-cells for any pair of paths not passing through either $(d_{odd})_i$ or $(d_{even})_i$ by taking representatives. For example, the cell $U^{a_{odd}, b_{odd}^{\alpha_1}}_{b_{odd}^{\alpha_1}, c_{odd}}$ is equal to $U^{a_1, b_1}_{b_1, c_1} = [2]_q$, where as $U^{a_{odd}, b_{odd}^{\alpha_2}}_{b_{odd}^{\alpha_2}, c_{odd}}$ is equal to $U^{a_1, b_1}_{b_1, c_3}=0$. Note that the representatives of $\alpha_1$ and $\alpha_2$ do not connect in $\Gamma_{\Lambda_1}^{4,6,\subset}$, which implies $U^{a_{odd}, b_{odd}^{\alpha_1}}_{b_{odd}^{\alpha_2}, c_{odd}}=0$. Together we deduce
\[  U^{a_{odd}}_{\quad c_{odd}} =  \begin{blockarray}{cc}
b_{odd}^{\alpha_1}  & b_{odd}^{\alpha_2}  \\
\begin{block}{[cc]}
 [2]_q  & 0\\
0   & 0\\
\end{block}
\end{blockarray}.    \]

This non-rigorous procedure gives us the U-cells for any paths which do not pass through either $(d_{odd})_i$ or $(d_{even})_i$. Recall the U-cells for the graph $\Gamma_{\Lambda_1}^{4,6,\subset}$ satisfied $\mathbb{Z}_{10}$ symmetry. As we only used $\mathbb{Z}_5$ symmetry for the orbifold, we naivly assume that the U-cells for $\Gamma_{\Lambda_1}^{4,6,\mathbb{Z}_5}$ have $\mathbb{Z}_2$ symmetry under the following graph automorphism:
\[ a_{odd}\leftrightarrow a_{even}, \quad b_{odd}\leftrightarrow b_{even} , \quad c_{odd}\leftrightarrow c_{even}, \quad (d_{odd})_i\leftrightarrow (d_{even})_i, \quad \alpha_i \leftrightarrow \beta_i ,\quad \gamma_i \leftrightarrow \lambda_i .  \]

This gives us enough of a seed to solve the remaining U-cells, with the help of the additional equation Tr($U_1$) (for this case, the equation Tr($U_1U_2$) gives no additional information). The remaining cells (up to the above symmetry) are:

%\[
%U^{(d_{odd})_i}_{\quad (d_{even})_j}=\begin{cases}
%[2]_q & i =  j \pm 1 \pmod 5\\
%0  &\text{otherwise}
%\end{cases}\]

%\begin{align*}
%U^{b_{odd}}_{\quad (d_{odd})_i} &= \frac{1}{[2]_q}\begin{blockarray}{cc}
%{}^{\alpha_1} c_{odd}  & {}^{\alpha_2} c_{odd}  \\
%\begin{block}{[cc]}
%1  & \zeta_5^i\sqrt{[3]_q}\\
%\zeta_5^{-i} \sqrt{[3]_q}  & [3]_q  
%\end{block}
%\end{blockarray}\\
%U^{(d_{odd})_i}_{\quad b_{odd}} &= \frac{1}{[2]_q}\begin{blockarray}{cc}
%c_{even}^{\gamma_1}  &c_{even}^{ \gamma_2  }\\
%\begin{block}{[cc]}
%1  & -\zeta_5^i\sqrt{[3]_q}\\
%-\zeta_5^{-i} \sqrt{[3]_q}  & [3]_q  
%\end{block}
%\end{blockarray}
%\end{align*}
\[\resizebox{\hsize}{!}{$U^{b_{odd}}_{\quad (d_{odd})_i} = \frac{1}{[2]_q}\begin{blockarray}{cc}
{}^{\alpha_1} c_{odd}  & {}^{\alpha_2} c_{odd}  \\
\begin{block}{[cc]}
1  & \zeta_5^i\sqrt{[3]_q}\\
\zeta_5^{-i} \sqrt{[3]_q}  & [3]_q  \\
\end{block}
\end{blockarray}\quad 
U^{(d_{odd})_i}_{\quad b_{odd}} = \frac{1}{[2]_q}\begin{blockarray}{cc}
c_{even}^{\gamma_1}  &c_{even}^{ \gamma_2  }\\
\begin{block}{[cc]}
1  & -\zeta_5^i\sqrt{[3]_q}\\
-\zeta_5^{-i} \sqrt{[3]_q}  & [3]_q  \\
\end{block}
\end{blockarray}\quad U^{(d_{odd})_i}_{\quad (d_{even})_j}=\begin{cases}
[2]_q & i =  j \pm 1 \pmod 5\\
0  &\text{otherwise}
\end{cases}$}\]
\[  \resizebox{\hsize}{!}{$U^{c_{odd}}_{\quad c_{even}}=\begin{blockarray}{ccccccccc}
(d_{odd})_0  & (d_{odd})_1 & (d_{odd})_2& (d_{odd})_3& (d_{odd})_4  &  {}^{\lambda_1}b_{even}^{\beta_1}&  {}^{\lambda_1}b_{even}^{\beta_2}&  {}^{\lambda_2}b_{even}^{\beta_1}&  {}^{\lambda_2}b_{even}^{\beta_2}\\
\begin{block}{[ccccccccc]}
\frac{\sqrt{[3]_q}}{[2]_q}&\frac{-1}{[2]_q\sqrt{[3]_q}}&\frac{-\zeta_5}{[2]_q\sqrt{[3]_q}}&\frac{-\zeta_5^{-1}}{[2]_q\sqrt{[3]_q}}&\frac{1}{[2]_q\sqrt{[3]_q}} & \frac{1}{\sqrt{[2]_q[4]_q}} &\frac{-1}{\sqrt{[2]_q[3]_q[4]_q}}& \frac{1}{\sqrt{[2]_q[3]_q[4]_q}}& \frac{1}{\sqrt{[2]_q[4]_q}} \\
\frac{-1}{[2]_q\sqrt{[3]_q}}&\frac{\sqrt{[3]_q}}{[2]_q}&\frac{-1}{[2]_q\sqrt{[3]_q}}&\frac{-\zeta_5}{[2]_q\sqrt{[3]_q}}&\frac{-\zeta_5^{-1}}{[2]_q\sqrt{[3]_q}} & \frac{\zeta_5^{4}}{\sqrt{[2]_q[4]_q}} &\frac{-\zeta_5}{-\sqrt{[2]_q[3]_q[4]_q}}& \frac{\zeta_5}{\sqrt{[2]_q[3]_q[4]_q}}& \frac{\zeta_5^{3}}{\sqrt{[2]_q[4]_q}} \\
\frac{-\zeta_5^{-1}}{[2]_q\sqrt{[3]_q}}&\frac{-1}{[2]_q\sqrt{[3]_q}}&\frac{\sqrt{[3]_q}}{[2]_q}&\frac{-1}{[2]_q\sqrt{[3]_q}}&\frac{-\zeta_5}{[2]_q\sqrt{[3]_q}} & \frac{\zeta_5^{3}}{\sqrt{[2]_q[4]_q}} &\frac{-\zeta_5^{2}}{\sqrt{[2]_q[3]_q[4]_q}}& \frac{\zeta_5^{2}}{\sqrt{[2]_q[3]_q[4]_q}}& \frac{\zeta_5}{\sqrt{[2]_q[4]_q}} \\
\frac{-\zeta_5}{[2]_q\sqrt{[3]_q}}&\frac{-\zeta_5^{-1}}{[2]_q\sqrt{[3]_q}}&\frac{-1}{[2]_q\sqrt{[3]_q}}&\frac{\sqrt{[3]_q}}{[2]_q}&\frac{-1}{[2]_q\sqrt{[3]_q}}& \frac{\zeta_5^{2}}{\sqrt{[2]_q[4]_q}} &\frac{-\zeta_5^{3}}{\sqrt{[2]_q[3]_q[4]_q}}& \frac{\zeta_5^{3}}{\sqrt{[2]_q[3]_q[4]_q}}& \frac{\zeta_5^{4}}{\sqrt{[2]_q[4]_q}} \\
\frac{-1}{[2]_q\sqrt{[3]_q}}&\frac{-\zeta_5}{[2]_q\sqrt{[3]_q}}&\frac{-\zeta_5^{-1}}{[2]_q\sqrt{[3]_q}}&\frac{-1}{[2]_q\sqrt{[3]_q}}&\frac{\sqrt{[3]_q}}{[2]_q} & \frac{\zeta_5}{\sqrt{[2]_q[4]_q}} &\frac{-\zeta_5^{4}}{\sqrt{[2]_q[3]_q[4]_q}}& \frac{\zeta_5^{4}}{\sqrt{[2]_q[3]_q[4]_q}}& \frac{\zeta_5^{2}}{\sqrt{[2]_q[4]_q}} \\
\frac{1}{\sqrt{[2]_q[4]_q}}&\frac{\zeta_5}{\sqrt{[2]_q[4]_q}}&\frac{\zeta_5^{2}}{\sqrt{[2]_q[4]_q}}&\frac{\zeta_5^{3}}{\sqrt{[2]_q[4]_q}}&\frac{\zeta_5^{4}}{\sqrt{[2]_q[4]_q}}& \frac{[2]_q}{[3]_q}&0 &0 &0  \\ 
\frac{-1}{\sqrt{[2]_q\q{3}[4]_q}}&\frac{-\zeta_5^{4}}{\sqrt{[2]_q\q{3}[4]_q}}&\frac{-\zeta_5^{3}}{\sqrt{[2]_q\q{3}[4]_q}}&\frac{-\zeta_5^{2}}{\sqrt{[2]_q\q{3}[4]_q}}&\frac{-\zeta_5}{\sqrt{[2]_q\q{3}[4]_q}}&0 &\frac{[4]_q}{[3]_q^{\frac{3}{2}}} &-\frac{[4]_q}{[3]_q^{\frac{3}{2}}} & 0 \\ 
\frac{1}{\sqrt{[2]_q\q{3}[4]_q}}&\frac{\zeta_5^{4}}{\sqrt{[2]_q\q{3}[4]_q}}&\frac{\zeta_5^{3}}{\sqrt{[2]_q\q{3}[4]_q}}&\frac{\zeta_5^{2}}{\sqrt{[2]_q\q{3}[4]_q}}&\frac{\zeta_5}{\sqrt{[2]_q\q{3}[4]_q}}&0 &-\frac{[4]_q}{[3]_q^{\frac{3}{2}}} & \frac{[4]_q}{[3]_q^{\frac{3}{2}}}& 0 \\ 
\frac{1}{\sqrt{[2]_q[4]_q}}&\frac{\zeta_5^{2}}{\sqrt{[2]_q[4]_q}}&\frac{\zeta_5^{4}}{\sqrt{[2]_q[4]_q}}&\frac{\zeta_5}{\sqrt{[2]_q[4]_q}}&\frac{\zeta_5^{3}}{\sqrt{[2]_q[4]_q}}&0 &0 &0 & \frac{[4]_q}{[3]_q}  \\ 
\end{block}
\end{blockarray} $}  \]

This is all fairly straightforward to solve, apart from the $9\times 9$ block which requires a little bit of lucky guesswork. The full solution to the U-cells can be found in ``k=6/Orbifold/Solutions/w=1/Solution.nb". The ordering of our vertices in this file is the following:
\[\resizebox{\hsize}{!}{$\left(a_{odd}, b_{odd}, a_{even}, b_{even}, c_{odd}, c_{even}, (d_{even})_{4} ,(d_{odd})_{0},(d_{even})_{1} ,(d_{odd})_{2},(d_{even})_{3} ,(d_{odd})_{4} ,(d_{even})_{0} ,(d_{odd})_{1},(d_{even})_{2} ,(d_{odd})_{3}\right)  $}  \]

We use a computer to verify relations (R1), (R2), (R3), and (H) in under 30 minutes. A record of this verification can be found in ``k=6/Orbifold/Solutions/w=1/Verification.nb".

Solving the linear systems (RI) and (BA) yields (as always) a 1-dimensional solution space for the B-cells. Using (N) we pin our solution down to a free phase, which we make a natural choice for. The B-cells satisfy projective $\mathbb{Z}_2$ symmetry with respect to the earlier graph automorphism, with factor $-1$. The solution for the B-cells can also be found in ``k=6/Orbifold/Solutions/w=1/Solution.nb". A computer verifies relations (RI), (BA), and (N) in under 15 seconds. A record of this verification can also be found in ``k=6/Orbifold/Solutions/w=1/Verification.nb".

\begin{thm}\label{thm:SU6C}
There exists a rank 16 module category $\mathcal{M}$ for $\mathcal{C}(\mathfrak{sl}_4, 6)$ such that the fusion graph for action by $\Lambda_1$ is $\Gamma_{\Lambda_1}^{4,6,\subset_{\mathbb{Z}_5}}$ 
\end{thm}

We also find KW cell system solutions on the graph $\Gamma_{\Lambda_1}^{4,6,\subset_{\mathbb{Z}_5}}$ when $\omega\in \{-1,\mathbf{i}, -\mathbf{i}\}$. The solutions and verification of these solutions can be found in the folder ``k=6/Orbifold/Solutions".

\begin{thm}
For each $\omega \in \{-1, \mathbf{i}, -\mathbf{i}\}$ there exists a rank 16 module category $\mathcal{M}$ for $\overline{\operatorname{Rep}(U_{e^{2\pi i \frac{1}{20}}}(\mathfrak{sl}_4))^\omega}$ such that the fusion graph for action by $\Lambda_1$ is $\Gamma_{\Lambda_1}^{4,6,\subset_{\mathbb{Z}_5}}$.
\end{thm}

Note that the U-cells for these KW cell systems solutions respect the $\mathbb{Z}_2$ symmetry from the start of this subsection. However the B-cells only respect the symmetry when $\omega = \pm \mathbf{i}$. This suggests that when $\omega = \pm \mathbf{i}$ we can orbifold the KW cell system solutions to obtain KW cell system solutions on the following $\mathbb{Z}_2$ orbifold graph of $\Gamma_{\Lambda_1}^{4,6,\subset_{\mathbb{Z}_5}}$.
\[  \Gamma_{\Lambda_1}^{4,6,\subset_{\mathbb{Z}_{10}}}:= \raisebox{-.5\height}{ \includegraphics[scale = .6]{su4610L1.pdf}}   \]
Note that by classification \cite{ModulesPt2} there can't be a module associated to this graph when $\omega = 1$. This example is interesting as it suggests that the classification of exceptional modules over $\overline{\operatorname{Rep}(U_{q}(\mathfrak{sl}_N))^\omega}$ is richer when $\omega \neq 1$.

As expected, we find solutions to the KW cell system on $\Gamma_{\Lambda_1}^{4,6,\subset_{\mathbb{Z}_{10}}}$ when $\omega = \pm \mathbf{i}$. These solutions and their verifications can be found in ``k=6/Orbifold2/Solutions". There appears to be no solution when $\omega = -1$.

\begin{thm}
For each $\omega \in \{\mathbf{i}, -\mathbf{i}\}$ there exists a rank 8 module category $\mathcal{M}$ for $\overline{\operatorname{Rep}(U_{e^{2\pi i \frac{1}{20}}}(\mathfrak{sl}_4))^\omega}$ such that the fusion graph for action by $\Lambda_1$ is $\Gamma_{\Lambda_1}^{4,6,\subset_{\mathbb{Z}_{10}}}$.
\end{thm}

\subsection{An Exceptional Module for $SU(4)$ at Level $8$}\label{subsec:8C}
We will construct a cell system on the following graph for $N=4$ and $q = e^{2\pi i \frac{1}{24}}$ (i.e. a module category for $\mathcal{C}(\mathfrak{sl}_4, 8)$).
\[\Gamma^{4,8,\subset_{\mathbb{Z}_2}}_{\Lambda_1}=\raisebox{-.5\height}{ \includegraphics[scale = .6]{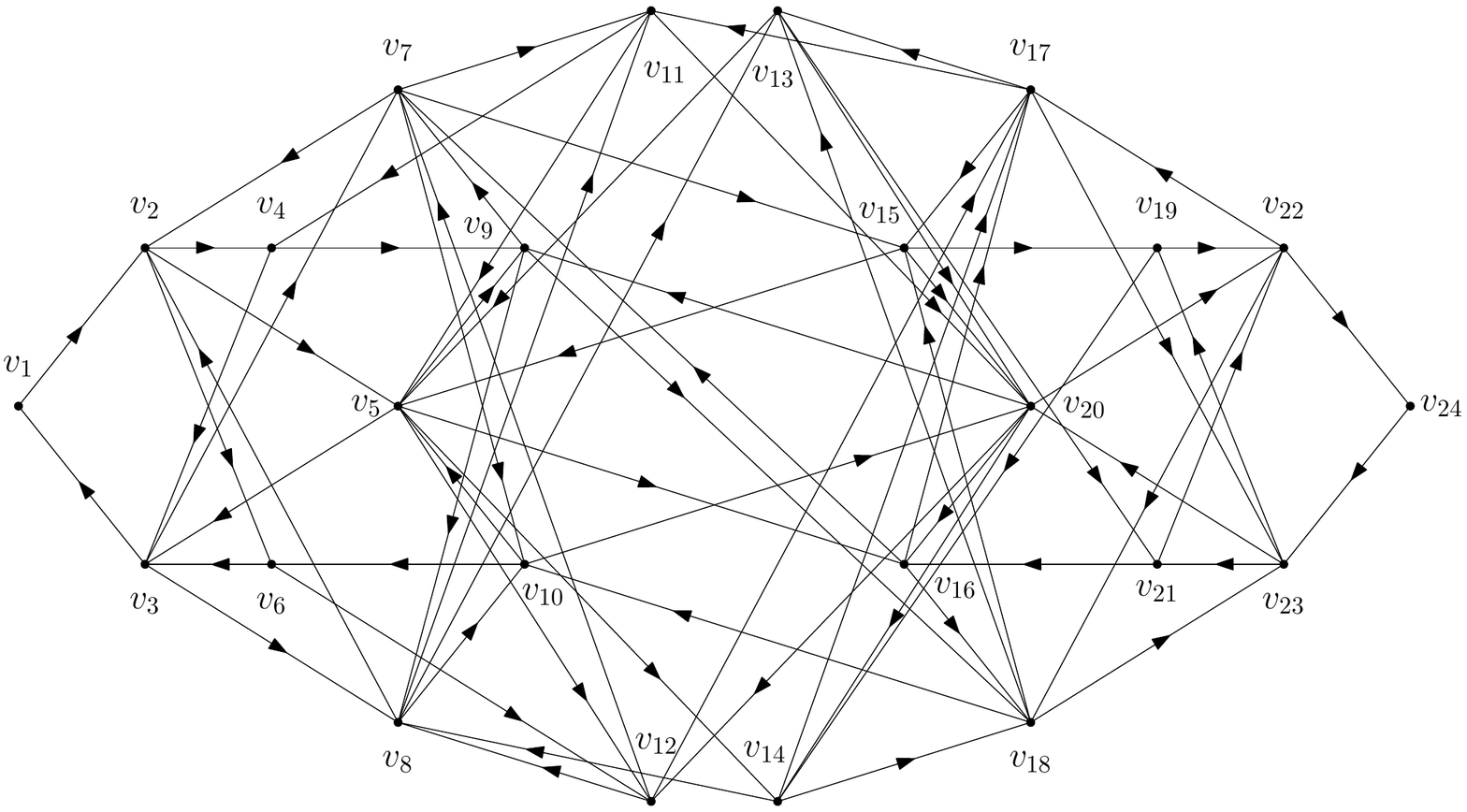}}\]

This graph has the positive eigenvector $\lambda$ (with eigenvalue $[4]_q$):
\begin{align*}
\lambda_{v_1} = \lambda_{v_{24}}& = 1,\quad\hspace{17em}
\lambda_{v_{2}}= \lambda_{v_{3}}= \lambda_{v_{22}}= \lambda_{v_{23} } &&=\q{4},\\
\lambda_{v_{4}}= \lambda_{v_{6}}= \lambda_{v_{19}}= \lambda_{v_{21} } &=\frac{ [4]_q [5]_q}{[2]_q  [3]_q},\quad\hspace{20em}
\lambda_{v_{5}}= \lambda_{v_{20}} &&=\frac{ [4]_q [5]_q}{ [2]_q}\\
\lambda_{v_{7}}= \lambda_{v_{8}}= \lambda_{v_{17}}= \lambda_{v_{18} } &=\frac{[2]_q [3]_q [5]_q}{[6]_q}\qquad
\lambda_{v_{9}}= \lambda_{v_{10}}= \lambda_{v_{11}}= \lambda_{v_{12} } =\lambda_{v_{13}}= \lambda_{v_{14}}= \lambda_{v_{15}}= \lambda_{v_{16} }&&=\frac{[3]_q  [4]_q [5]_q}{ [2]_q[6]_q}\\
\end{align*}

The graph for action by $\Lambda_2$ we assume to be:
\[\Gamma^{4,8,\subset_{\mathbb{Z}_2}}_{\Lambda_2}=\raisebox{-.5\height}{ \includegraphics[scale = .6]{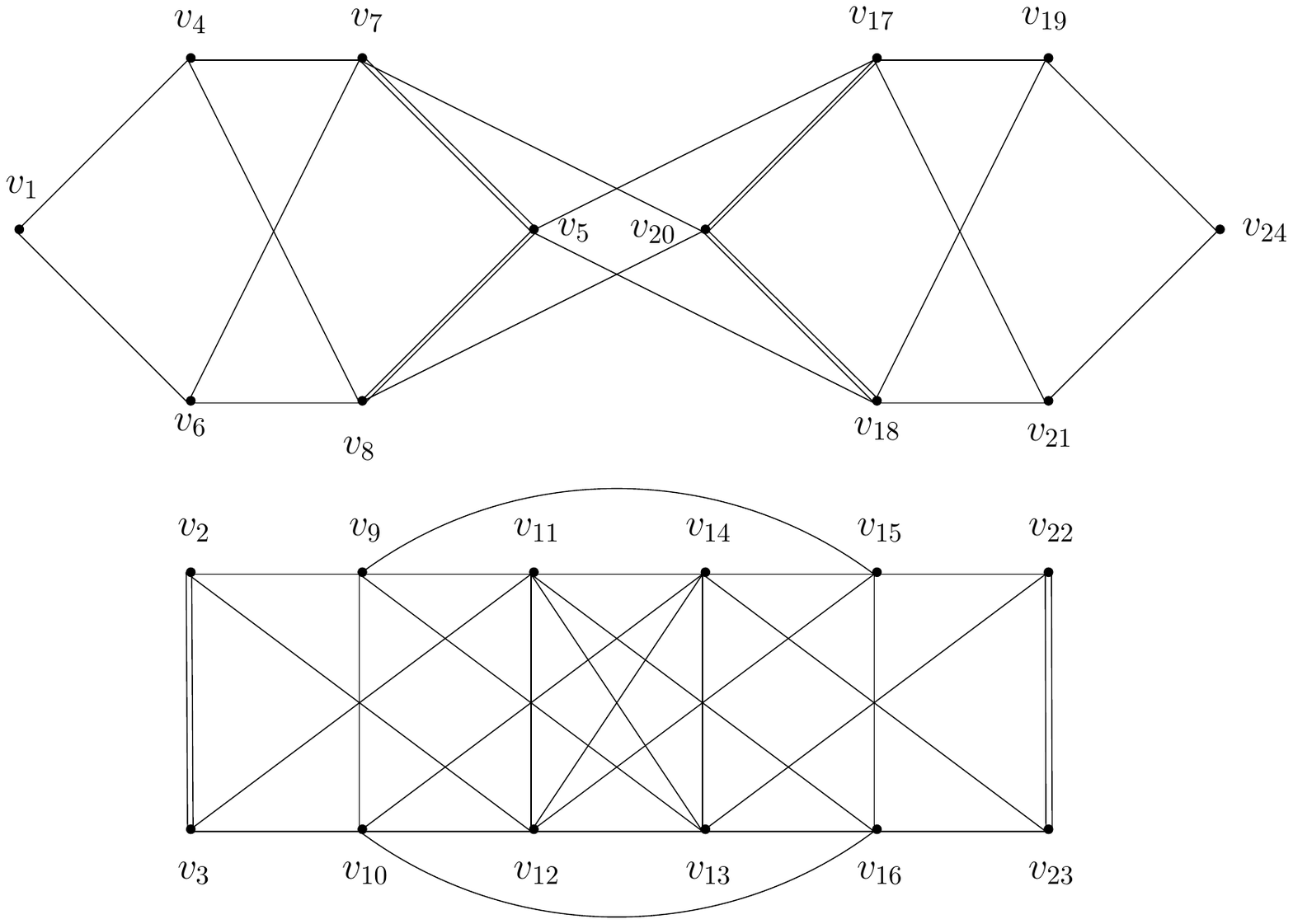}}\]
%We obtain this graph from the work of Ocneanu. Its correctness is not required for our calculations to be valid, as we only use it to obtain the additional linear (Tr($U_1$)) and quadratic (Tr($U_1U_2$)) equations of our KW cells.

To try find a cell system on $\Gamma^{4,8,\subset_{\mathbb{Z}_2}}_{\Lambda_1}$, we begin by assuming that the coefficients of the $U$ cells are all real. We initially tried to find a solution invariant under the rotational 180 degree symmetry. However it appears that a solution with such symmetry does not exist.

To obtain a solution for the $U$ cells, we solve the linear systems (R1), (R2), and (Tr($U_1$)), and then numerically approximate (Tr($U_1U_2$)). This yields two solutions to a subset of the $U$ cells. We pick one of these solutions, and guess exact values for the coefficients (as products of half-integer powers of quantum integers). We then solve (Hecke) to determine nearly all of the remaining coefficients up to sign. Taking a subset of the equations from (R3) allows us to pin down the remaining coefficients, and to choose signs for our coefficients. The solution we obtain is too large (568 coefficients) to include here, so we include it in the Mathematica notebook ``k=8/Orbifold/Solutions/w=1/Solution.nb'', attached to the arXiv submission of this article.

A computer verifies relations (R1), (R2), (Hecke), and (R3) in just under 5 hours for our solution. A record of this verification can be found in ``k=8/Orbifold/Solutions/w=1/Verification.nb''

Solving the linear systems (RI) and (BA) gives a unique solution for the $B$ cells up to scalar, and we solve (N) to pin down the norm of this scalar. Fixing a natural gauge gives us a solution, which can be found in ``k=8/Orbifold/Solutions/w=1/Solution.nb''.

A computer verifies (RI), (BA), and (N) for this solution in under 30 seconds. A record of this verification can be found in ``k=8/Orbifold/Solutions/w=1/Verification.nb''. 

\begin{thm}
There exists a rank 24 module category $\mathcal{M}$ for $\mathcal{C}(\mathfrak{sl}_4, 8)$ such that the fusion graph for action by $\Lambda_1$ is $\Gamma^{4,8,\subset_{\mathbb{Z}_2}}_{\Lambda_1}$.
\end{thm}

We also find KW cell system solutions on the graph $\Gamma^{4,8,\subset_{\mathbb{Z}_2}}_{\Lambda_1}$ when $\omega\in \{-1,\mathbf{i}, -\mathbf{i}\}$. The solutions and verification of these solutions can be found in the folder``k=8/Orbifold/Solutions".

\begin{thm}
For each $\omega \in \{-1, \mathbf{i}, -\mathbf{i}\}$ there exists a rank 24 module category $\mathcal{M}$ for $\overline{\operatorname{Rep}(U_{e^{2\pi i \frac{1}{24}}}(\mathfrak{sl}_4))^\omega}$ such that the fusion graph for action by $\Lambda_1$ is $\Gamma^{4,8,\subset_{\mathbb{Z}_2}}_{\Lambda_1}$.
\end{thm}

\subsection{A Second Exceptional Module for $SU(4)$ at Level $8$}\label{sub:k8excep2}
In this subsection we construct a cell system with parameters $N=4$ and $q = e^{2\pi i \frac{1}{24}}$ on the following graph:
\[\Gamma_{\Lambda_1}^{4,8,\subset} =\raisebox{-.5\height}{ \includegraphics[scale = .6]{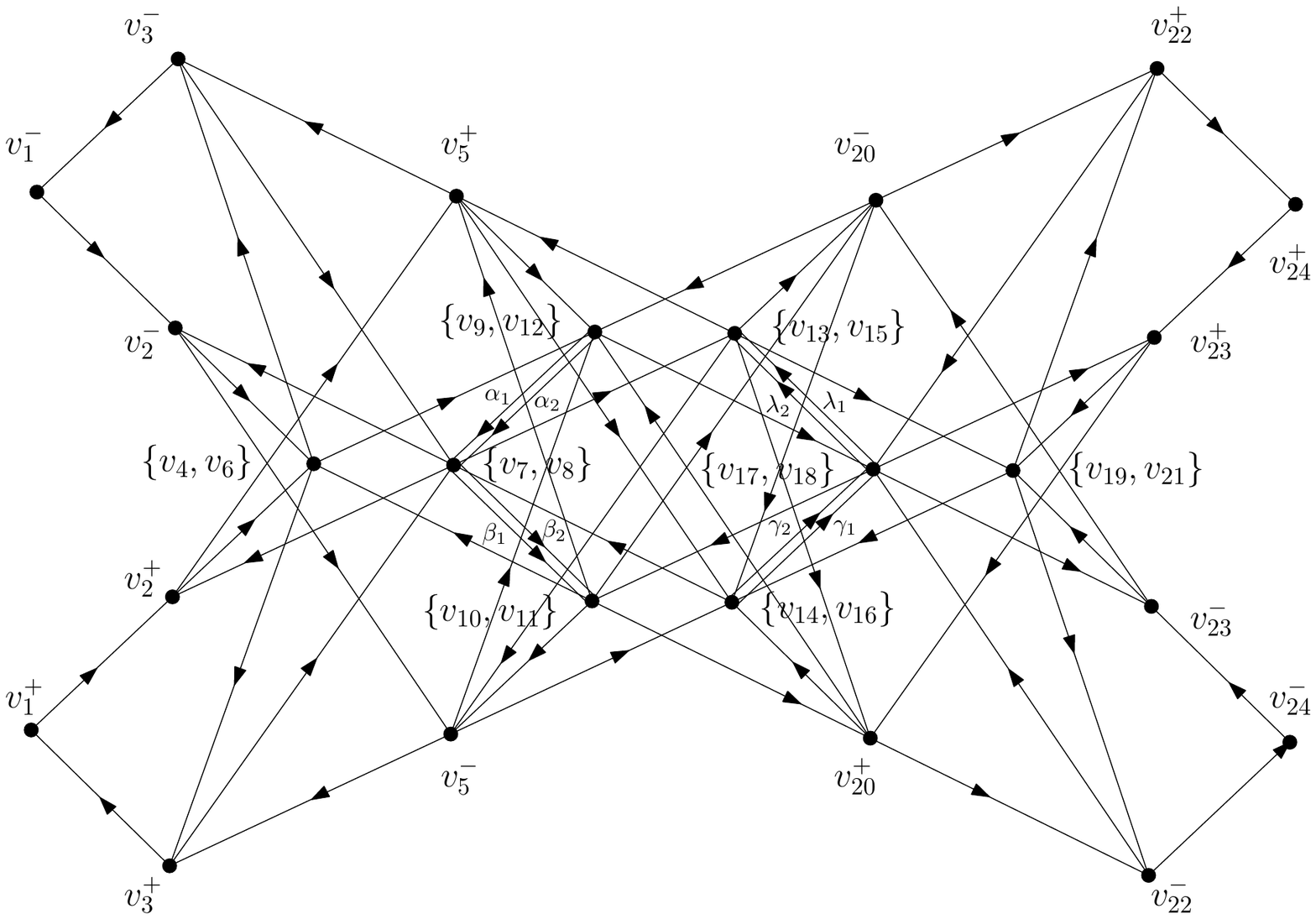}} \]
and hence an exceptional module category over $\overline{\operatorname{Rep}\left(U_{e^{2\pi i \frac{1}{24}}}(\mathfrak{sl}_4)\right)}$. This module will correspond the the conformal inclusion $(SU(4))_8 \subset (SO(20))_1$.

The positive eigenvector for $\Gamma_{\Lambda_1}^{4,8,\subset}$ is
\begin{align*}
&\lambda_{v_1^\pm} = \lambda_{v_{24}^\pm} &&=1 \qquad &&\lambda_{v_2^\pm} = \lambda_{v_3^\pm}=\lambda_{v_{22}^\pm} = \lambda_{v_{23}^\pm}&&=[4]_q \\
&\lambda_{v_5^\pm} = \lambda_{v_{20}^\pm} &&= \frac{\q{4}\q{5}}{\q{2}} \qquad &&\lambda_{\{v_{9}, v_{12}\}} = \lambda_{\{v_{10}, v_{11}\}}=\lambda_{\{v_{13}, v_{15}\}} = \lambda_{\{v_{14}, v_{16}\}}&&=\frac{\q{4}\q{5}\q{6}}{\q{2}\q{3}}\\
&\lambda_{\{v_{4}, v_{6}\}}=\lambda_{\{v_{19}, v_{21}\}} &&=\frac{\q{4}\q{6}}{\q{3}} \qquad &&\lambda_{\{v_{7}, v_{8}\}}=\lambda_{\{v_{17}, v_{18}\}} &&=\q{3}\q{5}
\end{align*} 

We assume the graph for action by $\Lambda_2$ is 
\[\Gamma_{\Lambda_2}^{4,8,\subset} =\raisebox{-.5\height}{ \includegraphics[scale = .6]{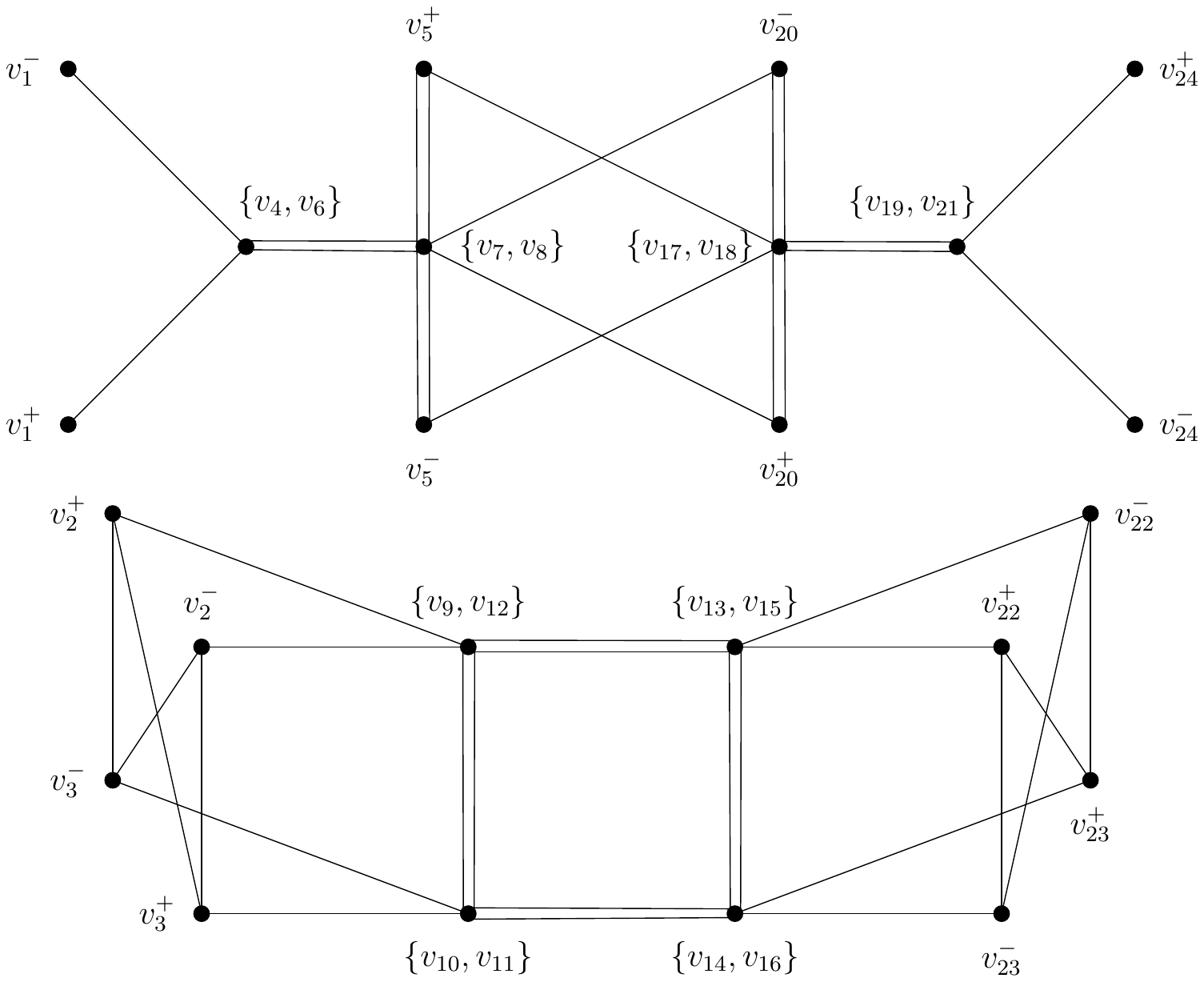}} \]

To assist with constructing a cell system on this graph, we observe that the KW-cell solution on $\Gamma_{\Lambda_1}^{4,8,\subset_{\mathbb{Z}_2}}$ of Subsection~\ref{subsec:8C} is gauge equivalent to a solution which is invariant under the following symmetry\footnote{We failed to find this symmetric solution when initially solving the system, as it requires making several coefficients non-real.}:
\[v_4 \leftrightarrow v_6 \quad v_7 \leftrightarrow v_8 \quad v_{9} \leftrightarrow v_{12} \quad v_{10} \leftrightarrow v_{11}\quad v_{13} \leftrightarrow v_{15} \quad v_{14} \leftrightarrow v_{16} \quad v_{17} \leftrightarrow v_{18}\quad v_{19} \leftrightarrow v_{21}.\]

We can then identify $\Gamma_{\Lambda_1}^{4,8,\subset}$ with the orbifold of $\Gamma_{\Lambda_1}^{4,8,\subset_{\mathbb{Z}_2}}$ under this symmetry (hence the suggestive vertex labels of $\Gamma_{\Lambda_1}^{4,8,\subset}$). Namely we have the natural identifications suggested by the labelling of the vertices. We have the following identifications for the edges with multiplicity:
\begin{align*}
    \alpha_1 &\leftrightarrow \{v_9 \to v_7  , v_{12}\to v_8\} \qquad &&\alpha_2 \leftrightarrow  \{v_9 \to v_8  , v_{12}\to v_7\}\\
    \beta_1 &\leftrightarrow \{v_7 \to v_{10}  , v_{8}\to v_{11}\} \qquad &&\beta_2 \leftrightarrow  \{v_7 \to v_{11}  , v_{8}\to v_{10}\}\\
    \gamma_1 &\leftrightarrow \{v_{14} \to v_{17}  , v_{16}\to v_{18}\} \qquad &&\gamma_2 \leftrightarrow \{v_{14} \to v_{18}  , v_{16}\to v_{17}\}\\
    \lambda_1 &\leftrightarrow \{v_{17} \to v_{13}  , v_{18}\to v_{15}\} \qquad &&\gamma_2 \leftrightarrow\{v_{17} \to v_{15}  , v_{18}\to v_{13}\}.
\end{align*} 
We can then use the (non-rigorous) orbifold procedure to deduce all the cells which only pass through the vertices
\[ \{v_4, v_6\},\quad \{v_7,v_8\},\quad\{v_9, v_{12}\},\quad\{v_{10}, v_{11}\},\quad \{v_{13}, v_{15}\},\quad \{v_{14}, v_{16}\},\quad \{v_{17}, v_{18}\},\quad \{v_{19}, v_{20}\}.  \]
In particular this determines the following two $5\times 5$ blocks:
\[U^{\{v_9,v_{12}\}}_{\qquad \quad\{v_{10},v_{11}\}}=\begin{blockarray}{ccccc}
{}^{\alpha_1}\{v_{7},v_{8}\}^{\beta_1} &{}^{\alpha_1}\{v_{7},v_{8}\}^{\beta_2} &{}^{\alpha_2}\{v_{7},v_{8}\}^{\beta_1}& {}^{\alpha_2}\{v_{7},v_{8}\}^{\beta_2}& \{v_{17},v_{18}\}    \\
\begin{block}{[ccccc]}
\frac{\q{3}\q{5}}{\q{2}\q{4}\q{6}} &0 & 0& \frac{\q{3}\q{5}}{\q{2}\q{4}\q{6}} & -\frac{\q{3}\q{5}^2}{\sqrt{\q{2}\q{4}^2\q{6}}}\\
0 & \frac{\q{5}}{\q{6}} - \frac{\sqrt{\q{4}\q{5}^2}}{\sqrt{\q{2}^2\q{3}^2\q{6}}}& -\frac{\q{5}}{\q{2}\q{3}} &0& 0\\
0&-\frac{\q{5}}{\q{2}\q{3}}& \sqrt{   \frac{\q{3}\q{5}}{\q{2}\q{6}} + \frac{\sqrt{\q{4}\q{5}^2}}{\q{2}\q{6}^2}} & 0& 0\\
\frac{\q{3}\q{5}}{\q{2}\q{4}\q{6}} &0 &0 & \frac{\q{3}\q{5}}{\q{2}\q{4}\q{6}} & -\frac{\q{3}\q{5}^2}{\sqrt{\q{2}\q{4}^2\q{6}}}\\
 -\frac{\q{3}\q{5}^2}{\sqrt{\q{2}\q{4}^2\q{6}}} & 0 & 0 & -\frac{\q{3}\q{5}^2}{\sqrt{\q{2}\q{4}^2\q{6}}} & \frac{\q{3}\q{5}}{\sqrt{\q{2}\q{4}^2\q{6}}}\\
\end{block}
\end{blockarray} \]
\[U^{\{v_{14},v_{16}\}}_{\qquad \quad\{v_{13},v_{15}\}}=\begin{blockarray}{ccccc}
{}^{\gamma_1}\{v_{17},v_{18}\}^{\lambda_1} &{}^{\gamma_1}\{v_{17},v_{18}\}^{\lambda_2} &{}^{\gamma_2}\{v_{17},v_{18}\}^{\lambda_1}& {}^{\gamma_2}\{v_{17},v_{18}\}^{\lambda_2}& \{v_{7},v_{8}\}    \\
\begin{block}{[ccccc]}
\frac{\q{3}\q{5}}{\q{2}\q{4}\q{6}} &0 & 0& \frac{\q{3}\q{5}}{\q{2}\q{4}\q{6}} & \frac{\q{3}\q{5}^2}{\sqrt{\q{2}\q{4}^2\q{6}}}\\
0 & \frac{\q{5}}{\q{6}} - \frac{\sqrt{\q{4}\q{5}^2}}{\sqrt{\q{2}^2\q{3}^2\q{6}}}& \frac{\q{5}}{\q{2}\q{3}} &0& 0\\
0&\frac{\q{5}}{\q{2}\q{3}}& \sqrt{   \frac{\q{3}\q{5}}{\q{2}\q{6}} + \frac{\sqrt{\q{4}\q{5}^2}}{\q{2}\q{6}^2}} & 0& 0\\
\frac{\q{3}\q{5}}{\q{2}\q{4}\q{6}} &0 &0 & \frac{\q{3}\q{5}}{\q{2}\q{4}\q{6}} & \frac{\q{3}\q{5}^2}{\sqrt{\q{2}\q{4}^2\q{6}}}\\
 \frac{\q{3}\q{5}^2}{\sqrt{\q{2}\q{4}^2\q{6}}} & 0 & 0 & \frac{\q{3}\q{5}^2}{\sqrt{\q{2}\q{4}^2\q{6}}} & \frac{\q{3}\q{5}}{\sqrt{\q{2}\q{4}^2\q{6}}}\\
\end{block}
\end{blockarray} \]

With this seed information, we can solve to find the remaining U-cells. Solving the linear relations (R1), (R2), and (Tr($U_1$)) determines all but $\approx 100$ coefficients. We then can fully solve (Tr($U_1U_2$)) which gives the diagonal entries of every $U$ matrix. From (Hecke), we can determine all of the $2\times 2$ and $3\times 3$ blocks up to some phase choices. Finally at this point (R3) contains enough linear equations to pin down the remaining coefficients, and pin down the earlier phases.

In the interest of space, we only present the remaining two $5\times 5$ blocks.
\[U^{\{v_{10},v_{11}\}}_{\qquad \quad\{v_{9},v_{12}\}}=\begin{blockarray}{ccccc}
\{v_4,v_6\} &v_5^+&v_5^-&v_{20}^+&v_{20}^-\\
\begin{block}{[ccccc]}
\frac{\q{6}}{\q{5}} &-\frac{1}{\sqrt{\q{3}\q{5}}} & -\frac{1}{\sqrt{\q{3}\q{5}}}& \frac{-\q{3} + \mathbf{i}\sqrt{\q{4}\q{5}\q{6}}}{\sqrt{\q{2}^2\q{3}\q{6}^2}} &  \frac{-\q{3} + \mathbf{i}\sqrt{\q{4}\q{5}\q{6}}}{\sqrt{\q{2}^2\q{3}\q{6}^2}}\\
-\frac{1}{\sqrt{\q{3}\q{5}}} & \frac{\q{4}}{\q{2}\q{6}}& -\frac{\q{5}}{\q{2}\q{6}} &-\frac{1+\mathbf{i}\frac{\sqrt{\q{4}}}{\sqrt{\q{6}}}    }{\q{3}}& \frac{\q{5}}{\q{2}\q{6}}\\
-\frac{1}{\sqrt{\q{3}\q{5}}}&-\frac{\q{5}}{\q{2}\q{6}}&\frac{\q{4}}{\q{2}\q{6}} & \frac{\q{5}}{\q{2}\q{6}}&-\frac{1+\mathbf{i}\frac{\sqrt{\q{4}}}{\sqrt{\q{6}}}    }{\q{3}}\\
 \frac{-\q{3} - \mathbf{i}\sqrt{\q{4}\q{5}\q{6}}}{\sqrt{\q{2}^2\q{3}\q{6}^2}} &-\frac{1-\mathbf{i}\frac{\sqrt{\q{4}}}{\sqrt{\q{6}}}    }{\q{3}} &\frac{\q{5}}{\q{2}\q{6}} &\frac{\q{5}}{\q{6}} & -\frac{1}{\q{6}}\\
  \frac{-\q{3} - \mathbf{i}\sqrt{\q{4}\q{5}\q{6}}}{\sqrt{\q{2}^2\q{3}\q{6}^2}} & \frac{\q{5}}{\q{2}\q{6}} & -\frac{1-\mathbf{i}\frac{\sqrt{\q{4}}}{\sqrt{\q{6}}}    }{\q{3}} & -\frac{1}{\q{6}} & \frac{\q{5}}{\q{6}}\\
\end{block}
\end{blockarray} \]
\[U^{\{v_{13},v_{15}\}}_{\qquad \quad\{v_{14},v_{16}\}}=\begin{blockarray}{ccccc}
v_5^+&v_5^-&\{v_{19},v_{21}\}&v_{20}^+&v_{20}^-\\
\begin{block}{[ccccc]}
\frac{\q{5}}{\q{6}} &-\frac{1}{\q{6}}& \frac{-\q{3} - \mathbf{i}\sqrt{\q{4}\q{5}\q{6}}}{\sqrt{\q{2}^2\q{3}\q{6}^2}}& \frac{\q{5}}{\q{2}\q{6}} &  -\frac{1-\mathbf{i}\frac{\sqrt{\q{4}}}{\sqrt{\q{6}}}    }{\q{3}}\\
-\frac{1}{\q{6}} &\frac{\q{5}}{\q{6}}& \frac{-\q{3} - \mathbf{i}\sqrt{\q{4}\q{5}\q{6}}}{\sqrt{\q{2}^2\q{3}\q{6}^2}}& -\frac{1-\mathbf{i}\frac{\sqrt{\q{4}}}{\sqrt{\q{6}}}    }{\q{3}} &  \frac{\q{5}}{\q{2}\q{6}}\\
 \frac{-\q{3} + \mathbf{i}\sqrt{\q{4}\q{5}\q{6}}}{\sqrt{\q{2}^2\q{3}\q{6}^2}} & \frac{-\q{3} +\mathbf{i}\sqrt{\q{4}\q{5}\q{6}}}{\sqrt{\q{2}^2\q{3}\q{6}^2}}&\frac{\q{6}}{\q{5}}&-\frac{1}{\sqrt{\q{3}\q{5}}} & -\frac{1}{\sqrt{\q{3}\q{5}}}\\
 \frac{\q{5}}{\q{2}\q{6}} & -\frac{1+\mathbf{i}\frac{\sqrt{\q{4}}}{\sqrt{\q{6}}}    }{\q{3}}&-\frac{1}{\sqrt{\q{3}\q{5}}}&\frac{\q{4}}{\q{2}\q{6}} &  -\frac{\q{5}}{\q{2}\q{6}}\\
 -\frac{1+\mathbf{i}\frac{\sqrt{\q{4}}}{\sqrt{\q{6}}}    }{\q{3}} & \frac{\q{5}}{\q{2}\q{6}}&-\frac{1}{\sqrt{\q{3}\q{5}}}& -\frac{\q{5}}{\q{2}\q{6}} & \frac{\q{4}}{\q{2}\q{6}}\\
\end{block}
\end{blockarray} \]

The full solution for the U-cells can be found in the Mathematica file ``k=8/Conformal Inclusion/Solution.nb". In this file we use the ordering:
\[\resizebox{\hsize}{!}{$\{v_{1}^+, v_{1}^-,v_{2}^+,v_{2}^-,v_{3}^+,v_{3}^-, \{v_{4},v_{6}\},v_{5}^+,v_{5}^-, \{v_{7},v_{8}\}, \{v_{9},v_{12}\}, \{v_{10},v_{11}\}, \{v_{13},v_{15}\}, \{v_{14},v_{14}\}, \{v_{17},v_{18}\}, \{v_{19},v_{21}\},v_{20}^+,v_{20}^-,v_{22}^+,v_{22}^-,v_{23}^+,v_{23}^-,v_{24}^+,v_{24}^-\}$.}\]

A computer verifies relations (R1), (R2), (R3), and (Hecke) in 2 hours. A record of this verification can be found in ``k=8/Conformal Inclusion/Verification.nb". Note that this verification removes any dependency we had on using the non-rigorous orbifold procedure to obtain certain values for our U-cells.

We now solve (RI) and (BA) to obtain a 1-dimensional solution space for our B-cells. The equations (N) determine this solution up to a choice of phase, which we pick a natural choice for. The solution for the B-cells can be found in ``k=8/Conformal Inclusion/Solution.nb". A computer verifies relations (RI), (BA), and (N) in just over 3 minutes for this solution. A record of this verification can be found in ``k=8/Conformal Inclusion/Verification.nb".

\begin{thm}
There exists a rank 24 module category $\mathcal{M}$ for $\mathcal{C}(\mathfrak{sl}_4, 8)$ such that the fusion graph for action by $\Lambda_1$ is $\Gamma_{\Lambda_1}^{4,8,\subset}$.
\end{thm}
We also find KW cell system solutions on the graph $\Gamma_{\Lambda_1}^{4,8,\subset}$ when $\omega\in \{-1,\mathbf{i}, -\mathbf{i}\}$. The solutions and verification of these solutions can be found in the folder ``k=8/Conformal Inclusions/Solutions".

\begin{thm}
For each $\omega \in \{-1, \mathbf{i}, -\mathbf{i}\}$ there exists a rank 24 module category $\mathcal{M}$ for $\overline{\operatorname{Rep}(U_{e^{2\pi i \frac{1}{24}}}(\mathfrak{sl}_4))^\omega}$ such that the fusion graph for action by $\Lambda_1$ is $\Gamma_{\Lambda_1}^{4,8,\subset}$.
\end{thm}

\subsection{A Third Exceptional Module for $SU(4)$ at Level $8$}
In this section we construct a cell system with parameters $N=4$ and $q = e^{2\pi i \frac{1}{24}}$ on the following graph
\[\Gamma_{\Lambda_1}^{4,8,\textrm{Twist}} =\raisebox{-.5\height}{ \includegraphics[scale = .4]{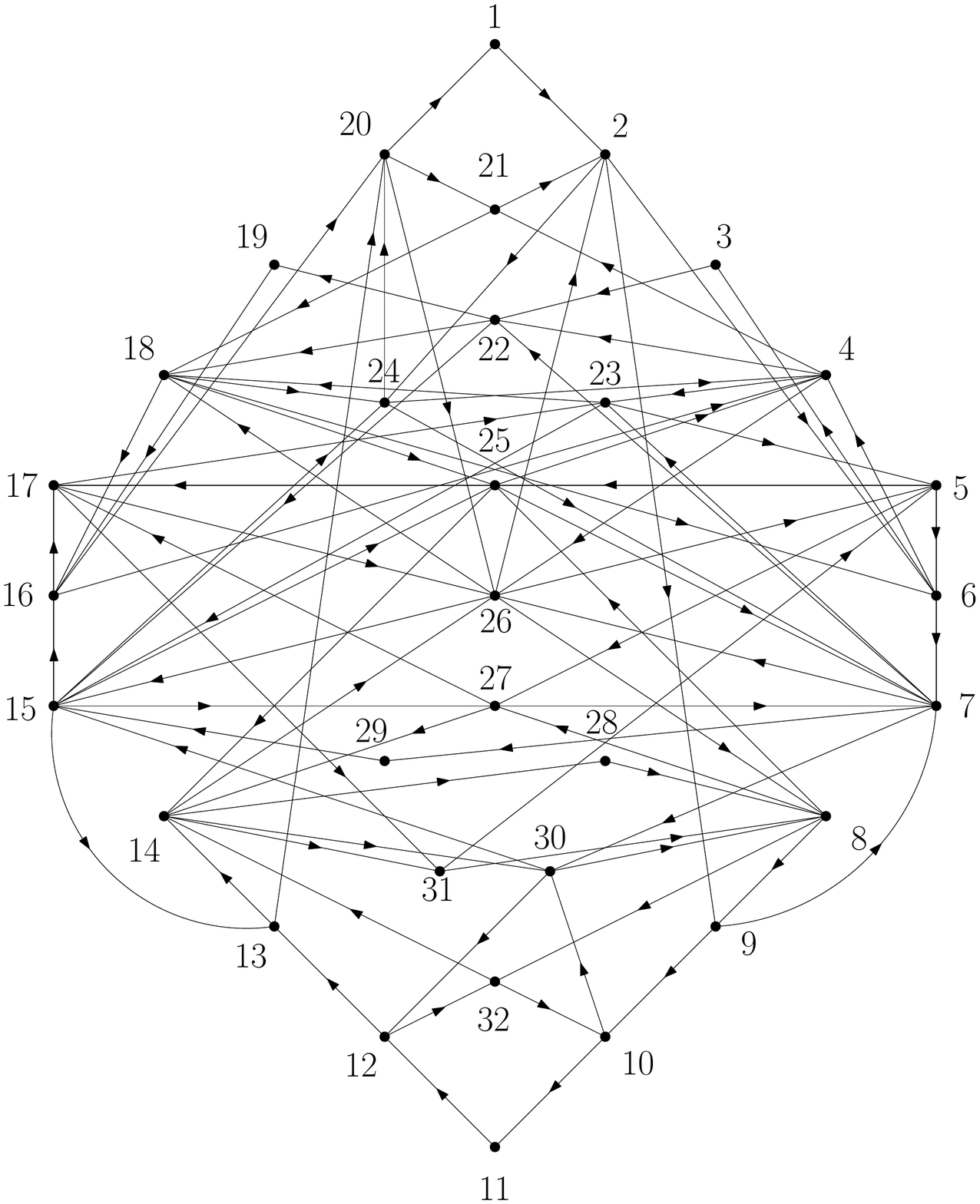}} \]
and hence an exceptional module category over $\overline{\operatorname{Rep}\left(U_{e^{2\pi i \frac{1}{24}}}(\mathfrak{sl}_4)\right)}$. This module corresponds to the exceptional triple found in \cite{ModulesPt1} constructed from the exceptional braided auto-equivalence of $\mathcal{C}(\mathfrak{sl}_4, 8)^0_{\operatorname{Rep}(\mathbb{Z}_4)}$

The Frobenius-Perron eigenvector of this graph is
%\begin{align*}\lambda = \{&1,\sqrt{3 \left(\sqrt{3}+2\right)},\sqrt{\frac{3}{2}},3 \sqrt{\sqrt{3}+2},\sqrt{6}+\frac{3}{\sqrt{2}},\frac{3}{2} \left(\sqrt{3}+1\right),2 \sqrt{3 \left(\sqrt{3}+2\right)},\sqrt{6}+\frac{3}{\sqrt{2}},\frac{1}{2} \left(\sqrt{3}+3\right),\sqrt{\frac{3}{2}},\\
%&\frac{1}{2} \left(\sqrt{3}-1\right),\sqrt{\frac{3}{2}},\frac{1}{2} \left(\sqrt{3}+3\right),\sqrt{6}+\frac{3}{\sqrt{2}},2 \sqrt{3 \left(\sqrt{3}+2\right)},\frac{3}{2} \left(\sqrt{3}+1\right),\sqrt{6}+\frac{3}{\sqrt{2}},3 \sqrt{\sqrt{3}+2},\sqrt{\frac{3}{2}},\\
%&\sqrt{3 \left(\sqrt{3}+2\right)},\sqrt{3}+1,\frac{3}{2} \left(\sqrt{3}+1\right),\frac{1}{2} \left(3 \sqrt{3}+5\right),\sqrt{3}+3,2 \sqrt{3}+3,2 \left(\sqrt{3}+2\right),\sqrt{3}+3,\frac{1}{2} \left(\sqrt{3}+1\right),\\
%&2,\sqrt{3}+2,\sqrt{3}+1,\sqrt{3}\}.\end{align*}
\begin{align*}
\lambda = \{&1,\q{4},\frac{\nothing{[4]_q} }{\nothing{[3]_q}},\frac{\nothing{[4]_q}^2  \nothing{[6]_q}}{ \nothing{[3]_q}^2},\frac{\nothing{[3]_q} \nothing{[4]_q} \nothing{[5]_q}}{\nothing{[2]_q} \nothing{[6]_q}},\nothing{[4]_q}^2  \nothing{[3]_q},\frac{\nothing{[4]_q} \nothing{[6]_q}}{\nothing{[2]_q}},\frac{\nothing{[3]_q} \nothing{[4]_q} \nothing{[5]_q}}{\nothing{[2]_q} \nothing{[6]_q}},\frac{\nothing{[4]_q} \nothing{[5]_q}}{\nothing{[2]_q} \nothing{[3]_q}},\frac{\nothing{[4]_q}}{ \nothing{[3]_q}},\frac{1}{ \nothing{[3]_q}},\\
&\frac{\nothing{[4]_q} }{\nothing{[3]_q}},\frac{\nothing{[4]_q} \nothing{[5]_q}}{\nothing{[2]_q} \nothing{[3]_q}},\frac{\nothing{[3]_q} \nothing{[4]_q} \nothing{[5]_q}}{\nothing{[2]_q} \nothing{[6]_q}},\frac{\nothing{[4]_q} \nothing{[6]_q}}{\nothing{[2]_q}},\frac{\nothing{[4]_q}^2 }{ \nothing{[3]_q}},\frac{\nothing{[3]_q} \nothing{[4]_q} \nothing{[5]_q}}{\nothing{[2]_q} \nothing{[6]_q}},\frac{[4]_q^2 [6]_q}{[3]_q^2},\frac{\nothing{[4]_q} }{\nothing{[3]_q}},\nothing{[4]_q} ,\frac{\nothing{[5]_q} \nothing{[6]_q}}{\nothing{[2]_q} \nothing{[3]_q}},\frac{\nothing{[4]_q}^2 }{ \nothing{[3]_q}},\\
&\frac{\nothing{[2]_q} \nothing{[3]_q} \nothing{[5]_q}}{\nothing{[6]_q}},\frac{\nothing{[4]_q} \nothing{[6]_q}}{\nothing{[3]_q}},\frac{\nothing{[4]_q} \nothing{[5]_q}}{\nothing{[2]_q}},\frac{\nothing{[5]_q} \nothing{[6]_q}}{\nothing{[2]_q}},\frac{\nothing{[4]_q} \nothing{[6]_q}}{ \nothing{[3]_q}},\frac{\nothing{[3]_q} \nothing{[5]_q}}{\nothing{[2]_q} \nothing{[6]_q}},\frac{\nothing{[6]_q}}{\nothing{[2]_q}},\frac{\q{3}\q{5}}{\q{2}\q{6}},\frac{\nothing{[5]_q} \nothing{[6]_q}}{\nothing{[2]_q} \nothing{[3]_q}},\frac{\nothing{[4]_q}}{\nothing{[2]_q}}\}
\end{align*}
We assume the graph for action by $\Lambda_2$ is 
\[\Gamma_{\Lambda_2}^{4,8,\textrm{Twist}} =\raisebox{-.5\height}{ \includegraphics[scale = .4]{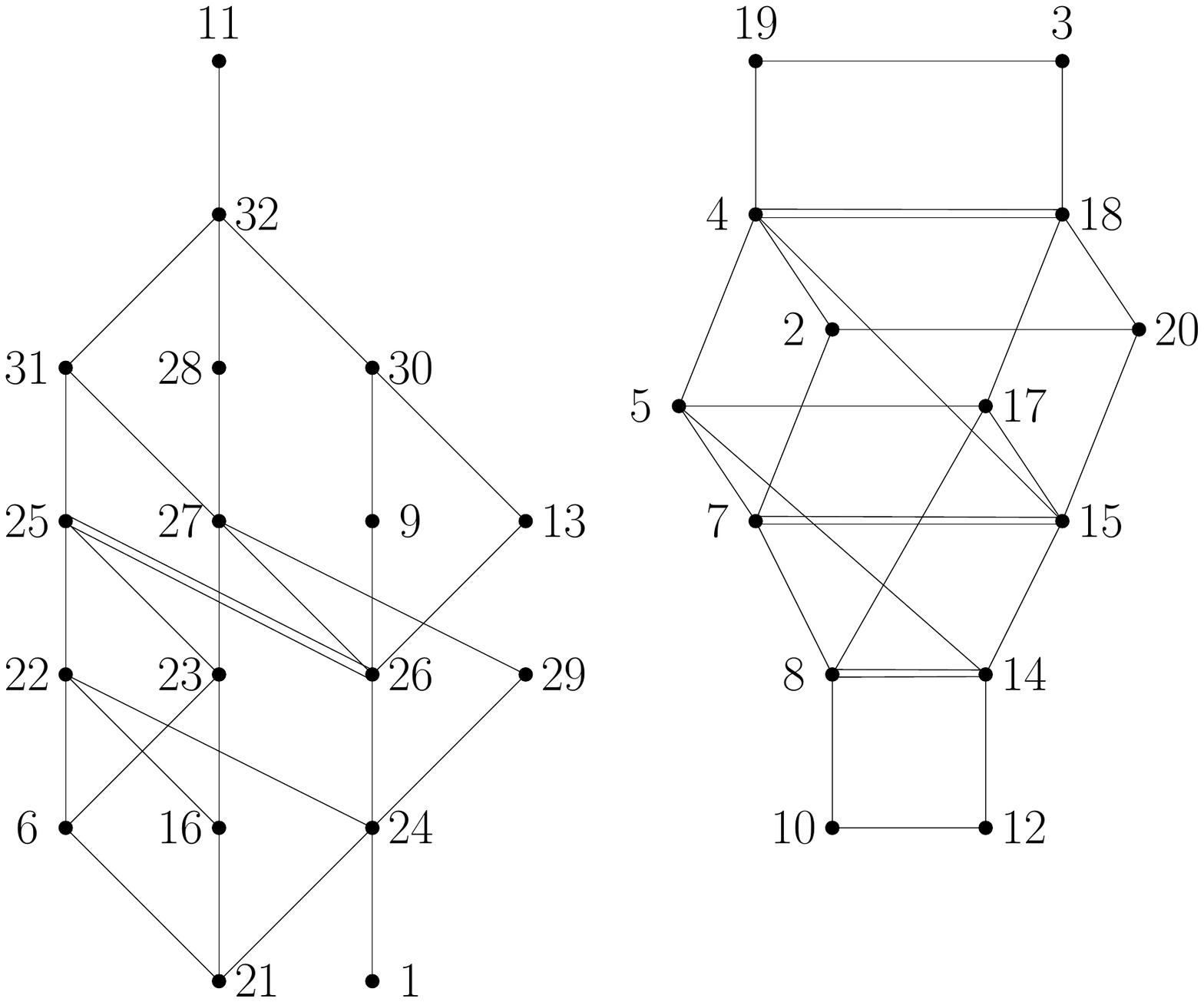}} \]
The data of these graphs can be found in ``k=8/AmbichiralTwist/Data.nb''.

To find a solution to the U-cells on $\Gamma_{\Lambda_1}^{4,8,\textrm{Twist}}$, we make the naive assumption that all our coefficients are real, and hence our $U$ matrices are symmetric by (R2). Solving the linear equations (R1) and (Tr($U_1$)) determines 417 of the 616 coefficients. 

The quadratic equation (Tr($U_1U_2$)) is especially useful for this example. Several of the equations coming from (Tr($U_1U_2$)) are in fact linear. Solving these linear equations, and plugging back in the values to (Tr($U_1U_2$)) then turns other equations into linear equations. This cascades, and allows us to completely solve (Tr($U_1U_2$)). This determines the diagonals of all the blocks in our solution, which gives 35 of the remaining coefficients.

We then use (Hecke) to determine the remaining 2x2 and 3x3 blocks, up to some sign ambiguities which we will resolve later. This leaves one 5x5 block, and five 4x4 blocks to determine. i.e. 40 coefficients. At this point, many of the equations from (R3) are now linear. Solving these determines the 4x4 and 5x5 blocks. We also use (R3) to resolve the sign ambiguities from earlier. Finally we choose an arbitrary real gauge to get a concrete solution.

The solution for the U-cells consists of 616 complex numbers. We just present the 5x5 block here in the interest of space, as this is the largest (and hence most difficult to determine) block in our solution. The rest of the solution can be found in the Mathematica notebook ``k=8/AmbichiralTwist/Solutions/w=1/Solution.nb".

\[U^{7}_{\quad 15}=\begin{blockarray}{ccccc}
22  & 23 &26& 29& 30    \\
\begin{block}{[ccccc]}
\frac{[4]_q^2[5]_q}{[2]_q^2[3]_q[6]_q} &-\frac{\q{5}}{\q{2}\q{3}\q{6}} & -\frac{\sqrt{\q{3}}}{\sqrt{\q{2}\q{6}}}& \frac{\sqrt{\q{3}\q{5}}}{\q{2}\q{6}} & \frac{\q{5}}{\sqrt{\q{3}\q{6}^2}}\\
-\frac{\q{5}}{\q{2}\q{3}\q{6}} & \frac{\q{5}}{\q{2}\q{3}} - \frac{\q{2}\q{5}}{\q{3}\q{4}\q{6}}& -\frac{\sqrt{\q{5}}}{\sqrt{\q{2}^2\q{3}\q{4}^2}} & \frac{\sqrt{\q{3}\q{5}^2}}{\q{4}\q{6}}& -\frac{\sqrt{\q{2}\q{3}\q{5}^3}}{\sqrt{\q{4}^2\q{6}^3}}\\
 -\frac{\sqrt{\q{3}}}{\sqrt{\q{2}\q{6}}}& -\frac{\sqrt{\q{5}}}{\sqrt{\q{2}^2\q{3}\q{4}^2}}& \frac{1}{\q{2}} & -\frac{\q{5}}{\q{2}\q{4}}& 0\\
 \frac{\sqrt{\q{3}\q{5}}}{\q{2}\q{6}} & -\frac{\sqrt{\q{5}}}{\sqrt{\q{2}^2\q{3}\q{4}^2}} &  -\frac{\q{5}}{\q{2}\q{4}} & \frac{\q{5}\q{6}}{\q{2}\q{3}\q{4}} & -\frac{\q{3}\q{5}}{\q{2}\q{4}\q{6}}\\
  \frac{\q{5}}{\sqrt{\q{3}\q{6}^2}} & -\frac{\sqrt{\q{2}\q{3}\q{5}^3}}{\sqrt{\q{4}^2\q{6}^3}} & 0 & -\frac{\q{3}\q{5}}{\q{2}\q{4}\q{6}} & \frac{\q{5}}{\q{6}}\\
\end{block}
\end{blockarray} \]

A computer verifies (R1), (R2), (R3), and (Hecke) for our solution in under two minutes. A record of this solution can be found in ``k=8/AmbichiralTwist/Solutions/w=1/Verification.nb"

Solving the linear equations (RI) and (BA) solves the B-cells up to a single scalar, and (N) determines this scalar up to a phase. We pick a natural choice for this phase to obtain a concrete solution for our B-cells. This solution consists of 536 complex scalars. This solution can also be found in ``k=8/AmbichiralTwist/Solutions/w=1/Solution.nb". A computer verifies relations (BA), (RI), and (N) in under 30 seconds for our solution. A record of this solution can also be found in ``k=8/AmbichiralTwist/Solutions/w=1/Verification.nb".

\begin{thm}
There exists a rank 32 module category $\mathcal{M}$ for $\mathcal{C}(\mathfrak{sl}_4, 8)$ such that the fusion graph for action by $\Lambda_1$ is $\Gamma_{\Lambda_1}^{4,8,\textrm{Twist}}$.
\end{thm}

Note that this theorem gives a construction of the exceptional braided auto-equivalence of $\mathcal{C}(\mathfrak{sl}_4, 8)^0_{\operatorname{Rep}(\mathbb{Z}_4)}$, independent from the construction in \cite{ModulesPt1}.

We also find KW cell system solutions on the graph $\Gamma_{\Lambda_1}^{4,8,\textrm{Twist}}$ when $\omega\in \{-1,\mathbf{i}, -\mathbf{i}\}$. The solutions and verification of these solutions can be found in the folder``k=8/AmbichiralTwist/Solutions".

\begin{thm}
For each $\omega \in \{-1, \mathbf{i}, -\mathbf{i}\}$ there exists a rank 32 module category $\mathcal{M}$ for $\overline{\operatorname{Rep}(U_{e^{2\pi i \frac{1}{24}}}(\mathfrak{sl}_4))^\omega}$ such that the fusion graph for action by $\Lambda_1$ is $\Gamma_{\Lambda_1}^{4,8,\textrm{Twist}}$.
\end{thm}
\subsection{Charge-Conjugation Modules for $SU(4)$ at Level $k$}\label{sec:charge}
In this subsection, we will construct a family of KW cell systems on the graphs $\Gamma^{4,k,*}_{\Lambda_1}$, where $N=4$ and $q = e^{2\pi i \frac{1}{2(4+k)}}$ for all $k\in \mathbb{N}$. These will construct the \textit{charge conjugation} module categories over $\mathcal{C}(\mathfrak{sl}_4, k)$ for all $k$.

Following \cite{MR2021644} we define the graph $\Gamma^{4,k,*}_{\Lambda_1}$ as follows. The vertices of $\Gamma^{4,k,*}_{\Lambda_1}$ are $a = (i,j)$ such that $i + 2j \leq k+3$. There is an edge in $\Gamma^{4,k,*}_{\Lambda_1}$ from $a\to b$ if 
\[a-b \in \{\varepsilon_{-2}:= (1,-1),\quad \varepsilon_{-1}:= (-1,0), \quad \varepsilon_{1}:= (1,0),\quad  \varepsilon_{2}:= (-1,1)   \}.\]

When $k$ is odd, the graph $\Gamma^{4,k,*}_{\Lambda_1}$ is of the form:
\[\raisebox{-.5\height}{ \includegraphics[scale = .6]{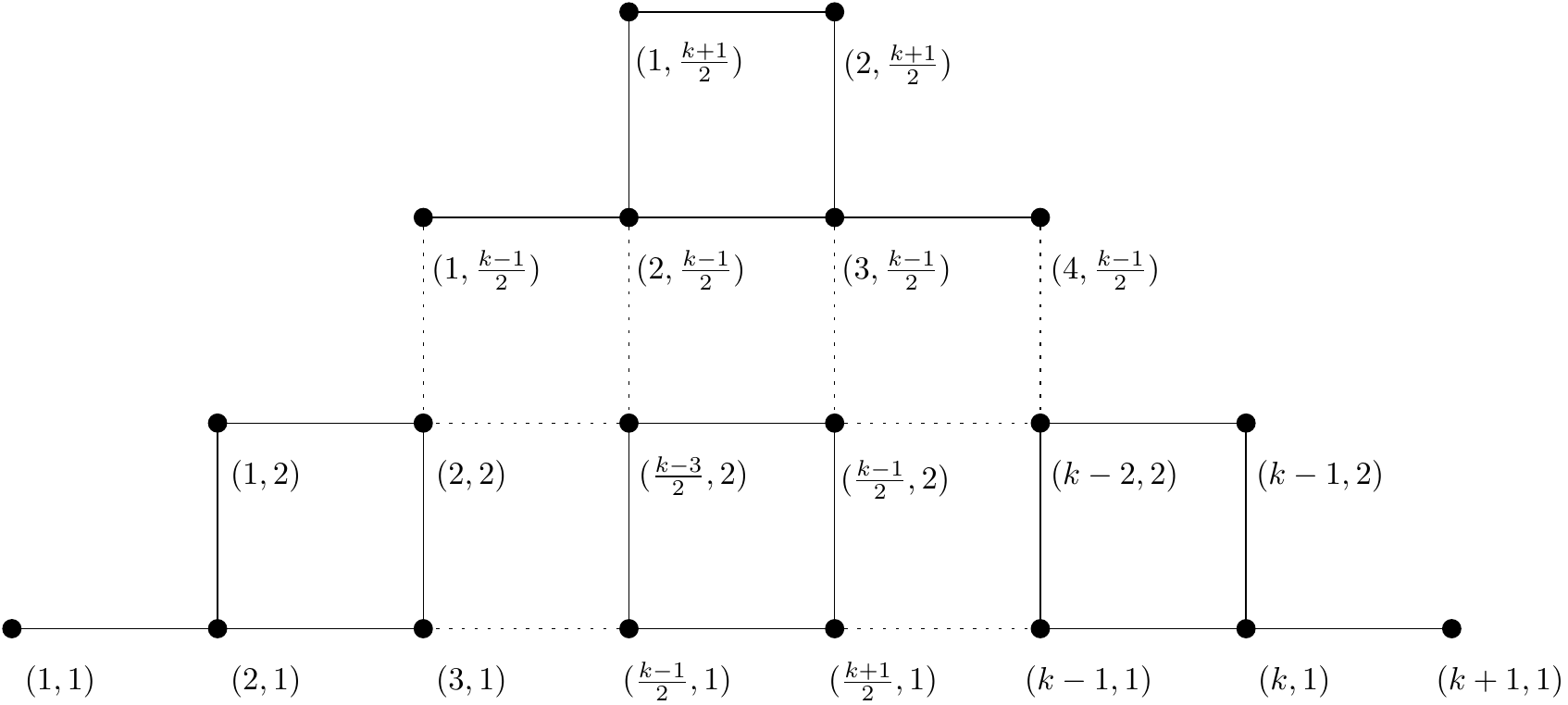}} \]
When $k$ is even, the graph $\Gamma^{4,k,*}_{\Lambda_1}$ is of the form:
\[\raisebox{-.5\height}{ \includegraphics[scale = .6]{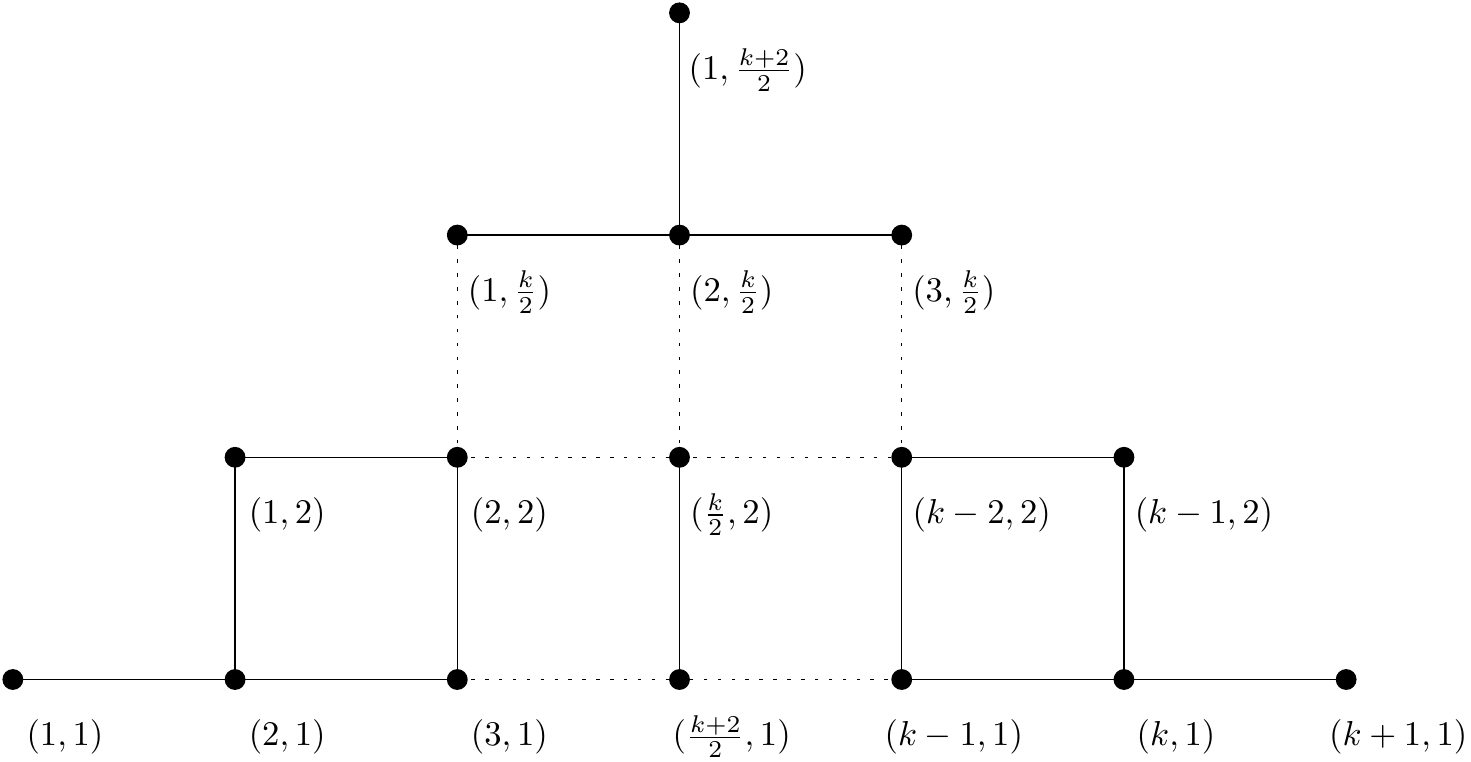}} \]
For $a\in \Gamma^{4,k,*}_{\Lambda_1}$ and $j\in \{1,2\}$, define $\overline{a}_j := \sum_{i=j}^2 a_i$, and $\overline{a}_{-j} = -\overline{a}_j$. The positive eigenvector $\lambda$ is given by the following formula:
\[  \lambda_a = [\overline{a}_1][\overline{a}_2] [\overline{a}_1-\overline{a}_2] [\overline{a}_1+\overline{a}_2]. \]
The following values will be useful
\begin{align*}
\lambda_{a \pm \epsilon_1} & = [\overline{a}_1\pm 1][\overline{a}_2] [\overline{a}_1-\overline{a}_2\pm 1] [\overline{a}_1+\overline{a}_2 \pm 1]\\
\lambda_{a \pm \epsilon_2} & = [\overline{a}_1][\overline{a}_2 \pm 1] [\overline{a}_1-\overline{a}_2\mp 1] [\overline{a}_1+\overline{a}_2 \pm 1]\\
\lambda_{a \pm (\epsilon_1+ \epsilon_2)} & = [\overline{a}_1\pm 1][\overline{a}_2 \pm 1] [\overline{a}_1-\overline{a}_2] [\overline{a}_1+\overline{a}_2 \pm 2]\\
\lambda_{a \pm (\epsilon_1- \epsilon_2)} & = [\overline{a}_1\pm 1][\overline{a}_2 \mp 1] [\overline{a}_1-\overline{a}_2\pm 2] [\overline{a}_1+\overline{a}_2 ].
\end{align*}

%Let $\hat{i}$ be the involution exchanging $1\leftrightarrow 2$ and $-1\leftrightarrow -2$, and $\Delta(a) = (-1)^{a_1}$ ((i.e. alternating signs on the veritices of $\Gamma$).

A solution for the U-cells on $\Gamma^{4,k,*}_{\Lambda_1}$ is computed in \cite{MR2021644}. They find
\begin{align*}
   U^{a}_{\quad a+2\varepsilon_i} &=  \begin{blockarray}{c}
 a+ \varepsilon_i  \\
\begin{block}{[c]}
  0\\
\end{block}
\end{blockarray}\\
    U^{a}_{\quad a+\varepsilon_i+\varepsilon_j} &=\frac{1}{ [\overline{a}_i -\overline{a}_j]} \begin{blockarray}{cc}
a+  \varepsilon_i & a+\varepsilon_j  \\
\begin{block}{[cc]}
  [\overline{a}_i -\overline{a}_j+1]   & \sqrt{[\overline{a}_i -\overline{a}_j+1][\overline{a}_i -\overline{a}_j-1]}\\
  \sqrt{[\overline{a}_i -\overline{a}_j+1][\overline{a}_i -\overline{a}_j-1]} &[\overline{a}_i -\overline{a}_j-1]\\
\end{block}
\end{blockarray}\quad \text{if $i\neq \pm j$}\\
U^{a}_{\quad a} &=\frac{1}{\lambda_a}\scalemath{.8}{\begin{blockarray}{cccc}
a+  \varepsilon_{-2} & a+  \varepsilon_{-1} &a+  \varepsilon_{1}&a+  \varepsilon_{2} \\
\begin{block}{[cccc]}
w_{a,-2}&\sqrt{\lambda_{a+\varepsilon_{-1}}\lambda_{a+\varepsilon_{-2}}}\frac{1}{[\overline{a}_{-1} +\overline{a}_{-2}  +1]}&-\sqrt{\lambda_{a+\varepsilon_{1}}\lambda_{a+\varepsilon_{-2}}}\frac{1}{[\overline{a}_{1} +\overline{a}_{-2}  +1]}&-\sqrt{\lambda_{a+\varepsilon_{2}}\lambda_{a+\varepsilon_{-2}}}\frac{1}{[\overline{a}_{2} +\overline{a}_{-2}  +1]}\\
 \sqrt{\lambda_{a+\varepsilon_{-1}}\lambda_{a+\varepsilon_{-2}}}\frac{1}{[\overline{a}_{-1} +\overline{a}_{-2}  +1]}&w_{a,-1}&-\sqrt{\lambda_{a+\varepsilon_{-1}}\lambda_{a+\varepsilon_{1}}}\frac{1}{[\overline{a}_{-1} +\overline{a}_{1}  +1]}&-\sqrt{\lambda_{a+\varepsilon_{-1}}\lambda_{a+\varepsilon_{2}}}\frac{1}{[\overline{a}_{-1} +\overline{a}_{2}  +1]}\\
 -\sqrt{\lambda_{a+\varepsilon_{1}}\lambda_{a+\varepsilon_{-2}}}\frac{1}{[\overline{a}_{1} +\overline{a}_{-2}  +1]}&-\sqrt{\lambda_{a+\varepsilon_{1}}\lambda_{a+\varepsilon_{-1}}}\frac{1}{[\overline{a}_{1} +\overline{a}_{-1}  +1]}&w_{a,1}&\sqrt{\lambda_{a+\varepsilon_{1}}\lambda_{a+\varepsilon_{2}}}\frac{1}{[\overline{a}_{1} +\overline{a}_{2}  +1]}\\
 -\sqrt{\lambda_{a+\varepsilon_{2}}\lambda_{a+\varepsilon_{-2}}}\frac{1}{[\overline{a}_{2} +\overline{a}_{-2}  +1]}&-\sqrt{\lambda_{a+\varepsilon_{2}}\lambda_{a+\varepsilon_{-1}}}\frac{1}{[\overline{a}_{2} +\overline{a}_{-1}  +1]}&\sqrt{\lambda_{a+\varepsilon_{2}}\lambda_{a+\varepsilon_{1}}}\frac{1}{[\overline{a}_{2} +\overline{a}_{1}  +1]}&w_{a,2}\\
\end{block}
\end{blockarray}}
\end{align*}
where
\[w_{a,i} := \begin{cases}
\frac{\lambda_a [2\overline{a}_i+2] + \lambda_{a+\epsilon_i}}{[2\overline{a}_i+1]}&\text{ if } [2\overline{a}_i+1]\neq 0\\
\frac{\lambda_a [2\overline{a}_i-2] -\sum_{j\neq i} \frac{[\overline{a}_i + \overline{a}_j-3]}{[\overline{a}_i + \overline{a}_j+1]}\lambda_{a+\epsilon_j}  }{[2\overline{a}_i-3]}&\text{ if } [2\overline{a}_i+1]= 0
\end{cases}\]
satisfies (R2), (R3), and (Hecke). We can directly verify that this solution also satisfies (R1).

\begin{lem}\label{lem:R1Conj}
The above solution for the U-cells on $\Gamma^{4,k,*}_{\Lambda_1}$ satisfies (R1).
\end{lem}
\begin{proof}
Let $a$ be any vertex in $\Gamma^{4,k,*}_{\Lambda_1}$. We have to show for all $i \in \{-2,-1,1,2\}$ that
\[ \frac{\omega_{a,i}}{\lambda_a} + \sum_{j\neq \pm i} \frac{[\overline{a}_i - \overline{a}_j+1]}{[\overline{a}_i - \overline{a}_j]} \frac{\lambda_{a+\epsilon_i + \epsilon_j}}{\lambda_{a+\epsilon_i}} = [3].   \]
In all four cases, this can equality be verified for generic $q$ by a computer\footnote{This could be done by hand, but a computer is faster, and more accurate.} for both forms of $\omega_{a,i}$ by expanding out the expression in terms of the variable $q$. For the first form of $\omega_{a,i}$, we have to simplify $\frac{[2\overline{a}_i+1]}{[2\overline{a}_i+1]}$, and for the second for we have to simplify $\frac{[2\overline{a}_i-3]}{[2\overline{a}_i-3]}$. Hence this equality holds for all $q$.
\end{proof}

To obtain a solution for the B-cells, we solve the linear system from (BA) and (RI), and normalise with (N). This gives the following: 

\begin{align*}
B_{a,\underline{\hspace{.5em}}, a+2\epsilon_i, \underline{\hspace{.5em}}} &= \begin{blockarray}{cc}
a+  \varepsilon_{i}  &\\
\begin{block}{[c]c}
 0 &a+  \varepsilon_{i} \\
\end{block}
\end{blockarray}\\
B_{a,\underline{\hspace{.5em}}, a+\epsilon_i+\epsilon_j, \underline{\hspace{.5em}}} &= \frac{1}{\sqrt{[4]!}}\operatorname{sign}(i)\operatorname{sign}(j)\begin{blockarray}{ccc}
a+  \varepsilon_{i}  &a+  \varepsilon_{j}&\\
\begin{block}{[cc]c}
 \frac{[\overline{a}_{j}-\overline{a}_{i}-1]}{[\overline{a}_{j}-\overline{a}_{i}]}\sqrt{\frac{\lambda_{a+\epsilon_i+\epsilon_j}}{\lambda_a}} &  \sqrt{\frac{[\overline{a}_{i}+\overline{a}_{j}][\overline{a}_{i}+\overline{a}_{j}+2]}{[\overline{a}_{i}+\overline{a}_{j}+1]^2}}\sqrt{\frac{\lambda_{a+\epsilon_i}\lambda_{a+\epsilon_j}}{\lambda_a^2}} & a + \epsilon_i\\
  \sqrt{\frac{[\overline{a}_{i}+\overline{a}_{j}][\overline{a}_{i}+\overline{a}_{j}+2]}{[\overline{a}_{i}+\overline{a}_{j}+1]^2}}\sqrt{\frac{\lambda_{a+\epsilon_i}\lambda_{a+\epsilon_j}}{\lambda_a^2}} &\frac{[\overline{a}_{i}-\overline{a}_{j}-1]}{[\overline{a}_{i}-\overline{a}_{j}]}\sqrt{\frac{\lambda_{a+\epsilon_i+\epsilon_j}}{\lambda_a}}&a+  \varepsilon_{j} \\
\end{block}
\end{blockarray}\\
B_{a,\underline{\hspace{.5em}}, a, \underline{\hspace{.5em}}} &= \scalemath{.5}{\frac{1}{\lambda_a\sqrt{[4]!}}\begin{blockarray}{ccccc}
a+  \varepsilon_{-2} & a+  \varepsilon_{-1} &a+  \varepsilon_{1}&a+  \varepsilon_{2} &\\
\begin{block}{[cccc]c}
 [\overline{a}_{-2}]\left(   [\overline{a}_{1}+1][\overline{a}_{1}-\overline{a}_{2}] - [\overline{a}_{1}-1][\overline{a}_{1}+\overline{a}_{2}]   \right)&\sqrt{\lambda_{a+\epsilon_{-1}}\lambda_{a+\epsilon_{-2}}}\frac{[\overline{a}_{-1}+\overline{a}_{-2}]}{[\overline{a}_{-1}+\overline{a}_{-2}+1]}&-\sqrt{\lambda_{a+\epsilon_{1}}\lambda_{a+\epsilon_{-2}}}\frac{[\overline{a}_{1}+\overline{a}_{-2}]}{[\overline{a}_{1}+\overline{a}_{-2}+1]}&0&a+  \varepsilon_{-2}\\
  \sqrt{\lambda_{a+\epsilon_{-2}}\lambda_{a+\epsilon_{-1}}}\frac{[\overline{a}_{-2}+\overline{a}_{-1}]}{[\overline{a}_{-2}+\overline{a}_{-1}+1]}& [\overline{a}_{-1}]\left(   [\overline{a}_{2}+1][\overline{a}_{1}-\overline{a}_{2}] + [\overline{a}_{2}-1][\overline{a}_{1}+\overline{a}_{2}]   \right)&0&-\sqrt{\lambda_{a+\epsilon_{2}}\lambda_{a+\epsilon_{-1}}}\frac{[\overline{a}_{2}+\overline{a}_{-1}]}{[\overline{a}_{2}+\overline{a}_{-1}+1]}&a+  \varepsilon_{-1}\\
     -\sqrt{\lambda_{a+\epsilon_{-2}}\lambda_{a+\epsilon_{1}}}\frac{[\overline{a}_{-2}+\overline{a}_{1}]}{[\overline{a}_{-2}+\overline{a}_{1}+1]}&0&[\overline{a}_{1}]\left(   [\overline{a}_{2}-1][\overline{a}_{1}-\overline{a}_{2}] + [\overline{a}_{2}+1][\overline{a}_{1}+\overline{a}_{2}]   \right)&\sqrt{\lambda_{a+\epsilon_{2}}\lambda_{a+\epsilon_{1}}}\frac{[\overline{a}_{2}+\overline{a}_{1}]}{[\overline{a}_{2}+\overline{a}_{1}+1]}&a+  \varepsilon_{1}\\
    0&-\sqrt{\lambda_{a+\epsilon_{-1}}\lambda_{a+\epsilon_{2}}}\frac{[\overline{a}_{-1}+\overline{a}_{2}]}{[\overline{a}_{-1}+\overline{a}_{2}+1]}&\sqrt{\lambda_{a+\epsilon_{1}}\lambda_{a+\epsilon_{2}}}\frac{[\overline{a}_{1}+\overline{a}_{2}]}{[\overline{a}_{1}+\overline{a}_{2}+1]}& [\overline{a}_{2}]\left(   [\overline{a}_{1}-1][\overline{a}_{1}-\overline{a}_{2}] - [\overline{a}_{1}+1][\overline{a}_{1}+\overline{a}_{2}]   \right)&a+  \varepsilon_{2}\\
\end{block}
\end{blockarray}}
\end{align*}

\begin{lem}\label{lem:conjB}
The solutions for the U and B cells above satisfy (BA) and (RI) when $\omega = 1$.
\end{lem}
\begin{proof}
This is a direct computation.  

To show (BA), we have to show that each of the $B_{x,\underline{\hspace{.5em}}, y, \underline{\hspace{.5em}}}$ is an eigenmatrix for $U^{x}_{\quad y}$ with eigenvalue $[2]$. Clearly $B_{a,\underline{\hspace{.5em}}, a+2\epsilon_i, \underline{\hspace{.5em}}}$ satisfies this. For the matrix $B_{a,\underline{\hspace{.5em}}, a+\epsilon_i+\epsilon_j, \underline{\hspace{.5em}}}$ we compute the upper right corner of $U^{a}_{\quad a+\epsilon_i+\epsilon_j}\cdot B_{a,\underline{\hspace{.5em}}, a+\epsilon_i+\epsilon_j, \underline{\hspace{.5em}}} $ as follows (we assume $i = 1$ and $j = -2$ to ease notation):
\begin{align*}  
&\frac{-1}{\sqrt{[4]!}}\frac{[\overline{a}_1 -\overline{a}_{-2}+1][\overline{a}_{-2} -\overline{a}_{1}-1]}{[\overline{a}_1 -\overline{a}_{-2}][\overline{a}_{-2} -\overline{a}_1]}\sqrt{ \frac{[\overline{a}_1 + 1][\overline{a}_2 -1] [\overline{a}_1-\overline{a}_2+2] [\overline{a}_1+\overline{a}_2]}{[\overline{a}_1 ][\overline{a}_2 ] [\overline{a}_1-\overline{a}_2] [\overline{a}_1+\overline{a}_2]}  }-\frac{1}{\sqrt{[4]!}}\frac{\sqrt{[\overline{a}_1 -\overline{a}_{-2}+1][\overline{a}_1 -\overline{a}_{-2}-1]}}{[\overline{a}_1 -\overline{a}_{-2}]}\\
&\cdot \sqrt{\frac{[\overline{a}_{1}+\overline{a}_{-2}][\overline{a}_{1}+\overline{a}_{-2}+2]}{[\overline{a}_{1}+\overline{a}_{-2}+1]^2}}\sqrt{\frac{[\overline{a}_1+1][\overline{a}_2] [\overline{a}_1-\overline{a}_2+1] [\overline{a}_1+\overline{a}_2 +1] [\overline{a}_1][\overline{a}_2 -1] [\overline{a}_1-\overline{a}_2+1] [\overline{a}_1+\overline{a}_2 -1]}{[\overline{a}_1 ]^2[\overline{a}_2 ]^2 [\overline{a}_1-\overline{a}_2]^2 [\overline{a}_1+\overline{a}_2]^2}}
   \\
  &= \frac{-1}{\sqrt{[4]!}}\left(  [\overline{a}_1+\overline{a}_2+1] + [\overline{a}_1+\overline{a}_2-1]   \right)\frac{[\overline{a}_1+\overline{a}_2+1]}{[\overline{a}_1+\overline{a}_2]^2}\sqrt{  \frac{[\overline{a}_1+1][\overline{a}_2-1][\overline{a}_1-\overline{a}_2+2]}{[\overline{a}_1][\overline{a}_2][\overline{a}_1-\overline{a}_2]}  }\\
   &= [2]\frac{-1}{\sqrt{[4]!}}\frac{[\overline{a}_1+\overline{a}_2+1]}{[\overline{a}_1+\overline{a}_2]}\sqrt{  \frac{[\overline{a}_1+1][\overline{a}_2-1][\overline{a}_1-\overline{a}_2+2]}{[\overline{a}_1][\overline{a}_2][\overline{a}_1-\overline{a}_2]}  }
\end{align*}
which is exactly $[2]$ times the upper left corner of $B_{a,\underline{\hspace{.5em}}, a+\epsilon_i+\epsilon_j, \underline{\hspace{.5em}}} $. The other entries (and $i,j$ values) follow in a similar fashion.

To verify that $B_{a,\underline{\hspace{.5em}}, a, \underline{\hspace{.5em}}}$ is an eigenmatrix for $U^{a}_{\quad a}$ in the we enlist the help of a computer (as in Lemma~\ref{lem:R1Conj}) to express both $U^{a}_{\quad a}\cdot B_{a,\underline{\hspace{.5em}}, a, \underline{\hspace{.5em}}}$ and $[2]B_{a,\underline{\hspace{.5em}}, a, \underline{\hspace{.5em}}}$ in terms of a formal variable $q$. We find equality by considering terms. As our expression for $U^{a}_{\quad a}$ (in particular the $w_{a,i}$ term) changes depending on the value of $[2\overline{a}_i+1]$ we have to consider both cases here. In both cases we perform a division which is non-zero only if $[2\overline{a}_1 +1]\neq 0$ in the first case, and  $[2\overline{a}_1 -3]\neq 0$ in the second case.

The relation (RI) can easily be verified by hand. We compute case as an example. We let $b := a + \epsilon_2$, so that $\overline{b}_1 = \overline{a_1}$ and $\overline{b_2} = \overline{a_2}+1$.
\begin{align*}
\frac{\lambda_{a+\epsilon_2}}{\lambda_a}B_{a+\epsilon_2, a , a+\epsilon_1, a} &=\frac{\lambda_{a+\epsilon_2}}{\lambda_a} B_{b, b + \epsilon_{-2}, b + \epsilon_1 + \epsilon_{-2}, b + \epsilon_{-2}}\\
&=\frac{-1}{\sqrt{[4]!}} \frac{\lambda_{a+\epsilon_2}}{\lambda_a}\frac{[\overline{b}_1-\overline{b}_{-2}-1]}{[\overline{b}_1-\overline{b}_{-2}]}\sqrt{\frac{\lambda_{b + \epsilon_{1} + \epsilon_{-2}}}{\lambda_{b}}}\\
&= \frac{-1}{\sqrt{[4]!}}\frac{\lambda_{a+\epsilon_2}}{\lambda_a}\frac{[\overline{a}_1+\overline{a}_2]}{[\overline{a}_1+\overline{a}_2+1]}\sqrt{\frac{\lambda_{a + \epsilon_{1}}}{\lambda_{a+\epsilon_2}}}\\
&= \frac{-1}{\sqrt{[4]!}}\frac{1}{\lambda_a}\frac{[\overline{a}_1+\overline{a}_2]}{[\overline{a}_1+\overline{a}_2+1]}\sqrt{\lambda_{a + \epsilon_{1}}\lambda_{a + \epsilon_{2}}}\\
&= (-1)^{4+1}\cdot 1\cdot B_{a,a+\epsilon_2, a , a+\epsilon_1}.
\end{align*}
The remaining cases all follow the same procedure.
\end{proof}
We thus have a KW-cell system on $\Gamma^{4,k,*}_{\Lambda_1}$ for all $k\geq 1$. This gives the following theorem.

\begin{thm}
Let $k\in \mathbb{N}_{\geq 1}$. Then there exists a module category $\mathcal{M}$ for $\mathcal{C}(\mathfrak{sl}_4, k)$ such that the fusion graph for action by $\Lambda_1$ is $\Gamma^{4,k,*}_{\Lambda_1}$.
\end{thm}

\subsection{Unfolded Charge-Conjugation Modules for $SU(4)$ at Level $k$}
In this subsection, we will construct a family of KW cell systems on the graphs $\Gamma^{4,k,\mathbb{Z}_2^*}_{\Lambda_1}$, where $N=4$ and $q = e^{2\pi i \frac{1}{2(4+k)}}$ for all $k\in \mathbb{N}$. These will construct the \textit{unfolded charge conjugation} module categories over $\mathcal{C}(\mathfrak{sl}_4, k)$ for all $k$.

We define the graph $\Gamma^{4,k,\mathbb{Z}_2^*}$ as follows. The vertices of $\Gamma^{4,k,\mathbb{Z}_2^*}_{\Lambda_1}$ are $a = (i,j,\delta)$ such that $i + 2j \leq k+3$, and $\delta = \pm$. There is an edge in $\Gamma^{4,k,\mathbb{Z}_2^*}_{\Lambda_1}$ from $(a,\delta_1)\to (b,\delta_2)$ if 
\[a-b \in \{\varepsilon_{-2}:= (1,-1),\quad \varepsilon_{-1}:= (-1,0), \quad \varepsilon_{1}:= (1,0),\quad  \varepsilon_{2}:= (-1,1)   \},\]
and
\[ \delta_1\delta_2 = \begin{cases}
    + & \text{ if } \overline{a}_1 - \overline{a}_2 \text{ is odd}\\
    - & \text{ if } \overline{a}_1 - \overline{a}_2 \text{ is even}\\
    \end{cases}\]

When $k$ is odd, the graph $\Gamma^{4,k,\mathbb{Z}_2^*}_{\Lambda_1}$ is of the form:
\[\raisebox{-.5\height}{ \includegraphics[scale = .6]{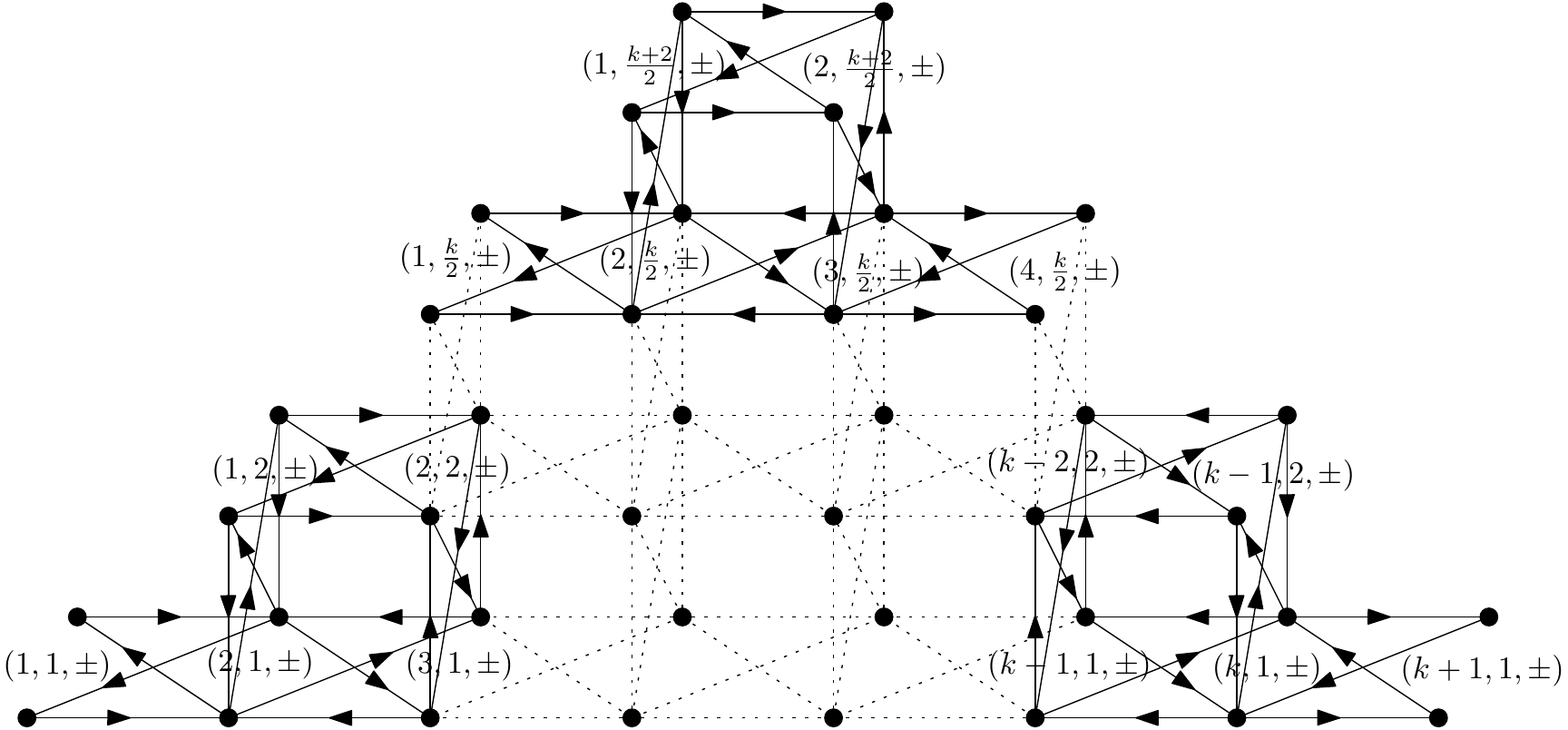}} \]
When $k$ is even, the graph $\Gamma^{4,k,\mathbb{Z}_2^*}_{\Lambda_1}$ is of the form:
\[\raisebox{-.5\height}{ \includegraphics[scale = .6]{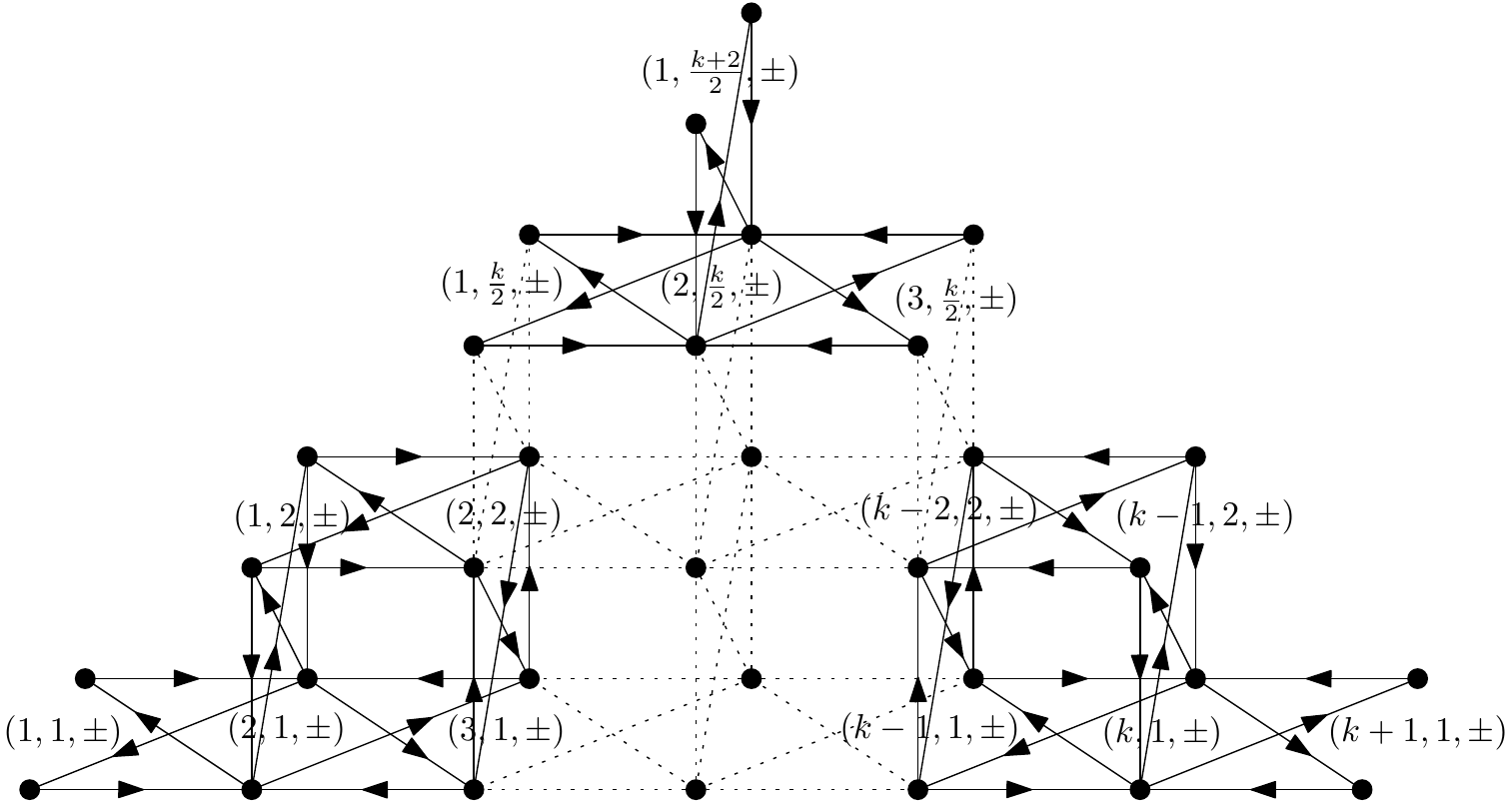}}. \]

From the definition of the graph $\Gamma^{4,k,\mathbb{Z}_2^*}_{\Lambda_1}$ we have that 
\[ \begin{tikzcd}
(a,\delta_1)\arrow[r] \arrow[d]
&  (b, \delta_2)  \arrow[d]\\
 (c,\delta_3)\arrow[r]&(d, \delta_4)
\end{tikzcd}  \]
is a valid path in $\Gamma^{4,k,\mathbb{Z}_2^*}_{\Lambda_1}$ if and only if 
\[ \begin{tikzcd}
a \arrow[r] \arrow[d]
& b \arrow[d]\\
c\arrow[r]&d
\end{tikzcd}  \]
is a valid path in $\Gamma^{4,k,*}_{\Lambda_1}$. Furthermore, for each valid path in $\Gamma^{4,k,*}_{\Lambda_1}$ of the above form, there are exactly two valid paths in $\Gamma^{4,k,\mathbb{Z}_2^*}_{\Lambda_1}$ of the above form (corresponding to the choice of $\delta_1 = \pm$).

The positive eigenvector $\lambda$ for $\Gamma^{4,k,\mathbb{Z}_2^*}_{\Lambda_1}$ is essentially the same as for the graph $\Gamma^{4,k,*}_{\Lambda_1}$.
\[   \lambda_{a,\delta} = [\overline{a}_1][\overline{a}_2] [\overline{a}_1-\overline{a}_2] [\overline{a}_1+\overline{a}_2].   \]

The graph $\Gamma^{4,k,\mathbb{Z}_2^*}_{\Lambda_1}$ inherits the cell system from the graph $\Gamma^{4,k,*}_{\Lambda_1}$. More precisely, we have
\[
U^{(a,\delta_1),(b, \delta_2) }_{(c,\delta_3),(d, \delta_4) }:= U^{a,b}_{c,d}\qquad B_{(e,\delta_5), (f, \delta_6), (g,\delta_7), (h, \delta_8)} := B_{e,f,g,h}
\]
where $U^{a,b}_{c,d}$ and $B_{e,f,g,h}$ are the cell system solution on $\Gamma^{4,k,*}_{\Lambda_1}$ from Subsection~\ref{sec:charge}.

\begin{lem}
The KW cell system on $\Gamma^{4,k,\mathbb{Z}_2^*}_{\Lambda_1}$ above satisfies (R1), (R2), (R3), (Hecke), (BA) and (RI) with $\omega = 1$.
\end{lem}
\begin{proof}
This follows by construction of $\Gamma^{4,k,\mathbb{Z}_2^*}_{\Lambda_1}$, along with the fact that the positive eigenvector for $\Gamma^{4,k,\mathbb{Z}_2^*}_{\Lambda_1}$ is essentially the same as for $\Gamma^{4,k,*}_{\Lambda_1}$. For an example, we will show (R1) holds. The other relations follow in the same manner.

Let $(a,\delta_1), (b, \delta_2)\in \Gamma^{4,k,\mathbb{Z}_2^*}_{\Lambda_1}$, we compute
\begin{align*} \sum_{(p,\delta_3): (b, \delta_2)\to(p,\delta_3) }  \frac{\lambda_{(p, \delta_3)}}{\lambda_{(b, \delta_2)}}KW\left(\raisebox{-.5\height}{ \includegraphics[scale = .4]{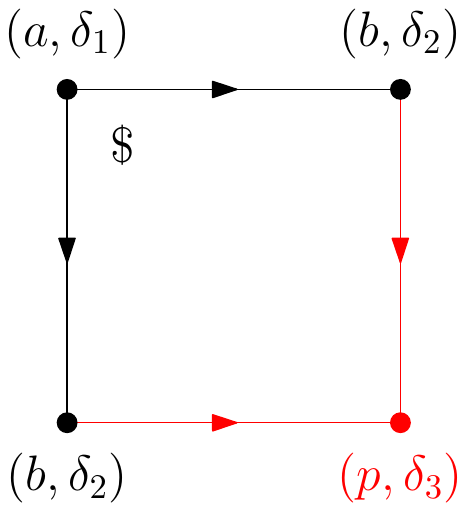}}\right)&=  \sum_{(p,\delta_3): (b, \delta_2)\to(p,\delta_3) }  \frac{\lambda_{p}}{\lambda_{b}}KW\left(\raisebox{-.5\height}{ \includegraphics[scale = .4]{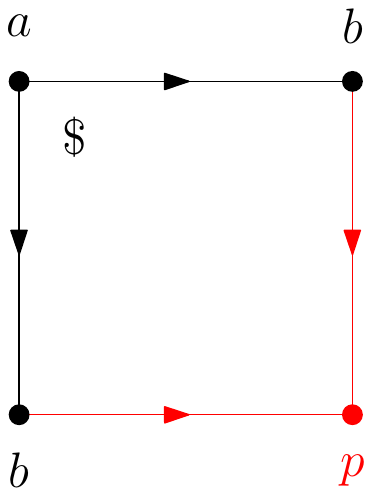}}\right) \\
&= \sum_{p: b\to p }  \frac{\lambda_{p}}{\lambda_{b}}KW\left(\raisebox{-.5\height}{ \includegraphics[scale = .4]{Ucharge2.pdf}}\right)\\
&= [3].
\end{align*}
Here we overload the symbols $\lambda$ and $KW$ for the positive eigenvector and KW cell systems of both $\Gamma^{4,k,\mathbb{Z}_2^*}_{\Lambda_1}$ and $\Gamma^{4,k,*}_{\Lambda_1}$. The indexing resolves this overloading of notation. The first equality holds by definition of our cell system on $\Gamma^{4,k,\mathbb{Z}_2^*}_{\Lambda_1}$, and the fact that the eigenvectors for the two graphs are essentially the same. The second equality holds as there is a path in $\Gamma^{4,k,\mathbb{Z}_2^*}_{\Lambda_1}$ from $(b, \delta_2)\to(p,\delta_3)$ if and only if there is a path in $\Gamma^{4,k,*}_{\Lambda_1}$ from $b\to p$. The final equality holds as we have shown our KW cell system on $\Gamma^{4,k,*}_{\Lambda_1}$ satisfies (R1) in Subsection~\ref{sec:charge}.
\end{proof}
In consequence we construct the unfolded charge conjugation modules.
\begin{thm}
Let $k\in \mathbb{N}_{\geq 1}$. Then there exists a module category $\mathcal{M}$ for $\mathcal{C}(\mathfrak{sl}_4, k)$ such that the fusion graph for action by $\Lambda_1$ is $\Gamma^{4,k,\mathbb{Z}_2^*}_{\Lambda_1}$.
\end{thm}
\subsection*{Unfolded Orbifold Charge-Conjugation Modules for $SU(4)$ at Even Levels}

Let $k$ be an even integer. We will build a cell system with $N=4$ and $q = e^{2\pi i \frac{1}{2(N+k)}}$ on the graph $\Gamma^{4,k,{\mathbb{Z}_4^*}}_{\Lambda_1}$. The graph $\Gamma^{4,k,{\mathbb{Z}_4^*}}_{\Lambda_1}$ is defined as an orbifold of the graph $\Gamma^{4,k,{\mathbb{Z}_2^*}}_{\Lambda_1}$ with respect to the horizontal flip symmetry.

Explicitly, the vertices of $\Gamma^{4,k,{\mathbb{Z}_4^*}}_{\Lambda_1}$ are of the form
\[    \left\{  ((i,j),\pm)   : i+j < \frac{k+4}{2} \right\}  \bigcup \left\{  ((i,j),\pm,
\chi_{\pm 1})   : i+j = \frac{k+4}{2}  \right\}. \]
We will refer to the second group as the split vertices. 

There is an edge $(a, \delta_1) \to (b, \delta_2)$, or $(a, \delta_1,\chi_\ell) \to (b, \delta_2)$, or $(a, \delta_1) \to (b, \delta_2,\chi_\ell)$ if the edge  $(a, \delta_1) \to (b, \delta_2)$ exists in $\Gamma^{4,k,{\mathbb{Z}_2^*}}_{\Lambda_1}$. There is an edge $(a, \delta_1,\chi_{\ell_1}) \to (b, \delta_2,\chi_{\ell_2})$ if there is an edge $(a, \delta_1) \to (b, \delta_2)$ in $\Gamma^{4,k,{\mathbb{Z}_2^*}}_{\Lambda_1}$, and 
\[ \chi_{\ell_1}\cdot \chi_{\ell_2} = \begin{cases}
-1 &\quad \text{if } \delta_1  = -1 \quad\text{and}\quad  \overline{a}_1 - \overline{a}_2 \text{ is even}    \\
1 &\quad \text{otherwise}.
\end{cases}      \]
The graph $\Gamma^{4,k,{\mathbb{Z}_4^*}}_{\Lambda_1}$ is of the form:
\[\raisebox{-.5\height}{ \includegraphics[scale = .6]{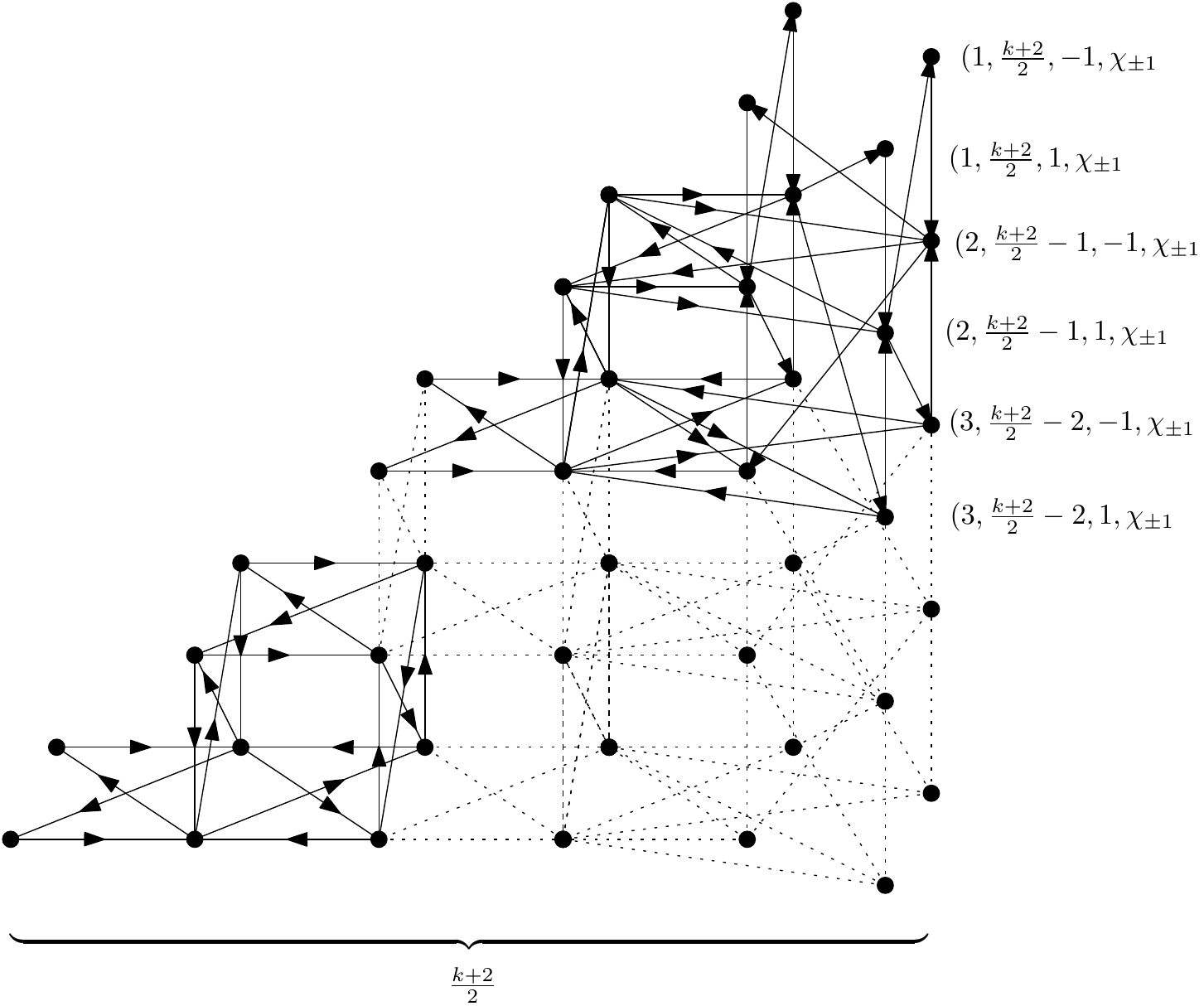}}. \]
The Frobenius-Perron eigenvector for $\Gamma^{4,k,{\mathbb{Z}_4^*}}_{\Lambda_1}$ is easily described in terms of the eigenvector for eigenvector for $\Gamma^{4,k,{*}}_{\Lambda_1}$. Note that we will overload the symbol $\lambda$ for both these eigenvectors, as the indexing resolves the ambiguity. We have
\begin{align*}
\lambda_{(a, \delta)} &= \lambda_{a} = [\overline{a}_1][\overline{a}_2] [\overline{a}_1-\overline{a}_2] [\overline{a}_1+\overline{a}_2]\\
\lambda_{(a, \delta, \chi_\ell)} &= \frac{\lambda_{a}}{2} = \frac{[\overline{a}_1][\overline{a}_2] [\overline{a}_1-\overline{a}_2] [\overline{a}_1+\overline{a}_2]}{2}.\\
\end{align*}

To build a cell system solution on $\Gamma^{4,k,{\mathbb{Z}_4^*}}_{\Lambda_1}$, we observe that the U-cells for our cell system solution on $\Gamma^{4,k,{\mathbb{Z}_2^*}}_{\Lambda_1}$ respects the horizontal flip symmetry (however the B-cells do not). This means we can orbifold the U-cell solution on $\Gamma^{4,k,{\mathbb{Z}_2^*}}_{\Lambda_1}$ to get a potential U-cell solution on $\Gamma^{4,k,{\mathbb{Z}_4^*}}_{\Lambda_1}$. From the results of \cite{MR1301620} we get that this solution will satisfy (R2), (R3), and (Hecke). Following their construction gives the following.

For any path that doesn't pass through a split vertex, the corresponding U-cell is directly inherited from the solution on $\Gamma^{4,k,{\mathbb{Z}_2^*}}_{\Lambda_1}$. Let $a = (i,j)$. For the remaining U-cells we obtain:

If $\overline{a}_1 = \frac{k}{2}$, then
\begin{align*}
   U^{(a ,\delta)}_{\quad (a+2\varepsilon_1,-\delta, \chi_{ l})} &=  \begin{blockarray}{c}
(a+\varepsilon_1,\delta')  \\
\begin{block}{[c]}
  0\\
\end{block}
\end{blockarray}\\
\end{align*}
If $\overline{a}_1 = \frac{k+2}{2}$, then
\begin{align*}
U^{( a,\delta)}_{\quad (a+\epsilon_1 \pm \epsilon_2,-\delta,\chi_{ l })} &= \frac{1}{ [\overline{a}_1 -\overline{a}_2]} \begin{blockarray}{cc}
(a+  \varepsilon_1, \delta', \chi_{l'}) & (a\pm\varepsilon_2, \delta')  \\
\begin{block}{[cc]} 
  [\overline{a}_1 \mp \overline{a}_2+1]   & \sqrt{[\overline{a}_1 \mp\overline{a}_2+1][\overline{a}_1 \mp \overline{a}_2-1]}\\
  \sqrt{[\overline{a}_1 \mp \overline{a}_2+1][\overline{a}_1 \mp \overline{a}_2-1]} &[\overline{a}_1 \mp \overline{a}_2-1]\\
\end{block}
\end{blockarray}\\
U^{( a,\delta)}_{\quad ( a,-\delta)} &=\frac{1}{\lambda_a}\scalemath{.6}{\begin{blockarray}{ccccc}
(a+  \varepsilon_{-2}, \delta') & (a+  \varepsilon_{-1},\delta') &(a+  \varepsilon_{1},\delta', \chi_1)&(a+  \varepsilon_{1},\delta', \chi_{-1})&(a+  \varepsilon_{2},\delta')  \\
\begin{block}{[ccccc]}
w_{a,-2}&\sqrt{\lambda_{a+\varepsilon_{-1}}\lambda_{a+\varepsilon_{-2}}}\frac{1}{[\overline{a}_{-1} +\overline{a}_{-2}  +1]}&-\frac{1}{\sqrt{2}}\sqrt{\lambda_{a+\varepsilon_{1}}\lambda_{a+\varepsilon_{-2}}}\frac{1}{[\overline{a}_{1} +\overline{a}_{-2}  +1]}&-\frac{1}{\sqrt{2}}\sqrt{\lambda_{a+\varepsilon_{1}}\lambda_{a+\varepsilon_{-2}}}\frac{1}{[\overline{a}_{1} +\overline{a}_{-2}  +1]}&-\sqrt{\lambda_{a+\varepsilon_{2}}\lambda_{a+\varepsilon_{-2}}}\frac{1}{[\overline{a}_{2} +\overline{a}_{-2}  +1]}\\
 \sqrt{\lambda_{a+\varepsilon_{-1}}\lambda_{a+\varepsilon_{-2}}}\frac{1}{[\overline{a}_{-1} +\overline{a}_{-2}  +1]}&w_{a,-1}&-\frac{1}{\sqrt{2}}\sqrt{\lambda_{a+\varepsilon_{-1}}\lambda_{a+\varepsilon_{1}}}\frac{1}{[\overline{a}_{-1} +\overline{a}_{1}  +1]}&-\frac{1}{\sqrt{2}}\sqrt{\lambda_{a+\varepsilon_{-1}}\lambda_{a+\varepsilon_{1}}}\frac{1}{[\overline{a}_{-1} +\overline{a}_{1}  +1]}&-\sqrt{\lambda_{a+\varepsilon_{-1}}\lambda_{a+\varepsilon_{2}}}\frac{1}{[\overline{a}_{-1} +\overline{a}_{2}  +1]}\\
 -\frac{1}{\sqrt{2}}\sqrt{\lambda_{a+\varepsilon_{1}}\lambda_{a+\varepsilon_{-2}}}\frac{1}{[\overline{a}_{1} +\overline{a}_{-2}  +1]}&-\frac{1}{\sqrt{2}}\sqrt{\lambda_{a+\varepsilon_{1}}\lambda_{a+\varepsilon_{-1}}}\frac{1}{[\overline{a}_{1} +\overline{a}_{-1}  +1]}&\frac{1}{2} w_{a,1}&\frac{1}{2} w_{a,1}&\frac{1}{\sqrt{2}}\sqrt{\lambda_{a+\varepsilon_{1}}\lambda_{a+\varepsilon_{2}}}\frac{1}{[\overline{a}_{1} +\overline{a}_{2}  +1]}\\
  -\frac{1}{\sqrt{2}}\sqrt{\lambda_{a+\varepsilon_{1}}\lambda_{a+\varepsilon_{-2}}}\frac{1}{[\overline{a}_{1} +\overline{a}_{-2}  +1]}&-\frac{1}{\sqrt{2}}\sqrt{\lambda_{a+\varepsilon_{1}}\lambda_{a+\varepsilon_{-1}}}\frac{1}{[\overline{a}_{1} +\overline{a}_{-1}  +1]}&\frac{1}{2} w_{a,1}&\frac{1}{2} w_{a,1}&\frac{1}{\sqrt{2}}\sqrt{\lambda_{a+\varepsilon_{1}}\lambda_{a+\varepsilon_{2}}}\frac{1}{[\overline{a}_{1} +\overline{a}_{2}  +1]}\\
 -\sqrt{\lambda_{a+\varepsilon_{2}}\lambda_{a+\varepsilon_{-2}}}\frac{1}{[\overline{a}_{2} +\overline{a}_{-2}  +1]}&-\sqrt{\lambda_{a+\varepsilon_{2}}\lambda_{a+\varepsilon_{-1}}}\frac{1}{[\overline{a}_{2} +\overline{a}_{-1}  +1]}&\frac{1}{\sqrt{2}}\sqrt{\lambda_{a+\varepsilon_{2}}\lambda_{a+\varepsilon_{1}}}\frac{1}{[\overline{a}_{2} +\overline{a}_{1}  +1]}&\frac{1}{\sqrt{2}}\sqrt{\lambda_{a+\varepsilon_{2}}\lambda_{a+\varepsilon_{1}}}\frac{1}{[\overline{a}_{2} +\overline{a}_{1}  +1]}&w_{a,2}\\
\end{block}
\end{blockarray}}
\end{align*}
If $\overline{a}_1 = \frac{k+4}{2}$, then
\begin{align*}
 U^{(a ,\delta,\chi_l)}_{\quad (a-2\varepsilon_1,-\delta, \chi_{ l'})} &=  \begin{blockarray}{c}
(a-\varepsilon_1,\delta'' )  \\
\begin{block}{[c]}
  0\\
\end{block}
\end{blockarray}\\
U^{(a,\delta,\chi_{ l })}_{\quad (a \pm 2\epsilon_2,-\delta,\chi_{ l' })} &=  \begin{blockarray}{c}
(a\pm \epsilon_2,\delta'',\chi_{ l'' })  \\
\begin{block}{[c]}
  0\\
\end{block}
\end{blockarray}\\
U^{(a,\delta,\chi_{ l })}_{\quad (a ,-\delta,\chi_{ l' })} &=  \begin{blockarray}{c}
(a -  \epsilon_1,\delta')  \\
\begin{block}{[c]}
  [2]\\
\end{block}
\end{blockarray}\quad &&\text{ if $l\cdot l' = -\delta$}\\
U^{(a,\delta,\chi_{ l })}_{\quad (a ,-\delta,\chi_{ l' })} &=  \begin{blockarray}{ccc}
(a + \varepsilon_{-2},\delta'',\chi_{l'}) & (a + \varepsilon_{-1},\delta'')  &(a + \varepsilon_{2},\delta'',\chi_{l''}) \\
\begin{block}{[ccc]}
 w_{a,-2} & \sqrt{2}\frac{\sqrt{\lambda_{a + \varepsilon_{-2}}\lambda_{a + \varepsilon_{-1}}}}{[\overline{a}_{-2} + \overline{a}_{-1}+1]} &-\sqrt{\lambda_{a + \varepsilon_{-2}}\lambda_{a + \varepsilon_{2}}} \\
  \sqrt{2}\frac{\sqrt{\lambda_{a + \varepsilon_{-2}}\lambda_{a + \varepsilon_{-1}}}}{[\overline{a}_{-2} + \overline{a}_{-1}+1]} & w_{a,-1} - \sqrt{\lambda_{a + \varepsilon_{-1}}\lambda_{a + \varepsilon_{1}}}& -\sqrt{2}\frac{\sqrt{\lambda_{a + \varepsilon_{-1}}\lambda_{a + \varepsilon_{2}}}}{[\overline{a}_{-1} + \overline{a}_{2}+1]}\\
  -\sqrt{\lambda_{a + \varepsilon_{-2}}\lambda_{a + \varepsilon_{2}}} & -\sqrt{2}\frac{\sqrt{\lambda_{a + \varepsilon_{-1}}\lambda_{a + \varepsilon_{2}}}}{[\overline{a}_{-1} + \overline{a}_{2}+1]} & w_{a,2}\\
\end{block}
\end{blockarray}\quad &&\text{ if $l\cdot l' = \delta$}\\
U^{(a,\delta,\chi_{ l })}_{\quad (a-\varepsilon_1 \pm \varepsilon_2 ,-\delta)}&= \frac{1}{ [\overline{a}_{-1} \mp \overline{a}_2]} \begin{blockarray}{cc}
(a+  \varepsilon_{-1}, \delta'') & (a\pm \varepsilon_2,\delta'', \chi_{l'})  \\
\begin{block}{[cc]}
  [\overline{a}_{-1} \mp\overline{a}_2+1]   & \sqrt{[\overline{a}_{-1} \mp \overline{a}_2+1][\overline{a}_{-1} \mp\overline{a}_2-1]}\\
  \sqrt{[\overline{a}_{-1} \mp\overline{a}_2+1][\overline{a}_{-1} \mp\overline{a}_2-1]} &[\overline{a}_{-1} \mp\overline{a}_2-1]\\
\end{block}
\end{blockarray}
\end{align*}
As the results of \cite{MR1301620} only establish that their construction gives an embedding of path algebras (and not of the entire graph planar algebra), we still have to verify relation (R1), which occurs outside of the path algebra.

\begin{lem}
The above U-cells for $\Gamma^{4,k,{\mathbb{Z}_4^*}}_{\Lambda_1}$ satisfy (R1).
\end{lem}
\begin{proof}
This is a direct computation as done in Lemma~\ref{lem:R1Conj}. As the U-cell solution is the same for paths that do not pass through a split vertex, we only have to consider vertices at most distance 2 away from a split vertex. This has to be done in several cases. Let us demonstrate the case where $s(i) = (a,1,\chi_{ 1 })$, and $r(i) =  (a+\varepsilon_2,1,\chi_{1 })$. Note that $\overline{a}_1- \overline{a}_2$ is odd for this edge to exist. We also have that $\overline{a}_1 = \frac{k+4}{2}$ . We have to show that
\[   \scalemath{.9}{ \frac{ \lambda_{(a, -1, \chi_{1})}}{\lambda_{(a+\varepsilon_2, 1, \chi_{1})} }  U^{(a,1,\chi_{1}),(a+\varepsilon_2,1,\chi_{1}) }_{(a+\varepsilon_2,1,\chi_{1}), (a ,-1,\chi_{ 1 })} + \frac{ \lambda_{(a+2\varepsilon_2, -1, \chi_{1})}}{\lambda_{(a+\varepsilon_2, 1, \chi_{1})} }  U^{(a,1,\chi_{1}),(a+\varepsilon_2,1,\chi_{1}) }_{(a+\varepsilon_2,1,\chi_{1}), (a+2\varepsilon_2 ,-1,\chi_{ 1 })}  + \frac{ \lambda_{(a-\varepsilon_1+\varepsilon_2, -1)}}{\lambda_{(a+\varepsilon_2, 1, \chi_{1})} }  U^{(a,1,\chi_{1}),(a+\varepsilon_2,1,\chi_{1}) }_{(a+\varepsilon_2,1,\chi_{1}), (a-\varepsilon_1+\varepsilon_2 ,-1)}= [3]. }    \]
Expanding out the left hand side gives
\[  \frac{[\overline{a}_1][\overline{a}_2][\overline{a}_1- \overline{a}_2][\overline{a}_1+ \overline{a}_2]}{[\overline{a}_1][\overline{a}_2+1][\overline{a}_1- \overline{a}_2-1][\overline{a}_1+ \overline{a}_2+1] }\frac{w_{a,2}}{\lambda_a} +  \frac{2 [\overline{a}_1-1][\overline{a}_2+1][\overline{a}_1-\overline{a}_2-2][\overline{a}_1+\overline{a}_2]}{[\overline{a}_1][\overline{a}_2+1][\overline{a}_1- \overline{a}_2-1][\overline{a}_1+ \overline{a}_2+1] }\cdot \frac{[\overline{a}_{-1} - \overline{a}_2-1]}{[\overline{a}_{-1} - \overline{a}_2]}.   \]
Note that $w_{a,2}$ must be of the first form in this situation, as $[2\overline{a}_2 + 1] \neq 0$. We can then use a computer to expand this in the formal variable $q$ (and using that $\overline{a}_1 = \frac{k+4}{2}$) to obtain the uninspiring
\[   \frac{\frac{\left(q^2-1\right) \left(q^{k+2}-1\right) q^{k+4}}{q^{2 {\overline{a}_2}}-q^{k+2}}-\frac{\left(q^2-1\right) \left(q^{k+6}-1\right)}{q^{k+2 {\overline{a}_2}+6}-1}+q^{k+4}+2 q^{k+6}-2 q^4-1}{q^2 \left(q^{k+4}-1\right)}.\]
However recalling that $q = e^{2 \pi i \frac{1}{2(4+k)}}$, we have that $q^{4+k} = -1$. This simplifies the above expression to give $1 + q^2 + q^{-2} = [3]$ as desired.

All the remaining cases are resolved in the same manner.
\end{proof}

Now that we have a solution for our U-cells, we can solve the linear system (RI) + (BA) to obtain the following solution for the B-cells. For any path not passing through a split vertex, the B-cell value is directly inherited from the corresponding B-cell value on $\Gamma^{4,k,{\mathbb{Z}_4^*}}_{\Lambda_1}$. For the remaining cells, we have:

If $\overline{a}_1 = \frac{k}{2}$, then
\begin{align*}
B_{(a, \delta),\underline{\hspace{.5em}}, (a+2\epsilon_1, -\delta, \chi_{\pm 1}), \underline{\hspace{.5em}}} &= \begin{blockarray}{cc}
(a+  \varepsilon_{1}, \delta')  &\\
\begin{block}{[c]c}
 0 &(a+  \varepsilon_{1}, -\delta')  \\
\end{block}
\end{blockarray}\\
\end{align*}
If $\overline{a}_1 = \frac{k+2}{2}$, then
\begin{align*}
B_{(a,\delta),\underline{\hspace{.5em}}, (a+\epsilon_1 \pm \epsilon_2,-\delta,\chi_{l}), \underline{\hspace{.5em}}} &= \scalemath{.9}{\pm \frac{1}{\sqrt{2}}\frac{1}{\sqrt{[4]!}}\begin{blockarray}{ccc}
(a+  \varepsilon_{1}, \delta',\chi_{l'})  &  ( a \pm   \varepsilon_{2},\delta')&\\
\begin{block}{[cc]c}
 \frac{[\overline{a}_{\pm 2}-\overline{a}_{1}-1]}{[\overline{a}_{\pm 2}-\overline{a}_{1}]}\sqrt{\frac{\lambda_{a+\epsilon_1+\epsilon_{\pm 2}}}{\lambda_a}} &  \sqrt{\frac{[\overline{a}_{1}+\overline{a}_{\pm 2}][\overline{a}_{1}+\overline{a}_{\pm 2}+2]}{[\overline{a}_{1}+\overline{a}_{\pm 2}+1]^2}}\sqrt{\frac{\lambda_{a+\epsilon_1}\lambda_{a+\epsilon_{\pm 2}}}{\lambda_a^2}} & (a+  \varepsilon_{1}, \delta',\chi_{l'}) \\
  \sqrt{\frac{[\overline{a}_{1}+\overline{a}_{\pm 2}][\overline{a}_{1}+\overline{a}_{\pm 2}+2]}{[\overline{a}_{1}+\overline{a}_{\pm 2}+1]^2}}\sqrt{\frac{\lambda_{a+\epsilon_1}\lambda_{a+\epsilon_{\pm 2}}}{\lambda_a^2}} &\frac{[\overline{a}_{1}-\overline{a}_{\pm 2}-1]}{[\overline{a}_{1}-\overline{a}_{\pm 2}]}\sqrt{\frac{\lambda_{a+\epsilon_1+\epsilon_{\pm 2}}}{\lambda_a}}& ( a \pm   \varepsilon_{2},\delta') \\
\end{block}
\end{blockarray}}\\
B_{(a,\delta),\underline{\hspace{.5em}}, (a, -\delta), \underline{\hspace{.5em}}} &= \scalemath{.35}{\frac{1}{\lambda_a\sqrt{[4]!}}\begin{blockarray}{cccccc}
(a+  \varepsilon_{-2},\delta') & (a+  \varepsilon_{-1},\delta') &(a+  \varepsilon_{1}, \delta', \chi_1) &(a+  \varepsilon_{1}, \delta', \chi_2)&(a+  \varepsilon_{2},\delta') &\\
\begin{block}{[ccccc]c}
 [\overline{a}_{-2}]\left(   [\overline{a}_{1}+1][\overline{a}_{1}-\overline{a}_{2}] - [\overline{a}_{1}-1][\overline{a}_{1}+\overline{a}_{2}]   \right)&\sqrt{\lambda_{a+\epsilon_{-1}}\lambda_{a+\epsilon_{-2}}}\frac{[\overline{a}_{-1}+\overline{a}_{-2}]}{[\overline{a}_{-1}+\overline{a}_{-2}+1]}&-\frac{1}{\sqrt{2}}\sqrt{\lambda_{a+\epsilon_{1}}\lambda_{a+\epsilon_{-2}}}\frac{[\overline{a}_{1}+\overline{a}_{-2}]}{[\overline{a}_{1}+\overline{a}_{-2}+1]}&-\frac{1}{\sqrt{2}}\sqrt{\lambda_{a+\epsilon_{1}}\lambda_{a+\epsilon_{-2}}}\frac{[\overline{a}_{1}+\overline{a}_{-2}]}{[\overline{a}_{1}+\overline{a}_{-2}+1]}&0&(a+  \varepsilon_{-2},\delta')\\
  \sqrt{\lambda_{a+\epsilon_{-2}}\lambda_{a+\epsilon_{-1}}}\frac{[\overline{a}_{-2}+\overline{a}_{-1}]}{[\overline{a}_{-2}+\overline{a}_{-1}+1]}& [\overline{a}_{-1}]\left(   [\overline{a}_{2}+1][\overline{a}_{1}-\overline{a}_{2}] + [\overline{a}_{2}-1][\overline{a}_{1}+\overline{a}_{2}]   \right)&0&0&-\sqrt{\lambda_{a+\epsilon_{2}}\lambda_{a+\epsilon_{-1}}}\frac{[\overline{a}_{2}+\overline{a}_{-1}]}{[\overline{a}_{2}+\overline{a}_{-1}+1]}&(a+  \varepsilon_{-1},\delta')\\
     -\frac{1}{\sqrt{2}}\sqrt{\lambda_{a+\epsilon_{-2}}\lambda_{a+\epsilon_{1}}}\frac{[\overline{a}_{-2}+\overline{a}_{1}]}{[\overline{a}_{-2}+\overline{a}_{1}+1]}&0&\frac{1}{2}[\overline{a}_{1}]\left(   [\overline{a}_{2}-1][\overline{a}_{1}-\overline{a}_{2}] + [\overline{a}_{2}+1][\overline{a}_{1}+\overline{a}_{2}]   \right)&\frac{1}{2}[\overline{a}_{1}]\left(   [\overline{a}_{2}-1][\overline{a}_{1}-\overline{a}_{2}] + [\overline{a}_{2}+1][\overline{a}_{1}+\overline{a}_{2}]   \right)&\frac{1}{\sqrt{2}}\sqrt{\lambda_{a+\epsilon_{2}}\lambda_{a+\epsilon_{1}}}\frac{[\overline{a}_{2}+\overline{a}_{1}]}{[\overline{a}_{2}+\overline{a}_{1}+1]}&(a+  \varepsilon_{1}, \delta', \chi_1)\\
      -\frac{1}{\sqrt{2}}\sqrt{\lambda_{a+\epsilon_{-2}}\lambda_{a+\epsilon_{1}}}\frac{[\overline{a}_{-2}+\overline{a}_{1}]}{[\overline{a}_{-2}+\overline{a}_{1}+1]}&0&\frac{1}{2}[\overline{a}_{1}]\left(   [\overline{a}_{2}-1][\overline{a}_{1}-\overline{a}_{2}] + [\overline{a}_{2}+1][\overline{a}_{1}+\overline{a}_{2}]   \right)&\frac{1}{2}[\overline{a}_{1}]\left(   [\overline{a}_{2}-1][\overline{a}_{1}-\overline{a}_{2}] + [\overline{a}_{2}+1][\overline{a}_{1}+\overline{a}_{2}]   \right)&\frac{1}{\sqrt{2}}\sqrt{\lambda_{a+\epsilon_{2}}\lambda_{a+\epsilon_{1}}}\frac{[\overline{a}_{2}+\overline{a}_{1}]}{[\overline{a}_{2}+\overline{a}_{1}+1]}&(a+  \varepsilon_{1}, \delta', \chi_2)\\
    0&-\sqrt{\lambda_{a+\epsilon_{-1}}\lambda_{a+\epsilon_{2}}}\frac{[\overline{a}_{-1}+\overline{a}_{2}]}{[\overline{a}_{-1}+\overline{a}_{2}+1]}&\frac{1}{\sqrt{2}}\sqrt{\lambda_{a+\epsilon_{1}}\lambda_{a+\epsilon_{2}}}\frac{[\overline{a}_{1}+\overline{a}_{2}]}{[\overline{a}_{1}+\overline{a}_{2}+1]}&\frac{1}{\sqrt{2}}\sqrt{\lambda_{a+\epsilon_{1}}\lambda_{a+\epsilon_{2}}}\frac{[\overline{a}_{1}+\overline{a}_{2}]}{[\overline{a}_{1}+\overline{a}_{2}+1]}& [\overline{a}_{2}]\left(   [\overline{a}_{1}-1][\overline{a}_{1}-\overline{a}_{2}] - [\overline{a}_{1}+1][\overline{a}_{1}+\overline{a}_{2}]   \right)&(a+  \varepsilon_{2},\delta')\\
\end{block}
\end{blockarray}}
\end{align*}
If $\overline{a}_1 = \frac{k+4}{2}$, then
\begin{align*}
B_{(a,\delta,\chi_{ l}),\underline{\hspace{.5em}}, (a,-\delta,\chi_{l'}), \underline{\hspace{.5em}}} &= \scalemath{1}{ \frac{1}{\lambda_a}\frac{1}{\sqrt{[4]!}}\begin{blockarray}{cc}
(a-  \varepsilon_{1}, -\delta') &\\
\begin{block}{[c]c}
\sqrt{2}\sqrt{\lambda_{a+\varepsilon_{-1}}\lambda_{a+\varepsilon_{-2}}}\frac{[\overline{a}_{-1}+\overline{a}_{-2}  ]}{[\overline{a}_{-1}+\overline{a}_{-2} +1 ]} &  (a+  \varepsilon_{-2}, \delta',\chi_{l''}) \\
\hspace{.01em}[\overline{a}_{-1}]([\overline{a}_{2}+1][\overline{a}_{1}-\overline{a}_{2}] +[\overline{a}_{2}-1][\overline{a}_{1}+\overline{a}_{2}]) &  (a+  \varepsilon_{-1}, \delta') \\
-\sqrt{2}\sqrt{\lambda_{a+\varepsilon_{-1}}\lambda_{a+\varepsilon_{2}}}\frac{[\overline{a}_{-1}+\overline{a}_{2}  ]}{[\overline{a}_{-1}+\overline{a}_{2} +1 ]} &  (a+  \varepsilon_{2}, \delta',\chi_{l''}) \\
\end{block}
\end{blockarray}}\quad\text{ if $l\cdot l' = \delta$}\\
B_{(a,\delta,\chi_{ l}),\underline{\hspace{.5em}}, (a,-\delta,\chi_{l'}), \underline{\hspace{.5em}}} &= \scalemath{.55}{ \frac{1}{\lambda_a}\frac{1}{\sqrt{[4]!}}\begin{blockarray}{cccc}
(a+  \varepsilon_{-2}, \delta', \chi_{l''}) &(a+  \varepsilon_{-1}, \delta') & (a+  \varepsilon_{2}, \delta', \chi_{l''}) &\\
\begin{block}{[ccc]c}
\sqrt{2}\sqrt{\lambda_{a+\varepsilon_{-1}}\lambda_{a+\varepsilon_{-2}}}\frac{[\overline{a}_{-1}+\overline{a}_{-2}  ]}{[\overline{a}_{-1}+\overline{a}_{-2} +1 ]}&
\hspace{.01em}[\overline{a}_{-1}]([\overline{a}_{2}+1][\overline{a}_{1}-\overline{a}_{2}] +[\overline{a}_{2}-1][\overline{a}_{1}+\overline{a}_{2}])&
-\sqrt{2}\sqrt{\lambda_{a+\varepsilon_{-1}}\lambda_{a+\varepsilon_{2}}}\frac{[\overline{a}_{-1}+\overline{a}_{2}  ]}{[\overline{a}_{-1}+\overline{a}_{2} +1 ]} &  (a+  \varepsilon_{2}, \delta',\chi_{l''}) \\
\end{block}
\end{blockarray}}\quad\text{ if $l\cdot l' = -\delta$}\\
B_{(a, \delta,\chi_l),\underline{\hspace{.5em}}, (a+2\epsilon_{-1}, -\delta), \underline{\hspace{.5em}}} &= \begin{blockarray}{cc}
(a+  \varepsilon_{-1}, \delta')  &\\
\begin{block}{[c]c}
 0 &(a+  \varepsilon_{-1}, -\delta')  \\
\end{block}
\end{blockarray}\\
B_{(a,\delta,\chi_l),\underline{\hspace{.5em}}, (a-\epsilon_1 \pm \epsilon_2,-\delta), \underline{\hspace{.5em}}} &= \scalemath{.85}{\pm \sqrt{2}\frac{1}{\sqrt{[4]!}}\begin{blockarray}{ccc}
(a+  \varepsilon_{-1}, \delta')  &  ( a \pm \varepsilon_2 , \delta', \chi_{l'}) &\\
\begin{block}{[cc]c}
 \frac{[\overline{a}_{\pm 2}+\overline{a}_{1}-1]}{[\overline{a}_{\pm 2}+\overline{a}_{1}]}\sqrt{\frac{\lambda_{a+\epsilon_{-1}+\epsilon_{\pm 2}}}{\lambda_a}} &  \sqrt{\frac{[\overline{a}_{-1}+\overline{a}_{\pm 2}][\overline{a}_{-1}+\overline{a}_{\pm 2}+2]}{[\overline{a}_{-1}+\overline{a}_{\pm 2}+1]^2}}\sqrt{\frac{\lambda_{a+\epsilon_{-1}}\lambda_{a+\epsilon_{\pm 2}}}{\lambda_a^2}} & (a+  \varepsilon_{-1},-\delta') \\
  \sqrt{\frac{[\overline{a}_{-1}+\overline{a}_{\pm 2}][\overline{a}_{-1}+\overline{a}_{\pm 2}+2]}{[\overline{a}_{-1}+\overline{a}_{\pm 2}+1]^2}}\sqrt{\frac{\lambda_{a+\epsilon_{-1}}\lambda_{a+\epsilon_{\pm 2}}}{\lambda_a^2}} &\frac{[\overline{a}_{-1}-\overline{a}_{\pm 2}-1]}{[\overline{a}_{-1}-\overline{a}_{\pm 2}]}\sqrt{\frac{\lambda_{a+\epsilon_{-1}+\epsilon_{\pm 2}}}{\lambda_a}}& ( a \pm \varepsilon_2 , -\delta', \chi_{l'}) \\
\end{block}
\end{blockarray}}\\
\end{align*}
Despite the seemingly complicated form of these cells, we can re-use many of the computations for $\Gamma_{\Lambda_1}^{4,4,*}$ to show the above solution satisfies the equations to be a KW-cell system on $\Gamma_{\Lambda_1}^{4,4,\mathbb{Z}_4*}$. 

\begin{lem}
The above solution for the U and B cells satisfy (BA) and (RI).
\end{lem}
\begin{proof}
We begin by showing that (BA) always holds. Recall this is exactly showing that $B_{X,\underline{\hspace{.5em}}, Y, \underline{\hspace{.5em}}}$ is an eigenmatrix for $U^{X}_{\hspace{1em} Y}$ with eigenvalue $[2]$. For many cases this immediately follows from the computations done in Subsection~\ref{sec:charge}. For example the matrix $B_{(a,\delta,\chi_l),\underline{\hspace{.5em}}, (a-\epsilon_1 \pm \epsilon_2,-\delta), \underline{\hspace{.5em}}}$ is a scalar multiple of the matrix $B_{a,\underline{\hspace{.5em}}, a-\epsilon_1 \pm \epsilon_2, \underline{\hspace{.5em}}}$ from Subsection~\ref{sec:charge}. We also have that $U^{(a,\delta,\chi_{ l })}_{\quad (a-\varepsilon_1 \pm \varepsilon_2 ,-\delta)}$ is exactly equal to the matrix $U^{a}_{\quad a-\varepsilon_1 \pm \varepsilon_2}$.  The result that $B_{(a,\delta,\chi_l),\underline{\hspace{.5em}}, (a-\epsilon_1 \pm \epsilon_2,-\delta), \underline{\hspace{.5em}}}$ is an eigenmatrix with eigenvalue $[2]$ for $U^{(a,\delta,\chi_{ l })}_{\quad (a-\varepsilon_1 \pm \varepsilon_2 ,-\delta)}$ then immediately follows from Lemma~\ref{lem:conjB}.

The other cases follow in the same manner, with a little more work. For example when $\overline{a}_1 = \frac{k+2}{2}$, to show the eigenmatrix condition for $B_{(a,\delta),\underline{\hspace{.5em}}, (a, -\delta), \underline{\hspace{.5em}}}$, we have to show (among other computations) that 
\begin{align*}
    &\frac{1}{\lambda_a}( -\frac{1}{\sqrt{2}}\sqrt{\lambda_{a+\epsilon_{1}}\lambda_{a+\epsilon_{-2}}}\frac{[\overline{a}_{1}+\overline{a}_{-2}]}{[\overline{a}_{1}+\overline{a}_{-2}+1]}\cdot\left(-\frac{1}{\sqrt{2}}\sqrt{\lambda_{a+\varepsilon_{1}}\lambda_{a+\varepsilon_{-2}}}\frac{1}{[\overline{a}_{1} +\overline{a}_{-2}  +1]}\right) \\ +&\frac{1}{2}[\overline{a}_{1}]\left(   [\overline{a}_{2}-1][\overline{a}_{1}-\overline{a}_{2}] + [\overline{a}_{2}+1][\overline{a}_{1}+\overline{a}_{2}]   \right)\cdot\left(\frac{1}{2} w_{a,1}\right)
    +\frac{1}{2}[\overline{a}_{1}]\left(   [\overline{a}_{2}-1][\overline{a}_{1}-\overline{a}_{2}] + [\overline{a}_{2}+1][\overline{a}_{1}+\overline{a}_{2}]   \right)\cdot\left(\frac{1}{2} w_{a,1}\right)\\
    +&\frac{1}{\sqrt{2}}\sqrt{\lambda_{a+\epsilon_{1}}\lambda_{a+\epsilon_{2}}}\frac{[\overline{a}_{1}+\overline{a}_{2}]}{[\overline{a}_{1}+\overline{a}_{2}+1]}\cdot\left(\frac{1}{\sqrt{2}}\sqrt{\lambda_{a+\varepsilon_{1}}\lambda_{a+\varepsilon_{2}}}\frac{1}{[\overline{a}_{1} +\overline{a}_{2}  +1]}\right))  = [2]\cdot\frac{1}{2}[\overline{a}_{1}]\left(   [\overline{a}_{2}-1][\overline{a}_{1}-\overline{a}_{2}] + [\overline{a}_{2}+1][\overline{a}_{1}+\overline{a}_{2}]   \right).
\end{align*}   
To verify this, note that the left hand side simplifies to 
\begin{align*}
    &\frac{1}{2}\frac{1}{\lambda_a}(\sqrt{\lambda_{a+\epsilon_{1}}\lambda_{a+\epsilon_{-2}}}\frac{[\overline{a}_{1}+\overline{a}_{-2}]}{[\overline{a}_{1}+\overline{a}_{-2}+1]}\cdot\left(\sqrt{\lambda_{a+\varepsilon_{1}}\lambda_{a+\varepsilon_{-2}}}\frac{1}{[\overline{a}_{1} +\overline{a}_{-2}  +1]}\right) \\ +&[\overline{a}_{1}]\left(   [\overline{a}_{2}-1][\overline{a}_{1}-\overline{a}_{2}] + [\overline{a}_{2}+1][\overline{a}_{1}+\overline{a}_{2}]   \right)\cdot( w_{a,1})
+\sqrt{\lambda_{a+\epsilon_{1}}\lambda_{a+\epsilon_{2}}}\frac{[\overline{a}_{1}+\overline{a}_{2}]}{[\overline{a}_{1}+\overline{a}_{2}+1]}\cdot\left(\sqrt{\lambda_{a+\varepsilon_{1}}\lambda_{a+\varepsilon_{2}}}\frac{1}{[\overline{a}_{1} +\overline{a}_{2}  +1]}\right)).
\end{align*}
We then have from Lemma~\ref{lem:conjB} that this simplifies to 
\[  [2]\cdot\frac{1}{2}[\overline{a}_{1}]\left(   [\overline{a}_{2}-1][\overline{a}_{1}-\overline{a}_{2}] + [\overline{a}_{2}+1][\overline{a}_{1}+\overline{a}_{2}]  \right)
\]
as desired. The other cases follow in a similar manner.

Finally verifying (RI) is done in an easy, but tedious, case-by-case computation. The details here are routine, so we neglect to write them down.
\end{proof}

We now have a KW-cell system on the graph $\Gamma_{\Lambda_1}^{4,4,\mathbb{Z}_4^*}$. In consequence we have the following.
\begin{thm}
Let $k\in \mathbb{N}_{\geq 1}$ be even. Then there exists a module category $\mathcal{M}$ for $\mathcal{C}(\mathfrak{sl}_4, k)$ such that the fusion graph for action by $\Lambda_1$ is $\Gamma^{4,k,\mathbb{Z}_4^*}_{\Lambda_1}$.
\end{thm}

\section{Subfactors}
One of the main goals of this paper is to explicitly construct subfactors from the interesting module categories of $\overline{\operatorname{Rep}(U_q(\mathfrak{sl}_N))^\omega}$. Assuming $q = e^{2 \pi i \frac{1}{2(N+k)}}$ for some $k\geq 1$, then $\overline{\operatorname{Rep}(U_q(\mathfrak{sl}_N))^\omega}$ is unitary, and a solution to a KW cell system on a graph gives a unitary module category by \cite{JamsHans,Whale}. It then follows from \cite{BigPopa} that we get a subfactor of the hyperfinite type $\textrm{II}_1$ factor for each simple object of the module category.

The first class of subfactors are constructed from $\overline{\operatorname{Rep}(U_q(\mathfrak{sl}_N))^\omega}$ acting on $\mathcal{M}$. We get an irreducible subfactor of the hyperfinite type $\textrm{II}_1$ factor for each simple object $M\in \mathcal{M}$. We label this subfactor $\rho_{M}$.

Note that for this subfactor, the category $\overline{\operatorname{Rep}(U_q(\mathfrak{sl}_N))^\omega}$ is the even part of the subfactor. This consists of certain $N-N$ bimodules. The odd part, which consists of certain $M-M$ bimodules is the dual module category. See \cite{Bisch} for more details.

The principal graph $\Gamma^M$ for $\rho_{M}$ has even vertices the simple objects of $\overline{\operatorname{Rep}(U_q(\mathfrak{sl}_N))^\omega}$, and odd vertices the simple objects of $\mathcal{M}$. The number of edges from $V \in \overline{\operatorname{Rep}(U_q(\mathfrak{sl}_N))^\omega}$ to $M' \in \mathcal{M}$ is given by 
\[   \Gamma^{M}_{V\to M'} =  \dim\Hom_\mathcal{M}(V \otimes M\to M')  .     \]
We can get an explicit formula for these multiplicities from the KW-cell solutions.

\begin{lem}\label{lem:edge}
Let $p_V$ be a projection onto $V$ in the planar algebra $\overline{\mathcal{P}_{
\operatorname{Rep}(U_q(\mathfrak{sl}_N))^\omega,\Lambda_1} }
$, and $KW(p_V)$ the image of this projection in the associated graph planar algebra. Then 
\[\dim\Hom_\mathcal{M}(V \otimes M\to M') =  \operatorname{Tr}\left( KW(p_{V})|_{{M}\to{ M'}}   
 \right)  \]
 where $\operatorname{Tr}\left( KW(p_{V})|_{{M}\to{ M'}}  \right)$ is the trace of the linear operator $KW(p_{V})\in oGPA(\Gamma)$ restricted to loops beginning at $M$ and ending at $M'$.
\end{lem}
\begin{proof}
This is precisely Lemma~\ref{lem:Hans}.
\end{proof}

In practice, one just needs to determine the multiplicities $\dim\Hom_\mathcal{M}(V \otimes M\to M')$ where $V$ runs over the fundamental representations of $\overline{\operatorname{Rep}(U_q(\mathfrak{sl}_N))^\omega}$. From this information, fusion ring arguments can determine the remaining multiplicities. Explicitly, let $G_{\Lambda_i}$ be the matrix for action of $\Lambda_i$ on $\mathcal{M}$, and let $V_\lambda$ be a simple object of $\overline{\operatorname{Rep}(U_q(\mathfrak{sl}_N))^\omega}$ indexed by a partition $\lambda$. Then the matrix for the action by $V_\lambda$ on $\mathcal{M}$ is given by the formula \cite[Equation 3.5]{Mac}:
\[     G_{V_{\lambda}} = \det \left[ G_{\Lambda_{\left(\lambda^{\textrm{tr}}_i - i + j\right)}} \right]_{1\leq i,j \leq l(\lambda^{\textrm{tr}})}      \]

Formulas for the projections onto $p_{\Lambda_i}$ are given in Subsection~\ref{sec:preKW}, and hence the matrices $G_{\Lambda_i}$ can be explicitly computed in all examples using Lemma~\ref{lem:edge}. 

In practice, these subfactors are huge. We can obtain smaller subfactors with the following observation. Consider the graded subcategory $\overline{\operatorname{Rep}(U_q(\mathfrak{sl}_N))^\omega}^{ad} \subset \overline{\operatorname{Rep}(U_q(\mathfrak{sl}_N))^\omega}$. Over this subcategory, the module $\mathcal{M}$ may no longer be simple. Let us write $\bigoplus_i \mathcal{M}_i$ for the simple decomposition of $\mathcal{M}$ as a $\overline{\operatorname{Rep}(U_q(\mathfrak{sl}_N))^\omega}^{ad}$ module. We then get a subfactor of $\mathcal{R}$ for each $i$ and simple object $M\in \mathcal{M}'$. We label this subfactor $\rho_{\mathcal{M}_i, M} $.

The principal graph has even vertices the simple objects of $\overline{\operatorname{Rep}(U_q(\mathfrak{sl}_N))^\omega}^{ad}$, and odd vertices the simple objects of $\mathcal{M}_i$. The number of edges from $V$ to $M'$ is then equal to 
\[\dim\Hom_{\mathcal{M}_i}(V \otimes M\to M') = \dim\Hom_{\mathcal{M}}(V \otimes M\to M').\]
Hence we can determine the principal graph using Lemma~\ref{lem:edge}.

\begin{remark}
We wish to point out that the explicit structure of this second class of subfactors has been worked out for the charge-conjugation modules by Wenzl \cite{HansMod}.
\end{remark}

As an example, we determine the structure of one of the subfactors corresponding to the second exceptional module of $\mathcal{C}(\mathfrak{sl}_4,6)$.

\begin{ex}
Let $\mathcal{M}$ be the module constructed in Theorem~\ref{thm:SU6C}. The object $M$ is chosen as $a_{odd}$. The above construction yields a subfactor with the following principal graph:
\[\raisebox{-.5\height}{ \includegraphics[scale = .4]{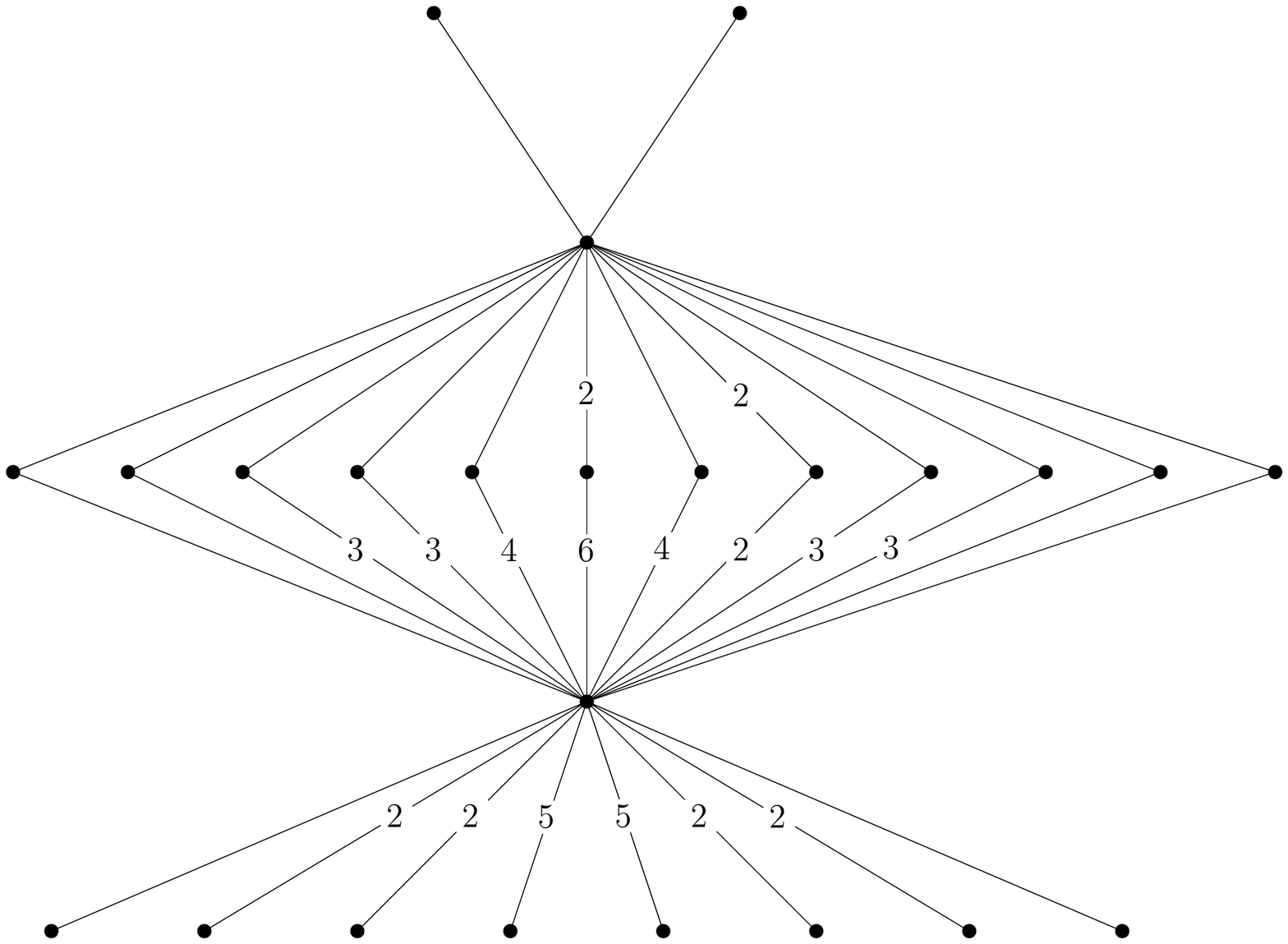}}\]
Note that directly building a flat connection on this graph would be a computational nightmare due to the high multiplicity.
\end{ex}
\bibliography{Cells}
\bibliographystyle{alpha}
\end{document}